\documentclass[11pt]{amsart}
\usepackage[margin=1in]{geometry}
\usepackage{geometry}
\usepackage{pgf,tikz,pgfplots,color}
\pgfplotsset{compat=1.15}
\usepackage{mathrsfs} 
\usetikzlibrary{arrows}
\usepackage{fancyhdr}
\usepackage{lastpage} 

\usepackage[toc,page]{appendix}
\usepackage[utf8]{inputenc}
\usepackage{hyperref}
\hypersetup{
	colorlinks=true,
	linkcolor=blue,
	citecolor=brown,
	filecolor=magenta, 
	urlcolor=blue,
}
\usepackage{amsmath}
\usepackage{amsfonts}
\usepackage{amssymb}
\usepackage{amsthm}
\usepackage{setspace}
\usepackage{inputenc}
\usepackage[english]{babel} 
\usepackage{comment}
\usepackage[shortlabels]{enumitem}
\numberwithin{equation}{section}

\makeatletter
\def\@tocline#1#2#3#4#5#6#7{\relax
  \ifnum #1>\c@tocdepth 
  \else
    \par \addpenalty\@secpenalty\addvspace{#2}%
    \begingroup \hyphenpenalty\@M
    \@ifempty{#4}{%
      \@tempdima\csname r@tocindent\number#1\endcsname\relax
    }{%
      \@tempdima#4\relax
    }%
    \parindent\z@ \leftskip#3\relax \advance\leftskip\@tempdima\relax
    \rightskip\@pnumwidth plus4em \parfillskip-\@pnumwidth
    #5\leavevmode\hskip-\@tempdima
      \ifcase #1
       \or\or \hskip 1em \or \hskip 2em \else \hskip 3em \fi%
      #6\nobreak\relax
    \hfill\hbox to\@pnumwidth{\@tocpagenum{#7}}\par
    \nobreak
    \endgroup
  \fi}
\makeatother

\newcommand{\dist}{\operatorname{dist}}

\newtheorem{thm}{Theorem}[section]
\newtheorem{cor}[thm]{Corollary} 
\newtheorem{prop}[thm]{Proposition}
\newtheorem{lemma}[thm]{Lemma}

\theoremstyle{definition}
\newtheorem{defn}[thm]{Definition}
\newtheorem{exmp}[thm]{Example}
\newtheorem{ques}[thm]{Question}

\theoremstyle{remark}
\newtheorem{rem}[thm]{Remark}
\newtheorem*{clm}{Claim}
\newtheorem*{ack}{Acknowledgment}

\renewcommand{\th}[1]{\begin{thm}\label{#1}}
	\renewcommand{\eth}{\end{thm}}
\newcommand{\co}[1]{\begin{cor}\label{#1}}
	\newcommand{\eco}{\end{cor}}
\newcommand{\pr}[1]{\begin{prop}\label{#1}}
	\newcommand{\epr}{\end{prop}}

\newcommand{\df}[1]{\begin{defn}\label{#1}}
	\newcommand{\edf}{\end{defn}}
\newcommand{\ex}[1]{\begin{exmp}\label{#1}} 
	\newcommand{\eex}{\end{exmp}}
\newcommand{\qu}[1]{\begin{ques}\label{#1}}
	\newcommand{\equ}{\end{ques}}  
\newcommand{\mk}{\begin{rem}}
	\newcommand{\emk}{\end{rem}}
\newcommand{\cl}{\begin{clm}}
	\newcommand{\ecl}{\end{clm}} 
\newcommand{\ac}{\begin{ack}}
	\newcommand{\eac}{\end{ack}} 

\newcommand{\ga}{\begin{gather}}
\newcommand{\ega}{\end{gather}}
\newcommand{\gan}{\begin{gather*}}
\newcommand{\egan}{\end{gather*}}
\newcommand{\al}{\begin{gngn}}
	\newcommand{\eal}{\end{align}}
\newcommand{\aln}{\begin{align*}}
\newcommand{\ealn}{\end{align*}}
\newcommand{\eq}[1]{\begin{equation}\label{#1}}
\newcommand{\eeq}{\end{equation}}

\newcommand{\pa}{\partial{}}

\newcommand{\zbd}{\partial_{\ov{z}}}

\newcommand{\we}{\wedge}

\newcommand{\ra}{\longrightarrow}

\newcommand{\sm}{\setminus}

\newcommand{\DD}[2]{\frac{\partial #1}{\partial #2}}

\newcommand{\R}{\mathbb{R}} 
\newcommand{\C}{\mathbb{C}}
\newcommand{\D}{\mathbb{D}}

\newcommand{\U}{\mathcal{U}} 

\newcommand{\mc}{\mathcal}
\newcommand{\mf}{\mathfrak}

\newcommand{\tit}{\textit}

\newcommand{\ov}{\overline}

\newcommand{\wti}{\widetilde}
\newcommand{\hht}{\widehat}

\newcommand{\dbar}{\overline\partial}

\newcommand{\all}{\alpha}

\newcommand{\del}{\delta}
\newcommand{\Del}{\Delta}
\newcommand{\var}{\varphi}

\newcommand{\ve}{\varepsilon}
\newcommand{\om}{\omega}
\newcommand{\Om}{\Omega}

\newcommand{\la}{\lambda}

\newcommand{\gm}{\gamma}
\newcommand{\Gm}{\Gamma}
\newcommand{\Si}{\Sigma}

\newcommand{\yh}{\frac{1}{2}}

\newcommand{\yf}{\frac{1}{4}}

\newcommand{\re}[1]{(\ref{#1})}

\newcommand{\rl}[1]{Lemma~\ref{#1}}

\newcommand{\rp}[1]{Proposition~\ref{#1}}
\newcommand{\rt}[1]{Theorem~\ref{#1}}
\newcommand{\rd}[1]{Definition~\ref{#1}}

\newcommand{\nn}{\nonumber}
\newcommand{\nid}{\noindent}

\newcounter{pp}
\newcommand{\bpp}{\begin{list}{$\hspace{-1em}\alph{pp})$}{\usecounter{pp}}}
	\newcommand{\epp}{\end{list}}

\newcounter{ppp}
\newcommand{\bppp}{\begin{list}{$\hspace{-1em}(\roman{ppp})$}{\usecounter{ppp}}}
	\newcommand{\eppp}{\end{list}}

\begin{document}
	
	\title[]{Oblique derivative boundary value problems on families of planar domains} 
	\author[]{Ziming Shi}
	
	
	\address{Department of Mathematics,
		Rutgers University - New Brunswick, Piscataway, NJ 08854} 
	\email{zs327@math.rutgers.edu}


	\begin{abstract}
		We consider second-order elliptic equations with oblique derivative boundary conditions, defined on a family of bounded domains in $\mathbb{C}$ that depend smoothly on a real parameter $\lambda \in [0,1]$. We derive sharp regularity properties of the solutions in all variables, including the parameter $\lambda$. More specifically we show that the solution and its derivatives are continuous in all variables, and the H\"older norms of the space variables are bounded uniformly in $\lambda$.  
	\end{abstract}

	\thanks{Supported in part by NSF grant DMS-1500162}  	
	\maketitle
	\tableofcontents

	
	\setcounter{thm}{0}\setcounter{equation}{0}
	
	\section{Introduction} 
	In this paper we study the domain dependence of solutions to the oblique boundary value problem on the plane. We consider the following problem on a family of ``smoothly varying" bounded domains $\Om^{\la} $ in $\C$ that depend on some parameter $\la \in [0,1]$: 
	\eq{odpintro} 
	\mc{L}^{\la} u^{\la} = f^{\la} \quad \text{in $\Om^{\la}$ },  \qquad \frac{du^{\la}}{d\Upsilon^{\la}} = \gm^{\la} \quad \text{on $b \Om^{\la}$. } 
	\eeq
	Here $\mc{L}^{\la}$ is a second-order elliptic operator defined by 
	\eq{Llambda}
	\mc{L}^{\la} u^{\la} = a^{\la} (x,y) u^{\la}_{xx}+ 2b^{\la} (x,y) u^{\la}_{xy} + c^{\la} (x, y) u^{\la}_{yy} + 
	d^{\la}(x,y) u^{\la}_x + e^{\la}(x,y) u^{\la}_y,  
	\eeq
	and $\Upsilon^{\la}$ is an arbitrary non-vanishing directional vector field along $b \Om^{\la}$. We assume $a^{\la} >0$ and the uniform ellipticity condition: 
	\eq{alamba}
	a^{\la} c^{\la} - (b^{\la})^2 \geq \del_0 > 0,\quad\forall \; \lambda\in[0,1]. 
	\eeq
	We show that if all the coefficient functions $a^{\la}, b^{\la}, c^{\la}, d^{\la}, e^{\la}, F^{\la}, \gm^{\la}$ depend on the parameter $\la$ ``smoothly". Then with certain suitable conditions to guarantee the existence and uniqueness of the problem \re{odpintro}, the family of solutions $u^{\la}$ will also depend on the parameter ``smoothly". 
	
	We now make the concept of domain dependence precise. 
	Let $\mathbb N$ be the set of non-negative integers.  Let $\Om$ be a bounded domain in $\mathbb C$, and let $k, j\in \mathbb N$ satisfy $k \geq j$. Given a family of functions $u^\la$ defined on $\Om$,   
	we say that $u^{\la}$ is in the space $\mc{C}^{k+\mu,j}(\ov{\Om})$ if for every integer $i$ with $0 \leq i \leq j$, $[\la \mapsto \pa_{\la}^{i} u^{\la}]$ is a continuous map from $[0,1]$ to $C^{k-i} (\ov{\Om})$ and a bounded map from $[0,1]$ to $C^{k-i+\mu} (\ov{\Om})$. Our main result is 
	\begin{thm} \label{mtintro} 
	  Let $k, j$ be in $\mathbb N\cup\{\infty\}$ with $k \geq j$, and let $0 < \mu < 1$.
		Let $\Om$ be a bounded domain in $\C$ with $\mc{C}^{k+2+\mu}$ boundary. Let $\Gm^{\la}$ be a family of maps that embed $\Om$ onto $\Om^{\la}$, with $\Gm^{\la} \in \mc{C}^{k+2+\mu,j}(\ov{\Om})$ and $\Gamma^0$ being the identity map. 
		Suppose that $a^{\la}, b^{\la}, c^{\la}, d^{\la}, e^{\la}, f^{\la}$ are functions in $\Om^{\la}$, such that $a^{\la} \circ \Gm^{\la}, b^{\la} \circ \Gm^{\la}, c^{\la} \circ \Gm^{\la} \in \mc{C}^{k+1+\mu,j} (\ov{\Om})$, $d^{\la} \circ \Gm^{\la}, e^{\la} \circ \Gm^{\la}, f^{\la} \circ \Gm^{\la} \in \mc{C}^{k + \mu,j} (\ov{\Om})$, and suppose that $\Upsilon^{\la}, \gm^{\la}$ are functions on $b \Om^{\la}$ such that $\Upsilon^{\la} \circ \Gm^{\la}, \gm^{\la} \circ \Gm^{\la} \in \mc{C}^{k+1+\mu,j} (b \Om)$. 
		Suppose certain well-posedness conditions (see \rp{c1ut}) are given so that \re{odpintro} with \re{alamba}  has a unique solution for each $\la$. Then $u^{\la} \circ \Gm^{\la} \in \mc{C}^{k+2+\mu,j} (\ov{\Om})$.  
	\end{thm}
	
	The existence and uniqueness theory for fixed $\la$ are classical results and we shall review them in Section 3. 
	
	In Bertrand-Gong \cite{B-G14} the analogous results were proved for the Dirichlet and Neumann problems for the Laplacian equation. As a corollary, they proved a parameter version of the Riemann mapping theorem, namely, for any family of simply connected domains $\Om^{\la}$ given by the embeddings $\Gm^{\la}: \ov{\Om} \to \ov{\Om^{\la}}$ with $\Gm^{\la} \in \mc{C}^{k+1+\mu,j}(\ov{\Om})$, there exists a family of Riemann mappings $R^{\la}: \Om^{\la} \to \D$ such that $R^{\la} \circ \Gm^{\la} \in \mc{C}^{k+1+\mu,j} (\ov{\Om})$, where $\D$ is the unit disk. 
	We shall use this result in our proof of \rt{mtintro} when the domains are simply-connected. 
	
	For fixed $\la$, we apply the classical theory of Vekua to solve equation \re{odpintro}. 
	The idea is to first find an isothermal coordinates in which the problem takes the form  
	\eq{lapeqnintro}
	\Del U + a(x,y)  U_x + b(x,y)U_y = f(x,y), \quad 	\all U_{x} - \beta U_{y} =\gm. 
	\eeq
	By setting $w = u+iv$ for $u =  U_x$ and $v = - U_y$, \re{lapeqnintro} takes the form
	\begin{gather} \label{pAeqnintro}
	\pa_{\ov z} w + Aw + B \ov{w} = F, \quad \text{in $\Om$,}  
	\\ \label{pAbcintro} 
	Re[\ov{l} w] = \gm,  \quad \text{on $b \Om$,}
	\end{gather} 
	where we assume $l = \all + i \beta$ is nowhere vanishing on $b \Om$. Problem \re{pAeqnintro}-\re{pAbcintro} is called a \emph{generalized Riemann-Hilbert problem}. (The classical Riemann-Hilbert problem is the special case when $A,B,F \equiv 0$.) 
	Vekua tackled this problem by introducing the theory of generalized analytic functions. His idea is to find an integral representation formula for a solution $w$ satisfying equation \re{pAeqnintro}. The boundary condition \re{pAbcintro} then translates into a singular integral equation on curves for which there is well-known theory. 
	
	For our problem on domains with parameter, we will first reduce the problem to the $\dbar$ form \re{pAeqnintro}-\re{pAbcintro}. If the domains are simply-connected, we can pull back the problem onto the fixed unit disk $\D$ by using the parameter version of the Riemann mapping theorem proved in \cite{B-G14}. The solution is shown to satisfy a Fredholm integral equation on $\D$, and the smoothness in parameter follows from a standard compactness argument. 
	
	The problem becomes more difficult for families of multiply-connected domains. In this case one may also pull back the equations onto some fixed domain. However, now the pull back maps are no longer biholomorphic, so this procedure will destroy the $\dbar$ form \re{pAeqnintro}. Thus we are forced to deal with variable domains in this case. We use Vekua's idea to reduce the problem to a singular integral equation on $b \Om^{\la}$, and then reduce it into a Fredholm equation in order to use compactness argument.  
	
We mention that one could certainly obtain an estimate of the form 
		\[
		\| w^{\la}  \|_{C^{k+1+\mu} (\ov{\Om^{\la}})} \leq C_{\la} \| f^{\la} \|_{C^{k+\mu} (\ov{\Om^{\la}})},
		\]
		for each fixed $\la$, by solving the problem on each fixed domain. However this estimate is not useful since in general we do not have any information as to how the the constant $C_{\la}$ depends on $\la$.  
	
	Following Vekua, we will henceforth call \re{pAeqnintro}-\re{pAbcintro} Problem A for $\dbar$ or simply Problem A. 
	There is a vast literature on the oblique derivative boundary value problem. See for example Miranda \cite[Chapter 3]{MA70} and  H\"ormander \cite[Chapter 10]{Hor63}.
	
	\medskip
	The paper is organized as follows: 
	
	In Section 2, we introduce the notations used in our paper. We review the definition of H\"older spaces with parameter as were used in \cite{N-W63} and \cite{B-G14}, and we show some basic properties of these spaces. In particular, we prove a result (\rp{cinhd}) which will be used repeatedly throughout the paper to simplify our arguments. 
	Section 3 contains an exposition of Vekua's theory on the solvability and regularity of Problem A without parameter. 
	The solvability properties are determined by the index of the problem (see \rd{pA_index}) and the number of connected components of the boundary $b \Om$.  
	All the results in this section can be found in \cite{VA62}. 
	In Section 4, we prove \rt{mtintro} for families of simply-connected domains $\Om^{\la}$. 
	In Section 5 we examine how the kernel in Vekua's integral solution formula depends on $ \la$. These results will be used in Section 6 to prove \rt{mtintro} for families of multiply-connected domains $\Om^{\la}$. For our proof we are able to avoid taking derivatives on the kernel as was done in \cite{B-G14}. 
	
	In the end we attach an appendix, where we provide either complete proofs or precise references to the various claims used in earlier proofs. For the reader's convenience we also include a few of Vekua's results that we used frequently in the paper. 
	For the constants appearing in our estimate, we use the notation $C_{1,0}$ (resp. $C_{k+\mu, j}$) to indicate that it depends on  $\| \Gm^{\la} \|_{\mc{C}^{1,0}(\Om)}$ (resp. $\| \Gm^{\la} \|_{\mc{C}^{k+\mu,j} (\Om )}$). For a family of functions $u^{\la}$ defined on $\Om^{\la}$, we shall write either $u^{\la} \in \mc{C}^{k+\mu, j} (\ov{\Om^{\la}} )$ or $u^{\la} \circ \Gm^{\la} \in \mc{C}^{k+\mu, j}(\ov \Om)$. We shall write $|f|_{k+\mu}$ instead of $|f|_{C^{k+\mu}(\ov \Om)}$ whenever the domain $\Om$ is clear from the context. 
	
	\ac
	The author would like to express his deep gratitude to Professor Xianghong Gong and Professor Xiaojun Huang for providing innumerable inspiration, encouragement and guidance through many years, without which this project would never be completed. The author would also like to thank the anonymous referee for reading the paper carefully and for many valuable suggestions. 
	\eac
	
	\section{H\"{o}lder spaces with parameter} 
	
	In this section we introduce H\"{o}lder spaces with parameter as in \cite{B-G14}, which can be used to study boundary value problems on families of domains. 
	
	We denote by $C^{r}(\ov{\Om})$, $r \geq 0$ the space of  complex-valued functions which are H\"older continuous with exponent $r$ on a bounded domain $\Om \subseteq \C$. 
	When the domain $\Om$ in question is clear, we also write $| f |_{r}$ in place of $| f |_{C^{r}(\ov{\Om})}$.  
	The space $C^{r-}(\ov \Om)$ consists of those functions that are in $C^{r-\ve}(\ov \Om)$, for all $\ve>0$. 
	We write $\D$ for the unit disk in $\C$, and we write $d A(z) = \frac{i}{2} dz \we d \ov{z}$ for the area element in $\C$.  
	The Cauchy-Green or $\dbar$ solution operator on a domain $\Om \subseteq \C$ is defined by 
	\begin{equation} \label{Topt}
	T_{\Om}f(z) := - \frac{1}{\pi} \iint_{\Om}  \frac{f(\zeta)}{\zeta -z} \, dA(\zeta), \quad z \in \Om. 
	\end{equation} 
	For $1 \leq p \leq \infty$ and $\nu \in \R$, we denote by $L_{p, \nu}(\C)$ the space of functions $f$ satisfying
	\eq{Lpqdef} 
	f  \in L^{p} (\D), \quad \quad  f_{\nu} (z) := |z|^{-\nu} f \left( \frac{1}{z} \right)  \in L^{p}(\D),  \quad 
	\eeq
	with the norm $| f |_{L_{p,\nu}(\C)} = | f |_{L^{p} (\D)} + | f_{\nu} |_{L^{p} (\D)} $. 
	The space $L_{p,\nu} (\C)$ is a Banach space with this norm. Intuitively the lower index $\nu$ describes the rate of growth of $f$ at $\infty$, and we have $L_{p, \nu_{1}}(\C) \subset L_{p, \nu_2}(\C)$, if $\nu_1 > \nu_2$.  It is easy to see that $f_{\nu} \in L^{p} (\D)$ if and only if 
	\[ 
	\int_{\C \sm \D} | u |^{\nu p - 4} | f(u) |^{p} \, dA(u) < \infty. 
	\]		
	For $0 < \mu < 1$, we let $L_\mu^p (\Om)$ denote the set of functions $f$ that satisfy $f \in L^{p} (\Om)$, $f \equiv 0$ outside $\Om$, and 
	\[
	B(f, \Om, \mu, p) = \sup_{h \in \C} \frac{  \left( \iint_{\Om} \left| f(z + h) - f(z) \right|^{p}   \, d A(z) \right) ^{\frac{1}{p}}}{|h|^\mu} < \infty. 
	\]
	Then $L_\mu^{p} (\Om)$ is a Banach space with the norm:
	\eq{Lpall}  
	|f|_{L_\mu^{p} (\Om)} = |f|_{L^{p}(\Om)} + B(f, \Om, \mu, p). 
	\eeq
	We say that $w$ is a solution to the equation 
	\begin{equation} \label{Sec2weq} 
	\pa_{\ov{z}} w +  Aw + B\ov{w} = F, \quad \text{in $\Om$}, 
	\end{equation}
	if $w, \pa_{\ov{z}} w \in L^1 (\Om)$, and the above equation holds in the sense of distributions in $\Om$. We denote by $\mc{L}_{p}(A,B,F, \Om)$ (resp. $\mc{L}_{p,2}(A,B,F)$) the set of $w \in L^1(\Om)$ that satisfy \re{Sec2weq} with $A,B,F \in L^{p}(\Om)$ (resp. $L_{p,2} (\C )$)  If $F \equiv 0$, we write $\mc{L}_{p}(A,B, \Om)$ (resp. $\mc{L}_{p,2}(A,B)$) for the spaces defined as above. 
	
	We adopt the following notations for differences of functions: 
	\eq{del_note}
	\del_{z,\zeta} (f (\cdot, \la)) := f(z, \la) - f(\zeta, \la), \quad \quad \del^{\la_{1}, \la_{2}} f (z,  \cdot) := f(z, \la_{1}) - f(z, \la_{2}).
	\eeq
	We denote the difference quotient with respect to the parameter by
	\eq{diff_quot_not}   
	D^{\la, \la_{0}} f (\cdot) := \frac{f(\cdot, \la) - f(\cdot, \la_{0})}{\la - \la_{0}} . 
	\eeq
	
	To study boundary value problems on families of domains, we use the H\"older space with parameter $\mc{C}^{k+\mu,j}(\ov{\Om_{\Gm}})$ from \cite{B-G14}(our space is the same as the $\mc{B}$ space in \cite{B-G14}.) First we define the spaces for a fixed domain: 
	\df{fdspacedef}
 Let $k, j$ be in $\mathbb N\cup\{\infty\}$ with $k \geq j$, and let $0 < \mu < 1$. Let $\Om$ be a bounded domain in $\C$ with $C^{k+\mu}$ boundary, and let $\{ u^{\la} \}_{\la \in [0,1]}$ be a family of functions on $\Om$. We say that $\{ u^{\la} \}_{\la \in [0,1]}$ is in the class $\mc{C}^{k+\mu,j}_{\ast} (\ov{\Om})$, if for all $ 0 \leq i \leq j$, $|\pa_{\la}^{i}u^{\la}|_{k+\mu}$ are uniformly bounded in $\la$, and $\pa_{\la}^{i}u^{\la} $ is continuous in $\la$ in the $C^k(\ov{\Om})$ norm. The norm on $\mc{C}^{k+\mu,j}_{\ast} (\ov{\Om})$ is defined to be
	\eq{fdnorm} 
	|u^\la|_{k+\mu,j} := \sup_{{0 \leq i \leq j, \la \in [0,1]}}  \{ |\pa_{\la}^{i}u^{\la}|_{k+\mu} \}. 
	\eeq 
	Similarly we define $\mc{C}^{k+\mu,j}_{\ast} (b \Om)$ by replacing $\Om$ with $b \Om$ in the above expressions. 
	\edf
	
	For our problem with parameter, we consider functions defined on a family of domains $\Om^{\la}$, for $\la \in [0,1]$. We assume there is an embedding $\Gm^{\la}: \Om \to \Om^{\la}$, where $\Gm^0$ is the identity map from $\Om$ to itself.

	Now we make the following observation. Suppose for each $\la \in [0,1]$, $u^{\la}$ is defined on an open set $V$ with $\cup_{\la} \Om^{\la} \subset \subset V$. Let $\Gm^{\la} \in \mc{C}^{k+\mu,j}_{\ast}( \ov{\Om} )$, and $u^{\la} \in \mc{C}^{k+\mu,j}_{\ast}(V)$. But then  $\wti{u}^{\la} (z) := u^{\la}(\Gm^{\la}(z))$ is not in the space $\mc{C}^{k+\mu,j}_{\ast}(\ov{\Om} )$, since when one takes derivative in $\lambda$, by the chain rule one needs to take the space derivative of $u^{\la}$. 
	Hence instead of $\mc{C}^{k+\mu,j}_{\ast}(\ov{\Om})$, we shall adopt the following definition in \cite{B-G14}  for a family of functions defined on a family of domains:
	\begin{defn}
	 Let $k, j,\mu$ be as in \rt{mtintro}. Let $\{\Gm^{\la}: \ov{\Om} \to \ov{\Om^{\la}}\}_{\la \in [0,1]}$ be in the class $ \mc{ C}^{k+\mu,j}_{\ast} (\ov{\Om})$. Let $u^{\la}$ be a family of functions defined on $\Om^{\la}$, and put $\wti{u} (z, \la) := u^{\la}(\Gm^{\la}(z))$. We say that 
		\[ 
		u^{\la} \in \mc{C}_{\ast}^{k+\mu,j}(\ov{\Om_{\Gm}}) \   \text{ if $\wti{u} (z, \la) \in \mc{C}_{\ast}^{k+\mu,j}(\ov{\Om})$},\quad 
		u^{\la} \in \mc{C}_{\ast}^{k+\mu,j}(b \Om_{\Gm})\   \text{if $\wti{u} (z, \la) \in C_{\ast}^{k+\mu,j}(b \Om)$}. 
		\]
		For $0 \leq j \leq k$, define the spaces 
		\[ \mc{C}^{k+1+\mu,j}(\ov{\Om}_{\Gm}) := \bigcap_{i=0}^{j} \mc{C}_{\ast}^{k-i+\mu,i}(\ov{\Om}_{\Gm}), \quad  
		\mc{C}^{k+1+\mu,j}(b \Om_{\Gm}) := \bigcap_{i=0}^{j} \mc{C}_{\ast}^{k-i+\mu,i}(b \Om _{\Gm}).  \]
		The space $\mc{C}^{k+1+\mu,j}(\ov{\Om}_{\Gm})$ is a Banach space with the norm defined by 
		$
		\| u^\la \|_{k+\mu,j} :=
		\max_{0 \leq i \leq j}$ $\{ | \wti{u} (z, \la)|_{k-i+\mu,i} \}  
		$. In other words, we say that $u^
		\la \in \mc{C}^{k+1+\mu,j}(\ov{\Om}_{\Gm})$ if $|\pa_{\la}^{i} \wti{u} |_{k-i+\mu}$ is uniformly bounded in $\la$, and $\pa_{\la}^{i}\wti{u}$ is continuous in the $C^{k-i} (\ov{\Om})$ norm, for each $i$ with $0 \leq i \leq j$.  
		Since $| \pa_{\la}^{i'} u |_{k-i + \mu} \leq | \pa_{\la}^{i'} u |_{k-i'+\mu}$, for any $i'$, $0 \leq i' \leq i$, we have
		\[ 
		\| u^\la \|_{k+\mu,j} = \max_{0 \leq i \leq j} \{ |\pa_{\la}^{i} \wti{u} (z, \la)|_{k-i+\mu} \}, \quad \wti u (z,\la) := u^\la(\Gm^\la(z)). 
		\]
	\end{defn}
	
	We denote the restriction of $\Gm^{\la}$ to $b \Om$ by $\rho^{\la}(s)$, where $s$ is the arclength element on $b \Om$. Then $dt^{\la} = (\rho^{\la})'(s) ds$.   
	\pr{cinhd}
	Let $k$ be a non-negative integer and $0 < \mu < 1$. Let $\Om$ be a bounded domain in $\R^{n}$.  \\
	(i) For any $0< \mu' < \mu$, $C^{k+\mu'} (\ov{\Om})$ is relatively compact in $C^{k+\mu} (\ov{\Om})$. 
	\\
	(ii) Let $\{f^{\la} \}$ be a family of functions in the class $C^{k+\mu}(\ov{\Om})$ such that $|f^{\la}|_{k+\mu}$ is uniformly bounded in $\la$, and $ f^{\la}$ is continuous in $\la$ in the $| \cdot |_{0}$ -norm. Then $f^{\la}$ is continuous in $\la$ in the $| \cdot |_{k+\mu'}$-norm, for any $0 < \mu' < \mu$. 
	\epr
	\begin{proof} 
		(i) Taking $\ve < \frac{\mu - \mu'}{\mu}$ (so that $ \mu(1- \ve) - \mu' >0$), we have
		\begin{align*}
		\| g_n \|_{k+\mu'} &= \sup_{x,y \in \Om} \frac{|D^k g_{n}(x) - D^k g_{n}(y)|}{|x-y|^{\mu'}} 
		\\ &\leq  \sup_{x,y \in \Om} \frac{| D^k g_{n}(x) - D^k g_{n}(y)|^{1-\ve}}{|x-y|^{\mu'}} \sup_{x,y \in \Om} | D^k g_{n}(x) - D^k g_{n}(y)|^{\ve} \\
		&\leq C \left( \| g_n \|_{k+\mu, 0} \right) \sup_{x,y \in \Om} |x-y|^{\mu(1-\ve) - \mu'}  
		| g_{n}|_k^{\ve} 
		\leq C \left( \| g_n \|_{k+\mu, 0} \right) |g_{n}|_k^{\ve}, 
		\end{align*}
		where the constant $C$ depends on $\Om$, $\mu, \mu', \ve$. By the Ascoli-Arzela theorem, there exists a subsequence $\{g_{n_j}\}$ which converges uniformly to some function $g_\ast$ in the $|\cdot|_k$ norm. Using the above estimates, the sequence $\{g_{n_j}\}$ is Cauchy in the $| \cdot|_{k+\mu'}$ norm, hence must converge to some $C^{k+\mu'}(\ov \Om)$ function, which is just $g_\ast$. 
		
		\medskip
		\noindent
		(ii) Seeking contradiction, suppose there exists some $\la_0$ and a sequence $\la_{j} \to \la_0$ such that $|f^{\la_{j}} - f^{\la_0} |_{k + \mu'} \geq \del >0$. By part (i) and passing to a subsequence if necessary, $f^{\la_{j}}$ converges to some function $\hht{f}$ in the $| \cdot|_{k + \mu'}$ -norm, and thus 
		\eq{k+all}
		|\hht{f} - f^{\la_0}|_{k+ \mu'} \geq \del > 0.   
		\eeq
		On the other hand, we have
		$
		| \hht{f} - f^{\la_0} |_0 \leq | \hht{f} - f^{\la_j} |_0  + | f^{\la_j} - f^{\la_0} |_0. 
		$
		The right-hand side converges to $0$ by assumption, so $\hht{f} \equiv f^{\la_0}$, which contradicts \re{k+all}. Hence $f^{\la}$ is continuous in $\la$ in the $| \cdot |_{k+\mu'}$ norm.  
	\end{proof}	  
	The following result shows that the definition $\mc{C}^{k+\mu,j}  (\ov{\Om}_{\Gm})$ is independent of the embedding $\Gm^{\la}$. For this reason we shall from now on write $\mc{C}^{k+\mu,j}  (\ov{\Om^{\la}})$ instead of $\mc{C}^{k+\mu,j}  (\ov{\Om}_{\Gm})$.      
	\begin{prop}[{\cite[Lemma 2.2]{B-G14}}] \label{embindle}
	Let $k, j,\mu$ be as in \rt{mtintro}.
	Let $\ov{\Om}$ be a bounded domain with $b \Om \in C^{k+\mu,j} \cap C^{1}$. Let $\Gm^{\la}_1$ and $\Gm^{\la}_2$ be two embeddings from $\ov{\Om}$ to $\ov{\Om^{\la}}$. 
	If $\Gm_i^{\la} \in  \mc{C}^{k+ \mu,j} (\ov{\Om }) \cap \mc{C}^{1,0} (\ov{\Om})$. Then
	$\mc{C}^{k+\mu,j} (\ov{\Om}_{\Gm_1}) = \mc{C}^{k+\mu,j} (\ov{\Om}_{\Gm_2}) $. 
	\end{prop} 
	
	\begin{lemma} \label{iftp}
	Let $k, j,\mu$ be as in \rt{mtintro}. Let $\Om, \Om' \subset \R^{2}$, and $F: \ov{\Om} \times [0,1] \to \ov{\Om'} \times [0,1] $ be a $\mc{C}^{k+1+\mu, j}$-mapping, defined by
	\[ 
	F(x,y,\la) = (f_{1}(x,y,\la), f_{2}(x,y,\la), \la) := (\xi, \eta, \tau). 
	\]
	Suppose that $ f(\cdot, \cdot, \la) =(f_{1}(\cdot, \cdot, \la),f_{2}(\cdot, \cdot, \la))$ are bijections from $ \ov{\Om}$ to $\ov{\Om'}$ for each fixed $\la$, and the Jacobian matrix $DF$ has non-vanishing determinant in $\ov{\Om} \times [0,1] $. Then $F^{-1} \in \mc{C}^{k+1+\mu,j} (\ov{\Om'} \times [0,1] )$.
	\end{lemma} 
	The proof can be done by induction on $k,j$. The details are left to the reader. 
	
	\begin{prop}[{\cite[Corollary 9.4]{B-G14}}]  \label{Rmp}
	Let $k, j,\mu$ be as in \rt{mtintro}. Let $ \Om $ be a simply-connected domain in $\C$ with $b \Om \in C^{k+1+\mu}$, and let $\Gm^{\la}: \ov{\Om} \to \ov{\Om^{\la}}$ be in the class $\mc{C}^{k+1+\mu,j}(\ov{\Om})$. Then there exists a family of Riemann mappings $R^{\la}: \Om^{\la} \to \D$ such that $\wti{R} (z, \la) := R^{\la} \circ \Gm^{\la} (z) \in \mc{C}^{k+1+\mu,j}(\ov{\Om}).$
	\end{prop}

	\pr{Ri}
	Keep the assumptions in the last theorem. Let $R^{\la}$ be the family of Riemann mappings in \rp{Rmp}. For each $\la \in [0,1]$, let $(R^{\la})^{-1}: \D \to \Om^{\la}$ be the inverse of $R^{\la}$. Then $(R^{\la})^{-1}  \in \mc{C}^{k+1+\mu,j}(\ov{\D} )$. 
	\epr
	\begin{proof}
		Define $F \in \mc{C}^{k+1+\mu,j}(\ov{\Om} \times [0,1] ): \ov{\Om} \times [0,1] \to \ov{\D} \times [0,1]$ by 
		\[ 
		F(z,\la) = (\wti{R} (z, \la), \la) := (R(\Gm(z, \la), \la), \la). 
		\]
		Since $\Gamma(\cdot, \la)$ is an embedding from $\Om$ to $\Om^{\la}$, and $R(\cdot, \la)$ is a Riemann mapping, the map $\wti{R}$ is a bijection for each fixed $\la$, and its Jacobian is invertible. It follows that the Jacobian of $F$ is invertible (when $j\geq 0$.) By Lemma \re{iftp}, $F$ is invertible and the inverse $F^{-1}$ is in $\mc{C}^{k+1+\mu,j}(\ov{\D} \times [0,1 ] )$. $F^{-1}$ is given by
		\[ F^{-1}(w,\tau) = ( \wti{R}^{-1} (w, \tau), \tau). \]
		Write $ R^{-1} (z, \la) = \Gm( \wti{R}^{-1} (z, \la), \la)$. 
		Since $\Gamma \in \mc{C}^{k+1+\mu,j}(\ov{\Om})$, and $\wti{R}^{-1} \in \mc{C}^{k+1+\mu, j}(\ov{\D})$, we obtain $R^{-1} \in \mc{C}^{k+1+\mu,j}(\ov{\D})$.
	\end{proof}
	
	In order to reduce our oblique derivative boundary value problem to Problem A for $\dbar$, we need to change coordinates to reduce the second-order terms in $\mc{L} u$ in the elliptic operator to Laplacian. This is the classical problem of finding isothermal coordinate on the plane. We state some of the results which are well-known.  
	\pr{isocord} 
	Let $\Om$ be a bounded domain in $\C$, and consider the following elliptic equation defined on $\Om$: 
	\eq{soee}
	a(x,y) u_{xx} + 2 b (x,y) u_{xy} + c (x,y) u_{yy} + F \left( x,y,u, u_x, u_y \right) = 0,
	\eeq  
	where the derivatives are taken in the sense of distributions. We assume the following uniform ellipticity condition:  
	\[
	d = a c - b^{2} \geq d_{0} > 0 \quad \text{a.e. in $\Om$. } 
	\]  
	(i) Suppose $ \Om$ has Lipschitz boundary, and $a, b, c \in W^{1,p} (\Om)$, for $2< p < \infty$. Then there exists a 1-1 map $\psi: z = x+iy \mapsto w = \xi + i \eta$ on $\Om$ such that $\psi \in W^{2,p} ( \Om) \subset C^{1+\all_p} (\ov{\Om})$, for $\all_p := \frac{p-2}{p}$. Moreover, if we set $u(x,y) = \wti{u} (\psi(x,y)) $, then $u$ is a solution to equation \re{soee} if and only if $\wti{u}$ is a solution to the following equation
    \[ 
	\wti u_{\xi \xi} + \wti{u}_{\eta \eta} + F_{1} \left( \xi, \eta, \wti{u} , \wti u_\xi, \wti u_\eta \right) = 0 
	\]  
	on the domain $\psi (\Om) $. 
	\\
	(ii) For $k \geq 0$ and $0 < \mu < 1$, let $\Om$ be a bounded domain in $\C$ with $C^{k+1+\mu}$ boundary. Suppose $a, b, c \in C^{k+1+\mu}(\ov{\Om})$. Then the coordinate map $\psi$ constructed in $(i)$ belongs to the class $C^{k+2 + \mu}(\ov{\Om})$. 
	\\
	(iii)  The Jacobian of the map $\psi$ is positive at each point of $\Om$: 
	\[
	J(z) = \left| \pa_z w  \right|^2  - \left| \pa_{\ov z} w \right|^2 > 0, \quad z \in \Om.
	\] 
	\epr
	The reader can refer to \cite{VA62} Theorems 2.5, 2.9 and 2.12 for the proofs of above statements. 
	
	\begin{lemma} \label{indle}
	Let $\Om$ be a bounded domain in $\C$ with $C^1$ boundary, and let $\psi $ be the map from \rp{isocord}. In particular $ \psi \in C^1 (\ov{\Om})$. Then $\psi$ defines a one-to-one onto homeomorphism from $\Om$ to $\Om'$. Let $I$ be a non-vanishing vector field along $b \Om$, and $\psi_{\ast} (I)$ be the push-forward of $I$ by $\psi$, so that $\psi_{\ast}$ is a non-vanishing vector field along $b \Om'$. Then
	\[
	[ I ]_{b \Om} = [ \psi_{\ast} I ]_{b \Om' },  
	\]
	where $[I]_{b \Om}$ denotes the change in the argument of $I(z)$ as $z$ moves along $b \Om$ once in the positive direction.  
	\end{lemma} 
	
	The proof is left to the reader. 
	We now prove a result on the parameter version of isothermal coordinates:  
	\pr{isop}
	Let $k, j,\mu$ be as in \rt{mtintro}.
		Let $ \Om$ be a bounded domain in $\C$ with $C^{k+2+\mu}$ boundary, and $\Gm^{\la}:\Om \to \Om^{\la}$ be a family of embeddings in the class $\mc{C}^{k+2+\mu,j} (\ov{\Om})$, with $\Gm^0$ being the identity map. 
	Keep the assumptions in \rt{mtintro}.  Consider the second-order elliptic equation with oblique derivative boundary condition: 
	\begin{equation} \label{isoodpeqn}   
	\begin{gathered}  
	a^{\la} (x,y) u^{\la}_{xx}+ 2b^{\la} (x,y) u^{\la}_{xy} + c^{\la} (x, y) u^{\la}_{yy} + 
	d^{\la}(x,y) u^{\la}_x + e^{\la}(x,y) u^{\la}_y = f^{\la} (x,y),  \text{in $\Om^{\la}$.} \\ 
	\all^{\la} u_x^{\la} + \beta^{\la} u_y^{\la} = \gm^{\la}, \quad \text{on $b \Om^{\la}$.} 
	\end{gathered} 
	\end{equation}
	Suppose $a^{\la}, b^{\la}, c^{\la} \in \mc{C}^{k+1+\mu,j} (\ov{\Om^{\la}})$, $d^{\la}, e^{\la}, f^{\la} \in \mc{C}^{k+\mu,j} (\ov{\Om^{\la}})$, $\all^{\la}, \beta^{\la}, \gm^{\la} \in \mc{C}^{k+1+\mu,j} (\ov{\Om^{\la}})$. Then for all $| \la| < \ve_0$ for some small $\ve_0$, there exists a family of coordinate maps $\var^{\la}: \Om^{\la} \to D^{\la}$ with $\var^{\la} \circ \Gm^{\la} \in \mc{C}^{k+1+\mu,j}(\ov{\Om})$, and the functions  $\mc{U}^{\la} (\tau^{\la}) := u^{\la}( (\var^{\la})^{-1} (\tau^{\la}) )$ ($\tau^{\la} = s^{\la} + it^{\la} $) satisfies: 
	\begin{gather*} 
	\Del \, \mc{U}^{\la} + p^{\la} (s^{\la}, t^{\la}) \, \mc{U}^{\la}_{s^{\la}} + q^{\la} (s^{\la}, t^{\la}) \, \mc{U}^{\la}_{t^{\la}} = h^{\la} ( s^{\la}, t^{\la}), \quad \text{in $D^{\la}$, } \\   
	\nu_1^{\la} \, \mc{U}^{\la}_{s} + \nu_2^{\la} \, \mc{U}^{\la}_{t} =  g^{\la} (s^{\la}, t^{\la}), \quad \text{on $bD^{\la}$, } 
	\end{gather*}
	where $p^{\la}, q^{\la}, h^{\la} \in \mc{C}^{k+\mu,j} (\ov{D^{\la}})$, $\nu_1^{\la}, \nu_2^{\la}, g^{\la} \in \mc{C}^{k+1+\mu,j} (\ov{D^{\la}})$.  
	Moreover, the winding number of the vector field $\U^{\la} = (\all^{\la},  \beta^{\la} )$ along $b \Om^{\la}$ is the same as the winding number of $\U_1^{\la} = (\nu_1^{ \la }, \nu_2^{\la} )$ along $b D^{\la}$. 
	
	\epr   
	\begin{proof}
		First we pull back \re{isoodpeqn} by the maps $\Gm^{\la}$ to the fixed domain $\Om$. Let $v^{\la} (\xi, \eta) := u^{\la} \circ \Gm^{\la} (\xi, \eta)$. Then $v^{\la}$ satisfies  
		\begin{gather} \label{odppbeqn}     
		\wti{a}^{\la} (\xi, \eta) v^{\la}_{\xi \xi}+ 2\wti{b}^{\la} (\xi, \eta)  v^{\la}_{\xi \eta} + \wti{c}^{\la} (\xi, \eta)  v^{\la}_{\eta \eta} + \wti{d}^{\la} (\xi, \eta)  v^{\la}_{\xi} + \wti{e}^{\la}(\xi, \eta) v^{\la}_{\eta} = \wti{f}^{\la} (\xi, \eta) , \; \text{in $\Om$.} \\ \label{odppbbc} 
		\wti{\all}^{\la} v_{\xi}^{\la} + \wti{\beta}^{\la} v_{\eta}^{\la} = \wti{\gm}^{\la}, \quad \text{on $b \Om$.} 
		\end{gather} 
		The second-order coefficients $\wti{a}^{\la}, \wti{b}^{\la}, \wti{c}^{\la}$ are given by 
		\begin{equation} \label{abc_ti} 
		\begin{gathered}  
		\wti{a}^{\la} (\xi,\eta)= (a^{\la} \circ \Gm^{\la}) (\xi^{\la}_{x})^2 + 2 (b^{\la} \circ \Gm^{\la} ) \xi^{\la}_x \xi^{\la}_y + (c^{\la} \circ \Gm^{\la}) (\xi^{\la}_{y})^2,  \\ 
		\wti{b}^{\la} (\xi, \eta) = (a^{\la} \circ \Gm^{\la}) \xi^{\la}_x \eta^{\la}_x + (b^{\la} \circ \Gm^{\la}) \left( \xi^{\la}_x \eta^{\la}_y + \xi^{\la}_y \eta^{\la}_x \right) + (c^{\la} \circ \Gm^{\la} )(\xi^{\la}_y \eta^{\la}_y), 
		\\ 
		\wti{c}^{\la} (\xi,\eta) = (a^{\la} \circ \Gm^{\la}) (\eta^{\la}_{x})^2 + 2 (b^{\la} \circ \Gm^{\la} ) \eta^{\la}_x \eta^{\la}_y + (c^{\la} \circ \Gm^{\la}) (\eta^{\la}_{y})^2. 
		\end{gathered} 
		\end{equation}
		Hence $\wti{a}^{\la}, \wti{b}^{\la}, \wti{c}^{\la} \in \mc{C}^{k+1+\mu,j}(\ov{\Om})$. The first-order coefficients $\wti{d}^{\la}, \wti{e}^{\la}$ are linear combinations of the  products of $a^{\la}, b^{\la}, c^{\la}, d^{\la}, e^{\la}$ and the first and second derivatives of $(\Gm^{\la})^{-1}$, and hence $\wti{d}^{\la}, \wti{e}^{\la} \in \mc{C}^{k+\mu,j} (\ov{\Om})$. Since $\wti{f}^{\la} = f^{\la} \circ \Gm^{\la}$, we have $\wti{f}^{\la} \in \mc{C}^{k+\mu,j}(\ov{\Om})$. 
		
		For the boundary coefficients, we have  
		\[
		\begin{pmatrix}
		\wti{\all}^{\la} (\xi, \eta) \\ \wti{\beta}^{\la} (\xi, \eta) 
		\end{pmatrix}
		= \begin{pmatrix}
		\DD{ ( \Gm^{\la})^{-1}  } {x} (\Gm^{\la} (\xi, \eta) ) & \DD{ (\Gm^{\la})^{-1} }{y} (\Gm^{\la} (\xi, \eta) ) \\
		\DD{ (\Gm^{\la})^{-1} }{x} (\Gm^{\la} (\xi, \eta) )& \DD{  (\Gm^{\la})^{-1} }{y} (\Gm^{\la} (\xi, \eta) )
		\end{pmatrix}  
		\begin{pmatrix} 
		\all^{\la}  (\Gm^{\la} (\xi, \eta) ) \\ \beta^{\la} (\Gm^{\la} (\xi, \eta) ) 
		\end{pmatrix},  
		\]
		or put it differently, $\wti{\U}^{\la} = (\Gm^{\la})^{-1}_{\ast}  (\U^{\la})$.  
		It follows that $\wti{\all}^{\la}, \wti{\beta}^{\la} \in \mc{C}^{k+1+\mu,j}(\ov{\Om})$, and $\wti{\gm}^{\la} = \gm^{\la} \circ  \Gm^{\la} \in \mc{C}^{k+ 1+ \mu,j} (\ov{\Om})$. 
		Since $\Gm^{0}$ is the identity map, the above matrix is close to the identity matrix (in matrix norm) when $\la$ is close to $0$. Consequently we observe that for $| \la| $ small, the winding number of the vector field $\wti{\U}^{\la}(\wti{\all}^{\la}, \wti{\beta}^{\la} )$ along $b \Om$ is the same as that of $(\all^{\la}, \beta^{\la} )$ along $b \Om^{\la}$: 
		$
		\frac{1}{2 \pi} [ \wti{\U} ^{\la} ]_{b \Om} = \frac{1}{2 \pi} [ \U^{\la} ]_{b  \Om^{\la}} = \chi_0.     
		$
		Since $\Gm^0$ is the identity map, by \re{abc_ti} we have $\wti{a}^{0} = a^0$, $\wti{b}^{0} = b^0$ and $\wti{c}^{0} = c^0 $. Therefore for $|\la| < \ve$, we have $\wti{d}^{\la} := \wti{a}^{\la} \wti{c}^{\la} - (\wti{b}^{\la})^2 \geq \ve_0 > 0$, in other words equation \re{odppbeqn} is uniformly elliptic for $\la$ close to $0$.  
		Applying \rp{isocord} for $\la=0$, we get a map $\Psi_0: \Om=\Om^0 \to D^0$, such that $ \Psi_0 \in C^{k+2+\mu}(\ov{\Om_0})$. Writing $\Psi_0 (\xi, \eta) = (\xi', \eta')$, one can transform equations \re{odppbeqn} and \re{odppbbc} to 
		\begin{gather} \label{mcVeqn}  
		\Del \mc{V}^{\la} + a_{0}^{\la} \mc{V}_{\xi' \xi '}^{\la} + 2b_{0}^{\la} \mc{V}_{\xi' \eta '}^{\la} + c_{0}^{\la} \mc{V}_{\eta' \eta '}^{\la} + d_0^{\la} \mc{V}^{\la}_{\xi'} +  e_0^{\la} \mc{V}^{\la}_{\eta'} = f_0^{\la}, \quad  \text{in $D^{0}$; } \\ \label{mcVbc} 
		\all_0^{\la} \mc{V}^{\la}_{\xi'} + \beta_0^{\la} \mc{V}^{\la}_{\xi'} = \gm_0^{\la}, \quad \text{on $bD^0$.}  
		\end{gather} 
		Here  $\mc{V} (\xi', \eta') = v \circ \Psi_0^{-1} (\xi', \eta')$, $a_0^{0} = b_0^{0} = c_0^{0} = 0$, and $|a_0^{\la}|, |b_0^{\la}|, |c_0^{\la}| < \ve$ if $| \la| < \del_0$ for some small positive $\del_0$. 
		By expressions similar to \re{abc_ti}, we see that the second-order coefficients $a_0^{\la}, b_0^{\la}, c_0^{\la}$ are linear combinations of the products of $\wti{a}^{\la} \circ \Psi^{-1}_0, \wti{b}^{\la}\circ \Psi^{-1}_0, \wti{c}^{\la}\circ \Psi^{-1}_0$ and the first space derivatives of $\Psi_0$, and thus they are in the class $\mc{C}^{k+1+\mu, j} (\ov{D_0})$. The first-order coefficients $d_0^{\la}, e_0^{\la}$ are linear combinations of products of $\wti{a}^{\la} \circ \Psi^{-1}_0, \wti{b}^{\la} \circ \Psi^{-1}_0, \wti{c}^{\la} \circ \Psi^{-1}_0, \wti{d}^{\la} \circ \Psi^{-1}_0, \wti{e}^{\la} \circ \Psi^{-1}_0$ and the first and second derivatives of $\Psi_0$, so $d_0^{\la}, e_0^{\la} \in \mc{C}^{k+\mu,j} (\ov{D_0})$. 
		In a similar way $f^{\la}_0 \in \mc{C}^{k+\mu,j} (\ov{D^0 })$ and $\all_0^{\la}, \beta_0^{\la}, \gm_0^{\la} \in \mc{C}^{k+1+\mu, j} (\ov{D^0})$.  
		
		Let $\U_0^{\la} = (\all^{\la}_0, \beta^{\la}_0 )$. Then $\U_0^{\la} = ( \Psi_0)_{\ast} \wti{\U }^{\la}$.   
		By \rl{indle}, the winding number of $\U_0^{\la} = (\all^{\la}_0, \beta^{\la}_0 )$ is the same as the winding number of $\wti{\U}^{\la} $: $\frac{1}{2 \pi} [\U_0^{\la} ]_{ b D^0} = \frac{1}{2 \pi} [\wti{\U}^{\la} ]_{b \Om} = \chi_0 $.  
		Rewrite equation \re{mcVeqn} as 
		\eq{mcVeqn'}
		a_{1}^{\la} \mc{V}_{\xi' \xi '}^{\la} + 2b_{0}^{\la} \mc{V}_{\xi' \eta '}^{\la} + c_{1}^{\la} \mc{V}_{\eta' \eta '}^{\la} + d_0^{\la} \mc{V}^{\la}_{\xi'} +  e_0^{\la} \mc{V}^{\la}_{\eta'} = f_0^{\la}, \quad  \text{in $D^{0}$,} 
		\eeq
		where $a_1^{\la} = 1+ a_0^{\la}$ and $c_1^{\la} = 1+ c_0^{\la}$.  For $| \la| < \del_1$, the determinant for the above equation satisfies
	    $
		d_0^{\la} = a_1^{\la} c_1^{\la} - 2b_0^{\la} \geq \ve_1 >0. 
		$
		By above we have that $a_1^{\la}, b_0^{\la}, c_1^{\la}, d^{\la}_0 \in \mc{C}^{k+1+\mu,j} (D^0)$.  
		Let $\psi^{\la}: \zeta' = \xi ' + i \eta' \mapsto \tau^{\la} = s^{\la} + i t^{\la}$ be a family of coordinate systems that satisfies:  
		\[
		\tau^{\la}_{\ov{\zeta'}} + q^{\la} (\zeta') \tau^{\la}_{\zeta'}  = 0, \quad 
		q^{\la} = \frac{a_1^{\la} - \sqrt{d_0^{\la}} - ib_0^{\la} }{ a_1^{\la}  +\sqrt{d_0^{\la}} + ib_0^{\la} }. 
		\]
		Note that for $\la=0$, the map $\psi^0$ is the identity map, $q^{0} \equiv 0$, and the above equation is trivially satisfied. 
		
		By \rl{indle}, $\psi^{\la} \in \mc{C}^{k+2+\mu,j}(\ov{D^0})$. Let $\psi^{\la}: D^0 \to D^{\la}$. 
		Then  \re{mcVeqn'} and \re{mcVbc} are transformed into
		\begin{gather*} 
		\Del \, \mc{U}^{\la} + p^{\la} (s^{\la}, t^{\la}) \, \mc{U}^{\la}_{s^{\la}} + q^{\la} (s^{\la}, t^{\la}) \, \mc{U}^{\la}_{t^{\la}} = h^{\la} ( s^{\la}, t^{\la}), \quad \text{in $D^{\la}$;} \\   
		\nu_1^{\la} \, \mc{U}^{\la}_{s} + \nu_2^{\la} \, \mc{U}^{\la}_{t} =  g^{\la} (s^{\la}, t^{\la}), \quad \text{on $bD^{\la}$, } 
		\end{gather*}  
		where $h^{\la} (s^{\la}, t^{\la}) = \frac{J^{\la} f^{\la}_0 }{4 \sqrt{d_0^{\la}}} $, and
		\begin{gather*}
		p^{\la} = \frac{J^{\la}}{4 \sqrt{d_0^{\la}}} \left( 1 + a_1^{\la} s^{\la}_{\xi' \xi'}  + 2 b_0^{\la} s^{\la}_{\xi' \eta'} + c_1^{\la} s^{\la}_{\eta' \eta'} + d_0^{\la} s^{\la}_{\xi'} +e^{\la}_0 s^{\la}_{\eta'} \right), 
		\\  
		q^{\la} = \frac{J^{\la}}{4 \sqrt{d_0^{\la}}} \left( 1 + a_1^{\la} t^{\la}_{\xi' \xi'}  + 2 b_0^{\la} t^{\la}_{\xi' \eta'} + c_1^{\la} t^{\la}_{\eta' \eta'} + d_0^{\la} t^{\la}_{\xi'} +e^{\la}_0 t^{\la}_{\eta'} \right). 
		\end{gather*} 
		Hence $p^{\la}, q^{\la} \in \mc{C}^{k+\mu,j} ( \ov{D^{\la} })$. Since $f^{\la}_0 \in \mc{C}^{k+\mu,j} (\ov{D^0})$, we have $h^{\la} \in \mc{C}^{k+\mu,j} (\ov{D^{\la}})$. Similarly it is easy to see $\nu_1^{\la}, \nu_2^{\la}, g^{\la} \in \mc{C}^{k+1+\mu,j} (\ov{D^{\la}})$. 
		Let $\U_1^{ \la} = (\nu_1^{\la}, \nu_2^{\la} )$. Then $\U_1^{\la} = \psi^{\la}_{\ast} (\U_0^{\la}) $. Since $\psi^{0}_{\ast}$ is the identity matrix, we see that the winding number of $\U_1^{\la}$ is the same as that of $ \U_0^{\la}$ for $|\la| $ small: 
		$
		\frac{1}{2 \pi} [\U^{\la} ]_{ b \Om^{\la} } = \frac{1}{2 \pi} [ \U_1^{\la} ]_{b D^{\la} } = \chi_0.
		$
		
		We now put together the series of transformation from earlier 
		\[
		\var^{\la} = \psi^{\la} \circ \Psi^0 \circ (\Gm^{\la})^{-1}: \Om^{\la} \ra D^{\la}. 
		\]  
		Since $ \Gm^{\la} \in \mc{C}^{k+2+\mu,j} (\ov{\Om})$, $\Psi^{0} \in \mc{C}^{k+2+\mu, j} (\ov{\Om_0})$, and $\psi^{\la} \in \mc{C}^{k+2+\mu,j}(\ov{D_0})$, we have $\var^{\la} \circ \Gm^{\la}: \Om \to D^{\la} \in \mc{C}^{k+2+\mu,j} (\ov{\Om})$. 
		\hfill \qed 
		\section{A review of Vekua's theory on fixed domains} 
		In this section we review Vekua's theory for solving oblique derivative boundary value problem. First we introduce Problem A (as in \cite{VA62}) on a fixed domain. 
		\\
		\noindent
		\textbf{Problem A} (Generalized Riemann-Hilbert problem). 
		Let $\Om$ be a bounded domain in $\C$. Given functions $A,B,F$ defined in $\Om$ and $\all,\beta,\gm$ defined on $b\Om$. Find a solution $w(z) = u(z) + iv(z)$ of the equation 
		\begin{equation}  \label{probA} 
		\begin{cases}  
		\mc{L}_{A,B} (w) :\equiv \pa_{\ov{z}} w  + A(z) w(z)  + B(z) \ov{w(z)} = F(z) \quad \text{in $\Om$;}  \\
		\all u + \beta v  \equiv  Re[\ov{l(z)} w] = \gamma \quad \text{on $b \Om$, }
		\end{cases} 
		\end{equation}
		where $ l:= \all + i \beta$ and is assumed to have unit length. 
		\begin{defn}   \label{pA_index} 
			The \emph{index of Problem A}, denoted by $n$, is the winding number of $l$ along $b \Om$: 
			\[ 
			n :=  \frac{1}{2 \pi} [l]_{b\Om}. 
			\] 
			Here we denote by $[l]_{b\Om}$ the change in the argument of $l(t)$ as $t$ loops around the boundary $b \Om$ once counterclockwise. 
		\end{defn}
		For now let us assume that 
		\begin{gather}
		\label{ABF}
		b \Om \in C^{1+\mu}, \quad A, B, F \in L^{p}(\Om), \quad 2 < p < \infty; 
		\quad 
		l, \gm \in C^{\mu}(\ov{\Om}), \quad   0<\mu<1.
		\end{gather}
		A priori, we seek a solution $w$ in the class $C^0 (\ov{\Om})$.  
		
		First let us consider the case when $\Om$ is simply-connected. By the Riemann-mapping theorem it suffices to assume $\Om = \D$. In this case $n = \frac{1}{2 \pi}  [l ]_{S^1}$, where $S^1$ is the boundary circle. 
		We now reduce the  problem to a simpler form. 
		Set $q(t) = - \arg( l(t)) + n \, \arg(t) $. Then $q$ is a well-defined function on $S^1$ and $q \in   C^{\mu}(S^1)$. We have
		\begin{align*}
		\ov{l (z)} &= e^{-i \arg(l (z))} 
		= e^{p(z) + i(- \arg (l (z) ) + n \arg z)} e^{-p(z)} e^{-i ( n\arg(z) ) } 
		= e^{\chi(z)} e^{-p(z)} z^{-n}. 
		\end{align*}
		Here $\chi $ is the holomorphic function in $\D$ taking the value $p + iq$ on $S^1$. It is defined by the Schwarz integral:
		\eq{chi}
		\chi(z) := \frac{1}{2 \pi } \int_{S^1} q (t) \frac{t+z}{t-z} \frac{dt}{t}, \quad \quad  z \in \D. 
		\eeq
		By \rp{Cint}, $\chi \in C^{\mu}  (\ov{\Om})$. 
		Using the substitutions
		\eq{pArbc1}
		\ov{l(z)} =z^{-n}e^{\chi(z)} e^{-p(z)}, \quad w_{\ast}(z) =e^{\chi(z)} w(z), 
		\eeq
		$w_{\ast}$ is then a solution to the problem
		\eq{pAbcrf} 
		\begin{cases}
			\pa_{\ov{z}} w_{\ast} + A_{\ast} w_{\ast} + B_{\ast} \ov{w_{\ast}} = F_{\ast},  \quad \text{in $\D$; }
			\\	
			Re[ z^{-n} w_{\ast} ] = \gm_{\ast}, \quad \text{on $S^1$,}
		\end{cases}
		\eeq
		where
		\eq{pAbcrfcf}
		A_{\ast} = A, \quad B_{\ast} = B e^{2i Im \, \chi(z)},  \quad F_{\ast} = Fe^{\chi(z)}, \quad \gm_{\ast}(z) = \gm(z) e^{p(z)} \quad \text{in $\D$.}
		\eeq
		Conversely, given any solution $w_{\ast}$ to the problem \re{pAbcrf}, a solution to the original problem \re{probA} is given by 
		$w(z) = e^{-\chi(z)} w_{\ast} (z)$. From now on we can assume the problem takes the form \re{pAbcrf}, and we write $w, A, B, F, \gm$ without the $\ast$. 
		
		By assumption $\pa_{\ov{z}} w \in L^{1}(\ov{\Om})$ and \rp{dbarie}, $w$ satisfies the integral equation
		\eq{gCG}
		w + T_{\D} (Aw + B \ov{w}) = \Phi + T_{\D}F, \quad \text{in $\D$,}
		\eeq
		where $\Phi$ is some function holomorphic in $\D$. 
		Now in view of \rp{dbar_opt_dist}, any continuous solution $w$ of the above equation satisfies 
		$
		\mc L_{A,B} (w) = F(z).  
		$
		We would like to choose $\Phi$ in such a way that the solution to equation \re{gCG} also satisfies the boundary condition $Re[ z^{-n} w ] = \gm$ on $S^1$, where $n$ is the index of \re{probA}. Then such solution will solve the original Problem A.
		
		First suppose $n\geq 0$. We write $\Phi$ in the form
		\begin{align} \label{Phi0}
		\Phi(z) &= \Phi_{0}(z) + \frac{z^{2n+1}}{\pi} \iint_{\D} \frac{ \ov{A(\zeta) w(\zeta)} + \ov{B(\zeta)} w(\zeta)}{1-\ov{\zeta}z} \, d A(\zeta)   - \frac{z^{2n+1}}{\pi} \iint_{\D } \frac{\ov{F(\zeta)}}{1 - \ov{\zeta}z} \, d A(\zeta) , 
		\end{align}
		where $\Phi_{0}$ is some function to be determined, and it is holomorphic in $\D$. Substituting \re{Phi0} into equation \re{gCG}, we obtain the following integral equation for $w$: 
		\eq{pAie1}
		w(z) + P_{n}(Aw + B \ov{w}) = \Phi_{0}(z) + P_{n}F, 
		\eeq
		where $P_{n}$ is the operator defined by:
		\eq{Pnf1}
		P_{n}\var (z):= - \frac{1}{\pi} \iint_{\D} \left( \frac{\var(\zeta)}{\zeta -z} +  \frac{z^{2n+1}\ov{\var(\zeta)}}{1- \ov{\zeta}z } \right) \, d A(\zeta).  
		\eeq
		The integral equation \re{pAie1} is of Fredholm type, which we now describe. 
		Denoting by $\D^{c}$ the exterior of the unit disk, we can write $P_{n} \var$ in the form
		\eq{Pnf2}
		P_{n} \var (z) = T_{\D} \var (z) + z^{2n+1}T_{\D^{c}}\var_{1} (z), \quad \var_{1}(\zeta) = \zeta^{-1} (\ov{\zeta})^{-2} \ov{\var(\ov \zeta^{-1})}. 
		\eeq
		If $\var \in L^{p}(\D), 2 < p \leq \infty$, $\var_1$ satisfies 
		\[
		\int_{\D} \left| \var_1 \left( \zeta^{-1} \right) \right|^{p} |\zeta|^{-2p} dA(\zeta) 
		= \int_{\D} \left| \var \left( \zeta \right) \right|^{p} |\zeta|^{3p-2p} dA(\zeta) 
		< \infty. 
		\]
		According to \rp{dbarubd}, $T_{\D^{c}} \var_1$ is $\all_p$-H\"older continuous in the entire plane, for $\all_p = \frac{p-2}{p}$. By \rp{dbarot} the function $T_{\D} \var$ is $\all_p$-H\"older continuous in the plane. Hence from \re{Pnf2} $P_n \var$ is $\all_p$-H\"older continuous. It is easy to check that
		$Re[z^{-n} P_{n} \var] = 0 $ on $S^1$. 
		By assumption, $F \in L^{p} (\D)$ for $2 < p < \infty$, so by letting $F$ be the $\var$ above, 
		we see that $P_n F$ is H\"older continuous in the entire plane, and $Re[z^{-n} P_n F (z) ] = 0$ on $S^1$. As we will see below, the solution $w$ is H\"older continuous in $\ov{\D}$, and thus $Aw + B \ov{w} \in L^{p} (\D)$, for $p >2$ and $Re[z^{-n} P_n (Aw + B \ov{w}) ] = 0$ on $S^1$.  
		
		Now by \re{pAie1}, if the holomorphic function $\Phi_{0}$ satisfies the boundary condition
		\eq{bcPhi0}
		Re[z^{-n}\Phi_{0}(z)] = \gm \quad \quad \text{on $S^1$,} 
		\eeq
		then the solution of the integral equation \re{pAie1} is a solution of the boundary value problem \re{pAbcrf}.   
		
		By \rp{RHpind}, the general solution of the problem \re{bcPhi0} is given by 
		\[\Phi_{0}(z) = \frac{z^{n}}{2 \pi i} \int_{S^1} \gm(t) \frac{t+z}{t-z} \frac{dt}{t} + \sum_{k=0}^{2n} c_{k}z^{k}, \]
		where the $c_{k}$'s are complex numbers satisfying the relations
		\eq{ckr}
		c_{2n-k} = - \ov{c_{k}}, \quad k=0,1, \dots, n.
		\eeq
		Thus for $n \geq 0$, the problem is reduced to the   equivalent integral equation of Fredholm type:
		\eq{pAie2}
		w + Q_{n} w = P_{n}F + \frac{z^{n}}{2 \pi i} \int_{S^1} \gm(t) \frac{t+z}{t-z} \frac{dt}{t} + \sum_{k=0}^{2n} c_{k}z^{k}, 
		\eeq 
		(see~\cite[p.  225, (1.10)]{VA62}), where operator $Q_{n}$ and $c_k$ are defined by
		\begin{gather}\label{Qn}
		Q_{n} w := P_{n}(Aw + B \ov{w}), \\
		\label{ckzkRe}
		\sum_{k=0}^{2n} c_{k}z^{k} = \sum_{k=0}^{n-1} a_{k} (z^{k} - z^{2n-k}) + i b_{k} (z^{k} + z^{2n-k}) + ic_{0}  z^{n}, 
		\end{gather}
		with $a_k, b_k, c_0$ being arbitrary real constants. 
		As shown earlier, $P_{n} F \in C^{\all_p} (\ov{\D})$ for $\all_p = \frac{p-2}{p}$. The second term on the right-hand side of equation \re{pAie2} can be written as $z^{n} \mc{S} \gm$, where $\mc{S}$ is the Schwarz integral operator.  
		By \rp{Cint} applied to $\gm \in C^{\mu} (S^1)$, one has $\mc{S} \gm \in C^{\mu}(\ov{\D})$. Hence the right-hand side of the equation lies in the class $C^{\nu} (\ov{\D})$, for $\nu:= \min(\all_p, \mu)$. 
		
		It is shown in \cite[p.~296]{VA62} that equation \re{pAie2} has a unique solution $w$ in the class $L^{q}(\D)$, $ \frac{p}{p-1} < q < \infty$, for any right-hand side function in the class $L^q(\D)$.   
		By Theorem~1.25 in \cite[p.~50]{VA62}, one has $Q_n w \in C^{0}(\ov{\D})$, and so it follows from equation \re{pAie2} that $w \in C^{0}(\ov{\D})$. 
		By \rp{dbarot} and \rp{dbarubd} applied to expression \re{Pnf2} where we replace $\var$ by $w$, we have $Q_n w \in C^{\all_p}(\ov{\D})$. By equation \re{pAie2} again we obtain that $w \in C^{\nu} (\ov{\D})$, $\nu:= \min(\all_p, \mu)$. 	
		
		Therefore when the index $n \geq 0$, the solution to the non-homogeneous Problem A exists for any boundary data, and the homogeneous Problem~A ($F \equiv 0$, $\gm \equiv 0$) admits $2n+1$ $\R$-linearly independent solutions. 
		
		If $n<0$. Referring to \cite[p.~302]{VA62}, the solution to Problem A (see \re{probA}) exists if and only if the following conditions are satisfied:
		\eq{scnnsc}
		\frac{1}{2 i} \int_{b \Om} l(t) w'_{i} (t) \gm(t) \, dt - 
		Re \iint_{\Om} w'_{i} (z) F (z) \, dx \, dy = 0, 
		\eeq
		for $i=1, \dots, 2k-1$, $k = -n$, and $(w'_{1}), \dots,(w'_{2k-1})$ are linearly-independent solutions of the homogeneous adjoint problem $A'$:
		\begin{equation} \label{ahpAeqn} 
		\begin{gathered} 
		\zbd w - A w - \ov{B w} = 0, \quad \text{in $\Om$; } 
		\quad 
		Re[l(z) z'(s) w(z)] = 0, \quad \text{ on $b \Om$. }  
		\end{gathered} 
		\end{equation} 
		Here we write $z'(s) = \frac{dz(s)}{ds}$. In this case if the solution for Problem A \re{probA} exists, then it is unique, and the solution can be expressed as the unique solution to the  integral equation (see \cite[p. 300]{VA62}): 
		\eq{pAnnie}
		w + (Q_{k}^{\ast}) w = P_{k}^{\ast} F + \frac{1}{\pi i} \int_{b \Om} \frac{\gm(t) \, dt}{t^{k}(t-z)}, \quad k=-n, 
		\eeq
		where $P_{k}^{\ast}$ and $Q_{k}^{\ast}$ are defined to be
		\begin{gather} \label{Past} 
		P_{k}^{\ast}f := -\frac{1}{\pi} \iint_{\Om} \left( \frac{f(\zeta)}{\zeta -z} + \frac{\ov{\zeta}^{2k-1} \ov{f(\zeta)}}{1 - z \ov{\zeta}} \right) \, d A(\zeta),  
		\quad 
		(Q_{k}^{\ast}) f := P_{k}^{\ast}(B \ov{f}). 
		\end{gather}

		The above method of reducing the problem to a Fredholm integral equation on the domain no longer applies for multiply-connected domains. The geometry of the circle allows for the vanishing of the term $Re[z^{-n}P_{n} \var]$ on $S^1$. To deal with the general case, Vekua introduced the theory of generalized analytic functions. We state here a few of his results without much elaborations. The reader can refer to \cite{VA62} for details. 
		
		We say a function $w$ belongs to the class $\mc{L}_{p,2}(A,B, \C)$ if $w$ satisfies
		$ \mc L_{A,B}(w) = 0$ 
		in the sense of distributions, for $A, B \in \mc{L}_{p,2}(\C)$. (The latter is defined in \re{Lpqdef}.) 
		In our application, $A,B$ are defined only on $\Om$ in which case we extend $A,B$ trivially to $\C$ by setting them to be $0$ on $\C\setminus\overline\Om$.

		Following Vekua, we call such $w$ a \emph{generalized analytic function} of the class $\mc{L}_{p,2}(A,B, \C)$. 
		For fixed $t \in \C$, and  $A,B \in  L_{p,2}(\C)$, consider the following integral equations (\cite[p~167]{VA62}):
		\eq{X1inteq}
		X_i(z,t) - \frac{1}{\pi} \iint_{\C} \frac{A(\zeta)X_i(\zeta,t) + B(\zeta)\ov{X_i(\zeta,t)}}{\zeta - z} \, d A(\zeta) = g_i(z), \quad i = 1,2, 
		\eeq
		where $g_1= \yh(t-z)^{-1} , g_2=  \frac{1}{2i} (t-z)^{-1} $. 
		It is proved in \cite[p.~156]{VA62} that equation \re{X1inteq} admits a unique solution in the class $L_{q,0}(\C)$ if $A,B \in L_{p,2} (\C)$ and the right-hand side lies in the class $L_{q,0} (\C)$, for $q \geq \frac{p}{p-1} = p'$. 
		Since $2 < p < \infty$, we have $1 < p' < 2$, and $\frac{1}{t-z} \in L_{q,0} (\C)$, for any $ p' \leq q < 2$. Equation \re{X1inteq} admits solution $X_i \in L_{q,0}(\C)$ for any $q \in [ p', 2 ) $. For $z \neq t$, $X_{i}$ satisfies  
		\[ 
		\pa_{\ov z} X_{i} (z,t)+ A(z) X_{i}(z,t) + B(z) \ov{X_{i}(z,t)} = 0, \quad i = 1,2. 
		\] 
		Furthermore, we have the following representation for $X_{i}$: 
		\eq{Xjrep}
		X_{1}(z,t) = \frac{e^{\om_{1}(z,t)}}{2(t-z)}, \quad \quad X_{2}(z,t) = \frac{e^{\om_{2}(z,t)}}{2i(t-z)},
		\eeq
	where 
		\eq{omi} 
		\om_{i}(z,t) = \frac{t-z}{\pi} \iint_{\C} \frac{1}{(\zeta-z)(t - \zeta)} \left[ A(\zeta) + B(\zeta) \frac{\ov{X_{i}(\zeta,t)}}{X_{i}(\zeta,t)}\right] \, dA (\zeta) ; 
		\eeq 
		\eq{omiHne}
		|\om_{i}(z_{1},t) - \om_{i}(z_{2}, t)| \leq C_{p} |z_{1} - z_{2}|^{\all_p}, \quad \all_p = \frac{p-2}{p};
		\eeq
		\eq{omde}
		|\om_{i} (z,t) | \leq C_{p} |z-t|^{\all_p}, \quad i= 1,2. 
		\eeq 
		Hence by \re{Xjrep}, $X_j (\cdot ,t)$ is 
		H\"older continuous everywhere in the plane except at $z = t$.

		Define $G_{i}$, the \emph{fundamental kernels of the class} $\mc{L}_{p,2}(A,B, \C)$ as follows:
		\begin{gather} \label{Gidef} 
		G_{1}(z,\zeta) = X_{1}(z,\zeta) + iX_{2}(z,\zeta), \quad G_{2}(z,\zeta) = X_{1}(z,\zeta) - iX_{2}(z,\zeta). 
		\end{gather} 
		By \re{Xjrep} we have
		\begin{gather} \label{Gi}
		G_{1}(z,\zeta) = \frac{e^{\om_1 (z, \zeta)} + e^{\om_2 (z, \zeta)} }{2 (\zeta - z)}, \qquad 
		G_{2}(z,\zeta) =  \frac{e^{\om_1 (z, \zeta)} - e^{\om_2 (z, \zeta)} }{2 (\zeta - z)}. 
		\end{gather} 
		
		The $G_{i}'s$ satisfy the relation: 
		\eq{Gieq}
			\pa_{\ov z}  G_i(z, \zeta) + A(z) G_i (z, \zeta) + B(z) \ov{G_i (z, \zeta)} = 0, \quad i=1,2, 
		\eeq 
		for any $z, \zeta \in \C$, $z \neq \zeta$. In addition we have the estimates
		\eq{fke}
		G_{1}(z, \zeta) = \frac{1}{\zeta -z} + O(|z - \zeta|^{-\frac{2}{p}}), \quad  G_{2}(z,\zeta) = O(|z - \zeta|^{-\frac{2}{p}}). 
		\eeq
		Hence $G_1$ behaves like the Cauchy kernel, modulo a mild singularity.
		
		Denote by $G_{1}'(z,\zeta)$ and $G_{2}'(z, \zeta)$ the adjoint fundamental kernel of the adjoint equation 
		$\mc{L}'(w') = \pa_{\ov z}  w' - A w' - \ov{B} \ov{w}' = F' $. 
		The following relations hold: (\cite[p.174]{VA62})
		\eq{fkr}
		G_{1}(z, \zeta) = - G_{1}' (\zeta,z), \quad \quad G_{2}(z, \zeta) = - \ov{G_{2}' (\zeta,z)}. 
		\eeq
		We will denote by $G_{1}(z,t, \Om)$ and $G_{2}(z,t,\Om)$ the fundamental kernels of the class $\mc{L}_{p,2}(A,B, \C)$, if $A \equiv B \equiv 0$ in $\Om^{c} = \C \setminus \ov{\Om}$. Given $A,B \in L^{p} (\Om)$, we let $\wti{A}, \wti{B}$ be the trivial extensions which vanish outside $\Om$. 
		
		For 
		the parameter problem, we need the following representation formula.
		\begin{prop}{\cite[p. 176]{VA62}} \label{nhgaerfe}
		Suppose $w$ is continuous in $\ov{\Om}$ and satisfies the equation $\mc L_{A,B}(w) = F$ in $\Om$,
		where $A,B,F \in L^{p}(\Om), p>2$. Then $w$ has the representation formula:
		\begin{align} \label{nhgaerf}
		w(z) &= \frac{1}{2 \pi i} \int_{b \Om} G_{1} (z,\zeta) w(\zeta) \, d\zeta 
		- G_{2}(z, \zeta) \ov{w(\zeta)} \, \ov{d\zeta}
		\\ \nonumber &\quad - \frac{1}{\pi} \iint_{\Om} G_{1}(z,\zeta) F(\zeta) + G_{2}(z,\zeta) \ov{F(\zeta)} \, d A(\zeta),  \quad 
		z \in \Om. 
		\end{align}
		\end{prop} 
		
		\mk{}
		When $A \equiv  B \equiv 0$, by formula \re{omi} we have $\om_1 (z,t) \equiv \om_2(z,t) \equiv 0$. Hence from \re{Gi} we get $G_1 = \frac{1}{\zeta -z}$ and $G_2 \equiv 0$, and \re{nhgaerf} reduces to the Cauchy-Green formula.  
		\emk
		Next we state some existence and uniqueness results for Problem A. For homogeneous boundary value problems, we call a finite set of linearly independent solutions which spans the solution space a \emph{complete system}. 
		\begin{prop}[{\cite[p. 253]{VA62} Theorem 4.10, Theorem 4.12.}] \label{euhh't}
		Consider Problem A on a multiply-connected domain $\Om$ whose boundary $b \Om$ has $m+1$ connected components. Let $h, h'$ be the numbers of linearly-independent solutions to the homogeneous Problem A and the homogeneous adjoint Problem $A'$ \re{ahpAeqn}, respectively. Let $n$ be the index of Problem A. 
	\begin{enumerate}[(i)] 
	   \item 
	   Suppose that $n< 0$. Then
		\[ 
		h = 0, \quad  h'= m-2n-1.
		\]   
		The non-homogeneous Problem A has a (unique) solution if and only if the following relations are satisfied:
		\eq{pAnnsc}
		\frac{1}{2 i} \int_{b \Om} l(t) w'_{i} (t) \gm(t) \, dt - 
		Re \iint_{\Om} w'_{i} F (z) \, dx dy = 0,
		\eeq
		for $i=1,...,m-2n-1$, and $w_{1}',..., w'_{m-2n-1}$ is a complete system of solutions of the adjoint  homogeneous Problem $A'$ \re{ahpAeqn}. 
	\item
		Suppose that $n > m-1$. Then
		\[ 
		h = 2n + 1 - m,  \quad h' = 0. 
		\]  
		The non-homogeneous Problem A is always solvable and its general solution is given by: 
		\eq{pAposindgs}
		w(z) = w_{0}(z) + \sum_{j=1}^{2n+1-m} c_{j} w_{j}, \quad  \quad \text{$c_{j}$ are real constant, }
		\eeq
		where $\{ w_{1},..., w _{2n+1-m}\} $ is a complete system of solutions of the homogeneous problem A, and $w_{0}$ is a particular solution of the non-homogeneous problem A. 
	\end{enumerate}
		\end{prop} 

		\mk{}
		When the domain is simply-connected, we have $m=0$, and the above theorem reduces to the results we obtained earlier. When $0 \leq n \leq m-1$ ($m \geq 1$), the existence and uniqueness of solutions to Problem A are more subtle, and we refer the reader to \cite[chapter 4, \S 5]{VA62}. 
		\emk
		
		From now on we only consider the case $n > m-1$. We can impose some additional conditions on the solution of Problem A to make it unique. This will be necessary when later on we add parameter since regularity in parameter depends on the uniqueness of these solutions.  
		\begin{defn} \label{ndsd}
			Let $\Om$ be a bounded domain in $\C$ whose boundary $b \Om$ contains $m+1$ connected components, for $m$ a non-negative integer. Let $z_{1},...,z_{N_0}$ and $z_{1}',...,z_{N_1}'$ be fixed points of $\Om$ and its boundary $b \Om$, respectively, satisfying the following conditions: \\ 
			(i) The numbers $N_0$ and $N_1$ satisfy the relation:
			\[  
			2N_0 + N_1 = 2n + 1 - m. 
			\]
			(ii) There are $m$ curves, e.g. $\Gm_{i_{1}},..., \Gm_{i_{m}}$, among the $m+1$ boundary curves $b \Om_{0},..., b \Om_{m}$, on each of which an odd number of $z_s'$ are situated. 
			\\
			Then we call $ \{ \{z_{r} \}_{r=1}^{N_0}$, $\{z_{s}' \}_{s=1}^{N_1} \}$  a \emph{normally distributed set}, henceforth denoted by $(N_0,N_1,\ov{\Om})$.
		\end{defn}
		\mk{}
		In our choice of $N_0$ and $N_1$, two extreme cases are possible: 
		\begin{enumerate}
		    \item $N_0 = 0$ and $N_1 = 2n + 1 -m$. 
		    \item $N_0 =n-m$ and $N_1 =m+1$. 
		\end{enumerate}
		Consequently, we have $0 \leq N_0 \leq n-m$, $m+1 \leq N_1 \leq 2n+1-m$. 
		If $m= 0$, i.e. the domain is simply-connected,  condition (i) becomes $2N_0 + N_1 = 2n+1$, and thus in this case $N_1$ needs to be odd.   
		\emk
		
		We can specify on $ (N_0,N_1, \ov{\Om})$ the values of the unknown solution of Problem A:
		\eq{ndsc}
		\begin{gathered}
			w(z_{r}) = a_{r} + ib_{r}, \quad r=1,..., N_0; \quad 
			w(z_{s}') = l(z_{s}') (\gm(z_{s}') + ic_{s}), \quad s= 1,...,N_1.  
		\end{gathered}  
		\eeq
		\pr{c1ut}
		Let $n$ be the index of Problem A on a domain $\Om$ in $\C$, and let $m$ be the number of connected components of $b \Om$.  Assume that on a normally distributed point set $(N_0,N_1,\ov{\Om})$ conditions of the form \re{ndsc} are given. Then for $n > m-1$, Problem A has always a solution satisfying these conditions, and this solution is uniquely determined. 
		\epr
		The reader can refer to Vekua \cite[p. 285 - 287]{VA62} for the proof.
		
		The following regularity result for simply-connected domains follows from the above construction of the solution by Vekua, though it is not explicitly stated in his book. For the reader's convenience we provide details of the proof here. Furthermore, the same estimates hold for multiply-connected domains and one can prove it using a trick of Vekua to reduce the simply-connected case (See \cite[p. 228, p. 336]{VA62}.)
		
		\pr{pArt} 
		(i) Let $\Om \subset \C$ be a domain with $C^{1+ \mu}$ boundary, for $0 < \mu  < 1$. Let $2 < p < \infty$. Suppose $A,B,F \in L^{p}(\Om)$, $l, \gm \in C^{\mu}(b \Om)$. Then the solution to Problem A \re{probA}, if it exists, belongs to the class $C^{\nu}(\ov{\Om})$, for $\nu:= \text{min} \{\all_p, \mu\}$, $\all_p = \frac{p-2}{p}$, and there exists a constant $C$ depending on the coefficients $A, B, l, p$, such that 
		\[ 
		| w |_{\nu} \leq C(A,B,l,p) \left( |w|_0 + |F|_{L^p (\Om) } + | \gm|_{\mu} \right) . 
		\] 
		
		\nid
		(ii) Let $k \geq 0$, and $0 < \mu < 1$. Let $\Om$ be a domain in $\C$ with $C^{k+1+\mu}$ boundary. Suppose $A, B, F \in C^{k+\mu}(\ov{\Om})$, $l, \gm \in C^{k+1+\mu}(b \Om)$. Then the solution $w$ of Problem A, if it exists, belongs to the class $C^{k+1+\mu}(\ov{\Om})$, and there exists a constant $C_k$ depending on the coefficients $A, B, l$ such that
		\[ 
		| w |_{k+1+ \mu} \leq C_k (A, B, l) \left( |w|_0 + | F |_{k+\mu} + | \gm |_{k+1+\mu} \right). 
		\] 
		\epr
		We remark that the regularity results are in \cite[p. 228  and p.~336]{VA62}. 
		The main purpose of the following proof is to derive the above two estimates.
		The constants cannot be controlled since the proof uses the open mapping theorem.
		
		\nid \tit{Proof.}  
		We treat only the non-negative index case. The proof is similar if the index is negative. \\
		(i) By the Riemann mapping theorem it suffices to take $\Om = \D$. Moreover by the earlier remarks it suffices to assume that the Problem A has the form \re{pAbcrf}. 
		Let $\nu = \min (\all_p, \mu)$, $\all_p = \frac{p-2}{p}$. In view of equation \re{pAie2}, any solution $w$ has the form 
		\eq{wastws}  
		w = w_{\ast} + \sum_{j=0}^{2n}  d_j w_j 
		\eeq
		where $d_j$ are complex constants and $w_{\ast}$ is the (unique) solution to the integral equation 
		\eq{wastinteqn}  
		w_{\ast} + Q_n w_{\ast} = P_n F + \mc{S} (\gm),  
		\eeq
		and for each $0 \leq j \leq 2n$, $w_j$ is the (unique) solution to the integral equation  
		\eq{wseqn}
		w_j + Q_n w_j = g_j, \quad 0 \leq j \leq 2n,  
		\eeq  
		where the $g_j$-s are the set of $2n+1$ functions $\{ z^k - z^{2n-k}, z^k + z^{2n - k}, z^n \}_{k=0}^{n-1}$ from expression \re{ckzkRe}.

		First we estimate $w_{\ast}$. Since the operator $I + Q_n$ is a bijective bounded linear operator from the space $C^0 (\ov{\Om})$ to itself, by the open mapping theorem it has a bounded inverse  $(I+ Q_n)^{-1}$, whose operator norm depends only on $A$ and $B$. Hence in view of \rp{Pnhest1} and \rp{Cint} we have for $C_i=C_i(A,B)$
		\begin{align*}  
		|w_{\ast} |_0 
		\leq C_1  |P_n F  + \mc{S} (\gm) |_0  
	    \leq C_2 \left( |P_n F|_0  + |\mc{S} (\gm) |_0  \right)  
		\leq C _3  \left(  |F|_{L^p(\D)} + | \gm |_{\mu} \right) . 
		\end{align*} 
		
		Next we estimate the H\"older norm of $w_{\ast}$. From \re{wastinteqn} we have 
		\begin{align*}
		| w_{\ast} |_{\nu} 
		&\leq |Q_n w_{\ast}  |_{\all_p} + |P_n F|_{\all_p} + |\mc{S} \gm |_{\mu} . 
		\end{align*}
		By \rp{Pnhest1} we have
		\begin{align*}
		|Q_n w_{\ast} |_{\all_p}   
		= | P_n (Aw_{\ast} + B \ov{w_{\ast}}) |_0  
		\leq C(p) | Aw_{\ast} + B \ov{w_{\ast}} |_{L^p (\D)} 
		\leq C(A,B,p)| w_{\ast} |_0, 
		\end{align*}
		and 
		$
		| P_n F|_{\all_p}  \leq C(p) | F |_{L^p (\D)}. 
		$ 
		By \rp{Cint} we have
		$
		|S \gm |_{\mu} \leq C | \gm |_{\mu}.   
		$
		Hence
		\begin{align} \label{wasthest}  
		|w_{\ast}|_{\nu} 
		&\leq C_0 (A,B,p) \left(  | w_{\ast} |_0 + | F |_{L^p (\D)} +  | \gm |_{\mu} \right) 
		\leq C_1 (A,B,p)  \left( | F |_{L^p (\D)} +  | \gm |_{\mu} \right). 
		\end{align} 
		
		We now estimate $w_j$.  
		Applying open mapping theorem to \re{wseqn} we have
		\begin{gather*} 
		|w_j|_{0} \leq C_j (A,B),  \quad 0 \leq j \leq 2n; 
		\\ 
		|w_j|_{\all_p} \leq | Q_n w_j |_{\all_p} + |g_j |_{\all_p}  \leq C (A,B,p) ( |w_j|_0 + 1 ) \leq C_j( A,B, p).  
		\end{gather*}
		It remains to estimate the coefficients $d_j$ in \re{wastws}. By the remarks in \cite[p.~286-288, equation (6.10)]{VA62} (see also \rp{pAscfdthm1} below), they can be expressed as 
		in the form $\frac{P}{Q}$, where $Q \neq 0$ and $P$ and $Q$ are linear combinations of the products of $w_j(z_r), \:  w_j (z_s')$ with terms like $ a_r , \: b_r, \: c_s, \: \gm(z_s'), \: w_{\ast} (z_r ), \: w_{\ast} (z_s' )$. Since $ | ( a_r, b_r, c_s ) | \leq |w|_0$, we have 
		\begin{align}  \label{dsest} 
		| d_j| &\leq C_j(A,B)  \left( |w|_0 + |\gm|_0 + |w_\ast|_0 \right)  
		\leq C_j(A,B) \left( |w|_0 + |F|_{L^p (\D)} + |\gm|_{\mu} \right). 
		\end{align} 
		Hence by 
		\re{wastws} and combining the above estimates we get
		\begin{align*}  
		|w|_{\nu} &\leq |w_{\ast}|_{\nu}  + \sum_{j=0}^{2n} |d_j| |w_j|_{\all_p}  
		\leq M(A,B, p) \left( |w|_0 + |F|_{L^p (\D)} + |\gm|_{\mu}  \right). 
		\end{align*}
		
		Finally, let $\wti{w}$ be a solution to the original Problem A \re{probA} on $\D$: 
		\eq{probAwti} 
		\begin{cases}
			\mc{L}_{\wti A, \wti B} (\wti{w}) \equiv \pa_{\ov{z}} \wti{w}  + \wti{A}(z) \wti{w}(z)  + \wti{B}(z) \ov{\wti{w}(z)} = \wti{F}(z) \quad \text{in $\Om$;}\\
			\all u + \beta v  \equiv  Re[\ov{l(z)} \wti{w}] = \wti{\gm}  \quad \text{on $b \Om$. }
		\end{cases} 
		\eeq
		Then we have the following relations: 
		\begin{equation} \label{redrel}
		\begin{gathered} 
		A  (z) = \wti{A} (z),  \quad B(z) =\wti{B} e^{2i \chi(z)},  \quad  F (z) = \wti{F}(z) e^{\chi(z )}, \: z \in \D;   
		\\
		\wti{w} (z) = e^{- \chi (z)} w(z), \: z \in D, \quad  \gm(\zeta) = \wti{\gm} (\zeta) e^{p(\zeta)}, \: \zeta  \in b S^1, 
		\end{gathered}
		\end{equation}
		where $\chi$ and $p = Re (\chi)$ are related to $l$ by formula \re{chi}. Hence 
		\[
		|\wti{w} |_{\nu} \leq C(\wti{A}, \wti{B}, l, p)  \left( |\wti{w}|_0 + |\wti{F}|_{L^p (\D) } + |\wti{\gm}|_{\mu} \right). 
		\]
		\medskip \nid 
		(ii) Let us first consider Problem A in the reduced form \re{pAbcrf}.  
		Write \re{wastinteqn} as
		\eq{pAie2'} 
		w_{\ast} = - Q_n w_{\ast} + P_n F + \mc{S} (\gm) 
		\eeq
		First let us prove the statement for $k=0$. We have
		\begin{align*}
		|w_{\ast}|_{1+\mu} 
		&\leq C \left( |Q_n w_{\ast}|_{1 + \mu} + |P_n F|_{1+ \mu} + |\mc{S} \gm |_{1+\mu} \right)  \\
		&= \left( |P_n (Aw_{\ast} + B\ov{w_{\ast}}) |_{1+ \mu} + |P_n F|_{1+\mu} + |\mc{S} \gm |_{1+\mu} \right)  \\
		&\leq C \left( |Aw_{\ast} + B\ov{w_{\ast}} |_{\mu} + |F|_{ \mu} + |\gm|_{1+\mu}  \right) \\
		&\leq C \left\{ \left( |A|_{\mu} + |B|_{\mu}  \right) | w_{\ast} |_{\mu} + |F|_{\mu} + |\gm|_{1+ \mu}  \right\} .  
		\end{align*}
	By 
		\re{wasthest}, we get 
		$
		|w_{\ast} |_{\mu} \leq C (A,B) \left( | F |_{\mu} +  |\gm |_{\mu} \right).   
		$
		Hence 
		$
		|w_{\ast} |_{1+\mu} \leq C_0(A,B) \left( |F|_{\mu} +  | \gm |_{1+ \mu} \right). 
		$ 
		
		For $k > 0$ we apply induction. By Propositions \ref{Pnhest2} and 
		\ref{Cint} applied to equation \re{pAie2'} one gets 
		\begin{align*}
		|w_{\ast}|_{k + 1+\mu} 
		&\leq C \left( |Q_n w_{\ast}|_{k + 1 + \mu} + |P_n F|_{k + 1+ \mu} + |\mc{S} \gm |_{k + 1+\mu} \right)  \\
		&= \left( |P_n (Aw_{\ast} + B\ov{w_{\ast}}) |_{k + 1+ \mu} + |P_n F|_{k + 1+\mu} + |\mc{S} \gm |_{k+1+\mu} \right)  \\
		&\leq C' \left( |Aw_{\ast} + B\ov{w_{\ast}} |_{k+\mu} + |F|_{k + \mu} + |\gm|_{k+ 1+ \mu} \right) \\
		&\leq C' \left\{ \left( |A|_{k+\mu} + |B|_{k+\mu}  \right) | w_{\ast} |_{k+ \mu} + |F|_{k+\mu} + |\gm|_{k+1+ \mu}  \right\}. 
		\end{align*}
		By the induction hypothesis, we have
		$
		| w_{\ast}|_{k+\mu} \leq C_{k-1} (A,B) \left( |F|_{k-1+\mu} + |\gm |_{k+\mu} \right) . 
		$ 
		Hence
		$
		|w_{\ast}|_{k+1+\mu} \leq C_k (A,B) \left( | F|_{k+\mu} + |  \gm|_{k+1+ \mu} \right). 
		$
		
		Next we estimate $w_j$, $0 \leq j \leq 2n$. From equation \re{wseqn} we have
		\eq{wseqn'}
		w_j =  - Q_n w_j + g_j.  
		\eeq
		First we prove for $k = 0$, so that $A, B \in C^{\mu}(\ov{\D})$. By part (i), we know that $w_s \in C^{\frac{p-2}{p}} (\ov{D})$ for any $2 < p < \infty$. Hence $w_s  \in C^{1-}(\ov{\Om})$. By \rp{Pnhest2} we have 
		\begin{align*}
		|w_j|_{1 + \mu} 
		&\leq C( |Q_n w_j|_{1 + \mu} + 1)   
		= C \left( |P_n (A w_j + B \ov{w_j} )|_{1+\mu}  + 1  \right) \\ 
		&\leq C\{ ( |A|_{\mu} + |B|_{\mu} ) |w_j |_{\mu}  +  1\}  
		\leq C(A,B).  
		\end{align*}
		By \re{wseqn'} again and iterate the process we get $w_j \in C^{k+1+\mu}$ and
		$
		| w_j |_{k+1+ \mu} \leq C_k(A,B) . 
		$
		Now for the coefficients $d_j$ we have estimate \re{dsest}, 
		$
		|d_j | \leq C_j(A,B)  \left(  |w|_0 + |F|_{\mu} + | \gm|_{\mu} \right) . 
		$
		Finally putting together the estimates we get
		\begin{align*} 
		| w |_{k+1+\mu} &\leq | w_{\ast} |_{k+1+ \mu} + \sum_{j=0}^{2n} |d_j| |w_j|_{k+1+\mu} \leq C_{k} (A,B) \left( |w|_0  + | F|_{k+\mu} + |  \gm|_{k+1+ \mu} \right). 
		\end{align*}

		Finally we let $\wti{w}$ be the solution to the original Problem A \re{probAwti} on $\D$. In view of the relations \re{redrel} we easily get 
		$
		|\wti{w} |_{k+1+\mu} \leq C_k (\wti{A}, \wti{B}, l) \left( | \wti{w} |_0 + |\wti{ F} |_{k+\mu} + |  \wti{\gm}|_{k+1+ \mu} \right).  
		$
		\hfill \qed

		\mk{}
		In our study of the boundary value problems with parameter, the above estimates imply that for each fixed $\la$, the solution $w^{\la}$ satisfies 
		\[
		| w^{\la} |_{k+1+ \mu} \leq C (A, B, l, \la)  \left(  |w^{\la}|_0 + |F^{\la}|_{k+ \mu} + | \gm^{\la} |_{k+1+\mu} \right) 
		\]
		on $b  \Om^{\la}$. However, this estimate is not useful since we do not know how the constant depends on $\la$.  
		\emk

		We end this section by connecting the oblique derivative boundary value problem for the Laplacian elliptic operator (hereafter called Problem B) to Problem A. \\
		\textbf{Problem B} \label{probB}   
		Consider the following boundary value problem: 
		\begin{equation} \label{pBe} 
	    \begin{gathered} 
		\Del U + a(x,y) U_{x} + b(x,y)U_{y} = f(x,y) ; 
		\\ 
		\all U_{x} + \beta U_{y} = \gm.   
		\end{gathered}   
		\end{equation}
		To solve this problem we follow Vekua and reduce it to Problem A. 
		Let 
		\[ 
		u = U_{x}, \quad v = -U_{y}, \quad w= u + iv. 
		\] 
		The equation \re{pBe} reduces to the following specialized Problem $A$:
		\begin{equation} \label{pBrpAe} 
		\begin{gathered} 
		\zbd w + \frac{1}{4}(a+ib) w + \frac{1}{4}(a-ib) \ov{w} =\yh f, \quad \text{in $\Om$;} \\ 
		Re[\ov{l} w] = \gm, \quad l = \all - i \beta, \quad \text{on $b \Om$. }
		\end{gathered}
		\end{equation} 
		If $w$ is the solution to \re{pBrpAe}, then the solution to Problem \re{pBe} is given by 
		\begin{align} \label{pBU} 
		U(x,y) &= c_{0} + Re \int_{z_{0}}^{z} w(\zeta) \, d\zeta 
		= c_{0} + \int_{z_{0}}^{z} U_{x}(x,y) \,dx + U_{y}(x,y) \, dy. 
		\end{align}
		Now, $U$ is globally well-defined only if the closed form $U_{x} \,dx + U_{y} \,   dy$ is exact. This occurs when the domain is simply-connected. Thus in the case of simply-connected domain, Problem $B$ is completely equivalent to Problem $A$. 
		
		In the case $b \Om$ has $m+1$ connected components, for $m>0$, in order to guarantee the right-hand side of \re{pBU} is single-valued, it is necessary and sufficient to add $m$ solubility conditions:
		\eq{odpsolc}  
		Re \int_{\Gm_{j}} w \, d\zeta 
		= \int_{\Gm_{j}} u \,dx - v\, dy
		=0,  \quad (j=1,2,\dots,m)  
		\eeq
		where $\Gm_0, \Gm_{1}, \dots, \Gm_{m}$ are simple closed contours bounding the domain $\Om$, $\Gm_{1}, \dots, \Gm_{m}$ being situated inside $\Gm_{0}$. Hence a solution $w$ to \re{pBrpAe} is a solution to the original Problem \re{pBe} if and only if it satisfies conditions \re{odpsolc}.

		\section{Families of simply-connected domains}
		In this section we prove \rt{mtintro} for families of simply-connected domains. First we need to define the index for oblique derivative boundary value problem:   
		
		Given the oblique derivative boundary value problem  \re{odpintro}, 
		we define its \emph{index $\varkappa$} as the winding number of $ \ov{ \U^{\la} } $: 
		\eq{odpind} 
		\varkappa: = \frac{1}{2 \pi} \left[  \ov{\Upsilon^{\la}} \right]_{b \Om^{\la} } = \frac{1}{2 \pi} \left[ \xi^{\la} - i \eta^{\la}  \right].    
		\eeq
		Note that since $\varkappa$ is an integer, it is stable under small perturbation of $\la$. 
		
		We formulate our results in two parts depending on the sign of $\varkappa$. First we consider the case $\varkappa \geq 0$. 
		\begin{thm}[Simply-connected, non-negative index case] \label{scposind} \ \\
		 Let $k, j,\mu$ be as in \rt{mtintro}. Let $\Om$ be a bounded, simply-connected domain in $\C$ with $\mc{C}^{k+2+\mu}$ boundary. Let $\Gm^{\la}$ be a family of maps that embed $\Om$ to $\Om^{\la}$, with $\Gm^{\la} \in \mc{C}^{k+2+\mu,j}(\ov{\Om})$. Consider problem \re{odpintro} with $\varkappa \geq 0$, and where $L^\la$ is given by \re{Llambda}.
			Suppose $a^{\la}, b^{\la}, c^{\la}, d^{\la}, e^{\la}, f^{\la}$ are functions in $\Om^{\la}$ such that $a^{\la} , b^{\la}, c^{\la} \in \mc{C}^{k+1+\mu,j}(\ov{\Om^{\la}})$, and $d^{\la}, e^{\la}, f^{\la} \in \mc{C}^{k+\mu,j} (\ov{\Om^{\la}})$. 
			Suppose that $\Upsilon^{\la}, \gm^{\la}$ are functions in $b \Om^{\la}$ such that $\Upsilon^{\la}, \gm^{\la} \in \mc{C}^{k+1+\mu,j} (b \Om^{\la})$. For each $\la \in [0,1]$, let $ \{ \{z_{r}^{\la} \}_{r=1}^{N_0}$, $(z_{s}^{\la})' \}_{s=1}^{N_1} \}$ be a normally distributed set for $\Om^{\la}$, and let $u^{\la}$ be the unique solution to \re{odpintro} on $\Om^{\la}$ satisfying 
			\begin{gather*} \label{scnds}  
			u^{\la}_x(z_r^{\la}) - i u^{\la}_y (z_r^{\la}) = a_r(\la) + i b_r(\la) \quad \quad (r =1,...,N_0); 
			\\  
			u^{\la}_x((z_s')^{\la}) - i u^{\la}_y ((z_s')^{\la}) = l^{\la}((z_{s}')^{\la}) (\gm^{\la}((z_{s}')^{\la}) + ic_{s}(\la)),  \quad \quad (s = 1,...,N_1), 
			\\  
			u^{\la} (z_0^{\la}) = g(\la), \quad z_0^{\la} \in \Om^{\la}; 
			\end{gather*} 
			for some functions $a_r, b_r, c_s, g \in C^j ([0,1])$. Then $u^{\la} \in \mc{C}^{k+2+\mu,j} (\ov{\Om^{\la}})$.  
		\end{thm} 
		
		By \rp{isop}, we can reduce the problem to the following``canonical" form (Problem B) with the same index: 
		\begin{gather*}  
		\Del U^{\la} +  a^{\la} (x,y) U^{\la}_x +  b^{\la} (x,y) U^{\la}_y = f^{\la}  (x,y), \quad  \text{in $\Om^{\la}$; }
		\\  
		\all^{\la} U_x^{\la}  + \beta^{\la} U^{\la}_y = \gm^{\la}, \quad \text{on $b \Om^{\la}$,}
		\end{gather*} 
		where $a^{\la}, b^{\la}, f^{\la} \in \mc{C}^{k+\mu,j} ( \ov{\Om^{\la}} )$, $\all^{\la}, \beta^{\la}, \gm^{\la} \in \mc{C}^{k+1+ \mu, j} (\ov{\Om^{\la}})$. Furthermore, by the remark in the last section (see \re{probB}), if one sets $w^{\la} = U_x^{\la} - i U_y^{\la} $ then
		\eq{wlaint} 
		U^{\la} (x,y) = c_0 (\la) + Re \int_{z_0}^z w^{\la} (\zeta) \, d \zeta.  
		\eeq
		Hence it suffices to consider the corresponding Problem A for $w^{\la}$: 
		\begin{equation} \label{pAsc} 
	      \begin{gathered}  
	      \pa_{\ov{z}} w^{\la} +  A^{\la}  w^{\la} + B^{\la} \ov{w^{\la}}= F ^{\la} , \quad  \text{in $\Om^{\la}$; }
	      \\  
	      Re [\ov{l^{\la}} w^{\la} ] = \gm^{\la} , \quad l^{\la} = \all^{\la} - i \beta^{\la}, \quad \text{on $b \Om^{\la}$,}
	      \end{gathered}  
		\end{equation} 
		with the corresponding index $n = \frac{1}{2 \pi} [l^{\la}]_{b \Om^{\la}} = \frac{1 }{2 \pi} [\ov{\U^{\la}} ]_{b \Om^{\la}} = \varkappa \geq 0 $, $A^{\la}, B^{\la}, F^{\la} \in  \mc{C}^{k+\mu,j} ( \ov{\Om^{\la}} )$, and $l^{\la}, \gm^{\la} \in  \mc{C}^{k+ 1 + \mu,j} ( \ov{\Om^{\la}} )$.  
		
		Our strategy is to use a smooth family of Riemann mappings proved in \cite{B-G14} to reduce the problem to one on a fixed domain, where only the coefficients and the right-hand side of the equation depend on the parameter. 
		
		Let us first consider the following homogeneous Problem A on a fixed domain:
		\eq{pAhomfd}
		\begin{cases}
			\mc{L} (w^{\la}) \equiv \pa_{\ov z} w^{\la}  + A^{\la}(z) w^{\la} + B^{\la}(z) \ov{w^{\la}} = 0 \quad \text{in $\Om$; }\\
			Re[\ov{l^{\la}(z)} w^{\la}] = 0 \quad \text{on $b \Om$. }
		\end{cases} 
		\eeq
		Let $w_s^{\la}$, $1 \leq s \leq 2n+1$ be the solutions to the integral equation 
		$
		w^{\la} + Q_{n}^{\la} w^{\la} = g_s, 
		$
		where $g_s$ is one of the functions:
		\eq{pAscgell}
		z^\ell - z^{2n-\ell}, \quad i(z^\ell + z^{2n - \ell}), \quad iz^{n}, \quad \ell = 0, \dots, n-1. 
		\eeq
	By 
		\re{pAie2} and \re{ckzkRe}, $\{ w_s^{\la} \}_{s=1}^{2n+1}$ forms a basis of $\R$ linearly-independent solutions to Problem \re{pAhomfd}. We now prove regularity of $w^{\la}_s$ with respect to $\la$.   
		
		\pr{pAfdhsct}
		Let $k$ be a non-negative integer and let $0 < \mu <  1$. Let $\Om \subset \C$ be a simply connected domain with $C^{k+1+\mu}$ boundary. Let $w_{1}^{\la},\cdots,w_{2n+1}^{\la} $ be linearly-independent solutions to the homogeneous Problem A \re{pAhomfd} with parameter as given above, with index $n \geq 0$. Suppose $A^{\la}, B^{\la} \in \mc{C}^{k+\mu,0}(\ov{\Om} )$, $l^{\la} \in \mc{C}^{k+1+\mu,0}(\ov{\Om} )$. Then $w_s^{\la} \in \mc{C}^{k+1+\mu,0}(\ov{\Om})$, for $s = 1, \dots, 2n+1$. 
		\epr
		\begin{proof}
			First we reduce the problem to one on the unit disk. Since $\Om$ is a simply-connected domain, there exists a biholomorphic map $\Gm: \Om \to \D$ which is $C^{k+1+\mu}$ up to the boundary. Define $\wti{w}$ on $\D$ by $ w(z) = \wti{w} \circ \Gm(z)$. Then $\wti{w}(\zeta)$ is the solution to the problem: 
			\eq{pAhomfd'}   
			\begin{cases}
				\pa_{\ov \zeta} \wti{w}^{\la}  +\wti{A}^{\la}(\zeta) \wti{w}^{\la} + \wti{B}^{\la}(\zeta) \ov{\wti{w}^{\la}} = 0 \quad \text{in $\D$;}\\
				Re[\ov{\wti{l}^{\la}(\zeta)} \wti{w}^{\la}] = 0  \quad \text{on $S^1$, }
			\end{cases} 
			\eeq
			where $
			\wti{A}^{\la} = A^{\la} D^{-1}, \wti{B}^{\la} = B^{\la} D^{-1}, 
			D (\zeta) = \ov{\pa_z \Gm} \circ \Gm^{-1}(\zeta) \neq 0.  
		    $
			Since $\Gm \in C^{k+1+\mu} (\ov{\Om})$ and $\Gm^{-1} \in C^{k+1+\mu} (\ov{\D})$, we have $A^{\la}, B^{\la}, \in \mc{C}^{k+\mu,0}(\ov{\D})$, and $\wti{l}^{\la} \in \mc{C}^{k+1+\mu, 0} (\ov{\D}).$ Thus the problem is reduced to $\Om = \D$. 
			
	By 
			\re{pAbcrf}, it suffices to consider the reduced boundary value problem
		\eq{pA_unit_disc} 
			\begin{cases}
				\pa_{\ov{z}} w^{\la} + A^{\la} w^{\la} + B^{\la} \ov{w^{\la}} = 0,  \quad \text{in $\D$; }
				\\	
				Re[ z^{-n} w^{\la} ] = 0, \quad \text{on $S^1$, }
			\end{cases}
	    \eeq
			where 
			$
			w^{\la}(z) =e^{\chi^{\la}(z)} \wti{w}^{\la}(z),  A^{\la} = \wti{A}^{\la}, B^{\la} = \wti{B}^{\la} e^{2i Im \, \chi^{\la}(z)}, 
			$
			and $\chi^{\la}(z)$ is defined in \re{chi}. 
	By 
			\rp{Cint},
			we have $\chi^{\la} \in \mc{C}^{k+1+\mu,0}(\ov{\Om})$, and hence $A^{\la} , B^{\la} \in \mc{C}^{k+\mu,0}( \ov{\Om})$. 
			By 
			\re{pAie2}, for $s = 1, 2, \dots, 2n+1$, $w_s^{\la}$ is the unique solution satisfying the integral equation
			\eq{dcie}
			w_s^{\la} + Q_{n}^{\la} w_s^{\la} = g_s,  
			\eeq
			where $Q_{n}^{\la} w^{\la}_s := P_{n} (A^{\la} w^{\la} + B^{\la} \ov{w^{\la}}) $, and $g_s$ are given by \re{pAscgell}. 
			
			Let us first prove the statement for $k=0$. We claim $|w_s^{\la}|_0$ is bounded uniformly in $\la$.  Suppose this is not the case, that there exists some sequence $\la_{j} \in [0,1]$ with $N_{j} = |w_s^{\la_{j}}|_0 \to \infty$. Since $[0,1]$ is compact, passing to a subsequence if necessary we can assume that $\la_j$ converges to some $\la_0$. Dividing each side of the integral equation \re{dcie} by $N_{j}$, and writing $\hht{w}_s^{\la_{j}} = w_s^{\la_{j}} / N_{j}$, we obtain
			\eq{dcnie}   
			\hht{w_s}^{\la_{j}} + Q_{n}^{\la_{j}} \hht{w_s}^{\la_{j}} = N_j^{-1} g_s, \quad \quad |\hht{w_s}^{\la_{j}}|_0 = 1.
			\eeq
			By \rp{Pnhest1}, we have 		for any $0 < \tau < 1$ 
		\[ 
		|Q_{n}^{\la} \hht{w_{i}}^{\la_{j}}|_{\tau} \leq C \left( |A^{\la}|_{0} + |B^{\la}|_{0} \right) |\hht{w_{i}}^{\la_{j}}|_0 \leq C'.
		\] 
			By Arzela-Ascoli theorem, and passing to a subsequence if necessary, we have $Q_{n}^{\la_{j}} \hht{w_s}^{\la_{j}}$ converges to some function uniformly. It follows from equation \re{dcnie} that $ \hht{w_s}^{\la_{j}}$ converges to some function $\hht{w}_s$ uniformly. We show that $Q_{n}^{\la_{j}} \hht{w_s}^{\la_{j}}$ converges uniformly to $Q_{n}^{\la_0} \hht{w}$. Recall that 
			\[
			Q_n^{\la} w^{\la} = P_n  (A^{\la} w^{\la} + B^{\la} \ov{w^{\la}} ), 
			\]
			where $P_n$ is given by formula \re{Pnf1}.  Hence 
			\begin{align*}
			Q_n^{\la_1} w^{\la_1} - Q_n^{\la_2} w^{\la_2}   
			&= \left( Q_n^{\la_1} - Q_n^{\la_2} \right) w^{\la_1} + Q_n^{\la_2} (w^{\la_1} - w^{\la_2}) \\
			&= P_n \left(  (A^{\la_1} - A^{\la_2}) w^{\la_1} \right) +  P_n \left( (B^{\la_1} - B^{\la_2} )  \ov{w^{\la_1}} \right) \\
			&\quad+ P_n \left( A^{\la_2} (w^{\la_1} - w^{\la_2}) \right) +  P_n \left( B^{\la_2} (w^{\la_1} - w^{\la_2}) \right). 
			\end{align*} 
			By \rp{Pnhest1}, we have 
			for any $0 < \tau  < 1$
			\begin{align*} 
			\left| Q_{n}^{\la_{j}} \hht{w_s}^{\la_{j}} - Q_{n}^{\la_0} \hht{w} \right|_{\tau} 
			&\leq \left( \left| A^{\la_j} - A^{\la_0} \right|_{0} + \left| B^{\la_j} - B^{\la_0} \right|_{0} \right) | \hht{w_s}^{\la_j} |_{0} \\ 
			&\quad + \left( |A^{\la_0}|_0 + |B^{\la_0}|_0 \right) \left|  \hht{w_s}^{\la_{j}} - \hht{w_s} \right|_0.
			\end{align*}
			Since $A^{\la}, B^{\la} \in \mc{C}^{\mu,0} (\ov{\D})$, we have $A, B$ are  continuous in $\la$ in the sup norm. Hence by taking $\tau = 0$ in the above expression we get $ \left| Q_{n}^{\la_{j}} \hht{w_s}^{\la_{j}} - Q_{n}^{\la_0} \hht{w} \right|_0 \ra 0$. Letting $\la_j \to \la_0$ in   \re{dcnie} we get
			$
			\hht{w}_s + Q_{n}^{\la_0} \hht{w}_s = 0. 
			$
			By Vekua \cite[p.~296-298]{VA62}, there does not exist non-trivial continuous solutions for the above equation, thus $\hht{w}_s = 0$. 
			But we have $|\hht{w}_s|_0 = \lim_{j \to \infty} |\hht{w_s }^{\la_{j}}|_0 = 1$, which is a contradiction. Thus $|w_s ^{\la}|_0$ is bounded by some constant $C_0$ uniform in $\la$. It follows then 
			\[
			| Q_n^{\la} w_s^{\la} |_{\tau} \leq C (|A^{\la}|_0 + | B^{\la} |_0  )  |w_s^{\la}|_ 0  \leq C_0', 
			\]
			for any $0 < \tau < 1$. 
			By  \re{dcie} we obtain $|w^{\la}|_{\tau}$ is bounded uniformly in $\la$ by some constant $C_0''$.
			
			For $k > 0$ we apply induction. By   \re{dcie} and \rp{Pnhest2} one has 
			\begin{align*}
			| w_s^{\la} |_{k+1+ \mu} 
			&\leq |Q_n^{\la} w_s^{\la} |_{k+1+ \mu} + C_k 
			\\ &= |P_n (A^{\la} w_s^{\la} + B^{\la} \ov{w_s^{\la}} ) |_{k+1+ \mu} + C'_k \\
			&\leq C''_k \left(  | A^{\la} |_{k+\mu} + | B^{\la} |_{k+\mu} \right)  |w_s^{\la}|_{k+\mu} + C'''_k 
			\leq C''''_k. 
			\end{align*}
			Hence $|w_s^{\la} |_{k+1+ \mu}$ is bounded uniformly in $\la$. 
			
			Finally we show that $w_s^{\la}$ is continuous in the $C^{0}$-norm. Suppose this is not the case, we can then find a sequence $\la_{j} \to \la_0$, $|w_s^{\la_{j}} - w_s^{\la_0}|_{0} \geq \del > 0$. Since $|w_s^{\la_{j}}|_{k+1+\mu}$ is bounded, passing to a subsequence if necessary, $w_s^{\la_{j}}$ converges to some function $\hht{w_s}$ in the $C^{k+1}$-norm. We have shown above that $Q_{n}^{\la_{j}} w_s^{\la_{j}}$ converges uniformly to $Q_{n}^{\la_0} \hht{w_s}$. Letting $\la_j \to \infty$ in equation \re{dcie} we have
			\[ 
			\hht{w_s} + Q_{n}^{\la_0} \hht{w_s} = g_s .  
			\]
			This equation has a unique solution $w^{\la_0}_s$, thus $\hht{w_s} = w^{\la_0}_s$. But this contradicts the fact that $|\hht{w_s} - w^{\la_0}_s|_{0} \geq \del > 0$. Thus we have proved that $w_s^{\la}$ is continuous in $\la$ in the $C^{0}(\ov{\D})$-norm. The continuity of $w_s^{\la}$ in the $C^{k+1}(\ov{\D})$-norm then follows from \rp{cinhd}.   
		\end{proof}
		Next we show the smooth dependence in parameter for a particular solution of the non-homogenous Problem A. 
		\pr{pAfdpsct} 
		Let $k$ be a non-negative integer and let $0 < \mu <  1$. Consider Problem A with parameter on a simply-connected domain $\Om$:
		\eq{pAfd}
		\begin{cases}
			\pa_{\ov z} w^{\la}  + A^{\la}(z) w^{\la} + B^{\la}(z) \ov{w^{\la}} = F^{\la} \quad \text{in $\Om$;}\\
			Re[\ov{l^{\la}(z)} w^{\la}] = \gm^{\la} \quad \text{on $b \Om$, }
		\end{cases} 
		\eeq 
		where the index $n \geq 0$.  Suppose $A^{\la}, B^{\la}, F^{\la} \in \mc{C}^{k+\mu,0} (\ov{\Om})$, $l^{\la}, \gm^{\la} \in \mc{C}^{k+1+\mu,0}(b \Om )$. Then there exists a solution $w_0$ in the class $\mc{C}^{k+1+\mu,0} (\ov{\Om})$. Moreover there exists some constant $C$ independent of $\la$ such that
		\begin{gather} 
		\label{fdpsc0est}   
		\| w_0^{\la} \|_{0,0} := \sup_{\la \in [0,1] } | w_0^{\la} |_0 \leq C \left( \| F^{\la} \|_{0,0} + \| \gm^{\la} \|_{\mu,0} \right),  
		\\ \label{fdpshdest}   
		\| w_0^{\la} \|_{k+1+\mu,0} \leq C \left( \| F^{\la} \|_{k+\mu,0} + \| \gm^{\la} \|_{k+1+\mu,0} \right).  
		\end{gather}
		\epr
		\begin{proof}
			To prove estimate \re{fdpsc0est} (resp.\re{fdpshdest}), it suffices to show that $ \| w_0 \|_{0,0} \leq C$ (resp. $\| w_0 \|_{k+1+\mu,0} \leq C$ ) if each norm on the right-hand side is bounded by $1$, since the equations are linear. 
			As before we can assume the problem takes the form \re{pA_unit_disc}. 
		In view of equation \re{pAie2}, we can take $w_{0}^{\la}$ to be the unique solution in $C^0 (\ov{\D})$ to the integral equation in $\D$: 
			\eq{pAfdpsinteq} 
			w^{\la} + Q_{n}^{\la} w^{\la} =  P_{n} F^{\la} + \frac{z^{n}}{2 \pi i} \int_{S^1} \gm^{\la}(t) \frac{t+z}{t-z} \frac{dt}{t}, 
			\eeq 
			where 
			$
			P_{n} F^{\la} = T_{\D} F^{\la} + z^{2n+1} T_{\D^{c}} F_1^{\la},  Q_{n}^{\la} w^{\la} = P_{n} (A^{\la} w^{\la} + B^{\la} \ov{w^{\la}}), 
			$
			and $F_1^{\la}$ and $F^{\la}$ are related by the second expression in \re{Pnf2}.  
			
			The proof is identical to that of the previous theorem, once we show that the right-hand side of equation \re{pAfdpsinteq} belongs to   $\mc{C}^{k+1+\mu,0}(\ov{\D})$. By \rp{Pnhest2},  	\[ 
			|P_{n} F^{\la}|_{k+1+\mu} \leq C |F^{\la}|_{k+\mu}. 
			\] 
			The second term is $z^{n}$ times $\mc{S} \gm^{\la}$, the Schwarz transform of $\gm^{\la}$ on $\D$.  We have
			\[ 
			| \mc{S} \gm^{\la} |_{k+1+\mu} \leq C |\gm^{\la} |_{k+1+\mu} 
			\]
		by \rp{Cint}.	Thus $ |P_{n} F^{\la} + z^{n} \mc{S} \gm^{\la}|_{k+1+\mu}$ is bounded. Fix $\tau \in (0,1)$ we have
			\begin{gather*}
			|P_{n} F^{\la_{1}} - P_{n} F^{\la_{2}}|_{0} 
			\leq | P_n F^{\la_1} - P_n F^{\la_2} |_{\tau} 
			\leq C|F^{\la_{1}} - F^{\la_{2}}|_{0}, 
			\\
			|\mc{S} \gm^{\la_{1}} - \mc{S} \gm^{\la_{2}} |_{0} 
			\leq 
			|\mc{S} \gm^{\la_{1}} - \mc{S} \gm^{\la_{2}} |_{\tau}  
			\leq C|\gm^{\la_{1}} - \gm ^{\la_{2}}|_{\tau}. 
			\end{gather*}
			By the above two inequalities and the assumptions, we see that $P_{n} F^{\la} + z^{n} \mc{S} \gm^{\la}$ is continuous in the $C^{0}$-norm. By \rp{cinhd}, $P_{n}F^{\la} + S\gm^{\la} \in \mc{C}^{k+1+\mu,0}(\ov{\D})$. 
		\end{proof}
		We can now prove the following result for Problem A on fixed domain: 
		\pr{pAscfdthm1} 
		Let $k$ be a non-negative integer and let $0 < \mu <  1$. Consider the non-homogeneous Problem A \re{pAfd} on a simply-connected domain $\Om$, with index $n \geq 0$. Suppose $A,B,F \in \mc{C}^{k+\mu,0}(\ov{\Om})$, $l, \gm \in \mc{C}^{k+1+\mu,0}(b \Om)$, $a_r, b_r, c_s \in C^{0}( [0,1] )$. Then for each $\la$ there exists a unique solution $w^{\la}$ satisfying the following conditions on a normally distributed set (see \re{ndsd}).
		\eq{scfdnds}
		\begin{gathered}
			w^{\la}(z_r) = a_r(\la) + ib_r(\la), \quad \quad r=1,...,N_0; \\
			w^{\la}(z_s') = l^{\la}(z_s') (\gm^{\la}(z_s') + ic_s(\la)), \quad \quad s =1,...,N_1 
		\end{gathered} 
		\eeq
		where $2N_0 + N_1 = 2n+1$. The family $w^{\la}$ belongs to the class $ \mc{C}^{k+1+\mu,0}(\ov{\Om})$. Moreover there exists some constant $C$ independent of $\la$ so that 
		\begin{gather} \label{pAscfdsupe}    
		\| w^{\la} \|_{0,0} \leq  C \left( \| F^{\la} \|_{0,0} + \| \gm^{\la} \|_{\mu,0} + \sum_{r=1}^{N_0}  \left(|a_r |_0 + |b_r|_0\right) + \sum_{s=1}^{N_1} |c_s|_0 \right), 
		\\ \label{pAscfdhde}        
		\| w^{\la} \|_{k+1+\mu,0} \leq C \left( \| F^{\la} \|_{k+\mu,0} + \| \gm^{\la} \|_{k+1+\mu,0} + \sum_{r=1}^{N_0} \left( |a_r |_0 + |b_r|_0 \right)+ \sum_{s=1}^{N_1} |c_s|_0 \right), 
		\end{gather} 
		where we denote $|a_r|_0 := \sup_{\la \in [0,1]} | a_r(\la) |$ and similarly for $|b_r|_0$ and $|c_s|_0$. 
		\epr
		\nid \tit{Proof.}
		Let $w_{0}^{\la}$ be the particular solution to the non-homogeneous Problem A given in \rp{pAfdpsct}, and let $w^{\la}_\ell$-s be the $2n+1$ $\R$-linearly independent solution to the homogeneous Problem A given in \rp{pAfdhsct}. Then $w^{\la}$ is given by the formula:
		\begin{gather} \label{pAscfdc1wr}
		w^{\la}(z)= w_{0}^{\la}(z) + \sum_{\ell =1}^{2n+1} d_\ell(\la) w_\ell^{\la}(z), 
		\end{gather} 
		where $d_\ell$ are real-valued functions on $[0,1]$. 
		
		From \re{scfdnds}, we have the following linear system of equations for the determination of the functions $d_s (\la)$: 
		\begin{gather} \label{ipeq} 
		\sum_{j=1}^{2n+1} d_\ell (\la) w_\ell^{\la} (z_r) = a_r(\la) + i b_r (\la) - w^{\la}_0 (z_r)  \quad \quad r =1,...,N_0, 
		\\ \label{bpeq}
		\sum_{\ell= 1}^{2n+1} d_\ell (\la) w_\ell^{\la} (z_s') = l^{\la} (z_s') (\gm^{\la} (z_s') + i c_s (\la)) - w^{\la}_0 (z_s') \qquad s = 1,...,N_1. 
		\end{gather} 
		We can write the above as $2N_0 + 2N_1$ real equations. However, there are in fact only $2N_0+ N_1 = 2n+1$ equations, since the imaginary parts of equations \re{bpeq} are obtained from the real parts of equation \re{bpeq}, by multiplying $- \frac{\all^{\la} (z_s')}{\beta^{\la}(z_s')}$ (assume that $\beta^{\la}(z_s') \neq 0$.) Indeed, we have for $z_s' \in b \Om$, 
		\gan
		0 = Re[ \ov{l^{\la}} w_\ell^{\la} (z_s') ] = \all^{\la} (z_s') Re[w_\ell^{\la} (z_s')] + \beta^{\la} (z_s') Im[w_\ell^{\la} (z_s')] , \quad 1 \leq \ell \leq 2n+1, 
		\\
		\gm^{\la}(z_s') = Re[ \ov{l^{\la}} w_{0}^{\la} (z_s') ] = \all^{\la} (z_s') Re[w_{0}^{\la} (z_s')] + \beta^{\la} (z_s') Im[w_{0}^{\la} (z_s')], 
	\end{gather*}
	where we set $l^{\la}(z_s') = \all^{\la}(z_s')+ i \beta^{\la}(z_s')$. Thus we can view the linear system \re{ipeq}-\re{bpeq} as:
	\gan
	\begin{pmatrix}
		Re[w^{\la}_{1}(z_1)] & \hdots & Re[w^{\la}_{2n+1}(z_1)]  \\
		Im[w^{\la}_{1}(z_1)] & \hdots & Im[w^{\la}_{2n+1}(z_1)] \\
		\vdots & \vdots & \vdots \\
		Re[w^{\la}_{1}(z_{N_0})] & \hdots & Re[w^{\la}_{2n+1}(z_{N_0})] \\
		Im[w^{\la}_{1}(z_{N_0})] & \hdots & Im[w^{\la}_{2n+1}(z_{N_0})] \\ 
		Re[w^{\la}_{1}(z_1')] & \hdots & Re[w^{\la}_{2n+1}(z_1')] \\
		\vdots & \vdots & \vdots \\
		Re[w^{\la}_{1}(z_{N_1}')] & \hdots & Re[w^{\la}_{2n+1}(z_{N_1}')]
	\end{pmatrix}
	\times 
	\begin{pmatrix}
		d_{1} (\la) \\ 
		\vdots \\
		d_{2n+1} (\la) 
	\end{pmatrix}
	= 
	\begin{pmatrix}
		Re(h^{\la}(z_1)) \\
		Im(h^{\la}(z_1)) \\
		\vdots \\
		Re(h^{\la}(z_{N_0})) \\
		Im(h^{\la}(z_{N_0})) \\
		\vdots \\
		Re(e^{\la}(z_{N_1}')) \\
		\vdots \\
		Re(e^{\la}(z_{N_1}'))
	\end{pmatrix}
\end{gather*}
where 
$h^{\la} (z_r)$ and $e^{\la} (z_s')$ 
denote the right-hand side of \re{ipeq} and \re{bpeq}, respectively. 
Write the above matrix equation as $M x = \textbf{v}$. By
Vekua \cite[p.~286-287]{VA62}, for fixed $\la$ and a normally distributed set $\{N_0, N_1, \ov{\Om} \}$, $M$ is invertible, thus we can write $d_s (\la)$ as $\frac{P}{Q}$, $Q \neq 0$, and $P$ and $Q$ are linear combination of products of 
\[ 
Re[w_\ell^{\la} (z_r)], \quad Im[w_\ell^{\la} (z_r)], \quad Re[w_\ell^{\la} (z_s')], \quad Im[w_\ell^{\la} (z_s')], 
\] 
\[ 
(Re, Im) [a_r(\la) + i b_r(\la) - w^{\la}_0 (z_r)], \quad Re [l^{\la}(z_s') (\gm^{\la}(z_s') + ic_{j}(\la)) - w^{\la}_0(z_s')]. 
\] 
By 
the assumption, Proposition \ref{pAfdhsct} and \ref{pAfdpsct}, for fixed $z_r$ and $z_s'$ the above functions are continuous in $\la$, thus $d_\ell$-s are continuous functions of $\la$, and moreover
by 
\re{fdpsc0est} we have 
\eq{fddsest}
| d_\ell|_0 \leq C \left(  \sum_{r=1}^{N_0} \left(|a_r|_0 + |b_r|_0 \right) + \left( \sum_{s=1}^{N_1} |c_s|_0 \right) +  \| F \|_{0,0} + \| \gm \|_{\mu, 0} \right). 
\eeq
Since $w_{0}^{\la}, w_\ell^{\la} \in \mc{C}^{k+1+\mu,0}(\ov{\Om^{\la}} )$, it follow from \re{pAscfdc1wr} that $w^{\la} \in \mc{C}^{k+1+\mu,0}(\ov{\Om})$. Putting together estimates \re{fdpsc0est}, \re{fdpshdest} and \re{fddsest} we get \re{pAscfdsupe} and \re{pAscfdhde}.
\end{proof}

\mk{} 
Since our proof uses a contradiction argument, the constant $C$ in \re{pAscfdsupe} and \re{pAscfdhde} is not explicit. In particular we do not know how it depends on the coefficient functions $A^{\la}$, $B^{\la}$ and $l^{\la}$.
\emk

We now prove the analog of \rp{pAscfdthm1} when the index is negative. 
\pr{pAscfdc2t}
Let $k$ be a non-negative integer, and let $0 < \mu <  1$. 
Consider Problem A \re{pAfd} with negative index. Suppose for each $\la$, conditions \re{scnnsc} is fulfilled. Suppose $A^{\la},B^{\la} \in \mc{C}^{k+\mu,0} (\ov{\Om})$, $l^{\la}, \gm^{\la} \in \mc{C}^{k+1+\mu,0}(b \Om)$. Then for each $\la$ there exists a unique solution $w^{\la}$ on $\Om$. The family $w^{\la} $ lies in the class $\mc{C}^{k+1+\mu,0} (\ov{\Om})$, and the following estimate hold
\[  
 \| w \|_{k+1+\mu,0} \leq C(A,B,l) \left( \| F \|_{k+\mu,0} + \| \gm \|_{k+1+\mu,0} \right). 
\] 
\epr
\begin{proof}
Reducing the problem to the canonical form \re{pA_unit_disc}. By Vekua 
\cite[p.~300, (7.32), (7.33)]{VA62}, $w^{\la}$ satisfies the integral equation \re{pAnnie} 
\eq{pAnnie2}
w^{\la} + (Q_{n'}^{\ast})^{\la} w^{\la} = P_{k}^{\ast} F^{\la} + \frac{1}{\pi i} \int_{S^1} \frac{\gm^{\la}(t) \, dt}{t^{k}(t-z)}, \quad n'=-n, 
\eeq
where $Q^{\ast}_{n'} w = P^{\ast}_{n'} (A^{\la} w^{\la} + B^{\la} \ov{w^{\la}})$ and  $P^{\ast}_{n'}$ is given by \re{Past}. 
Vekua proved that for fixed $\la$, the above equation is uniquely soluble for $w^{\la}$ given any right-hand side in $L^{q}(\ov{\D})$, for $q \geq \frac{p}{p-1}$, and $Q^{\ast}_{n'}$ is a compact operator on $L^{q}(\D)$ satisfying the estimate
\[  
|(Q_{n'}^{\ast})^{\la} f |_{k+1+\mu} \leq C |B^{\la}|_{k+\mu} |f|_{k+1+\mu}. 
\] 
Following the same argument as for the positive index case, it suffices to show the right-hand side of \re{pAnnie2} lies in the space $\mc{C}^{k+1+\mu,0}(\ov{\D} \times [0,1] )$. For the first term we have the estimate
\[ 
|P_{k}^{\ast} F^{\la} |_{k+1+\mu}  \leq C |F^{\la}|_{k+\mu} , 
\quad
|P_{k}^{\ast} F^{\la_{1}} - P_{k}^{\ast} F^{\la_{2}}|_{ } 
\leq C |F^{\la_{1}} - F^{\la_{2}} |_{0}.   
\]
The second term is the Cauchy transform of the function $g^{\la} = \gm^{\la}/ t^{k}$, and we have the estimate by \rp{Cint}:   
\[ 
|\mf{C} g^{\la}|_{k+1+\mu} \leq C |\gm^{\la}|_{k+1+\mu}, 
\quad 
|\mf{C} g^{\la_{1}} - \mf{C} g^{\la_{2}} |_{0} \leq C |\gm^{\la_{1}} - \gm^{\la_{2}}|_{1}.  
\]
Since $F^{\la} \in \mc{C}^{k+\mu,0}(\D)$ and $\gm^{\la} \in \mc{C}^{k+1+\mu,0}(S^1)$, then $P^{\ast}_{n'} F^{\la}$ and $\mf{C} (g^{\la})$ are in 
$\mc{C}^{k+1+\mu,0}(\ov{\D})$. 
\end{proof}
We now combine \rp{pAscfdthm1} and \rp{Rmp} to prove the following result for Problem A on families of simply-connected domains. We state and prove the theorem only for $n \geq 0$. When $n<0$, the same conclusion holds if we assume conditions \re{scnnsc} hold. The proof can be done similarly as in \rt{pAscvdt} by using \rp{pAscfdc2t} and we leave the details to the reader. 

\th{pAscvdt}
 Let $k, j,\mu$ be as in \rt{mtintro}. Let $\Om^{\la}$ be a $\mc{C}^{k+1+\mu,j}$ family of bounded simply-connected domains. That is, there exists a family of maps $\Gm^{\la}: \ov{\Om} \to \ov{\Om^{\la}}$, with $\Gm^{\la} \in \mc{C}^{k+1+\mu,j}(\ov{\Om})$. 
 Suppose the index $n \geq 0$. Let $ \{ z_r \in \Om, z_s' \in b \Om \}$ be a normally distributed set in $\ov{\Om}$ (\rd{ndsd}). Set $z^{\la}_r = \Gm^{\la} (z_r)$, $(z'_s)^{\la} = \Gm^{\la} (z'_s)$. Then $\{ z^{\la}_r, (z'_s)^{\la} \}$ is a normally distributed set in $\ov{\Om^{\la}}$.   
Let $w^{\la}$ be the unique solution to Problem \re{pAsc} satisfying the following: 
\begin{equation} \label{pAscvdnds1}
\begin{gathered}  
w^{\la}(z_r^{\la}) = a_r(\la) + ib_r(\la) , \quad \quad (r =1,..., N_0); \\ 
w^{\la}((z_s')^{\la}) = l^{\la}( (z_s')^{\la})[\gm^{\la}( (z_s')^{\la}) + i c_s(\la) ], \quad \quad ( s = 1,...,N_1) 
\end{gathered} 
\end{equation}
where $2N_0 + N_1 = 2n+1$. 
\\ 
(i)
Suppose $A^{\la}, B^{\la}, F^{\la} \in \mc{C}^{k + \mu, 0}(\ov{\Om^{\la}})$, and $l^{\la}, \gm^{\la} \in \mc{C}^{k+1+\mu, 0}(b \Om^{\la})$, and $a_{r} , b_{r}, c_{s} \in C^{0}([0,1])$. Then $w^{\la} \in \mc{C}^{k+1+\mu,0}(\ov{\Om^{\la}})$, and there exists some constant $C$ independent of $\la$ such that
\begin{gather} \label{pAscsupest}    
\| w^{\la} \|_{0,0} \leq C \left( \| F^{\la} \|_{0,0} + \| \gm^{\la} \|_{\mu,0} + \sum_{r=1}^{N_0} \left( |a_r |_0 + |b_r|_0 \right) + \sum_{s=1}^{N_1} |c_s|_0 \right); 
\\ 
\label{pAschdest}   
\| w^{\la} \|_{k+1+\mu,0} \leq C \left( \| F^{\la} \|_{k+\mu,0} + \| \gm^{\la} \|_{k+1+\mu,0} + \sum_{r=1}^{N_0} \left( |a_r |_0 + |b_r|_0 \right) + \sum_{s=1}^{N_1} |c_s|_0 \right). 
\end{gather} 

\medskip
\noindent
(ii) Suppose $A^{\la}, B^{\la}, F^{\la} \in \mc{C}^{k +\mu, j}(\ov{\Om^{\la}})$, and $l^{\la}, \gm^{\la} \in \mc{C}^{k+1+\mu, j}(\ov{b \Om^{\la}})$. 
Let $w^{\la}$ be the solution to Problem \re{pAsc} satisfying the condition \re{pAscvdnds1}, where $a_{r} , b_{r}, c_{s} \in C^{j}([0,1])$. Then $w \in \mc{C}^{k+1+\mu,j}(\ov{\Om^{\la}})$.  
\eth 
\begin{proof}
(i)
By \rp{Rmp}, there exists a family of Riemann mappings $R^{\la}: \Om^{\la} \to \D$ with $\wti{R}^{\la} (z) := R^{\la} (\Gm^{\la} (z) ) \in \mc{C}^{k+1+\mu, j} (\ov{\Om})$. 
By \rp{Ri}, $(R^{\la})^{-1} \in \mc{C}^{k+1+\mu, j} (\ov{\D}) $. i.e. $(R^{\la})^{-1} : \D \to \ov{\Om^{\la}}$ is also a $\mc{C}^{k+1+\mu,j}$ family of embeddings. 
Define $v^{\la} (\zeta)=  w^{\la} \circ (R^{\la})^{-1}(\zeta)$, $\zeta \in \D$, and write $z^{\la} = (R^{\la})^{-1}(\zeta)$. Since $(R^{\la})^{-1}$ is holomorphic, we have the boundary value problem for $v^{\la}$ on $\D$:
\eq{pAscvdtbvp} 
\begin{cases} 
	\pa_{\ov \zeta} v^{\la}(\zeta) + A^{\la}_1 (\zeta) v^{\la} (\zeta) + B^{\la}_1 (\zeta) \ov{v^{\la}(\zeta)}  = F^{\la}_1 (\zeta) & \text{in $\D$;}  \\
	Re [ \ov{l_1^{\la} } (\zeta) v^{\la}(\zeta) ]  = \gm_1^{\la} (\zeta)   & \text{on $S^1$, } 
\end{cases}
\eeq
where $l^{\la}_1 := l^{\la} \circ (R^{\la})^{-1}$, $\gm^{\la}_1 := \gm^{\la} \circ (R^{\la})^{-1}$, and 
\begin{gather*}
A_1^{\la} = A^{\la} \circ (R^{\la})^{-1} (\zeta) D^{\la} (\zeta), \quad  B_1^{\la} = B^{\la} \circ (R^{\la})^{-1} (\zeta) D^{\la} (\zeta), \quad
\\ \nonumber F_1^{\la} = F^{\la} \circ (R^{\la})^{-1} (\zeta) D^{\la} (\zeta), \quad D^{\la} = \pa_{\ov \zeta} \ov{(R^{\la})^{-1}}. 
\end{gather*} 
By \rp{embindle}, $ A^{\la}_1, B^{\la}_1, F^{\la}_1 \in \mc{C}^{k+\mu,j} (\ov{\D})$, and $\ov{ l_1^{\la}}$, $\gm_1^{\la} \in \mc{C}^{k+1 +\mu,j}(\ov{\D})$. Also $D^{\la} \in \mc{C}^{k+\mu,j} (\ov{\D})$. Hence $A_1^{\la}$, $B_1^{\la}, F_1^{\la} \in \mc{C}^{k+ \mu,j} (\ov{\D})$. 
The index $n = \frac{1}{2\pi} [l^{\la}]_{b \Om}$ is an integer independent  of $\la$. 
Let
\[
\zeta_r^{\la} := R^{\la} \circ \Gm^{\la} (z_r), \quad  (\zeta_s')^{\la} :=  R^{\la} \circ \Gm^{\la} (z_s). 
\]
Then for each $\la$, $\{ \zeta^{\la}_r, (\zeta'_s)^{\la}  \} $ is a normally distributed set in $\Om^{\la}$. $v^{\la}$ satisfies the condition
\begin{gather*}
v^{\la} (\zeta_r^{\la}) = w^{\la} (z_r^{\la}) = a_r (\la) + i b_r(\la), \quad r =1,..., N_0;
\\
v^{\la}( (\zeta_s' )^{\la}) = w^{\la} ( (z_s')^{\la}) 
= l^{\la}( (z_s')^{\la})[\gm^{\la}( (z_s')^{\la}) + ic_s(\la) ], \quad s = 1,...,N_1. 
\end{gather*}
Since $v^{\la} = v_0^{\la} + \sum_{\ell} d_{\ell} (\la) v_{\ell}^{\la}$, we have
\begin{gather*}
\sum_{\ell=1}^{2n+1} d_{\ell}(\la) v_{\ell}^{\la}(\zeta_{r}^{\la}) =  a_{r}(\la) + i b_{r}(\la) - v_{0}^{\la}(\zeta_{r}^{\la} ), \\ 
\sum_{\ell=1}^{2n+1} d_{\ell}(\la) v_{\ell}^{\la}( (\zeta_{s}')^{\la} ) = l^{\la}( (z_{s}')^{\la})[\gm^{\la}( (z_{s}')^{\la}) + ic_{s}(\la) ]- v_{0}^{\la}( (\zeta_{s}')^{\la}), 
\end{gather*}
where $v_{0}^{\la}$ is a particular solution to the non-homogeneous problem \re{pAscvdtbvp} as given in \rp{pAfdpsct}, and $v_{1}^{\la},...,v_{2n+1}^{\la}$ are linearly independent solutions to the corresponding homogeneous problem as given in \rp{pAfdhsct}. We have $v^{\la}_{0}, v^{\la}_{1}$, $ \dots, v^{\la}_{2n+1} \in \mc{C}^{k+1+\mu,0}(\ov{\D})$. 
Since $\zeta^{\la}_r$ and $ (\zeta'_s)^{\la}$ depend on $\la$ continuously, $v_{0}^{\la}(\zeta_{r}^{\la})$, $v_{\ell}^{\la}(\zeta_{r}^{\la})$, $v_{0}^{\la}( (\zeta_{s}')^{\la})$, $v_{\ell}^{\la}( (\zeta_{s}')^{\la})$ depend on $\la$ continuously. 
Same reasoning as in the proof of \rp{pAscfdthm1} then shows that $d_{l} \in C^{0}( [0,1] )$, and consequently $v^{\la} \in \mc{C}^{k+1+\mu,0}(\ov{\D})$, and we have
\[
\| v^{\la} \|_{k+1+\mu,0} \leq C \left( \| F_1^{\la} \|_{k+\mu,0} + \| \gm_1^{\la} \|_{k+1+\mu,0} + \sum_{r=1}^{N_0} \left( |a_r |_0 + |b_r|_0 \right) + \sum_{s=1}^{N_1} |c_s|_0 \right). 
\] 

Now $w^{\la} := v^{\la} \circ R^{\la}$ is the unique solution to Problem \re{pAsc}, $w^{\la} \circ (R^{\la})^{-1} \in \mc{C}^{k+1+\mu,0}(\ov{\D})$, and $w^{\la}$ satisfies the estimate \re{pAschdest}. 
By \rl{embindle}, $w^{\la} \circ \Gm^{\la} \in \mc{C}^{k+1+\mu,0}(\ov{\Om})$, or $w^{\la} \in \mc{C}^{k+1+\mu, 0} (\ov{\Om^{\la}} )$. 

\medskip \noindent  
(ii) The $j=0$ case is done in (i). Assume now that $j \geq 1$, and we use induction. We use the notation $D^{\la, \la_{0}} v(\zeta) $ from \re{diff_quot_not}.
In view of \re{pAscvdtbvp} $D^{\la, \la_{0}} v (\zeta)$ is a solution to 
\eq{pAscdqbvp}
\begin{cases}
	\pa_{\ov \zeta} u + A^{\la_0}_1 u + B_1^{\la_0} \ov{u} = F_{\ast}^{\la, \la_0} \quad & \text{in $\D$; }
	\\ Re[ \ov{\wti{l}^{\la_0}} u] = \gm_{\ast}^{\la, \la_0}
	& \text{on $S^1$,}
\end{cases}
\eeq
where 
$
F_{\ast}^{\la, \la_0}:=  D^{\la, \la_0} F_1 - (D^{\la, \la_0} A_1) v^{\la} - (D^{\la, \la_0} B_1) \ov{v^{\la}}$ and $
\gm_{\ast}^{\la, \la_0} := D^{\la, \la_{0}} \wti{\gm} - Re[ (D^{\la, \la_{0}} \ov{\wti{l}}) v^{\la}]. 
$ 
We also have
\begin{equation} \label{hr_gs_la}   
\begin{gathered} 
D^{\la, \la_0} v (\zeta_r) = h_r^{\la, \la_0} := D^{\la, \la_0} a _r + i D^{\la, \la_0} b_r, 
\\  
D^{\la, \la_0} v ( \zeta_s' )= g_s^{\la, \la_0} := D^{\la, \la_0} \left\{ l^{\bullet}( (\zeta_{s}')^{\bullet})[\gm^{\bullet}( (\zeta_{s}')^{\bullet}) + ic_{s}(\bullet) ]
\right\}. 
\end{gathered}
\end{equation} 

For every $\la_0$, denote by $\pa_{\la} v (\cdot, \la_0)$ the solution to the problem:
\eq{pAscladbvp}
\begin{cases}
\pa_{\ov \zeta} u + A^{\la_0}_1 u + B_1^{\la_0} \ov{u} = F_{\ast}^{\la_0} \quad & \text{in $\D$; }
\\ Re[ \ov{\wti{l}^{\la_0}} u] = \gm_{\ast}^{\la_0}
& \text{on $S^1$,}
\end{cases}
\eeq
satisfying the condition
\eq{pAscladuc} 
u^{\la_0}(\zeta_r) = h_{r} (\la_0),     \quad  u(\zeta'_{s}, \la_0) = g_s (\la_0),  
\eeq
where 
\begin{gather} \label{pAscFast}
F_{\ast}^{\la_0} :=  \pa_{\la} F_1(\cdot, \la_0) -  \pa_{\la} A_1 (\cdot, \la_0) v^{\la_0} - \pa_{\la} B_1(\cdot, \la_0) \ov{v^{\la_0}}, 
\\ \label{pAscgmast}
\gm_{\ast}^{\la_0} := \pa_{\la} \wti{\gm} (\cdot, \la_0) - Re[( \pa_{\la} \ov{\wti{l}} (\cdot, \la_0 ) v^{\la_0}], 
\\ \label{pAscaast} 
h_r(\la_0) = a_{r}'(\la_0) + i b_{r}'(\la_0), 
\quad 
g_s (\la_0) = \pa_{\la_0} \{ \wti{l^{\la}} (\zeta_{s}') [ \wti{\gm^{\la}}( \zeta_{s}') + i c_{s}(\la) ]  \}.  
\end{gather}
Note that the solution to \re{pAscladbvp} exists since the index $\frac{1}{2 \pi} [\wti{l}^{\la_0}]$  is non-negative.  

By the induction hypothesis, we have $w^{\la} \in \mc{C}^{k+\mu,j-1}(\ov{\Om^{\la}})$, and therefore $v^{\la} = w^{\la} \circ (R^{\la})^{-1} (\zeta) \in \mc{C}^{k+\mu, j-1}(\ov{\D})$. As shown in (i), $A_1, B_1, F_1 \in \mc{C}^{k+\mu, j} (\ov{\D})$, thus $\pa_{\la} A_1$, $\pa_{\la}B_1$, $\pa_{\la} F_1 \in \mc{C}^{k-1+\mu, j-1} (\ov{\D})$. 
From \re{pAscFast} we have $F_{\ast}^{\la_0} \in \mc{C}^{k-1+\mu, j-1} (\ov{\D} ) $. On the other hand we have $\pa_{\la} \wti{\gm}, \pa_{\la} \ov{\wti{l}} \in \mc{C}^{k + \mu, j-1} (\ov{\D})$, thus by \re{pAscgmast}, $\gm_{\ast}^{\la_0} \in \mc{C}^{k+\mu, j-1} (\ov{\D} ) $. 

Since $a_r', b_r', c_s' \in C^{j-1}( [0,1] )$, by \re{pAscaast} and \re{pAscaast}, $h_r, g_s \in \mc{C}^{j-1}([0,1] )$. Applying (i) to problem \re{pAscladbvp}-\re{pAscladuc}, we obtain $ \pa_{\la} v(\cdot, \la_0) \in \mc{C}^{k+\mu,j-1}(\ov{\D})$. 

We now show that $\pa_{\la} v$ is indeed the $\la$-derivative of $v^{\la}$, namely $D^{\la, \la_{0}} v$ converges to $\pa_{\la} v(\cdot, \la_0)$ uniformly on $\D$, as $\la \to \la_{0}$. 

Now $u^{\la} := D^{\la, \la_{0}} v - \pa_{\la} v(\cdot, \la_0) $ is the unique solution to
\eq{pAscgbvp}
\begin{cases}
\pa_{\ov \zeta} u^{\la} + A_1^{\la_0} u^{\la} + B_1^{\la_0} \ov{u^{\la}} = F^{\la, \la_0}_{\ast} - F^{\la_0}_{\ast} & \text{in $\D$;}
\\
Re[\ov{\wti{l^{\la_{0}}}} u^{\la}] = \gm_{\ast}^{\la, \la_0} - \gm^{\la_0}_{\ast} \quad \quad \text{on $S^1$,}
\end{cases}
\eeq
which satisfies the condition: for $r=1,\dots, N_0$ and $s=1,\dots, N_1,$
\eq{pAscguc}
\begin{split}
u^{\la}(\zeta_{r}) = h_r^{\la, \la_0} - h_r (\la_0),
\quad 
u^{\la}(\zeta_s') = g_{s}^{\la, \la_0} - g_s (\la_0).
\end{split} 
\eeq
where the right-hand sides are defined by \re{hr_gs_la}. 
We now show that $ | u^{\la} |_{0}$ converges to $0$ as  $\la \to \la_{0}$.  Reducing the problem as before to the form
\eq{pAscgrbvp}
\begin{cases}
\pa_{\ov \zeta} \hht{u}^{\la}  + A^{\la_0}_{0} \hht{u}^{\la} + B^{\la_0}_{0} \ov{\hht{u}^{\la}} = F^{\la, \la_0}_0 & \text{in $\D$;}
\\
Re[\zeta^{-n} \hht{u}^{\la}] = \gm_0^{\la, \la_0} & \text{on $S^1$}
\end{cases}
\eeq
satisfying the conditions:  
\eq{pAscgruc}
\begin{split}
\hht{u}^{\la} (\zeta_{r}) = \hht{h}_r(\la), \quad (r =1 ,...,N_0); \quad 
\hht{u}^{\la} (\zeta_s') = \hht{g}_s (\la), \quad (s=1,...,N_1) 
\end{split} 
\eeq
where
\begin{gather} \nn
\hht{u}^{\la} = e^{\chi^{\la_0}}u^{\la}, \quad A_{0}^{\la_0} = A_1^{\la_0}, \qquad B_0^{\la_0} = B_1^{\la_0} e^{2i Im(\chi^{\la_0})}, 
\\ \label{Fgm0la}
F_{0}^{\la, \la_0} = \left(  F^{\la, \la_0}_{\ast} - F^{\la_0}_{\ast} \right) e^{\chi^{\la_0}}, \qquad
\gm_0^{\la, \la_0} = \left( \gm_{\ast}^{\la, \la_0} - \gm^{\la_0}_{\ast} \right) e^{p^{\la_0}}, 
\\ \label{hrhat}  
\hht{h}_r (\la) = h_r^{\la, \la_0} - h_r (\la_0), \qquad 
\hht{g}_s (\la) = g_s^{\la, \la_0} - g_s (\la_0), 
\end{gather}
where as before $\chi^{\la_0}$ is the holomorphic function in $\D$ defined by \re{chi}. 
The solution to \re{pAscgrbvp}-\re{pAscgruc} is given by  
\eq{pAscggs}
\hht{u}^{\la} = u_0^{\la} + \sum_{\ell =1}^{2n+1} d_{\ell}(\la) u_{\ell}^{\la}, 
\eeq
where $u_{0}^{\la}$ is a particular solution which uniquely solves the integral equation: 
\[ 
u_0^{\la} + Q_{n}^{\la_{0}} u_0^{\la} = P_{n}(F_0^{\la}) + \frac{z^{n}}{2 \pi i} \int_{b
\D} \gm_0^{\la} (t) \frac{t+z}{t-z} \frac{dt}{t}.
\] 
Here $Q_{n}^{\la_0}$ and $P_{n}^{\la_0}$ are the compact operators in the space $C^{0}(\ov{\D})$ defined by \re{Qn} and \re{Pnf2}, respectively. 
Since $I + Q^{\la_{0}}_{n}$ is an injective (Vekua \cite[p.~268]{VA62}), by Fredholm alternative, it is also surjective. By the open mapping theorem, the inverse $(I - Q^{\la_{0}}_{n})^{-1}$ is continuous. Denoting its operator norm by $M_0$, we get 
\[
| u_0^{\la} |_{0} \leq M_0 \left( | F_0^{\la, \la_0} |_{0} + | \gm_0^{\la, \la_0} |_{0} \right).  
\]
Since $v \in C^{0,0}(\ov{\D})$ from part (i), we obtain from expressions \re{Fgm0la} that the right-hand side of the above equation converges to $0$ as $\la \to \la_0$. Hence $|u_0^{\la}|_0 \to 0$. 

Finally we show $d_{\ell}(\la)$ converges to $0$. We have
\begin{gather*}
\hht{h}_r (\la) = \hht{u}^{\la} (\zeta_{r}) = u_0^{\la} (\zeta_{r}) + \sum_{\ell=1}^{2n+1} d_{\ell}(\la) u_{\ell}^{\la}(\zeta_{r}), \quad r = 1, ..., N_0; \\
\hht{g}_s (\la) = \hht{u}^{\la}(\zeta_{s}') = u_0^{\la} (\zeta_{s}') + \sum_{\ell=1}^{2n+1} d_{\ell}(\la) u_{\ell}^{\la}(\zeta_{s}'), \quad  s = 1, ..., N_1.
\end{gather*}
By 
\re{hrhat}, $\hht{h}_r(\la)$, $\hht{g}_s(\la)$ converge to $0$ as $\la \to \la_0$. Also $u_0^{\la} (\zeta_r)$ and $u_0^{\la} (\zeta_s')$ all converge to $0$ as $\la \to \la_{0}$, and therefore same reasoning as before shows that $d_{\ell}(\la)$ converges to $0$. By \re{pAscggs}, $|\hht{u}^{\la}|_0$ converges to $0$, which implies $D^{\la,\la_{0}} v$ converges to $\pa_{\la} v(\cdot, \la_0) \in \mc{C}^{k+\mu, j-1}(\ov{\D})$ in the sup-norm. This completes the induction and we obtain $v \in \mc{C}^{k+1+\mu, j} (\ov{\D})$. Since $w^{\la} = v^{\la} \circ R^{\la}$, $w \in \mc{C}^{k+1+\mu,j} (\ov{\Om^{\la}})$. 
\end{proof}

\nid \tit{Proof of \rt{scposind}.} \rp{pAscvdt} gives a solution $w^{\la} \in  \mc{C}^{k+1+\mu, j} (\ov{\Om^{\la}})$. To finish the proof we need to integrate $w^{\la}$ (as in \re{wlaint}) to get a solution $U^{\la}$ for the original elliptic boundary value problem. Denote the resulting integration constant by $c_0 (\la)$. Then by combining $w^{\la} \in  \mc{C}^{k+1+\mu, j} (\ov{\Om^{\la}})$ and the condition $u^{\la} (z_0^{\la}) \in C^{j}([0,1])$ in the statement of the theorem, we get $c_0 \in C^j [0,1]$, and the proof of \rt{scposind} is complete.   \hfill \qed 

\section{Kernels with parameter} 
Starting in this section we study the parameter problem on families of multiply-connected domains. As explained in the introduction, we can no longer fix the domain while preserving $\dbar$ in the highest order term, since there is no Riemann mapping theorem in this case. More significantly, unlike in the simply-connected case where the solutions satisfy certain Fredholm equation, the solutions on multiply-connected domains in general cannot be expressed as solutions of integral equations directly. Instead we will follow Vekua to represent solutions by integral formulas \re{nhgaerf}. In this section we show how the kernels $G_1$ and $G_2$ depend on $\la$. 

For a function $f^{\la}$ defined on the domain $ \Om^{\la}$ we will adopt the following notation: 
\[
f(z, \la) := f^{\la} \circ \Gm^{\la} (z), \quad z\in \Om,  \quad \la \in [0,1].  
\]
\begin{lemma} \label{HeinLp}
Let $\Om \subset \C$ be a bounded domain with $C^1$ boundary, and $\Gm^{\la}: \ov{\Om} \to \ov{\Om^{\la}}  $ be a $C^{1,0}$-embedding. Let $1 \leq p \leq 2$. For each $\la$, suppose $f^{\la} \in L^p (\Om^{\la}) $ and $| f (\cdot, \la) |_{L^p (\Om)}$ is uniformly bounded in $\la$.  Define 
\[
T_{\Om^{\la}} f(\cdot,\la) := \iint_{\Om} \frac{f(\zeta, \la)}{\zeta^{\la} - z^{\la}} \, d A (\zeta^{\la}).  
\]
Then the following inequality hold:
\eq{Lgmb}
| T_{\Om^{\la}} f(\cdot,\la)|_{L^{\gm}_{\beta}(\Om)} \leq C_{1,0} C(r,\gm,\Om) |f(\cdot,\la)|_{L^p(\Om)}, 
\eeq
where $\gm$ is an arbitrary number satisfying the inequality
$ 1 < \gm < \frac{2p}{2-p}, $
and
$
 \beta = \frac{1}{\gm} - \frac{2-p}{2p} > 0.
$
The constant $C(p, \gm, \Om)$ depends only on $p, \gm$ and $\Om$ and does not depend on $\la$.   
\end{lemma}
\begin{proof}  
Recall definition \re{Lpall} for the norm $| \cdot |_{L^{\gm}_{\beta}}$, so we have
\begin{align*}
&
| T_{\Om^{\la}} f^{\la}(\cdot)|_{L^{\gm}_{\beta}(\Om^{\la})} 
= \left( \int_{\Om} \left| \frac{f(\zeta, \la)}{\zeta^{\la} - z^{\la}} \right|^{\gm} a^{\la}(\zeta) \, d A (\zeta) \right)^{\frac{1}{\gm}}   \\ 
&\quad  + \sup_{h \in \C} \frac{1}{|h|^{\beta}} \left( \int_{\Om} |f(\zeta + h, \la) - f(\zeta, \la) |^{\gm}  a^{\la} (\zeta) \, d A (\zeta) \right)^{\frac{1}{\gm}}
\\
& \leq C_{1,0} \left\{ \left( \int_{\Om} \left| \frac{ f(\zeta,  \la)}{\zeta- z} \right|^{\gm}  \, d | \zeta | \right)^{\frac{1}{\gm}}
+ \sup_{h \in \C} \frac{1}{|h|^{\beta}} \left( \int_{\Om} |f(\zeta + h, \la) - f(\zeta, \la) |^{\gm}  \, d A (\zeta) \right)^{\frac{1}{\gm}} \right\}  \\
&= C_{1,0}  |T_{\Om} f (\cdot, \la) |_{L^{\gm}_{\beta} (\Om) } , 
\end{align*}
where $a^{\la} (\zeta) = | \det  (\Gm^{\la}) |$. 
By the result in \cite[p.~47]{VA62}, for $\gm, \beta, p$ satisfying the above conditions, we have
$
| T_{\Om } f^{\la}(\cdot, \la)|_{L^{\gm}_{\beta}(\Om)} 
\leq C (p,\gm,\Om) |f (\cdot, \la) |_{L^p(\Om) },  
$
from which estimate \re{Lgmb} then follows.  
\end{proof}

\begin{lemma} \label{ombdd} 
Let $2 < p < \infty$. Suppose $\la \mapsto A(\cdot, \la), B(\cdot, \la)$ are continuous maps from $[0,1]$ to $L^{p}(\C)$, where for each $\la \in [0,1 ]$, $A(\cdot, \la), B(\cdot, \la) \equiv  0$ outside $\Om$. For $j =1,2$, let 
\eq{omiexp} 
\om_{j}(z,t,\la) = \frac{t^{\la}-z^{\la}}{\pi} \iint_{\Om} \frac{A(\zeta, \la) X_j(\zeta, t, \la) + B(\zeta, \la) \ov{X_{j}(\zeta,t,\la)}}{(\zeta^{\la}-z^{\la})(t^{\la} - \zeta^{\la}) X_{j}(\zeta,t, \la)} \, d A (\zeta^{\la}). 
\eeq
Then we have
\eq{omjhnbdd}
| \om_j (\cdot, t, \la)|_{\all_p} \leq C_{1,0} C(p, \Om) \left( |A(\cdot, \la)|_{L^{p}(\Om)}  + |B(\cdot, \la)|_{L^{p}(\Om)} \right), \quad \all_p = \frac{p-2}{p}, 
\eeq
where the constants is uniform in $\la$ and $t \in K$, for any compact subset $K$ of $\C$. 
In particular, if $A^{\la} , B^{\la} \in \mc{C}^{0,0}(\ov{\Om^{\la}} )$, then
\[ 
|\om_{j}(\cdot, t, \la)|_{\tau} \leq C_{1,0} C (\tau, \Om) \left( \| A^{\la} \|_{0,0} + \| B^{\la} \|_{0,0} \right), 
\quad \quad \text{for any $0 < \tau < 1$}. 
\] 
\end{lemma}
\nid \tit{Proof.}
Let $a^{\la}(\zeta) = | \det \Gm^{\la} (\zeta)|$. We can rewrite $\om_j$ as 
\begin{align*}
\om_{j}(z, t,\la) 
&= \frac{1}{\pi} \int_{\Om} \left[ \frac{1}{\zeta^{\la} -z^{\la} } - \frac{1}{\zeta^{\la} - t^{\la} } \right]  
\left[ A(\zeta, \la) + B(\zeta,\la) \frac{\ov{X_{j}(\zeta,t,\la)}}{X_{j}(\zeta,t,\la)}\right] a^{\la} (\zeta) \, d A(\zeta). 
\end{align*}  
By H\"older's inequality and \rp{dbarot}, we have for any $z, t \in \C$, 
\eq{cinoC0e}
|\om_{j}(z,t,\la)| \leq C_{1,0} C(p,\Om) \left( |A(\cdot,\la)|_{L^{p}(\Om)} + |B(\cdot,\la)|_{L^{p}(\Om)}  \right); 
\eeq
\eq{cinoC0e2}
|\om_{j}(z,t,\la)| \leq C_{1,0} C_p \left( |A(\cdot,\la)|_{L^{p}(\Om)}  + |B(\cdot,\la)|_{L^{p}(\Om)}  \right) |z-t|^{\all_p};
\eeq 
\begin{align} \label{cinoCae}
|\om_{j}(z_1, t, \la) - \om_{j}(z_2, t, \la)|   
\leq C_{1,0}C_p \left( |A(\cdot,\la)|_{L^{p}(\Om)}  + |B(\cdot,\la)|_{L^{p}(\Om)} \right) | z_1- z_2|^{\all_p} 
\end{align}
for any $z_1, z_2, t \in \C$. \hfill \qed 
\begin{lemma} \label{cinLqfs}
Let $\Om$ be a bounded domain in $\C$ with $C^{1}$ boundary, and $\Gm: \Om \to \Om^{\la} $ be a $\mc{C}^{1,0}$ embedding. Let $2 < p < \infty$, and $p' = \frac{p}{p-1}$ be the conjugate of $p$, with $1 < p' < 2$. Suppose $\la \mapsto A(\cdot, \la), B(\cdot, \la) $ are continuous maps from $[0,1]$ to $L^{p}(\C)$, where for each $\la \in [0,1]$, $A(\cdot, \la)$, $B(\cdot, \la)$ $\equiv 0$ outside $\Om$ (In particular $A,B \in L_{p,2}(\C)$, $p>2$.) Let $K$ be a compact subset of $\C$. For any $q$ with $ 1 < p' \leq q < 2$, and for each $\la \in [0,1] $, let $X_{1}(z,t,\la)$ and $X_{2}(z,t,\la)$ be the unique solution  in $L_{q,0}(\C)$ to the integral equations (see \cite[p.~156]{VA62}): 
for $i=1,2$, 
\eq{cinLqie1}
X_j (z,t, \la) - \frac{1}{\pi} \iint_{\Om} \frac{A(\zeta,\la) X_j (\zeta,t,\la) + B(\zeta,\la) \ov{X_j(\zeta,t,\la)} }{\zeta^{\la}- z^{\la}} \, d A (\zeta^{\la})
= g_i(z,t,\la),   
\eeq where $g_1 (z,t,\la)=  \frac{1}{2(t^{\la}-z^{\la})}$ and $g_2 (z,t,\la) =  \frac{1}{2i(t^{\la}-z^{\la})}$.
Then the following statements hold:\\
(i) For each $t \in \C$, $\la \mapsto X_{j}(\cdot,t,\la)$ is a continuous map from $[0,1]$ to $L^{q}(\Om)$. Moreover, for all $t \in K$ we have
\eq{Xlqest} 
\left| X_j (\cdot, t, \la) \right|_{L^{q}(\Om)} 
\leq C_{1,0} C(p, \Om) \left( \left| A(\cdot, \la) \right|_{L^{p}(\Om)}  + \left| B(\cdot, \la) \right|_{L^{p}(\Om)} \right), 
\eeq
where the constants on the right-hand side is uniform in $\la$ and $ t \in K$.  
\\
(ii) 
\[
\sup_{t \in K} |X_j (\cdot, t, \la_1) - X_j ( \cdot, t, \la_2) |_{L^q (\Om) } \to 0  \quad \text{as $|\la_1 - \la_2| \to 0$.} 
\]  
(iii) For each $\ve > 0$, let $E_{\ve} (t) = \{z \in \Om: |z-t|> \ve \}$. Then we have 
\[
\sup_{t \in K} |X_{j}(\cdot, t, \la_{1}) - X_{j} (\cdot, t,\la_{2})|_{C^{0}(E_{\ve}(t))} \to 0 \quad \text{as $|\la_{1} - \la_{2}| \to 0$.} 
\]

\end{lemma} 
\begin{proof}
(i) We shall only prove the lemma for $X_{1}$, as the proof for $X_{2}$ is identical. First, we show the $L^q$ norm of $X_j$ is uniformly bounded in $\la$ and $t \in K$. Seeking a contradiction, suppose there exists a sequence $\la_{j} \in [0,1]$ such that $N_{j} := |X_{1}(\cdot, t, \la_{j})|_{L^{q}(\Om)} \to \infty$ as $j \to \infty$. Since $[0,1]$ is compact we can assume that $\la_{j} \to \la_0 \in [0,1]$ for some $\la_0$. Dividing both sides of equation \re{cinLqie1} by $N_{j}$, and writing $\hht{X_{1}}(\cdot, t, \la_{j}) = \frac{1}{N_{j}} X_{1}(\cdot, t, \la_{j})$, we obtain 
\eq{cinLqnie}
\hht{X_{1}}(\cdot, t, \la_{j}) - P^{\la_{j}} \hht{X_{1}}(\cdot, t, \la_{j})  =  \frac{1}{2(t^{\la_{j}}-z^{\la_{j}})N_{j}},  
\eeq 
where $|\hht{X_{1}}(\cdot, t, \la_{j})|_{L^{q}(\Om)} = 1$ and we set
\gan
P^{\la} f (z, t, \la) = T_{\C} [A(\cdot, \la)f(\cdot, t, \la) + B(\cdot, \la)\ov{f(\cdot, t, \la)}] (z, t, \la). 
\end{gather*}
Setting $r=1$ in \rl{HeinLp} and applying H\"older's inequality, we have
\begin{align*}
&\left| P^{\la_{j}} \hht{X_{1}}(\cdot,t,\la_{j}) \right|_{L^{\gm}_{\all}(\Om)}  
\leq C(p, \gm, \Om) \left|  (A+B)(\cdot,\la_{j})\hht{X}_{1}(\cdot,t,\la_{j}) \right|_{L^{1}(\Om)}
\\ & \qquad \leq C(p, \gm, \Om) \left( |A(\cdot,\la_{j})|_{L^{p}(\Om)} + |B(\cdot,\la_{j})|_{L^{p}(\Om)} \right)  |\hht{X}_{1}(\cdot,t,\la_{j})|_{L^{p'}(\Om)} 
\end{align*}
for $1 < \gm < \frac{2 \cdot 1 }{2 - 1}= 2$ and $\all = \frac{1}{\gm} - \yh >0$.
In particular we can choose $\gm = q$, since $1 < p' \leq q < 2$.
By \cite[Theorem 1.3]{VA62} applied to the sequence $P^{\la_{j}} \hht{X_{1}}(\cdot,  t, \la_{j})$, and passing to a subsequence if necessary, we can assume that $P^{\la_{j}} \hht{X_{1}}(\cdot,  t, \la_{j})$ converges to some limit in $L^{q}(\Om)$.  
Meanwhile the right-hand side of 
\re{cinLqnie} converges to $0$ in $L^q(\Om)$. Therefore by \re{cinLqnie} $\hht{X_{1}}(\cdot, t, \la_{j})$ converges to some $\hht{X_{1}}(\cdot, t)$ in $L^{q}(\Om) $. 

We show that $| P^{\la_{j}} \hht{X_{1}}(\cdot,t, \la_{j}) - P^{\la_0} \hht{X_{1}}(\cdot, t) |_{L^{q}(\Om)}\to 0$ as $\la_{j} \to \la_0$. 
In what follows we write $a^{\la} (\zeta)$ for $| \det D\Gm^{\la}(\zeta) |$. 
Writing $ D^{\la_j} \hht{X_{1}} (\zeta,t) = \hht{X_{1}} (\zeta, t, \la_{j}) - \hht{X_{1}} (\zeta,t)$, we have
\eq{plajcvg}   
|P^{\la_{j}} \hht{X_{1}}(\cdot, t, \la_{j}) - P^{\la_0} \hht{X_{1}}(\cdot, t)|_{L^{q}(\Om)}
\leq D_1 + D_2 + D_3 + D_4, 
\eeq 
where 
\gan
D_1 = \left| \int_{\Om} \frac{(\del^{\la_{j}, \la_0 } A) \hht{X_{1}}(\zeta, t, \la_{j}) + A(\zeta, \la_0) (D^{\la_j} \hht{X_{1}}(\zeta,t))}{\zeta^{\la_j} - z^{\la_j} }  \, d A (\zeta^{\la_j})  |  \right|_{L^{q}(\Om)}, 
\\ 
D_2 = \left| \int_{\Om} \frac{(\del^{\la_{j}, \la_0 } B) \hht{X_{1}}(\zeta, t, \la_{j}) + B(\zeta, \la_0) \ov{D^{\la_j} \hht{X_{1}}(\zeta,t)}}{\zeta^{\la_j} - z^{\la_j} }  \, d A (\zeta^{\la_j}) | \right|_{L^{q}(\Om)}, 
\\
D_3 = \left| \int_{\Om} \frac{A(\zeta,\la_0) \hht{X_{1}} (\zeta,t,\la_0) + B(\zeta,\la_0) \ov{\hht{X_{1}}(\zeta,t,\la_0)} }{\zeta^{\la_j} - z^{\la_j} }  \left\{ a^{\la_j} (\zeta) - a^{\la_0}(\zeta) \right\} \,   d  A(\zeta) \right|_{L^{q}(\Om)}, \\
D_4 = 
\left| \int_{\Om} \left( A(\zeta,\la_0) \hht{X_{1}} (\zeta,t,\la_0) + B(\zeta,\la_0) \ov{\hht{X_{1}}(\zeta,t,\la_0)} \right) \left( \frac{1}{\zeta^{\la_j} - z^{\la_j} } - \frac{1}{\zeta^{\la_0} - z^{\la_0} }  \right) \, d A(\zeta^{\la_0}) \right|_{L^{q} (\Om) }.
\end{gather*}
By estimate \re{Lgmb} applied to $\gm = q$ and $r =1$ and H\"older's inequality, we have
\begin{align*}
D_1 &\leq C_{1,0} C(q,\Om)  \left(  | (\del^{\la_{j}, \la_0} A) \hht{X_{1}} (\cdot,t,\la) |_{L^{1}(\Om)}  + |A(\cdot, \la_0)(D^{\la_j} \hht{X_{1}}(\cdot,t))|_{L^{1}(\Om)} \right)
\\ & \leq C_{1,0} C(q,\Om) \left( | \del^{\la_{j}, \la_0} A |_{L^{p}(\Om)} |\hht{X_{1}} (\cdot,t,\la)|_{L^{q}(\Om)} + |A(\zeta, \la_0)|_{L^{p}(\Om)} |D^{\la_j} \hht{X_{1}}|_{L^{q}(\Om)} \right), 
\end{align*}
where $ 1 < p'  \leq  q < 2$. By assumption, this expression goes to $0$. Similarly we can show that $D_2$ converges to $0$, and also
$$
D_3 \leq C_{1,0} C(q,\Om) \left| \Gm^{\la_j} - \Gm^{\la_0} \right|_1 \left\{ \left| A(\cdot, \la_0) \right|_{L^{p}(\Om)} + 
\left| B(\cdot, \la_0) \right|_{L^{p}(\Om)} \right\}  |\hht{X_{1}} (\cdot,t,\la_0)|_{L^{q}(\Om)},
$$
\begin{align*}
D_4 &\leq  \left| \int_{\Om} \left( A(\zeta,\la_0) \hht{X_{1}} (\zeta,t,\la_0) + B(\zeta,\la_0) \ov{\hht{X_{1}}(\zeta,t,\la_0)} \right) \frac{\del_{\zeta, z} (\Gm^{\la_0} - \Gm^{\la_j}) \, d A (\zeta^{\la_0}) }{(\zeta^{\la_j} - z^{\la_j})(\zeta^{\la_0} - z^{\la_0})} 
 \right|_{L^{q}_{z}(\Om)} \\
&\leq C_{1,0} C(q,\Om)\left| \Gm^{\la_0} - \Gm^{\la_j} \right|_{1} \left\{ \left| A(\cdot, \la_0) \right|_{L^{p}(\Om)} + 
\left| B(\cdot, \la_0) \right|_{L^{p}(\Om)} \right\}  |\hht{X_{1}}  (\cdot,t,\la_0)|_{L^{q}(\Om)}.  
\end{align*}
Hence $D_3, D_4 \to 0$ as $|\la_j - \la_0| \to 0$. By \re{plajcvg}, $|P^{\la_{j}} \hht{X_{1}}(\cdot, t, \la_{j}) - P^{\la_0} \hht{X_{1}}(\cdot, t)|_{L^{q}(\Om)}$ tends to $0$ as $| \la_j - \la_0 | \to 0$.  

Letting $j \to \infty$ in equation \re{cinLqnie}, we have
\begin{align*}
0 &= \hht{X_{1}}(z,t) - T_{\C} [A(\cdot, \la_0) \hht{X_{1}}(\cdot,t) + B(\cdot, \la_0) \ov{ \hht{X_{1}}(\cdot,t)} ] (z,t) 
\end{align*}
holds in $L^q(\Om)$. 
By the result in  \cite[p.~156]{VA62}, the equation has only the trivial solution, so $\hht{X_{1}} \equiv 0$. But $|\hht{X_{1}}(\cdot, t)|_{L^{q}(\Om)} = \lim_{j\to \infty} |\hht{X_{1}} (\cdot, t, \la_{j})|_{L^{q}(\Om)} = 1$, which is a contradiction. Thus $|X_{1}(\cdot,t, \la)|_{L^{q}(\Om)}$ is bounded, for any $q$ with $1< p' \leq q < 2$. 

Next, we show that $X_{1}(\cdot, t, \la)$ is continuous in the $L^{q}(\Om)$-norm, for any $q$ with $1< p' \leq q < 2$. Otherwise there exists some sequence $\la_{j} \to \la_0$, such that $|X_{1}(\cdot,t,\la_{j}) - X_{1}(\cdot, t, \la_0) |_{L^{q}(\Om)} \geq \del >0$. 
Since we have already shown that $|X_{1}(\cdot,t,\la)|_{L^{q}(\Om)}$ is bounded uniformly in $\la$, by the computation above, $|P^{\la_{j}} X_{1}(\cdot,t,\la_{j})|_{L^{q}_{\all}(\Om)}$ is bounded uniformly in $\la$. 
Passing to a subsequence if necessary, $P^{\la_{j}}X_{1}(\cdot, t,\la_{j})$ converges to some function in $L^{q}(\Om)$. The right-hand side of the integral equation \re{cinLqie1} converges to $\frac{1}{t^{\la_0} - z^{\la_0} }$ in $L^{q}( {\Om})$. Indeed, we have
\[ 
\left( \iint_{\Om} \left| \frac{1}{t^{\la_{j}}  - z^{\la_j} } - \frac{1}{t^{\la_0} - z^{\la_0} } \right|^{q} \, d A(z)  \right)^{\frac{1}{q}} 
\leq C_{1,0} |\Gm^{\la_{j}} - \Gm^{\la_0}|_{1} \left( \iint_{\Om} \frac{1}{|t-z|^{q}} \, d A(z) \right)^{\frac{1}{q}}, 
\]
which converges to $0$ as $| \la_j - \la_0 | \to 0$, since $ \Gm \in C^{1,0} (\ov{\Om})$.  
It follows from equation \re{cinLqie1} that $X_{1}(z,  t,\la_{j})$ converges to some $X^{\ast}_{1}(z,t)$ in $L^{q}(\Om)$. As before we can show that $|P^{\la_{j}} X_{1}(\cdot,t,\la_{j}) - P^{\la_0} X_{1}^{\ast}(\cdot, t)|_{L^{q}(\Om)} \to 0$.  Taking $j \to \infty$ in equation \re{cinLqie1} we obtain 
\[
X_{1}^{\ast}(z,t) - P^{\la_0} X_{1}^{\ast}(z,t) = \frac{1}{2(t^{\la_0}-z^{\la_0})}, \quad \quad \text{in $L^{q}(\Om)$}. 
\]
On the other hand, since $X_1(z,t, \la_0)$ satisfies the same equation. By uniqueness of solution  we have
$X_{1}(\cdot, t, \la_0) = X_{1}^{\ast}(\cdot,t)$ in $L^{q}(\Om)$. But $|X_{1}^{\ast}(\cdot,t) - X_{1}(\cdot,t, \la_0)|_{L^{q}(\ov{\Om})} \geq \del > 0$, which is a contradiction. This shows that of $X_{1}(\cdot, t, \la)$ in continuous in $\la$ in the $L^{q}(\Om)$- norm.  

Finally to obtain estimate \re{Xlqest} we write $X_1$ in the form \re{Xjrep}: 
\eq{X1rep}
X_1 (z,t, \la) = \frac{e^{\om_{1}(z,t,\la)}}{t^{\la}-z^{\la}}.  
\eeq
By   \re{omjhnbdd}, we have
$ 
|\om_{j}(\cdot, t, \la)|_{C^{\all_p}(\ov{\Om})}
\leq C_{1,0} C(p, \Om) \left( |A(\cdot, \la)|_{L^{p}(\Om)}  + |B(\cdot, \la)|_{L^{p}(\Om)} \right),
$ 
where $ \all_p = \frac{p-2}{p}$, and the constants are independent of $\la$ and $t \in K$. Estimate \re{Xlqest} then readily follows. 
\hfill \qed
\\ \\
(ii)
As in part (i) we only need to prove the statements for $X_1$. Let $D^{\la_1, \la_2}X_1 = X_1(z,t,\la_1) - X_1(z,t,\la_2)$. Then by 
\re{cinLqfs},  we see that $D^{\la_1, \la_2} X_1$ satisfies the following integral equation
\[
D^{\la_1, \la_2} X_1(z,t) + P^{\la_2} (D^{\la_1, \la_2} X_1 ) (z,t) = \sum_{i=1}^{4} R_i (z,t, \la_1), 
\] 
where 
$
P^{\la_2} f (z,t) = T_{\C} \left[ A(\cdot, \la_2) f(\cdot, t) + B(\cdot, \la_2) \ov{f(\cdot, t)} \right] $ and
\gan
R_1 (z,t, \la_1) = \iint_{\Om} \frac{\left( A(\cdot, \la_1) - A(\cdot, \la_2) \right) X_1(\cdot, t, \la_1) }{\zeta^{\la_1} - z^{\la_1} } \, dA (\zeta^{\la_1}), 
\\
R_2 (z,t, \la_1) = \iint_{\Om} \frac{A(\cdot, \la_2)X(\cdot, t, \la_1)(a^{\la_1} - a^{\la_2}) }{\zeta^{\la_1} - z^{\la_1}} \, d A(\zeta), 
\\ 
R_3 (z,t, \la_1) = \iint_{\Om} A(\cdot, \la_2) X(\cdot, t, \la_1) \left( \frac{1}{\zeta^{\la_1} - z^{\la_1} } - \frac{1}{\zeta^{\la_2} - z^{\la_2} } \right) \, d A (\zeta^{\la_2}), 
\\
R_4 (z,t, \la_1) = \frac{1}{t^{\la_1} - z^{\la_1}} - \frac{1}{t^{\la_2} - z^{\la_2}}. 
\end{gather*}  
Since $|X_1 (\cdot, t, \la) |_{L^{q}(\Om)}$ is bounded, we can show by the same proof as in (i) that $\sup_{t \in K}$ $|R_i (z,t, \la_1) |_{L^{q}(\Om)}$ converges to $0$ as $|\la_1 - \la_2| \to 0$, $1 \leq i \leq 4$. Let $(I + P^{\la_2})^{-1}: L^{q} (\Om) \to L^{q} (\Om)$ be the inverse operator of $I + P^{\la_2}$. By the open mapping theorem, $(I + P^{\la_2})^{-1}$ is continuous. Denoting its operator norm by $M_{\la_2}$, we have
\[
\sup_{t \in K} |D^{\la_1, \la_2} X_1 (\cdot, t) |_{L^{q}(\Om)} 
\leq M_{\la_2}  \sum_{i=1}^{4} \sup_{t \in K} \left| R_i (z,t, \la_1) \right|_{L^{q}(\Om)},
\]
which converges to $0$ as $|\la_1 - \la_2 | \to 0$. \hfill \qed
\\ \\ 
(iii) We write $X_1$ in the form \re{X1rep}, 
where $\om_1$ is given by expression \re{omiexp}. 
By \rl{ombdd}, if $|A(\cdot,\la)|_{L^{p}(\Om)}$, $|B(\cdot,\la)|_{L^{p}(\Om)}$ are bounded uniformly in $\la$, then $|\om_{1}(\cdot,t, \la)|_{C^{\all_p}(\ov{\Om})}$ is bounded uniformly in $\la \in [0,1]$ and $t \in K$, for 
$\all_p = \frac{p-2}{p}$. In view of \re{X1rep},  $|X_{1}(\cdot,t,\la)|_{C^{\all_p}(E_{\ve})}$ is bounded uniformly in  $\la$ and $t$. By part (ii), $\sup_{t \in K} |X_{1}(\cdot, t, \la_{1}) - X_{1}(\cdot, t, \la_{2}) |_{L^{1}(\Om)}$ converges to $0$ as $| \la_1 - \la_2| \to 0$. Applying \rp{L1HnC0} we get
\[ 
\sup_{t \in K} |X_{1}(\cdot, t, \la_{1}) - X_{1}(\cdot, t, \la_{2}) |_{C^{0}(\Om \cap E_{\ve})} \ra 0. \qquad \qed 
\]

\begin{lemma} \label{cino}
Let $2 < p < \infty$. Suppose $\la \mapsto A(\cdot, \la), B(\cdot, \la)$ are continuous maps from $[0,1]$ to $L^{p}(\C)$, where for each $\la \in [0,1 ]$, $A(\cdot, \la), B(\cdot, \la) \equiv  0$ outside $\Om$. For $j =1,2$, let 
\begin{equation*} 
\om_{j}(z,t,\la) = \frac{t^{\la}-z^{\la}}{\pi} \iint_{\Om^{\la}} \frac{A(\zeta, \la) X_j(\zeta, t, \la) + B(\zeta, \la) \ov{X_{j}(\zeta,t,\la)}}{(\zeta^{\la}-z^{\la})(t^{\la} - \zeta^{\la}) X_{j}(\zeta,t, \la)}\, d A (\zeta^{\la}). 
\end{equation*} 
Then for each $t$, $\la \mapsto \om_{j}(\cdot, t,\la) \in C^{0}(\ov{\Om})$ is a continuous map from $[0,1]$ to $C^0 (\ov{\Om})$. Moreover for any compact subset $K$ in $\C$, we have 
\eq{omjlacont}
\sup_{t \in K } |\om_{j}(\cdot, t, \la_{1}) - \om_{j}(\cdot,t, \la_{2})|_{C^{0}(\ov{\Om})} \to 0 \quad \text{as $| \la_1- \la_2| \to 0$.}
\eeq
\end{lemma} 
\nid \tit{Proof.} 
Rewrite $\om_j$ as 
\[
\om_{j}(z,t,\la) 
= \frac{1}{\pi} \iint_{\Om} \left[ \frac{1}{\zeta^{\la} -z^{\la} } - \frac{1}{\zeta^{\la} - t^{\la} } \right]  \\ 
\left[ A(\zeta, \la) + B(\zeta,\la) \frac{\ov{X_{j}(\zeta,t,\la)}}{X_{j}(\zeta,t,\la)}\right] a(\zeta, \la) \, dA(\zeta), 
\]
where $a(\zeta, \la) = | (\det \Gm^{\la}) (\zeta ) | $. We have
\eq{cinoe}
|\om_{j}(z,t, \la_{1})- \om_{j}(z,t, \la_{2})| \leq D_1  + D_2 + D_3, 
\eeq
where
\begin{align*} 
D_1 &= \iint_{\C} \left\{  \left| \frac{1}{\zeta^{\la_1} - z^{\la_{1}}} - \frac{1}{\zeta^{\la_2} - z^{\la_{2}}} \right|  + \left| \frac{1}{\zeta^{\la_1} - t^{\la_{1}}} - \frac{1}{\zeta^{\la_2} - t^{\la_{2}}} \right|  \right\} 
\\ &\qquad \times  \left| A(\zeta,\la_{1}) + B(\zeta,\la_{1}) \frac{\ov{X_{j}(\zeta,t,\la_{1})}}{X_{j}(\zeta,t,\la_{1})} \right| \, d A (\zeta^{\la_1});  
\end{align*}
\begin{align*} 
D_2 = \iint_{\C} \left| \frac{1}{\zeta^{\la_{2}}-z^{\la_{2}}} + \frac{1}{t^{\la_{2}} - \zeta^{\la_{2}}} \right| |(Aa)(\zeta, \la_{1}) - (Aa)(\zeta, \la_{2})| \, d A(\zeta); 
\end{align*}
\begin{align*} 
D_3 &= \iint_{\C} \left| \frac{1}{\zeta^{\la_{2}}-z^{\la_{2}}} + \frac{1}{t^{\la_{2}} - \zeta^{\la_{2}}} \right| \left| (Ba)(\zeta, \la_{1}) \frac{\ov{X_{j}(\zeta,t,\la_{1}) }}{X_{j}(\zeta,t,\la_{1})} - (Ba)(\zeta,\la_{2}) \frac{\ov{X_{j}(\zeta,t,\la_{2})} }{X_{j}(\zeta,t,\la_{2})}
\right| \, d A(\zeta) . 
\end{align*}
Let $p'$ be the conjugate of $p$, with $1< p' < 2$. By H\"older's inequality we have
\begin{align*} 
D_1 & \leq C_{1,0} |\Gm^{\la_{1}} - \Gm^{\la_{2}} |_{1} 
\int_{\Om} \left( \frac{1}{|\zeta -z|} + \frac{1}{|\zeta -t|} \right) 
\left| A(\zeta,\la_{1}) + B(\zeta,\la_{1}) \frac{\ov{X_{j}(\zeta,t,\la_{1})}}{X_{j}(\zeta,t,\la_{1})} \right| \, d A (\zeta^{\la_1})  \\
&\leq C_{1,0} |\Gm^{\la_{1}} - \Gm^{\la_{2}} |_{1} \left(  \| (\zeta-z)^{-1} \|_{L^{p'} (\Om)} + \| (\zeta-t )^{-1} \|_{L^{p'} (\Om) } \right) \left( |A(\cdot, \la)|_{L^{p}(\Om)} + |B(\cdot, \la)|_{L^{p}(\Om)} \right) 
\\ &\leq C_{1,0} C_p |\Gm^{\la_{1}} - \Gm^{\la_{2}} |_{1}  \left( |A(\cdot, \la)|_{L^{p}(\Om)} + |B(\cdot, \la)|_{L^{p}(\Om)} \right), 
\end{align*} 
and similarly for $D_2$:
\[  
D_2 \leq C_{1,0} C_p \left( |\del^{\la_{1}, \la_{2}} A|_{L^{p} (\Om)} + |A(\cdot, \la)|_{L^{p} (\Om)}  |\del^{\la_{1}, \la_{2}}  \Gm|_{1} \right). 
\] 
Hence by our assumption, $D_1$ and $D_2$ converge to $0$ and the convergence is uniform in $t \in K$.
For $D_3$ we have $D_3 \leq D_{31} + D_{32}$, 
where 
\begin{gather*}
D_{31} = 
\int_{\C} \left| \frac{1}{\zeta^{\la_{2}}-z^{\la_{2}} } + \frac{1}{t^{\la_{2}}-\zeta^{\la_{2}} } \right| \left| \del^{\la_1,\la_2} (Ba) \right| \left| \frac{\ov{X_{j}(\zeta,t, \la_{1}) }}{X_{j}(\zeta,t,\la_{1})} \right| \, d A(\zeta), 
\\
D_{32} =\int_{\C} \left| \frac{(Ba)(\zeta,\la_{2})}{\zeta^{\la_{2}}-z^{\la_{2}}} + \frac{(Ba)(\zeta,\la_{2})}{t^{\la_{2}}-\zeta^{\la_{2}} } \right|  \left|  \frac{ \ov{X_{j}(\zeta,t,\la_{1})} }{X_{j}(\zeta,t,\la_{1})} - \frac{\ov{X_{j}(\zeta,t,\la_{2}) }}{X_{j}(\zeta,t,\la_{2})} \right| \, d A(\zeta). 
\end{gather*}
By H\"older's inequality, 
$
D_{31} \leq C_{1,0} C_p |(Ba)(\cdot,\la_{1}) - (Ba)(\cdot,\la_{2}) |_{L^{p}(\Om)}, 
$ 
which by assumption goes to $0$ as $|\la_{1} - \la_{2}| \to 0$. On the other hand, 
\begin{align} \label{cinoie}
& D_{32} \leq 2  \int_{\C} \left| \frac{(Ba)(\zeta,\la_{2})}{\zeta^{\la_{2}}-z^{\la_{2}}} + \frac{(Ba)(\zeta,\la_{2})}{t^{\la_{2}}-\zeta^{\la_{2}}}  \right| \left| \frac{ X_{j}(\zeta,t,\la_{1}) - X_{j}(\zeta,t,\la_{2})}{X_{j}(\zeta,t,\la_{1})} \right| \, d A(\zeta).  
\end{align} 
By \re{cinoC0e}, $|\om_{j}(\zeta,t,\la)|$ is bounded uniformly for any $\zeta, t \in \C $ and $\la \in [0,1] $, and 
\[
X_{1}(\zeta,t,\la) = \frac{e^{\om_{1}(\zeta,t,\la)}}{2(t^{\la}-\zeta^{\la})}, \quad X_{2}(\zeta,t,\la) = \frac{e^{\om_{2}(\zeta,t,\la)}}{2i (t^{\la}-\zeta^{\la})}. 
\]
Hence $ 
\left| X_{j}(\zeta,t,\la) \right| \geq \frac{C_{1,0}}{|\zeta -t|}. 
$ 
In what follows we write $L^1_{\zeta} (\Om)$ to indicate the integral is with respect to $\zeta$. Using the above estimate in \re{cinoie} we have
\begin{align*} 
D_{32} &\leq C_{1,0} \left( \left| \frac{B(\zeta,\la_{2})}{\zeta^{\la_{2}} -z^{\la_{2}} } \left[ \del^{\la_{1}, \la_{2}} X_{j}(\zeta,t)  \right] \right|_{L^{1}_{\zeta}(\Om)} + \left| B(\zeta, \la_{2}) \left[ \del^{\la_{1}, \la_{2}} X_{j}(\zeta,t)  \right] \right|_{L^{1}_{\zeta}(\Om)} \right) \\ \nn 
& \leq C_{1,0} \left| B(\cdot, \la_{2}) \right|_{L^{p}(\ov{\Om})} \left( \left| \frac{\del^{\la_{1}, \la_{2}} X_{j}(\zeta,t) }{\zeta^{\la_{2}} - z^{\la_{2}}} \right|_{L^{p'}_{\zeta}(\ov{\Om})} + \left| \del^{\la_{1}, \la_{2}} X_{j}(\zeta,t)  \right|_{L^{p'}_{\zeta}(\ov{\Om})} \right) \\ 
& \leq C_{1,0} \left| B (\cdot,\la_{2}) \right|_{L^{p}(\ov{\Om})} \left\{  \left( \int_{ \{\zeta : |\zeta- z|< \del' \}} \left| \frac{\del^{\la_{1}, \la_{2}}X_{j}(\zeta,t) }{\zeta-z} \right|^{p'} \, d A(\zeta) \right)^{\frac{1}{p'}} \right. \\ \nn
& \quad + \left. \left( \int_{ \{\zeta : |\zeta- z| \geq \del' \}} \left| \frac{\del^{\la_{1}, \la_{2}}X_{j}(\zeta,t) }{\zeta-z} \right|^{p'} \, d A(\zeta) \right)^{\frac{1}{p'}} + \left| \del^{\la_{1}, \la_{2}}X_{j}(\zeta,t) \right|_{L^{p'}_{\zeta}(\ov{\Om})} \right\}, 
\end{align*}
where $\del' > 0$ is to be determined. 
By \rl{cinLqfs} (ii), $\sup_{t \in K} | \del^{\la_{1}, \la_{2}}X_{j}(\cdot,t)  |_{L^{p'}(\Om)} \to 0$ as $|\la_{1} - \la_{2}| \to 0$. So for fixed $\del'$ the last two terms in the above bracket converge to $0$.
By estimate \re{cinoC0e2},
\begin{align*}
|\om_{j}(z,t, \la_{1}) - \om_{j}(z,t, \la_{2}) | 
&\leq  |\om_{j}(z,t, \la_{1}) |  + | \om_{j}(z,t, \la_{2}) | \\ 
&\leq C_{1,0} C_p \left( |A(\cdot,\la)|_{L^{p}(\ov{\Om})}  + |B(\cdot,\la)|_{L^{p}(\ov{\Om})}  \right) |z-t|^{\all_p}. 
\end{align*}
For any $\ve>0$, we can pick $r>0$ such that for any $z \in \ov{\Om}$ and $t \in K$ with $|z-t|<r$, one has  $|\om_{j}(z,t, \la_{1}) - \om_{j}(z,t, \la_{2})| < \ve$. Then it suffices to consider the points $z$ for which $|z-t| \geq r$. 

Now for this fixed $r$, we can let $\del' = \frac{r}{2}$ be such that for any $z$ with $|z-t|\geq r$ and  $|\zeta-z| < \del'$, we have $|\zeta-t| \geq |t-z| - |z-\zeta| \geq \del'= \frac{r}{2}$, and so  
\[ 
\{ \zeta: | \zeta-z | < \del'  \}  \subset \{  \zeta: |\zeta  - t| \geq \del' \}.
\]
Hence 
\begin{align*} 
\left( \int_{ \{\zeta : |\zeta- z|< \del' \}} \left| \frac{\del^{\la_{1}, \la_{2}}X_{j}(\zeta,t) }{\zeta-z} \right|^{p'} \, d A(\zeta) \right)^{\frac{1}{p'}} 
&\leq \left( \int_{ \{\zeta : |\zeta- t| \geq \del' \}} \left| \frac{\del^{\la_{1},  \la_{2}}X_{j}(\zeta,t) }{\zeta-z} \right|^{p'} \, d A(\zeta) \right)^{\frac{1}{p'}} \\  
&\leq C_{p'} \sup_{t \in K } |\del^{\la_{1},  \la_{2}}X_{j}(\zeta,t) |_{C^{0}(\{ \zeta: |\zeta-t|\geq \del'\} )}. 
\end{align*}
The last expression goes to $0$ as $|\la_{1} - \la_{2}| \to 0$, by \rl{cinLqfs} (iii). This shows that $D_{32}$ converges to $0$ as $|\la_1 - \la_2| \to 0$ and the convergence is uniform in $t \in K$. The proof of \re{omjlacont} is now complete. 
\end{proof} 

\section{Families of multiply-connected domains}  

In this section we prove \rt{mtintro} for families of bounded multiply-connected domains. We will first follow Vekua to reduce the problem to a singular integral equation. We then obtain regularities in all variables by transforming it into an equivalent Fredholm equation, so that a compactness argument can be used.     

Just like in the simply-connected case, it suffices to consider the corresponding problem in the form \re{pAsc}. 

\begin{thm}[Problem A on multiply-connected domains] \label{RHmt}  
 Let $k, j,\mu$ be as in \rt{mtintro}.
Let $\Om$ be a bounded domain in $\C$ whose boundary is $C^{1+\mu}$ and which has $m+1$ connected components. 
Let $\Gm^{\la}: \Om \to \Om^{\la}$ be a family of embeddings of class $\mc{C}^{1,0}(\Om)$. 
Fix a normally distributed set (\rd{ndsd}) $ \{ \{z_{r}^{\la} \}_{r=1}^{N_0}$, $\{z_{s}^{\la} \}_{s=1}^{N_1} \}$ in $\Om$ so that $ \{ z_r^{\la} = \Gm^{\la} (z_r ), z_a^{\la} = \Gm^{\la} (z_s) \}$ forms a normally distributed set in $\Om^{\la}$.  For each $\la$, let $w^{\la}$ be the unique solution to 
\begin{equation} \label{pAmc} 
\begin{cases}
\pa_{\ov{z}} w^{\la}(z^{\la}) + A^{\la}(z^{\la}) w^{\la}(z^{\la}) + B^{\la}(z^{\la}) \ov{w^{\la}(z^{\la})} = F^{\la}(z^{\la}) & \text {in $\Om^{\la}$;} \\ 
Re[\ov{l^{\la}(z^{\la})} w^{\la}(z^{\la})] = \gm^{\la}(z^{\la}) & \text {on $b \Om^{\la}$, }
\end{cases}
\end{equation}
where the index $n$ (\rd{pA_index}) satisfies $n > m-1$ and $w^{\la}$ satisfies the additional conditions
\begin{equation} \label{mcwucip}
\begin{gathered} 
w^{\la} (z_r^{\la}) = a_r(\la) + ib_r(\la), \quad \quad (r =1,...,N_0); 
\\ 
w^{\la} ((z_s')^{\la}) = l_{j}(z_s', \la) (\gm_{j}(z_s', \la) + ic_s (\la)),  \quad \quad ( s= 1,...,N_1). 
\end{gathered} 
\end{equation} 
(i) Suppose for $2<p<\infty$, $\la \to A^{\la}, B^{\la}, F^{\la}$ are continuous maps from $[0,1]$ to $L^p (\Om)$, and that $l^{\la}, \gm^{\la} \in \mc{C}^{\mu,0}(b \Om^{\la})$, $a_r, b_r, c_s \in C^0 ([0,1])$.  Then $w^{\la} \in \mc{C}^{\nu,0}(\ov{\Om^{\la}})$, where $\nu = min(\all_p, \mu )$ and $\all_p = \frac{p-2}{p}$. 
\\ 
(ii) If $b \Om \in C^{k+1+\mu}$, $\Gm^{\la} \in \mc{C}^{k+1+\mu,0}(\ov{\Om})$, $A^{\la}, B^{\la}, F^{\la} \in \mc{C}^{k+\mu,0}(\ov{\Om^{\la}})$, $l^{\la}, \gm^{\la} \in \mc{C}^{k+1+\mu,0}(b \Om^{\la})$, then $w^{\la} \in \mc{C}^{k+1+\mu,0}(\ov{\Om^{\la}})$, and there exists some constant $C$ independent of $\la$ such that
\begin{gather*} 
\| w^{\la} \|_{0,0} \leq C \left( \| F^{\la} \|_{0,0} + \| \gm^{\la} \|_{\mu,0} + \sum_{r=1}^{N_0} \left( |a_r |_0 + |b_r|_0 \right) + \sum_{s=1}^{N_1} |c_s|_0 \right), 
\\  \nn 
\| w^{\la} \|_{k+1+\mu,0} \leq C \left( \| F^{\la} \|_{k+\mu,0} + \| \gm^{\la} \|_{k+1+\mu,0} + \sum_{r=1}^{N_0} (|a_r |_0 + |b_r|_0) + \sum_{s=1}^{N_1} |c_s|_0 \right). 
\end{gather*}  
\\
(iii) If $b \Om \in C^{k+1+\mu}$, $\Gm^{\la} \in \mc{C}^{k+1+\mu,j}(\ov{\Om})$, $A^{\la}, B^{\la}, F^{\la} \in \mc{C}^{k+\mu,j}(\ov{\Om^{\la}})$, $l^{\la}, \gm^{\la} \in \mc{C}^{k+1+\mu,j}(b \Om^{\la} )$, then $w^{\la} \in \mc{C}^{k+1+\mu,j}(\ov{\Om^{\la}} )$, and there exists some constant $C$ independent of $\la$ such that 
\[
\| w^{\la} \|_{k+1+\mu, j} \leq C \left( \| F^{\la} \|_{k+\mu,j} + \| \gm^{\la} \|_{k+1+\mu,j} + \sum_{r=1}^{N_0} \left( |a_r |_j + |b_r|_j \right) + \sum_{s=1}^{N_1}|c_s|_j \right).  
\]

\end{thm}

As an immediate consequence of \rt{RHmt}, we can now prove \rt{mtintro} for the case $\varkappa > m-1$, where $\varkappa$ is the index for the oblique derivative boundary value problem. (See Definition \ref{odpind}).   
\begin{thm}[Multiply-connected domains, $\varkappa > m-1$ ] \label{odthm} \ \\
 Let $k, j,\mu$ be as in \rt{mtintro}. Let $\Om$ be a bounded domain in $\C$ with $\mc{C}^{k+2+\mu}$ boundary. Let $\Gm^{\la}$ be a family of maps that embed $\Om$ to $\Om^{\la}$, with $\Gm^{\la} \in \mc{C}^{k+2+\mu,j}(\ov{\Om})$. Consider the following oblique derivative problem:
\begin{equation}  \label{mteqn}
\begin{gathered} 
a^{\la} (x,y) u^{\la}_{xx}+ 2b^{\la} (x,y) u^{\la}_{xy} + c^{\la} (x, y) u^{\la}_{yy} +
d^{\la}(x,y) u^{\la}_x + e^{\la}(x,y) u^{\la}_y = f^{\la},  \quad \text{in $ \Om^{\la}$; }   \\	
\frac{du^{\la}}{d \Upsilon^{\la}} = \gm^{\la} \quad \text{on $b \Om^{\la}$, } 
\end{gathered} 
\end{equation} 
with index $\varkappa > m-1$. Here $\Upsilon^{\la}= \xi^{\la} + i\eta^{\la}$ is a non-vanishing vector field along $b \Om^{\la}$. Suppose $a^{\la}, b^{\la}, c^{\la}, d^{\la}, e^{\la}, f^{\la}$ are functions in $\Om^{\la}$ such that $a^{\la}, b^{\la}, c^{\la} \in \mc{C}^{k+1+\mu,j}(\ov{\Om^{\la}})$, and $d^{\la} , e^{\la}, f^{\la} \in \mc{C}^{k+\mu,j} (\ov{\Om^{\la}})$. 
Suppose further that $\Upsilon^{\la}, \gm^{\la}$ are functions in $b \Om^{\la}$ such that $\Upsilon^{\la}, \gm^{\la} \in \mc{C}^{k+1+\mu,j} (b \Om^{\la})$. 
For each $\la \in [0,1]$, let $ \{ \{z_{r}^{\la} \}_{r=1}^{N_0}$, $\{z_{s}^{\la} \}_{s=1}^{N_1} \}$ be a normally distributed set for $\Om^{\la}$, and let $u^{\la}$ be the unique solution to \re{mteqn} on $\Om^{\la}$ satisfying the conditions (see \cite[Theorem 4.13]{VA62} )  
\begin{gather} \nn 
u^{\la}_x(z_r^{\la}) - i u^{\la}_y (z_r^{\la}) = a_r(\la) + i b_r(\la) \quad \quad (r=1,...,N_0); 
\\   \nn 
u^{\la}_x((z_s')^{\la}) - i u^{\la}_y ((z_s')^{\la}) = l^{\la}((z_{s}')^{\la}) (\gm^{\la}((z_{s}')^{\la}) + ic_{s}(\la)),  \quad \quad (s = 1,...,N_1);
\\ \label{obbvpuc3} 
u^{\la} (z_0^{\la}) = e (\la), \quad z_0^{\la} \in \Om^{\la}, 
\end{gather}
where $a_r, b_r, c_s, e \in C^j ([0,1])$. Then $u^{\la} \in \mc{C}^{k+2+\mu,j} (\ov{\Om^{\la}} )$.  
\end{thm}  
\begin{proof} 
Let $\var^{\la}$ be the maps defined as in \rp{isop}. Then $\mc{U}^{\la} := u^{\la} \circ (\var^{\la})^{-1}$ is a solution to the following problem  
\begin{gather*} 
\Del \, \mc{U}^{\la} + p^{\la} (\xi, \eta) \, \mc{U}^{\la}_{\xi} + q^{\la} (\xi, \eta) \, \mc{U}^{\la}_{\eta} = h^{\la} (\xi, \eta), \quad \text{in $D^{\la}$;} \\   
\nu_1^{\la} \, \mc{U}^{\la}_{\xi} + \nu_2^{\la} \, \mc{U}^{\la}_{\eta} =  g^{\la} (\xi, \eta), \quad \text{on $bD^{\la}$,} 
\end{gather*}  
where $p^{\la}, q^{\la}, h^{\la} \in \mc{C}^{k+\mu,j} (\ov{D^{\la}})$ and $\nu_1^{\la}, \nu_2^{\la}, g^{\la} \in \mc{C}^{k+1+\mu,j} (\ov{D^{\la}})$. Furthermore $\mc{U}^{\la}$ satisfies 
\begin{gather*} \label{Lapuc1}  
\mc{U}^{\la}_{\xi}(\zeta_r^{\la}) - i \, \mc{U}^{\la}_{\eta} (\zeta_r^{\la}) = \wti{a_r}(\la) + i \wti{b_r}(\la) \quad \quad (r=1,...,N_0); 
\\  \label{Lapuc2}  
\mc{U}^{\la}_{\xi}((\zeta_s')^{\la}) - i \, \mc{U}^{\la}_{\eta} ((\zeta_s')^{\la}) 
= \wti{c}_s (\la) + i \wti{d}_{s} (\la),  \quad \quad (s = 1,...,N_1); 
\\  \label{Lapuc3}   
\mc{U}^{\la} (\zeta_0^{\la}) = \wti{e}(\la), \quad z_0^{\la} \in \Om^{\la},    
\end{gather*} 
where $\{ \zeta_r^{\la}, (\zeta'_s)^{\la} \}$ is the image of $\{z^{\la}_r, (z'_s)^{\la} \}$ under the map $\var^{\la}$, and $\wti{a_r}, \wti{b_r}, \wti{c_s}, \wti{d_s}, \wti{e} \in C^j([0,1])$.  

Let $w^{\la} = \mc{U}^{\la}_{\xi} - i \, \mc{U}_{\eta}^{\la}$ be the unique solution in $\mc{C}^{k+1+\mu,j}(\ov{\Om})$ to problem \re{pAmc} as given in \rt{RHmt},  where we take
\begin{gather*}  
A^{\la} = \yf \left( p^{\la} + iq^{\la} \right) , \   B^{\la} = \yf \left( p^{\la} - iq^{\la} \right), 
\quad F^{\la} = \yh h^{\la}, 
\quad
l^{\la} = \nu_1^{\la} - i \nu_2^{\la}, \quad \gm^{\la} = g^{\la}.  
\end{gather*}   
and $w^{\la}$ satisfies
\begin{gather*}
w^{\la} (\zeta_{r}^{\la}) = \wti{a_r}(\la) + i \wti{b_r}(\la), 
\quad 
w^{\la} ( (\zeta'_s)^{\la}) = \wti{c}_s (\la) + i \wti{d}_{s} (\la). 
\end{gather*} 
for $r=1,\dots, N_0$ and $s=1,\dots, N_1$. 
Set 
\begin{gather*}
\mc{U}^{\la}( \zeta^{\la}) = c_0(\la) + Re \int_{\zeta_0^{\la}}^{z^{\la}} w^{\la}(\zeta) \, d \zeta^{\la}, \quad z^{\la} = x^{\la} +iy^{\la} \in \Om^{\la}. 
\\
u^{\la} (z^{\la}) = \mc{U}^{\la} (\var^{\la} (z^{\la})) 
= c_0 (\la)+ Re \int_{\zeta_0^{\la}}^{  \var^{\la} (z^{\la}) } w^{\la}(\zeta) \, d \zeta^{\la}. 
\end{gather*} 
By \re{obbvpuc3} and $e \in C^j([0,1])$, we have $c_0 \in C^j([0,1])$. Hence $u^{\la} \in \mc{C}^{k+2+\mu,j} (\ov{\Om^{\la}})$.   
\end{proof}

For the rest of the section we prove \rt{RHmt}. 
In \S 6.1 and \S 6.2 we prove part (i). We reduce problem A \re{pAmc} to a singular integral equation with parameter. We then transform the singular integral equation into a Fredholm equation and derive all the relevant estimates. In the end we apply a standard compactness argument to show that the solution $w^{\la} \in \mc{C}^{\nu',0}(b \Om^{\la})$, for any $0 < \nu' < \nu $, where $\nu = \min(\all_p, \mu)$, $\all_p = \frac{p-2}{p}$. We then use the $\mc{C}^{\nu-, 0 } (b \Om^{\la})$ estimate to prove the $\mc{C}^{\nu, 0} (b \Om^{\la})$ estimate. For our proof we apply a trick of Vekua to reduce the problem to the case for families of simply-connected domains.  
In \S 6.3, we derive the $\mc{C}^{k+1+\mu,0}(b  \Om^{\la})$ estimate for $w^{\la}$. 
In \S 6.4, we take the derivatives in $\la$ and derive the $\mc{C}^{k+1+\mu, j}(b  \Om^{\la})$ estimate for $w^{\la}$. This completes the proof. 

\subsection{The $\mc{C}^{\nu^{-},0}$ estimate} 
\ \\ \\
Recall the notation: $f(z, \la): = f^{\la}(z^{\la})$, for $(z, \la) \in \ov{\Om} \times [0,1]$.  
\\ \\
\textbf{Step 1: Setup of the singular integral equation.}  

For each $\la$, let $w^{\la}$ be the unique solution to the boundary value problem \re{pAmc} satisfying condition \re{mcwucip}. 
Since $Re[\ov{l(t,\la)} w(t,\la)] = \gm(t,\la)$ on $b \Om$, there exists a real-valued function $\eta$ such that 
\eq{wbdyvalue}
w(t, \la) = l(t, \la) \gm(t,\la) + i l(t,\la) \eta(t, \la),  \quad (t,\la)  \in b \Om \times [0,1]. 
\eeq   
Then the solution $w^{\la}$ can be represented by formula \re{nhgaerf}: 
\eq{pAcpwrf}  
w(z,\la) = w_{\gm}(z,\la) + w_{\eta}(z,\la) + w_{F}(z,\la), \quad z \in \Om,
\eeq 
where
\begin{align} \label{pAcpsrf1}
w_{\gm} (z,\la) &:= \frac{1}{2 \pi i} \int_{b \Om} G_{1}(z,t,\la) l (t,\la) \gm(t, \la) \, dt^{\la}
-\frac{1}{2 \pi i}  \int_{b \Om} G_{2}(z,t,\la) \ov{l(t,\la)} \gm(t,\la) \, \ov{dt^{\la}}, 
\end{align}
\begin{align}  \label{pAcpsrf2}
w_{\eta} (z,\la) &:= \frac{1}{2 \pi} \int_{b \Om} G_{1}(z,t,\la) l(t,\la) \eta(t, \la) \, dt^{\la} 
- \frac{1}{2 \pi} \int_{b \Om} G_{2}(z,t,\la) \ov{l (t,\la)} \eta(t,\la) \, \ov{dt^{\la}}, 
\end{align}
\begin{align} \label{pAcpsrf3}
w_{F}(z,\la) &:= - \frac{1}{\pi} \iint_{\Om} G_{1}(z, \zeta,\la) F(\zeta,\la) \, d A(\zeta^{\la}) 
- \frac{1}{\pi} \iint_{\Om} G_{2}(z, \zeta, \la) \ov{F(\zeta,\la)} \, d A(\zeta^{\la}).
\end{align}
Here $G_i (z, t, \la)$ is the fundamental kernel of the equation: 
\[
\pa_{\ov{z}} w + A(z, \la) w + B(z, \la) \ov{w} = 0, \quad z \in \C,  
\]
where for $z \notin \Om$, we take $A(z, \la) \equiv B (z, \la) \equiv 0$.  

According to \re{Gi}, $G_i(z,t,\la)$ has the representations: 
\[
G_1(z, t, \la) = \frac{e^{\om_1(z,t,\la)} + e^{\om_2(z,t,\la)} }{t^{\la} - z^{\la}}, 
\quad 
G_2(z, t, \la) = \frac{e^{\om_1(z,t,\la)} - e^{\om_2(z,t,\la)} }{t^{\la} - z^{\la}}, 
\]
where $\om_i(z,t, \la)-s$ are given by formula \re{omiexp}.

By 
\re{fke}, it can be shown easily that for fixed $\la$ and $F(\cdot, \la) \in L^{p}(\Om)$, $w_{F}(\cdot, \la)$ is continuous in $\mathbb C$. 

Referring to \cite[p.~234]{VA62}, $\eta$ satisfies the singular integral equation
\eq{pAfdsie}
\mc{K}^{\la} \eta (\zeta, \la) = \int_{b \Om} K (\zeta,t,\la)  \eta(t,\la) \,ds  = \gm_{0}(\zeta, \la), 
\eeq
where 
\begin{gather}\label{pAfdsieRHS} 
\gm_{0} (\zeta, \la) = \gm(\zeta, \la) - Re [\ov{l(\zeta,\la)}  w_{\gm}^{+}(\zeta, \la)] - Re[ \ov{l(\zeta, \la)} w_{F}(\zeta, \la)], 
\\
\nonumber
K(\zeta, t, \la) =  - \frac{1}{2 \pi} Re \left[ l(t, \la) (\rho^{\la})'(s)   
\left( G_{1}(\zeta,t,\la) \ov{l(\zeta,\la)} + \ov{G_{2}(\zeta,t,\la)} l(\zeta,\la) \right) \right]. 
\end{gather}
In addition, $\eta$ satisfies the following  conditions:
\begin{equation} \label{c1muucip}
\begin{gathered}  
w_{\eta}(z_r, \la) = a_r(\la) + i b_r(\la) - w_{\gm}(z_r, \la) - w_{F}(z_r, \la)  
\\  
w_{\eta}^{+} (z_s', \la) = l(z_s', \la)( \gm(z_s', \la) + i c_s(\la)) - w_{\gm}^{+} (z_s', \la) -   w_{F}(z_s', \la) 
\end{gathered}
\end{equation}
where $\{ z_r, z_s'\}$ for $ r =1, \dots, N_0$ and $ s= 1, \dots, N_1$ is a normally distributed set on $\Om$. Here we use $w_{\gm}^{+}$ (resp.$w_{\eta}^{+}$) to denote the limiting value of $w_{\gm}(z, \la)$ (resp.$w_{\eta}^{+} (z,\la)$) as $z$ approaches to $z_s' \in b \Om$ from the interior of $\Om$. For fixed $\la$,  $\eta(\cdot, \la) \in C^{\nu}(\ov{\Om})$ by \rp{pArt} (i), and thus according to \rp{gCijft}, $w_{\eta}^{+}$ is given by   \re{gCijf1}. 

Referring to \cite[p.~235]{VA62}, the homogeneous integral equation 
$
 \mc{K}^{\la} \eta (\zeta, \la) = 0
$
has $2n+1$ linearly independent solutions. 
We apply the theory of one dimensional singular integrals. The reader can refer to the classic book of Muskhelishvili~\cite{MI92}.
\\ \\
\nid \textbf{Step 2: Estimates for $w_{\gm}^{\la}$ and $w_F^{\la}$. }  
\ \\ 
We first show that the right-hand side of \re{pAfdsie} is in $C^{\nu,0}(\ov{\Om^{\la}})$. 
\begin{lemma} \label{wgmle} 
Let $\gm^{\la} \in \mc{C}^{\mu,0}(b \Om^{\la})$, and let $w_{\gm}$ be defined by formula \re{pAcpsrf1}, 
then $w_{\gm}^{\la} \in \mc{C}^{\mu, 0} (b \Om^{\la})$.  
\end{lemma}
\nid \tit{Proof.} 
By equations \re{Gieq} $w_{\gm}^{\la}$ satisfies the equation
\[ 
\pa_{\ov{z^{\la}}} w_{\gm}  + A(z, \la) w_{\gm}(z, \la)  + B(z, \la) \ov{w_{\gm}(z, \la)} = 0. 
\]
By \rp{gCith}, $w_{\gm}^{\la}$ is the unique solution to the integral equation: 
\eq{wgmie}
w_{\gm}(z, \la) + P^{\la} w_{\gm} (z, \la) = \Phi(z, \la), \quad z \in \Om, 
\eeq
where
\begin{gather*}  
P^{\la} w_{\gm} (z, \la) := T_ {\Om^{\la}}(A^{\la}w^{\la}_{\gm}+B^{\la}\ov{w^{\la}_{\gm}})(z^{\la}), \quad 
\Phi(z, \la) := \frac{1}{2 \pi i} \int_{b \Om} \frac{(l \gm) (t, \la) }{t^{\la} -z^{\la}} \, dt^{\la}. 
\end{gather*} 
By \rp{gCiHne} we know for fixed $\la$, $w_{\gm}(\cdot, \la) \in \mc{C}^{\nu}(\ov{\Om})$, where $\nu = \min (\all_p, \mu)$, $\all_p =\frac{p-2}{p}$.    

We make a few observations: \\
1) By \rp{dbarot}, we have
\begin{align} \label{PwgmHe}
\left| P^{\la} w_{\gm} (\cdot,\la) \right|_{C^{\all_p}(\ov{\Om})}  
&\leq C_{1,0} C_{p} (|A(\cdot, \la)|_{L^{p}(\ov{\Om})} + |B(\cdot, \la)|_{L^{p}(\ov{\Om})} ) |w_{\gm}(\cdot, \la)|_{C^{0}(\ov{\Om})}. 
\end{align}
In other words, the operator $P^{\la}$ is a compact operator for each fixed $\la$.  
\\ \\
2) By standard estimates for Cauchy integral (or see \cite[p.~176]{B-G14}), we have
\eq{wgmPhie}
\| \Phi \|_{0,0} \leq  C_{1,0}  \| l \gm \|_{\mu,0} , \quad 
\| \Phi \|_{\mu,0} \leq C_{1+\mu,0} \| l \gm \|_{\mu,0}. 
\eeq
3) The homogeneous integral equation 
\[ 
w_{\gm}(z, \la) + T_{\Om^{\la}}(A^{\la}w^{\la}_{\gm}+B^{\la}\ov{w^{\la}_{\gm}})(z^{\la}) = 0, \quad z \in \Om  
\]
has only the trivial solution for each $\la$. 
For the existence and uniqueness, see the argument in the proof of \rp{gCith}.

First, we show that $|w_{\gm}(\cdot, \la)|_0$ is bounded uniformly in $\la$. Seeking contradiction, suppose that for some sequence $\{\la_{j}\}$ converging to $\la_0$, one has $N_{j} := |\gm (\cdot,\la_{j})|_{0} \to \infty$. Dividing both sides of \re{wgmie} by $N_{j}$, we have  
\eq{pAnhomie}  
\hht{w_{\gm}}(z, \la_j) + P^{\la_j} \hht{w_{\gm}} (z, \la_j) = \frac{1}{N_j} \Phi(z, \la_j), \quad z \in \Om, 
\eeq 
where $ \hht{w_{\gm}}(z, \la_j) = \frac{1}{N_j} \wti{w_{\gm}} (z, \la_j)$. 
By estimate \re{PwgmHe}, $| P^{\la} \hht{w_{\gm} } (\cdot, \la)|_{\all_p}$ is bounded by a constant uniform in $\la$. 
By Arzela-Ascoli theorem and passing to a subsequence if necessary, $P^{\la_j} \hht{w_{\gm} }$ converges uniformly to some function. The right-hand side of the above equation converges to $0$ by \re{wgmPhie}, so in view of \re{pAnhomie} $\hht{w_{\gm}}$ converges uniformly to some function $\wti{w}$. One can check that $P^{\la_j} \hht{w_{\gm}}$ converges uniformly to $P^{\la_0} \wti{w}$, thus passing to the limit as $j \to \infty$ we get
$
\wti{w} + P^{\la_0} \wti{w} = 0. 
$  
By the above observation 3), $\wti{w}\equiv 0$ which contradicts the fact that $| \wti{w}|_0 =$ $\lim_{j \to \infty}$  $|\hht{w_{\gm}} (\cdot, \la)|_0 =  1$. Hence $|w_{\gm}(\cdot, \la)|_0$ is bounded uniformly in $\la$.
Applying estimate \re{PwgmHe} we get
\begin{align*}
\| w_{\gm}^{\la} \|_{\nu,0} 
\leq \| P^{\la} w_{\gm}^{\la} \|_{\all_p,0} + \| \Phi^{\la}  \|_{\mu,0}  
\leq C (\| w_{\gm}^{\la} \|_{0,0} + \| l^{\la} \gm^{\la} \|_{\mu,0}), 
\end{align*}
where $C$ is some constant independent of $\la$.

Next we show $w_{\gm}( \cdot, \la)$ is continuous in $\la$ in the $C^{0}(b \Om)$-norm. Suppose this is not the case, then there exists some $\la_0$ and a sequence $\la_{j} \to \la_0$, such that $|w_{\gm} (\cdot,\la_{j}) - w_{\gm}  (\cdot, \la_0)|_{0} \geq \ve > 0$. 
Since we have shown that $|w_{\gm}^{\la}|_{\nu}$ is bounded uniformly in $\la$. Passing to a subsequence if necessary, we have that $w_{\gm}^{\la_j} $ converges to some $w_{\ast}$ in the $C^{\nu'}(b \Om)$-norm, for $0 < \nu' < \nu$. 
As before we can show that $|P^{\la_j}w_{\gm}^{\la_j} - P^{\la_0}w_{\ast} |_{0} \to 0 $. 
Letting $| \la_j - \la_0 | \to 0$ in \re{wgmie} we get
$
w_{\ast}(z) + P^{\la_0} w_{\ast} (z, \la) = \Phi(z, \la_0)$ for $z \in \Om$. 
Since we know that $w^{\la_0}_{\gm}$ satisfies the same equation, again by the above observation 3) we must have 
$
 w_{\ast} = w_{\gm} (\cdot,  \la_0 ). 
$
But this contradicts the fact that $| w_{\gm} (\cdot, \la_j) - w_{\gm} (\cdot, \la_0) | \geq \ve > 0$. Hence $w_{\gm} (\cdot, \la)$ is continuous in $\la$ in the $C^0 (b \Om)$ norm. By \rp{cinhd}, the results imply 
$w_{\gm}^{\la} \in \mc{C}^{\nu,0} (b \Om^{\la})$.
\hfill \qed

Next we show that $w_F^{\la} \in \mc{C}^{\all_p,0}(b \Om^{\la})$.  
\begin{lemma} \label{wFle}
Let $2 < p < \infty$ and $\all_p = \frac{p-2}{p}$. Suppose $\la \mapsto A(\cdot, \la), B (\cdot, \la), F (\cdot, \la)$ are continuous maps from $[0,1]$ to $L^p(\Om)$. For $z \in \Om$, let $w^{\la}_F$ be given by formula \re{pAcpsrf3}. 
Then for each $\la$, $w_F^{\la}$ can be extended to a function in $C^{\all_p}(\ov{\Om})$, and $w^{\la}_F \in \mc{C}^{\all_p, 0} (\ov{\Om^{\la}})$. Moreover we have the following estimate
$$
\| w_F^{\la} \|_{\all_p,0} \leq  C(p, \Om) 
(\sup_{\la \in [0,1] }  |A (\cdot, \la) |_{L^p (\Om) }  
+ \sup_{\la \in [0,1] }  |B (\cdot, \la) |_{L^p (\Om) } ) , 
$$ 
          where the constant depends only on $p$ and $\Om$ and is independent of $\la$.  
\end{lemma}
\begin{proof} 
By H\"older's inequality we have
\begin{equation*}
\begin{aligned}
|w_{F}(z,\la)|
&\leq C \left( |G_{1}(\zeta, \cdot, \la) |_{L^{p'}(\Om)} + |G_{2}(\zeta, \cdot, \la) |_{L^{p'}(\Om)} \right) |F(\cdot, \la) |_{L^{p}(\Om)}  \\ 
&\leq C_{1,0} C_{p} \left( \Si_{i=1,2} \| \om_{i} \|_{0,0} \right)
|F(\cdot, \la)|_{L^{p}(\Om)} ,  \quad 2 < p < \infty, 
\end{aligned}
\end{equation*} 
where $\| \om_i \|_{0,0} = \sup_{\zeta \in \ov{\Om}, \la \in [0,1] } |\om_i(\cdot, t, \la)|_0$.  We have
\begin{equation*} 
\begin{aligned} 
&|w_{F}(z, \la) - w_{F}(z', \la)| 
\\  & \quad \leq C \sum_{i=1,2} \iint_{\Om} \left| G_{i}(z, \zeta, \la) - G_{i} (z', \zeta, \la) \right| |F(\zeta,\la)| \, d A(\zeta^{\la})  \\ 
& \quad \leq C \iint_{\Om} \frac{  \Si_{i=1,2} \left| e^{\om_{i}(z, \zeta,\la)} - e^{\om_{i}(z',\zeta,\la)} \right|}{|\zeta^{\la} - z^{\la}|}  |F(\zeta, \la)| \, d A(\zeta^{\la}) \\ 
& \qquad + C \iint_{\Om} \left( \Si_{i=1,2} \left| e^{\om_{i}(z', \zeta,\la)} \right| \right) \left| \frac{1}{\zeta^{\la}   - z^{\la}} - \frac{1}{\zeta^{\la}-(z')^{\la}} \right| |F(t,\la)| \, d A(\zeta^{\la}) .   
\end{aligned}
\end{equation*} 
Denote the first integral above by $I$ and the second integral by $II$. Applying H\"older's inequality for $p> 2$ we have
\begin{align*}
I &\leq C_{1,0} (\Si_{i=1,2} \|\om_{i} \|_{\all_p,0}) |z - z'|^{\all_p} \iint_{\Om} 
\frac{|F(\zeta,\la)|}{|z - \zeta|} \, dA(\zeta) 
\\& \leq C_{1,0} C_{p} (\Si_{i=1,2}  \| \om_{i} \|_{\all_p,0} ) |F(\cdot, \la)|_{L^{p}(\Om)} |z - z'|^{\all_p} , 
\end{align*}
where $\| \om_i \|_{\all_p,0} := \sup_{\zeta \in \ov{\Om}, \la \in [0,1] } | \om_i (\cdot , t, \la) |_{\all_p}$. Also
\begin{align*} 
II &= \left\{ \iint_{|z- \zeta| < 2|z - z'|} +  \iint_{| z - \zeta | \geq |z - z'|}  \right\} 
\left( \Si_{i=1,2} \left| e^{\om_{i}(z', \zeta,\la)} \right| \right) \left| \frac{1}{z^{\la} - \zeta^{\la}} - \frac{1}{(z')^{\la} - \zeta^{\la}} \right| |F(\zeta,\la)| \, d A(\zeta^{\la})  
\\ & \leq C_{1,0} (\Si_{i=1,2} \| \om_{i} \|_{0,0} ) \left\{ \iint_{|z - \zeta| < 2|z - z'|} \frac{|F(\zeta,\la)|}{|z - \zeta|} \, dA(\zeta) + \iint_{|z' - \zeta| < 3|z - z'|} \frac{|F(\zeta,\la)|}{|z' -  \zeta|} \, d A(\zeta) \right. 
\\ & \qquad \left. + \iint_{|z - \zeta| \geq 2| z - z'|} \left| \frac{1}{z^{\la} - \zeta^{\la}} - \frac{1}{(z')^{\la} - \zeta^{\la}} \right| |F(\zeta,\la)|  \,  d A (\zeta) \right\}  
\\ &\leq C_{1,0} (\Si_{i=1,2} \| \om_{i} \|_{0,0} ) \left\{  | F(\cdot, \la) |_{L^{p}(\ov{\Om})} |z - z'|^{\all_p} + |z - z'| \iint_{|z - \zeta| \geq 2|z - z'|}  \frac{|F(\zeta,\la)|}{|z- \zeta | |z' - \zeta |}  \, d A(\zeta)  \right\} 
\\ &\leq C_{1,0} C_{p} (\Si_{i=1,2} \| \om_{i} \|_{0,0} ) | F(\cdot, \la) |_{L^{p}(\ov{\Om})} \left\{ |z - z'|^{\all_p} + |z - z'| (C + |z - z'|^{-\frac{2}{p}})  \right\} 
\\ &\leq C_{1,0} C_{p} (\Si_{i=1,2} \| \om_{i} \|_{0,0} ) | F(\cdot, \la) |_{L^{p}(\ov{\Om})} | z - z'|^{\all_p}. 
\end{align*} 
Thus we obtain 
\begin{equation} \label{wflahest}
\begin{aligned}  
|w_F(\cdot, \la)|_{\all_p} 
&\leq C_{1,0} C_{p}  (\Si_{i=1,2} \| \om_{i} \|_{\all_p,0} ) | F(\cdot, \la) |_{L^{p} (\Om)}
\leq C  \left( |A(\cdot, \la)|_{L^{p}(\Om)}  + |B(\cdot, \la)|_{L^{p}(\Om)} \right) , 
\end{aligned}  
\end{equation} 
where we used \re{omjhnbdd}, and the constant $C$ depends only on $p$ and $\Om$ and is independent of $\la$. 

Next we show $w_{F}^{\la}$ is continuous in $C^{0}$-norm. For simplicity we omit the constant $- \frac{1}{\pi}$. Set
\[  
d A (\zeta^{\la}) = a(\zeta, \la) d A(\zeta),  \quad a (\zeta, \la) := |det(D \Gm^{\la})|. 
\] 
We have, using notation \re{del_note} for $\del^{\la_1,\la_2}$. 
\begin{align*}  
& |w_{F} (z, \la_{1}) -  w_{F} (z, \la_{2}) | 
\leq \sum_{i=1,2} \left| \iint_{\Om}  \del^{\la_{1}, \la_{2}} [G_{i}(z, \zeta) F(t) a (\zeta)] \, d A(\zeta) \right| \\ 
& \quad \leq \sum_{i=1,2} \iint_{\Om}  [\del^{\la_{1}, \la_{2}} G_{i}(z, t)] (Fa)(\zeta, \la_{1})  \, d A(\zeta) + \iint_{\Om} G_{i}(z, \zeta,\la_{2}) [\del^{\la_{1}, \la_{2}}(Fa)(t)] \, d A(\zeta) \\ 
& \quad \leq \iint_{\Om} \frac{ \Si_{i=1,2} |e^{\om_{i}(z, \zeta,\la_{1})} - e^{ \om_{i}(z, \zeta,\la_{2})}|}{|\zeta^{\la_{1}} - z^{\la_{1}} |} |(Fa)(\zeta, \la_{1})| + \iint_{\Om} \Si_{i} |e^{\om_{i}(z ,\zeta,\la_{2})}| |(Fa)(\zeta, \la_{1})| \\
& \qquad \left| \frac{1}{\zeta^{\la_{1}} - z^{\la_{1}} } - \frac{1}{ \zeta^{\la_{2}} - z^{\la_{2}} }\right| \, d A(\zeta) +  \iint_{\Om} \frac{ \Si_{i} |e^{\om_{i}(z, \zeta,\la_{2})}|}{|\zeta^{\la_{2}} - z^{\la_{2}}|} |\del^{\la_{1}, \la_{2}}(Fa)(t)| \, d A(\zeta). 
\end{align*} 
By H\"older's inequality, this is bounded by 
\[ 
C_{1,0}C_{p} \left\{ \Si_{i} \sup_{\zeta \in \ov{\Om}}|\del^{\la_{1}, \la_{2}} \om_{i}(\cdot, \zeta)|_{0} |F|_{L^{p}(\Om)} + 
(\Si_{i} \| \om_{i} \|_{0,0}) \left( |\Gm^{\la_{1}} - \Gm^{\la_{2}}|_{1} |F|_{L^{p}(\Om)}  + |  \del^{\la_{1},  \la_{2}}F|_{L^{p}(\Om)}  \right) \right\}, 
\] 
where the constants are independent of $\la$. 
By \rl{cino} (iii), $ \sup_{\zeta  \in \ov{\Om}} |\om_{i}(z, \zeta, \la_{1}) - \om_{i}(z, \zeta, \la_{2})|_{0}  \to 0$ , so the first term above goes to $0$. Since $\Gm^{\la} \in \mc{C}^{1,0}$, and $F(\cdot,\la)$ is continuous in $\la$ in the $L^{p}(\Om)$-norm, the second and third terms above converge to $0$ as $|\la_{1} - \la_{2}| \to 0$. 
\end{proof}
\textbf{Step 3: Reduction of the singular integral operator. }  
\\ \\
To estimate $w_{\eta}^{\la}$ (\re{pAcpsrf2}), we first need to estimate $\eta(\cdot, \la)$. 
Write 
\[ 
\mc{K}^{\la} \eta (\zeta, \la) = \mc{K}_1^{\la}\eta (\zeta, \la) + \mc{N}_1^{\la} \eta (\zeta, \la),   
\]
where
$ 
\mc{K}_{1}^{\la} \eta (\zeta, \la) := - \frac{1}{2 \pi} \int_{b \Om} Re \left[ l(t, \la) \ov{l(\zeta, \la)}  (\rho^{\la})'(s)  
G_{1}(\zeta,t, \la)  \right] \eta(t, \la) \, ds, 
$ and 
\begin{gather} \label{N1opt}  
\mc{N}_1^{\la} \eta_{1}(\zeta, \la) := - \frac{1}{2 \pi} \int_{b \Om} Re \left[ l(t, \la) l(\zeta, \la) (\rho^{\la})'(s)   \ov{G_{2}(\zeta,t,\la)}  \right]  \eta (t, \la) \, ds. 
\end{gather}   
The operator $\mc{N}_{1}^{\la}$ is compact, since by estimate \re{cinoC0e2}, 
\begin{align*} 
&|G_{2}(\zeta,t, \la) |
= \left| \frac{e^{\om_{1} (\zeta,t, \la)} - e^{\om_{2} (\zeta,t, \la)}}{2 (t^{\la} - \zeta^{\la})}  \right|  
\leq C \frac{| \om_{1} (\zeta,t, \la) - \om_{2} (\zeta,t, \la) | }{| t^{\la} - \zeta^{\la} | } \\
&\leq C \frac{| \om_{1} (\zeta,t, \la)| + |\om_{2} (\zeta,t, \la)| }{t- \zeta}  \leq  C_{1,0} C_p \left( |A(\cdot,\la)|_{L^{p}(\Om)}  + |B(\cdot,\la)|_{L^{p}(\Om)}  \right)  |t - \zeta|^{- \frac{2}{p}} , 
\end{align*} 
where $2 < p < \infty$. 
We shall call $\mc{N}_{1}^{\la}$ a \emph{Fredholm operator of the first kind}, using the same terminologies as in \cite{MI92}. 

On the other hand, using that $w_j (t,t, \la) = 0$ for all $t$ and estimate \re{cinoCae}, we have
\begin{align*}
G_{1}(\zeta,t, \la)  -  \frac{1}{t^{\la} - \zeta^{\la} }
&= \frac{e^{\om_{1} (\zeta,t, \la)} + e^{\om_{2} (\zeta,t,\la)}}{2 (t^{\la} - \zeta^{\la}) } - \frac{2}{2( t^{\la}- \zeta^{\la}) } \\ 
&= \frac{e^{\om_{1} (\zeta,t, \la)} - e^{\om_1(t, t, \la) }}{2 (t^{\la} - \zeta^{\la})} + \frac{e^{\om_{2} (\zeta,t, \la)} - e^{\om_2 (t, t, \la)}}{2 (t^{\la} - \zeta^{\la})} \\
&\leq  C_{1,0} C_p \left( |A(\cdot,\la)|_{L^{p}(\Om)}  + |B(\cdot,\la)|_{L^{p}(\Om)}  \right)|\zeta - t|^{- \frac{2}{p}}. 
\end{align*}
It follows that 
$ 
\mc{K}^{\la}_1 \eta (\zeta, \la)  = \mc{K}_{2}^{\la} \eta(\zeta, \la) + \mc{N}_{2}^{\la} \eta(\zeta, \la),  
$ 
where 
\begin{gather}
\mc{K}^{\la}_2 \eta(\zeta, \la) 
:= - \frac{1}{2 \pi} \int_{b \Om} Re\left[ \frac{l(t, \la) (\rho^{\la}) ' (s) \ov{l(\zeta, \la)} }{t^{\la}-\zeta^{\la} } \right] \eta(t, \la) \, ds; \nonumber\\
\label{N2opt}   
\mc{N}_2^{\la} \eta (\zeta, \la) 
:= - \frac{1}{2 \pi} \int_{b \Om} Re \left[ l(t, \la) \ov{l(\zeta, \la)} (\rho^{\la}) ' (s)   \sum_{i=1,2} \frac{e^{\om_{i} (\zeta,t, \la)} - e^{\om_i(t,t,\la)}}{2 (t^{\la} - \zeta^{\la} )} \right] \eta(t, \la) \, ds. 
\end{gather}
The operator $\mc{N}_{2}^{\la}$ is compact, and
\[  
\mc{K}_{2}^{\la} \eta(\zeta, \la) 
= - \frac{1}{4 \pi} \int_{b \Om} \left[ \frac{l(t, \la) (\rho^{\la}) '(s)  \ov{l(\zeta, \la)} }{t^{\la}-\zeta^{\la} } + 
\frac{\ov{l(t, \la)} \, \ov{(\rho^{\la}) ' (s) }l(\zeta, \la)}{\ov{t^{\la} - \zeta^{\la}}} \right] \eta(t, \la) \, ds. 
\] 
Since
\begin{align*}
\frac{ds}{\ov{t^{\la} - \zeta^{\la}} } 
= \frac{ ( \rho^{\la}) '(s) \, d \ov{t^{\la} } }{\ov{t^{\la} - \zeta^{\la} }} 
&=  (\rho^{\la} ) '(s) \, d \log(\ov{t} - \ov{\zeta^{\la} }) \\
&= (\rho^{\la} ) '(s) \, d \log(t^{\la} - \zeta^{\la} ) + (\rho^{\la} ) '(s) \, d \log \frac{\ov{t^{\la}- \zeta^{\la}  }}{t^{\la} -\zeta^{\la} } \\ 
&= \frac{ (\rho^{\la} ) '(s)  \, dt^{\la} }{t^{\la} - \zeta^{\la} } + ( \rho^{\la} ) '(s) \, d \log \frac{\ov{t^{\la} } - \ov{\zeta^{\la}} }{t^{\la}  -\zeta^{\la} }, 
\end{align*}
we can rewrite $\mc{K}_2^{\la} \eta$ as 
$ 
\mc{K}_{2}^{\la} \eta(\zeta, \la) = \mc{K}^{\la}_{3} \eta(\zeta, \la) + \mc{N}_{3}^{\la} \eta(\zeta, \la), $ with \begin{gather}
\nn
\mc{K}^{\la}_3 \eta(\zeta, \la) :=  - \frac{1}{4 \pi} \int_{b \Om} \frac{ \left( l(t, \la) \ov{l(\zeta, \la)} + \ov{l (t, \la)} l (\zeta, \la) \right) \eta(t, \la) }{t^{\la}-\zeta^{\la} }\,dt^{\la}
\\ \label{N3}
\mc{N}_{3}^{\la} \eta(\zeta, \la) :=  - \frac{1}{4 \pi} \int_{b \Om} \ov{l(t, \la)}l(\zeta, \la) \eta(t, \la ) \, d \, log \frac{\ov{t^{\la} - \zeta^{\la} }}{t^{\la}-\zeta^{\la}}. 
\end{gather} 
Now, $\mc{N}_{3}^{\la}$ is a compact operator, and we can write $\mc{K}_3^{\la} \eta^{\la}$ in the form
\[ 
\mc{K}_3^{\la} \eta(\zeta, \la) = A(\zeta, \la)  \eta(\zeta, \la) + \frac{1}{\pi i} \int_{b \Om}  
\frac{S(\zeta,t, \la) \eta(t, \la)}{t^{\la}-\zeta^{\la}}  \, dt^{\la}, 
\] 
where $A(\zeta, \la) \equiv 0$, and
\[ 
 S (\zeta, t, \la) := - \frac{i}{4} \,  \left(  l(t, \la) \ov{l(\zeta, \la)} + \ov{l (t, \la) } l (\zeta, \la)   \right). 
\] 
In particular we have $S(t,t, \la) = -\frac{i}{2}$.   
Now, the integral equation \re{pAfdsie} can be written as 
\eq{pAmuesie}
\mc{K}^{\la} \eta (\zeta,\la) = \mc{K}_3^{\la} \eta (\zeta,\la) + \mc{N}_{1}^{\la} \eta (\zeta,\la) + \mc{N}_{2}^{\la} \eta (\zeta,\la) + \mc{N}_{3}^{\la} \eta (\zeta,\la)  = \gm_{0}(\zeta, \la), 
\eeq
where $\mc{K}^{\la}_3$ is a singular integral operator and $N^{\la}_j$-s are compact operators. 
Following \cite[P.~118, (45.4)]{MI92}, the index of the singular integral operator $\mc{K}^{\la}_3$ is defined to be
\[ 
\frac{1}{2 \pi} \left[ \arg \frac{A(t, \la) - S(t, t, \la)}{A(t, \la) + S(t, t, \la)} \right]_{b \Om} 
= \frac{1}{2 \pi} [\arg (-1) ] = 0.  
\]
One can find a `` reducing operator" $\mc{M}^{\la}$ so that the composition $ \mc{F}^{\la} = \mc{M}^{\la} \circ \mc{K}^{\la}$ is a Fredholm operator.
Moreover by Vekua's equivalence theorem~\cite[p.~149, case 1]{MI92}, if the index of $\mc{K}_0^{\la}$ is nonnegative, one can choose $\mc{M}^{\la}$ so that the equation 
$
\mc{M}^{\la} \var = 0
$
has only the trivial solution $\var \equiv 0$. Therefore $\eta^{\la}$ is a solution to the Fredholm equation $\mc{F}^{\la} \eta^{\la} = \mc{M}^{\la} \gm_0$ if and only if $\eta^{\la}$ is a solution to the singular integral equation $ \mc{K}^{\la} \eta^{\la} = \gm_0$, in other words it suffices to consider the equation 
\[
\mc{F}^{\la} \eta^{\la} = \mc{M}^{\la} \gm_0^{\la}. 
\] 

There are many choices for such $\mc{M}^{\la}$, for example, one can take $\mc{M}^{\la}$ to be the dominant part of the adjoint operator of $\mc{K}_3^{\la}$ (See the remark after \rd{redoptdf}). For an operator $\mc{T}$, we use the notation 
\[
\mc{T} \eta (\zeta, \la) := \mc{T}^{\la} \eta^{\la} (\zeta^{\la}). 
\]
We have
\eq{pAro} 
\mc{M} \eta (\zeta, \la) 
:= - \frac{S(\zeta, \zeta, \la)}{\pi i} P.V. \int_{b \Om} \frac{\eta(t, \la) \, dt^{\la}}{ t^{\la} - \zeta^{\la}}  
= \frac{1}{2 \pi} P.V. \int_{b \Om} \frac{\eta(t, \la)\, dt^{\la}}{t^{\la}- \zeta^{\la}}. 
\eeq
Since $S$ is H\"older continuous in both $\zeta$ and $t$, we can apply formula \re{soprod} to get
\begin{align*}
&\mc{F}^{\la}\eta^{\la} (\zeta^{\la})= \mc{M} \mc{K}_3 \eta(\zeta, \la) \\
& \quad = - [S(\zeta,\zeta, \la)]^{2} \eta(\zeta, \la)
- \frac{S(\zeta,\zeta, \la)}{(\pi i)^{2}} \int_{b \Om} \left[ \int_{b \Om} \frac{S(t_{1},t, \la)}{(t_{1}^{\la}-\zeta^{\la})(t^{\la}-t_{1}^{\la})}  \, dt_{1}^{\la}  \right] \eta(t, \la) \, dt^{\la} \\
& \quad = \frac{1}{4} \eta(\zeta, \la) - \frac{i}{2 \pi^{2}} \int_{b \Om} 
\left[ \int_{b \Om} \frac{l(t, \la) \ov{l (t_{1}, \la)} + \ov{l(t, \la)} l (t_1, \la) }{(t_{1}^{\la} - \zeta^{\la})(t^{\la}-t_{1}^{\la})} \, dt_{1}^{\la} \right] \eta(t, \la) \, dt^{\la}, 
\end{align*} 
where we used that $S(t,t, \la) = -\frac{i}{2}$.   
We have
\begin{align} \label{Flait}
\int_{b \Om} \frac{l(t, \la) \ov{l(t_1, \la)} }{(t_1^{\la} - \zeta^{\la} ) (t^{\la} - t_1^{\la}) } 
= \frac{l(t, \la)}{t^{\la} - \zeta^{\la}}  [v_0(\zeta, {\la}) - v_0(t, \la)],   
\end{align} 
where we set 
\begin{align*} 
v_0 (\zeta,\la) 
&= P.V. \int_{b \Om} \frac{ \ov{l(\tau,\la)}}{ \tau^{\la}- \zeta^{\la}} \, d \tau^{\la} 
= \int_{b \Om} \frac{ \ov{l( \tau,\la)} - \ov{l( \zeta,\la)} }{ \tau^{\la}- \zeta^{\la}} \, d \tau^{\la} + (\pi  i) \ov{l(\zeta,\la)}. 
\end{align*}     
By \rl{pvintest}, $v_0 \in C^{\mu, 0} (b \Om^{\la})$ and 
\eq{v0Hest} 
\| v_0 \|_{\mu,0} \leq C_{1,0} \| l \|_{\mu,0}. 
\eeq  
In view of expression \re{Flait}, we have
\eq{pAroF}
\mc{F} \eta(\zeta, \la) = \frac{1}{4} \eta(\zeta, \la) + \int_{b \Om} \frac{k_0(\zeta,t, \la)}{ \zeta^{\la} -t^{\la}  } \eta(t, \la) \, dt^{\la}, 
\eeq
where $k_0(\zeta,t, \la)$ is given by 
\eq{k0}
k_0 (\zeta,t, \la) = l(t, \la) \left( v_0 ( \zeta, \la) - v_0 (t, \la) \right). 
\eeq
In view of \re{v0Hest} we have
\eq{k0est} 
| k_0 (\zeta, t, \la) | \leq A_0 | \zeta - t|^{\mu},   
\eeq 
where $A_0$ is some constant independent of $\la \in [0,1]$ and $t \in b \Om$.  

Next, we compute the composition operators $\mc{M^{\la}} \mc{N}_{1}^{\la} \eta^{\la}$, $\mc{M^{\la}} \mc{N}_{2}^{\la} \eta^{\la}$, $\mc{M^{\la}} \mc{N}_{3}^{\la} \eta^{\la}$. All of them turn out to be Fredholm operators of the first kind. For simplicity, we will omit all the non-zero constants in the following computations. In view of \re{N1opt}, we have
\begin{align*}
&\mc{M} \mc{N}_{1} \eta(\zeta, \la ) \\  
& \quad = \int_{b \Om} \frac{1}{t^{\la}-\zeta^{\la}}
\left[ \int_{b \Om} Re \left[ l(t_{1}, \la) l(t, \la) \ov{G_{2}(t,t_{1}, \la)} (\rho^{\la})'(s_{1}) \right] \eta(t_{1}, \la) \, ds_{1} \right] \, dt^{\la} \\  
& \quad = \int_{b \Om} 
\left[ \int_{b \Om} \frac{1}{t^{\la}- \zeta^{\la}}  Re \left(  l(t_{1}, \la)  l(t, \la) \ov{G_{2}(t,t_{1}, \la)} (\rho^{\la})' (s_{1}) \right) \, dt^{\la} \right]
\eta(t_{1}, \la) \, ds_{1}  \\
& \quad = \int_{b \Om} \left[ \int_{b \Om} \frac{1}{t^{\la}-\zeta^{\la}} \left(
l(t_{1}, \la) l(t, \la) \ov{G_{2}(t,t_{1}, \la)} (\rho^{\la})'(s_{1}) \right.  \right. 
\\ & \qquad +  \ov{l(t_{1}, \la)} \, \ov{l(t, \la)}  \, \ov{(\rho^{\la})'(s_{1})} G_{2}(t,t_{1}, \la) \left. \right)  \,dt^{\la} \left. \right] \eta (t_{1}, \la) \, ds_{1} ,  
\end{align*}
where we can easily justify the interchange of integrals. 
The inner integral can be expanded as 
\begin{align} \label{MN1init} 
&l(t_{1}, \la) (t_{1}^{\la})'(s_{1}) \int_{b \Om} \frac{ l(t, \la) \left( \ov{e^{\om_{1} (t,t_{1}, \la)}} - \ov{e^{\om_{2} (t,t_{1}, \la)}} \right) }{(t^{\la}-\zeta^{\la})(\ov{t_{1}^{\la}-t^{\la}})}  \, dt^{\la} 
\\  &\qquad \nn + \, \ov{l(t_{1}, \la)} \, \ov{(t_{1}^{\la})'(s_{1})} \int_{b \Om} \frac{\ov{l(t, \la)}\left( e^{\om_{1}(t,t_{1}, \la)} - e^{\om_{2} (t,t_{1}, \la)} \right)}{(t^{\la}-\zeta^{\la})(t_{1}^{\la}-t^{\la})}  
\, dt^{\la}. 
\end{align} 
Denoting the first and second  integrals in the above expression by $I_1$ and $I_2$ respectively, we have
\begin{align} \label{MN1I1} 
& I_{1}(\zeta, t_{1}, \la) = 
\int_{b \Om} \frac{l(t, \la) e^{2iarg(t_{1}^{\la}-t^{\la})}}{(t^{\la}-\zeta^{\la})(t_{1}^{\la}-t^{\la})} \left[ \ov{e^{\om_{1} (t,t_{1}, \la)}} - \ov{e^{\om_{2}(t,t_{1}, \la)}} \right] \, dt^{\la}
\\ \nn & \quad = \frac{1}{t_{1}^{\la} -\zeta^{\la}} \left\{ P.V. \int_{b \Om} \frac{ 
F(t, t_1, \la)  }{t^{\la}-\zeta^{\la}} dt^{\la} \right. 
-  P.V. \left.  \int_{b \Om} \frac{ F(t, t_1, \la)}{t^{\la} -t_{1}^{\la} } \, dt^{\la} \right\}.   
\\ \nn & \quad = \frac{1}{t_{1}^{\la} -\zeta^{\la} }[v'_1(\zeta, t_1, \la) - v'_1(t_{1},t_{1}, \la) ], 
\end{align}
where we set
\gan
v'_1(\zeta, t_1, \la):= \int_{b \Om} \frac{F(t, t_1, \la) - F(\zeta, t_1, \la)}{t^{\la} -\zeta^{\la} } \, dt^{\la}  + \pi i F (\zeta, t_1, \la)  , 
\\ 
v'_1(t_1, t_1, \la)
:= \int_{b \Om} \frac{F(t, t_1, \la) - F(t_1, t_1, \la)}{t^{\la} -t_1^{\la} } \, dt^{\la}  + \pi i F (t_1, t_1, \la) , 
\\
F(t, t_1, \la) := l(t, \la) e^{2iarg(t_{1}^{\la} -t^{\la} )} \left[ \ov{e^{\om_{1} (t,t_{1}, \la)}} - \ov{e^{\om_{2} (t,t_{1}, \la)}} \right]. 
\end{gather*}
Note that since $\om_1 (t_1, t_1) = \om_2 (t_1, t_1) = 0$, we have $F(t_1, t_1, \la) \equiv 0$ on $b \Om$. 
We show that $v_1'(\cdot, t_1, \la) \in C^{\nu,0} (b\Om^{\la})$. By \rl{pvintest}, it suffices to see that $F( \cdot, t, \la)$ is $\nu$-H\"older continuous in the first argument, and this follows from \rl{argfnt}, and estimate \re{omjhnbdd}. In fact we have 
\[  
|F(\cdot, t_1, \la) |_{\nu}  
\leq C_{1+\mu,0} | l (\cdot, \la)  |_{\mu} (\Si_{i=1,2} |\om_{i} (\cdot, t_{1}, \la) |_{\all_p}), \quad \nu = min(\all_p,\mu), 
\] 
where the constant is uniform in $\la$ and $t_1 \in b \Om$.   
Hence by \rl{pvintest}, 
\[ 
\|v'_1(\cdot, t_1, \la) \|_{\nu,0} \leq C_{1,0} \|F(\cdot, t_1, \la) \|_{\nu,0} \leq C_{1+\mu,0} \| l \|_{\mu,0} (\Si_{i=1,2} \| \om_{i} \|_{\all_p,0}). 
\] 
In a similar way, we can show that 
\eq{MN1I2} 
I_2 (\zeta, t_1, \la) = \frac{1}{t_{1}^{\la} -\zeta^{\la} }[  v''_1(t_{1}, \zeta, \la) - v''_1(t_{1},t_{1}, \la) ],  
\eeq 
where $v''_1$ satisfies  
$
\|v''_1(\cdot, t_1, \la) \|_{\nu,0} \leq C_{1+\mu,0} \| l \|_{\mu,0} (\Si_{i=1,2} \| \om_{i} \|_{\all_p,0}). 
$ 

Putting \re{MN1I1} and \re{MN1I2} in \re{MN1init}, we have
\eq{MNo1}
\mc{M} \mc{N}_{1} \eta(\zeta, \la) 
= \int_{b \Om} \frac{k_{1}(\zeta,t_{1},\la)}{ t_{1}^{\la} - \zeta^{\la} } \eta(t_{1}, \la) \, ds_{1}, \quad \quad \nu = min(\all_p, \mu)
\eeq
where $k_{1}(\zeta,t_{1}, \la)$ is given by
\begin{equation} \label{k1} 
 \begin{aligned} 
k_{1}(\zeta,t_{1}, \la)  &:= l(t_{1}, \la)  (\rho^{\la})'(s_{1}) [v_1'(\zeta, t_1, \la) - v'(t_{1},t_{1}, \la) ]
\\ &\quad + \ov{l(t_{1}, \la)} \ov{(\rho^{\la})'(s_{1})} [ v_1''(\zeta, t_1, \la) - v''(t_{1},t_{1}, \la) ], 
\end{aligned} 
\end{equation} 
and $k_1$ satisfies 
\eq{k1est} 
| k_1 (\zeta, t_1, \la) | \leq A_1 | \zeta - t_1|^{\nu}. 
\eeq
Here $A_1$ is some constant independent of $\la \in [0,1]$ and $t_1 \in b \Om$.  
By similar procedures applied to \re{N2opt}, we can show that the following holds (up to a nonzero constant) 
\eq{MNo2}
\mc{M} \mc{N}_{2} \eta(\zeta, \la) 
= \int_{b \Om} \frac{k_{2}(\zeta,t_{1},\la)}{ t_{1}^{\la} - \zeta^{\la} } \eta(t_{1}, \la) \, ds_{1}, \quad \quad \nu = min(\all_p, \mu)
\eeq
where $k_2 (\zeta, t_1, \la)$ is given by 
\begin{equation} \label{k2}
\begin{aligned}    
k_{2}(\zeta,t_{1}, \la)  &:= l(t_{1}, \la)  (\rho^{\la})'(s_{1}) [v_2'(\zeta, t_1, \la) - v'_2(t_{1},t_{1}, \la) ]
\\  
&\quad + \ov{l(t_{1}, \la)} \ov{(\rho^{\la})'(s_{1})} [ v_2''(\zeta, t_1, \la) - v''_2(t_{1},t_{1}, \la) ], 
\end{aligned}
\end{equation}
where $v_2'$ and $v_2''$ satisfy the estimates
\begin{gather*}
\|v'_2(\cdot, t_1, \la) \|_{\nu,0}, \: \|v''_2(\cdot, t_1, \la) \|_{\nu,0} \leq C_{1+\mu,0} \| l \|_{\mu,0} (\Si_{i=1,2} \| \om_{i} \|_{\all_p,0}). 
\end{gather*}
Hence 
\eq{k2est} 
| k_2 (\zeta, t_1, \la) | \leq A_2 | \zeta - t_1|^{\nu},   
\eeq
where $A_2$ is some constant independent of $\la \in [0,1]$ and $t_1 \in b \Om$.  
It remains to compute $\mc{M}^{\la} \mc{N}_{3}^{\la} \eta^{\la}$. From \re{N3}, we have
\begin{align*}
\mc{N} _{3} \eta(\zeta, \la)
&= - \frac{1}{4 \pi} \int_{b \Om} \ov{l(t, \la)}l(\zeta, \la) \eta(t, \la) \, d \, log \frac{\ov{t^{\la} - \zeta^{\la}}}{t^{\la}-\zeta^{\la}} \\ 
&= - \frac{(-2i)}{4 \pi} \int_{b \Om} \ov{l(t, \la)}l(\zeta, \la) \eta(t, \la)
\, d_{t^{\la}} \arg (t^{\la} - \zeta^{\la}) \\
&=  \frac{i}{2 \pi} \int_{b \Om} \ov{l(t, \la)}l(\zeta, \la) \eta(t, \la)
\pa_{\tau^{\la} (s)} \arg (t^{\la} - \zeta^{\la} ) \, ds, 
\end{align*}
where $\tau^{\la}$ is the unit tangent vector of $b \Om^{\la}$ at $t^{\la}$. 
Omitting the constant in the front, we have
\begin{align*}
\mc{M}\mc{N}_{3} \eta(\zeta, \la)
&= \int_{b \Om} \frac{1}{t^{\la}-\zeta^{\la}} \left[ \int_{b \Om} 
\ov{l(t_1, \la)} l(t, \la) [\pa_{\tau^{\la} (s_1)} \arg (t_1^{\la} - t^{\la})] \eta(t_{1}, \la) \, ds_{1}  \right] \, dt^{\la} 
\\ & = \int_{b \Om} \ov{l(t_{1}, \la)} \left[ \int_{b \Om} \frac{l(t,\la) [\pa_{\tau^{\la} (s_1)} \arg (t_1^{\la} - t^{\la})] (t_1 ^{\la} - t^{\la}) }{(t^{\la}-\zeta^{\la})(t_{1}^{\la}-t^{\la})} \, dt^{\la} \right] \eta(t_{1}, \la) \, ds_{1}. 
\end{align*}
Denote 
$
F(t, t_1, \la) = l (t, \la) (t_1^{\la} - t^{\la}) [\pa_{\tau^{\la} (s_1)} \arg (t_1^{\la} - t^{\la})] . 
$
As before we show that $| F(\cdot, t_1, \la)|_{\mu}$ is bounded uniformly in $\la \in [0,1]$ and $t_1 \in b \Om$. Since $l^{\la} \in \mc{C}^{\mu, 0} (b\Om^{\la})$, we only need to show the function 
\[  
[\pa_{\tau^{\la}(s_1) } \arg (t_1^{\la} - t^{\la} ) ] (t_1^{ \la} - t^{\la})
\] 
is H\"older $\mu$ continuous in $t$, and the H\"older norm is uniformly bounded in  $\la$ and $t_1$.  

If $| t - t_1| < 2 | t - \wti{t} | $, then $|\wti{t}- t_1| \leq |\wti{t} - t| + |t-t_1| \leq 3 | t- \wti{t}|$, and we have
\begin{align*} 
&\left| [\pa_{\tau^{\la} (s_1)} \arg (t_1^{\la} - t^{\la}) ] (t_1^{\la} - t^{\la} )  
- \pa_{\tau^{\la} (s_1)} \arg (t_1^{\la} - \wti{t}^{\la})]  (t_1^{\la} - \wti{t}^{\la} ) \right| \\
&\quad =  \left| \pa_{s_1} \arg (\rho^{\la}  (s_1) - t^{\la}) (t_1^{\la} - t^{\la} )   - [\pa_{s_1} \arg (\rho^{\la}  (s_1) - \wti{t}^{\la})](t_1^{\la} - \wti{t}^{\la} )   \right| 
\\
& \quad \leq \left| \pa_{s_1} \arg (\rho^{\la}  (s_1) - t^{\la})  \right| |t_1^{\la} - t^{ \la}| 
+ \left| \pa_{s_1} \arg (\rho^{\la}  (s_1) - \wti{t}^{\la})  \right| | t_1^{\la} - \wti{t}^{\la} |. 
\end{align*}   
By \rl{argfnt} (i), the above is bounded by 
\begin{align*}  
& C_{1+ \mu,0} \left( |t_1 - t|^{\mu-1} | t_1 - t| + |t_1 -  \wti{t}|^{\mu -1 } |t_1 - \wti{t}|  \right)  \\
&\quad \leq C_{1+ \mu,0} ( |t_1 - t|^{\mu} + | t_1 - \wti{t} |^{\mu} ) 
= C_{1+ \mu,0}' | t - \wti{t} |^{\mu}. 
\end{align*}  
On the other hand if $| t - t_1 | \geq 2 | t - \wti{t} | $, we have
\begin{align*} 
&\left| [\pa_{\tau^{\la} (s_1)} \arg (t_1^{\la} - t^{\la})] (t_1^{\la} - t^{\la}) 
- \pa_{\tau^{\la} (s_1)} \arg (t_1^{\la} - \wti{t}^{\la})] (t_1^{\la} - \wti{t}^{\la}) \right| \\
& \quad \leq \left| \left\{ \pa_{\tau^{\la} (s_1)} \arg (t_1^{\la} - t^{\la}) - \pa_{\tau^{\la} (s_1)} \arg (t_1^{\la} - \wti{t}^{\la})  \right\} (t_1^{\la} - t^{\la})\right| 
\\ & \qquad + \left| \pa_{\tau^{\la} (s_1)} \arg (t_1^{\la} - \wti{t}^{\la})] (t^{\la} - \wti{t}^{\la}) \right| . 
\end{align*}  
By \rl{argfnt} (ii), since $|t - t_1| \geq 2 | t - \wti{t} | $, the above is bounded by 
\begin{align*} 
C_{1+ \mu,0} \left( \frac{| t - \wti{t}|}{|t -t_1|^{2 - \mu} } |t_1 - t| + |t_1 - t|^{\mu-1} | t - \wti{t} | \right) 
\leq C_{1+\mu, 0}' | t- \wti{t} |^{\mu}. 
\end{align*}  

Now, by the same proof as that for $\mc{M}^{\la} \mc{N}_{1}^{\la}$ and $\mc{M}^{\la} \mc{N}_{2}^{\la}$, we can show that
\eq{MNo3}
\mc{M}\mc{N}_{3} \eta(\zeta, \la) 
= \int_{b \Om} \frac{k_{3}(\zeta, t_1, \la)}{t_1^{\la} - \zeta^{\la}} \eta(t_1, \la) \, ds_1, 
\eeq
where $k_{3}(\zeta, t_1, \la)$ is of the form
\eq{k3}
k_3(\zeta, t_1, \la) := \ov{l(t_1, \la) } \left( v_3 (\zeta,t_1, \la) - v_3 (t_1, t_1, \la)  \right),
\eeq 
and $v_3$ satisfies the estimate  
$
\| v_3 (\cdot, t_1, \la ) \|_{\mu, 0} \leq C_{1+\mu, 0} \| l \|_{\mu, 0}.  
$
Hence 
\eq{k3est} 
| k_3 (\zeta, t_1, \la) | \leq A_3 | \zeta - t_1|^{\mu},   
\eeq
where $A_3$ is some constant independent of $\la \in [0,1]$ and $t_1 \in b \Om$.  

We now apply $\mc{M}$ to both sides of integral equation \re{pAmuesie}. By \re{pAroF}, \re{MNo1}, \re{MNo2}, \re{MNo3} we have
\eq{pArie}
(I + \mc{N} ) \eta(\zeta, \la) :=  \eta(\zeta, \la) + \int_{b \Om} \frac{k(\zeta,t, \la)}{\zeta^{\la}-t^{\la}} \eta(t, \la) \, dt^{\la}
= \mc{M} \gm_{0} (\zeta, \la), 
\eeq
where $k(\zeta,t, \la) := \sum_{i=0}^3 k_i(\zeta,t, \la) $. By 
\re{k0est}, \re{k1est}, \re{k2est} and \re{k3est}, we have
\[ 
|k (\zeta, t, \la) | \leq A | \zeta - t|^{\nu}, \qquad \nu = \min \left( \mu, \: \all_p \right), 
\] 
for some constant $A$ independent of $\la \in [0,1]$ and $t \in b \Om$.  
\\ \\  
\nid 
\textbf{Step 4: Estimates for the right-hand side of the reduced equation. } 
\\  \\ 
We now prove $\mc{M}^{\la} \gm_{0}^{\la} \in \mc{C}^{\nu',0}(\ov{\Om^{\la}} )$, for any $0 < \nu' < \nu$ and $\nu = \min (\mu, \all_p )$. Recall that $\gm_{0}$ is given by \re{pAfdsieRHS}:
\[ 
\gm_{0} (\zeta, \la) = \gm(\zeta, \la) - Re [\ov{l(\zeta, \la)}  w_{\gm}^{+}(\zeta, \la)] - Re[ \ov{l(\zeta,\la)} w_{F}(\zeta,\la)]. 
\] 

Since for each $\la$, $l (\cdot, \la)$ and $\gm (\cdot, \la)$ are H\"older continuous, we can apply the jump formula for generalized Cauchy integrals (\rp{gCijft}) to get    
\begin{equation} \label{wgm+}
  \begin{aligned} 
w_{\gm}^{+}(\zeta, \la) 
&= \yh \,l(\zeta,\la) \gm(\zeta,\la) 
+ \frac{1}{2 \pi i} P.V. \int_{b \Om}  G_{1}(\zeta,t, \la) l(t, \la) \gm(t, \la) \, dt^{\la} \\  \nn
& \quad - \frac{1}{2 \pi i} \int_{b \Om} G_{2}(\zeta,t, \la) \ov{l(t, \la)} \gm(t, \la) \, \ov{dt^{\la}} 
\\ \nn & := P_{1}(\zeta, \la) + P_{2}(\zeta, \la) + P_{3}(\zeta, \la) + P_{4}(\zeta, \la),     
\end{aligned}    
\end{equation}
where we set
\begin{gather} \label{pAP1P2}
P_{1}(\zeta, \la) := \yh \, l(\zeta, \la) \gm(\zeta, \la),  \quad \quad P_{2}(\zeta,\la) := \pi i [l(\zeta, \la) \gm(\zeta, \la)], 
\end{gather} 
\begin{align} \label{pAP3}
P_{3}(\zeta, \la) :=&  \int_{b \Om} \frac{1}{2 (t^{\la}-\zeta^{\la})} 
\left\{ \left( \sum_{i=1}^{2} e^{\om_{i} (\zeta,t,\la)} \right) [l(t,\la) \gm(t,\la)] \right.
\\
&- \left. \left( \sum_{i=1}^{2} e^{\om_{i} (t,t,\la)} \right) [l(\zeta, \la) \gm(\zeta, \la)] \right\} \, dt^{\la}, 
\nonumber 
\end{align}
\begin{gather*}
P_{4}(\zeta, \la) :=  - \frac{1}{2 \pi i} \int_{b \Om} \frac{e^{\om_{1}(\zeta,t,\la)} - e^{\om_{2}(\zeta,t,\la)}}{2 (t^{\la}-\zeta^{\la})} \ov{l(t,\la)} \gm(t,\la) \, \ov{dt^{\la}}. 
\end{gather*} 
By assumption, $P_{1}^{\la}, P_2^{\la} \in C^{\mu,0}(b \Om^{\la} )$. Using \re{pAro}, we have 
\begin{align*}
\mc{M} (\ov{l} P_{1}) (\zeta, \la)
&= \frac{1}{2 \pi} P.V. \int_{b \Om} \frac{\ov{l} (l \gm) (t, \la)}{t^{\la}-\zeta^{\la}} \, dt^{\la}
= \frac{1}{2 \pi} P.V. \int_{b \Om} \frac{\gm(t, \la)}{t^{\la} - \zeta^{\la}} \,dt^{\la} \\
&= \frac{1}{2 \pi} \left\{ \int_{b \Om} \frac{\gm(t, \la) - \gm(\zeta, \la)}{t^{\la} - \zeta^{\la}} \, dt^{\la} +  \pi i  \gm(\zeta, \la) \right\}. 
\end{align*}
By \rl{pvintest}, we have
\eq{MP12laest}   
\| \mc{M}(\ov{l} P_{1}) \|_{\mu,0} \leq C_{1,0} \| \gm \| _{\mu}, \quad  \| \mc{M} (\ov{l} P_{2}) \|_{\mu,0} \leq C_{1,0} \| \gm \|_{\mu}. 
\eeq
For $P_{3}$, since $\om_i (t,t,\la) = 0$, we have
\begin{align} \label{pAP3'}
P_{3}(\zeta, \la) &= \sum_{i=1,2} \int_{b \Om} \frac{l(t,\la) \gm(t,\la) }{2 (t^{\la}-\zeta^{\la})} \left[ e^{\om_{i}(\zeta,t, \la)}-e^{\om_{i}(t,t,\la)} \right] \, dt^{\la}
\\ \nn &\quad + \int_{b \Om} \frac{1}{t^{\la}-\zeta^{\la}} [l(t, \la)\gm(t,\la) - l(\zeta,\la)\gm(\zeta, \la)] \, dt^{\la}, 
\end{align}
and 
\begin{align*}
&\mc{M} (\ov{l} P_{3}) (\zeta, \la)
= \sum_{i=1,2}  \int_{b \Om} \frac{\ov{l(t, \la)} }{t^{\la}-\zeta^{\la}} \left[ 
\int_{b \Om} \frac{l(t_{1}, \la) \gm(t_{1}, \la) }{2 (t_{1}^{\la}-t^{\la})} \left[ e^{\om_{i}(t,t_{1},\la)}-e^{\om_{i}(t_{1},t_{1},\la)}  \right] \, dt_{1}^{\la} \right] \, dt^{\la}
\\ & \qquad + \int_{b \Om} \frac{\ov{l(t, \la)} }{t^{\la}-\zeta^{\la}} \left[ \int_{b \Om} \frac{1}{t_{1}^{\la}-t^{\la}} [l(t_{1},\la)\gm(t_{1},\la) - l(t,\la)\gm(t,\la)] \, dt_{1}^{\la} 
\right] \, dt^{\la} 
\\& \quad = \sum_{i=1,2} \int_{b \Om} l(t_{1},\la) \gm(t_{1},\la) \left[ \int_{b \Om} \frac{ f_i(t_1, t, \la)  }{(t^{\la}-\zeta^{\la})(t_{1}^{\la}-t^{\la})} \, dt^{\la} \right] \, dt_{1}^{\la}
\\& \qquad + \int_{b \Om} \left[ \int_{b \Om} \frac{g(t_1, t, \la)}{(t^{\la} - \zeta^{\la})(t_{1}^{\la} - t^{\la})} \, dt^{\la}  \right] dt_{1}^{\la} 
\\ & \quad = \sum_{ i=1,2} \int_{b \Om} \frac{l(t_{1}, \la) \gm(t_{1},\la)}{t_{1}^{\la} -\zeta^{\la}} 
\left[ P.V.\int_{b \Om} \frac{f_{i}(t_{1},t,\la)}{t^{\la}-\zeta^{\la}}  \, dt^{\la} - \int_{b \Om} \frac{f_{i}(t_{1},t,\la)}{t^{\la}-t_{1}^{\la}} \, dt^{\la} \right] \, dt_{1}^{\la}
\\ &\qquad + \int_{b \Om} \frac{1}{t_1^{\la} - \zeta^{\la} } \left[ P.V. \int_{b \Om} \frac{g(t_1, t, \la)}{t^{\la} - \zeta^{\la}} \, dt^{\la} - \int_{b \Om} \frac{g(t_1, t, \la)}{t^{\la} - t_1^{\la}} \, dt^{\la} \right] \, dt^{\la}_1, 
\end{align*}
where we denote
$ 
f _i (t_1, t, \la) := \ov{l(t, \la)} \left( e^{\om_{i}(t,t_{1},\la)}-e^{\om_{i}(t_{1},t_{1},\la)} \right)$  and 
\gan g(t_{1}, t, \la) := \ov{l(t, \la)} l(t_{1}, \la) \gm(t_{1}, \la)  - \gm(t,\la). 
\end{gather*} 
We can write $\mc{M} (\ov{l}P_3)$ as 
\eq{MP3exp} 
\mc{M} (\ov{l}P_{3})(\zeta, \la) = \int_{b \Om} \frac{k (\zeta, t_1, \la)}{t_1^{\la} - \zeta^{\la}} \, dt_1^{\la}, 
\eeq
where
\begin{align*} 
k (\zeta, t_1, \la) &:= l(t_1, \la) \gm(t_1, \la) \sum_{i=1}^{2} \left( v_i(t_1, \zeta, \la) - v_i(t_1, t_1, \la) \right) 
+ \left( w(t_1, \zeta, \la) - w(t_1, t_1, \la) \right),  
\\
v_{i}(t_{1},\zeta,\la): 
&= P.V.\int_{b \Om} \frac{f_{i}(t_{1},t,\la)}{t^{\la}-\zeta^{\la}}  \, dt^{\la} 
= \int_{b \Om} \frac{f_i(t_1, t, \la) - f_i(t_1, \zeta, \la)}{t^{\la} - \zeta^{\la} } \, dt^{\la} + \pi i  f_i (t_1, \zeta, \la),\\  
w(t_{1},\zeta,\la) 
&:= P.V. \int_{b \Om} \frac{g(t_1, t, \la)}{t^{\la} - \zeta^{\la}} \, dt^{\la}  
= \int_{b \Om} \frac{g(t_1, t, \la) - g(t_1, \zeta, \la)}{t^{\la} - \zeta^{\la} } \, dt^{\la} + \pi i  g(t_1, \zeta, \la). 
\end{align*}  
Let $\| \om_i \|_{\all,0} := \sup_{t_1 \in b \Om, \la \in [0,1] } 
|\om_i (\cdot, t_1, \la) |_{\all_p}$. By \rl{pvintest} we have
\gan
\| v_i( t_1, \cdot, \la) \|_{\nu,0} \leq C_{1,0} \| f_i(t_1, \cdot, \la) \|_{\nu,0} \leq C_{1,0} \| \om_i \|_{\all_p, 0} | l(\cdot, \la) |_{\mu}. 
\\
\| w(t_{1},\cdot,\la) \|_{\mu,0} \leq C_{1,0} | g(t_1, \cdot, \la) |_{\mu} \leq C_{1,0}  \left( |\gm(\cdot, \la) |_{0} | l(\cdot, \la) |_{\mu} + | \gm(\cdot, \la) |_{\mu} \right), 
\end{gather*} 
 for $\nu = \min (\all_p, \mu)$, where the constants are uniform in $t_1 \in b \Om$. It follows from \re{MP3exp} and \rl{redopt} that given any $ 0 < \nu' < \nu$, one has
\begin{align} \label{MP3Hne}
\| \mc{M} (\ov{l} P_{3}) \|_{\nu',0} 
&\leq M_0, 
\end{align}
for some $M_0$ independent of $\la$.  
Similarly we can show 
\begin{align} \label{MP4Hne}
\| \mc{M} (\ov{l} P_{4}) \|_{\nu', 0}   
\leq M_1, 
\end{align} 
for some $M_1$ independent of $\la$.  
Combining \re{MP12laest}, \re{MP3Hne} and \re{MP4Hne}, we obtain 
\eq{mwgmlaest}
\| \mc{M} ( Re [ \ov{l^{\la}} (w_{\gm}^{\la})^{+} ] ) \|_{\nu',0} 
\leq M, \quad 0 < \nu ' < \nu,  
\eeq
where $M$ depends only on $\nu'$, $\|\om_i \|_{\all_p, 0}$, $ \| l \|_{\mu,0}$ and $ \| \gm \|_{\nu, 0}$.

Next, we show $\la \mapsto \left| \mc{M} \left( Re [ \ov{l^{\la}} (w_{\gm}^{\la})^{+} ]\right) \right|$ is a continuous map from $[0,1]$ to $C^{0}(b \Om)$. For simplicity, we omit the constant $\frac{1}{2 \pi}$ in the operator $\mc{M}$.
\begin{align*}
& \left| \mc{M} \left( Re \left[ \ov{l^{\la_{1}}} (w_{\gm}^{\la_{1}})^{+} \right]\right)
- \mc{M} \left( Re \left[ \ov{l^{\la_{2}}} (w_{\gm}^{\la_{2}})^{+} \right]\right) \right| \\ 
& \quad = \left| P.V. \int_{b \Om}  \frac{Re[ (\ov{l} w_{\gm}^{+})(t,\la_{1})]}{t^{\la_{1}}  -  \zeta^{\la_{1}}} \, dt^{\la_1}   
- P.V. \int_{b \Om}  \frac{Re[ (\ov{l} w_{\gm}^{+})(t,\la_{2})]}{t^{\la_{2}}  - \zeta^{\la_{2}}} \,dt^{\la_2} \right|. 
\end{align*}
We parametrize $b\Om^{\la}$ by $\rho^{\la}(s)$ and set $ \zeta^{\la} = \rho^{\la}(s_0)$. The above is bounded by  
\begin{equation} \label{Mwgmcinla} 
\begin{aligned} 
 &\left| (\pi i) \{ Re[\ov{l} w_{\gm}^{+}] (\zeta, \la_1) - Re[\ov{l} w_{\gm}^{+}] (\zeta, \la_2) \} \right|  
+  \left| \int_{b \Om} \frac{Re[ (\ov{l} w_{\gm}^{+})(t,\la_{1})] -  Re[ (\ov{l} w_{\gm}^{+})(\zeta,\la_{1})] }{t^{\la_1} - \zeta^{\la_1} } dt^{\la_1} \right. \\
&\quad - \left. \int_{b \Om} \frac{Re[ (\ov{l} w_{\gm}^{+})(t,\la_{2})] -  Re[ (\ov{l} w_{\gm}^{+})(\zeta,\la_2)] }{t^{\la_2} - \zeta^{\la_2} }  dt^{\la_2}  \right|.
\end{aligned}
\end{equation}

We can rewrite the above as 
\begin{equation} \label{Mwgmcinla'} 
\begin{aligned}  
& \pi \left|\left\{ Re[\ov{l} w_{\gm}^{+}] (\zeta, \la_1) - Re[\ov{l} w_{\gm}^{+}] (\zeta, \la_2) \right\} \right| 
+ \left| \int_{b \Om}\frac{F^{\la_1, \la_2} (s) - F^{\la_1, \la_2} (s_0) }{t^{\la_1}  - \zeta^{\la_1}} \, ds \right|   \\
& \quad + \left.\int_{b \Om} \left\{  Re[ (\ov{l} w_{\gm}^{+})(t,\la_{2})] -  Re[ (\ov{l} w_{\gm}^{+})(\zeta,\la_2)] \right\} 
\left( \frac{1}{t^{\la_1} - \zeta^{\la_1}} - \frac{1}{t^{\la_2} - \zeta^{\la_2}} \right) dt^{\la_2} \right|, 
\end{aligned}
\end{equation} 
where 
\eq{Fla1la2}  
F^{\la_{1}, \la_{2}} (s)
= Re[\ov{l} w_{\gm}^{+}] (s, \la_1) (\rho^{\la_1})'(s) - Re[\ov{l} w_{\gm}^{+}] (s, \la_2) (\rho^{\la_2})'(s), 
\eeq
and we write 
$
Re[\ov{l} w_{\gm}^{+}] (s, \la) := Re[\ov{l^{\la}} (w_{\gm}^{\la})^{+}] (t^{\la} ) 
= Re[\ov{l^{\la}} (w_{\gm}^{\la})^{+}] (\rho^{\la} (s)). 
$

We claim that: 
\begin{enumerate}
\item $w_{\gm}^+ \in \mc{C}^{\nu,0} (b \Om^{\la})$, that is,
\begin{enumerate}
\item $|w_{\gm}^+ ( \cdot, \la) |_{\nu} $ is uniformly bounded in $\la$.
\item
$|w_{\gm}^{+} (\cdot, \la_1) - w_{\gm}^{+} (\cdot, \la_2) |_0 \to 0$, as $| \la_1 - \la_2| \to 0$.\end{enumerate}
\item 
For any $0 < \nu' < \nu$, we have $|F^{\la_1, \la_2}|_{\nu'} \to 0$ as $|\la_1 - \la_2| \to 0$.
\end{enumerate}
Assuming the above for now, the expression \re{Mwgmcinla'} is then bounded by  
\[ 
\pi \left| \del^{\la_1, \la_2} Re[\ov{l} w^+_{\gm} ] \right|_0 
+ C_{1,0} |F^{\la_1, \la_2} |_{\nu'} \int_{b \Om} \frac{ds}{|t - \zeta|^{1- \nu'}}   
+ C_{1,0} \| l w_{\gm} \|_{\nu, 0} |\Gm^{\la_1} - \Gm^{\la_2} |_1 \int_{b \Om} \frac{dt}{|t- \zeta|^{1- \nu}} 
\]
which converges to $0$ uniformly in $\zeta$ as $|\la_1 - \la_2| \to 0$. 

We now verify the claims. In view of \re{Fla1la2} and our assumptions on $l^{\la}$ and $\rho^{\la}$, to prove (2) it suffices to show that $|w_{\gm}^+ |_{\nu'} \to 0$. 

By \rp{cinhd}, it suffices to verify claim (1). By \re{wgm+} we have
\[ 
w_{\gm}^{+}(\zeta, \la) := P_{1}(\zeta, \la) + P_{2}(\zeta, \la) + P_{3}(\zeta, \la) + P_{4}(\zeta, \la).  
\]
By assumption $P_{1}, P_{2} \in  C^{\nu,0}(b \Om^{\la})$, and 
\eq{P1P2Hne} 
\|P_{1} \|_{\mu,0} \leq \| l\gm \|_{\mu,0},  \quad \|P_{2} \|_{\mu,0} \leq \| l\gm \|_{\mu,0}. 
\eeq
Applying \rl{redopt} to the expression \re{pAP3'} for $P_3$, we have for any $0 < \nu' < \nu$, 
\eq{P3Hne} 
\| P_3 (\cdot, \la) \|_{\nu',0} \leq C_{\nu'} C_{1,0} \| l \gm \|_{\mu,0} , 
\eeq
where the constants are independent of $\la$. In the same way we have
\eq{P4Hne}
\| P_{4} \|_{\all',0 } \leq C_{\all'} C_{1,0}  \| l \gm \|_{0,0}  
\eeq
for any $0< \all_p' < \all_p$. Putting together \re{P1P2Hne}, \re{P3Hne} and \re{P4Hne}, we obtain 
$
\| w_{\gm}^{+} \|_{\nu',0} \leq A_{\nu'} C_{1,0} \| l \gm \|_{\mu,0} 
$
for any $0 < \nu' < \nu < 1$. This proves claim (1).

Next we show the continuity of $w_{\gm}^{+}$ in the $C^{0}$-norm. 
\eq{omglp}
\left| w_{\gm}^{+} (\zeta, \la_{1}) - w_{\gm}^{+} (\zeta, \la_{2}) \right|
\leq \sum_{i=1}^{4}  \left| P_{i}(\zeta, \la_{1}) - P_{i}(\zeta, \la_{2}) \right|. 
\eeq
In view of \re{pAP1P2}, we have
$
| P_{i} (\zeta, \la_{1}) - P_{i} (\zeta, \la_{2}) | \leq C |l(\zeta,\la_{1}) \gm(\zeta,\la_{1}) -  l(\zeta,\la_{2}) \gm(\zeta,\la_{2})|, 
$
 for $i = 1,2$, 
which converges to $0$ as $|\la_1 - \la_2| \to 0$, since $l, \gm \in C^{\mu, 0}(b \Om^{\la})$. 
By 
expression \re{pAP3'} for $P_3$, we have
\begin{align*}
|P_{3}(\zeta, \la_{1}) - P_{3}(\zeta,\la_{2})| 
&\leq \sum_{i=1,2} \int_{b \Om} \frac{(l \gm)(t, \la_1)}{2(t^{\la_1} - \zeta^{\la_1} )} \left( 
e^{\om_i (\zeta,t, \la_1) } - e^{\om_i (t ,t, \la_1) } \right)   (\rho^{\la_1})'(s)  \, ds \\
& \qquad - \int_{b \Om} \frac{(l \gm)(t, \la_2)}{2(t^{\la_2} - \zeta^{\la_2} )}  \left( e^{\om_i (\zeta,t, \la_2) } - e^{\om_i (t, t, \la_2) } \right) (\rho^{\la_2})'(s) \, ds \\
& \qquad + \int_{b \Om} \frac{1}{t^{\la_1} - \zeta^{\la_1} }  \left\{ (l \gm) (t, \la_1) - (l \gm) (\zeta, \la_1) \right\} (\rho^{\la_1})'(s)  \, ds \\
& \qquad - \int_{b \Om} \frac{1}{t^{\la_2} - \zeta^{\la_2}}  \left\{ (l \gm) (t, \la_2) - (l \gm) (\zeta, \la_2) \right\} (\rho^{\la_2 })'(s)  \, ds. 
\end{align*}
The above expression can be bounded by a sum of the following terms:
\eq{P3ladiff1}  
\sum_{i=1,2} \int_{b \Om} \left\{ (l \gm \rho')  (s,\la_1)  - (l \gm \rho') (s, \la_2) \right\} \frac{e^{\om_i (\zeta,t, \la_1) }  - e^{\om_i (t,t, \la_1) } }{t^{\la_1} - \zeta^{\la_1}} \, ds,   
\eeq
\eq{P3ladiff2} 
\sum_{i=1,2}\int_{b \Om} (l \gm) (s, \la_2) \frac{F_i^{\la_1, \la_2} (\zeta) - F_i^{\la_1, \la_2} (t) }{t^{\la_1} - \zeta^{\la_1}} \, dt^{\la_2}, 
\eeq
\eq{P3ladiff3}  
\sum_{i=1,2} \int_{b \Om}  (l \gm ) (s, \la_2) \left( e^{\om_i (\zeta,t, \la_2) }  - e^{\om_i (t,t, \la_2)} \right) \left( \frac{1}{t^{\la_1} - \zeta^{\la_1} } - \frac{1}{t^{\la_2} - \zeta^{\la_2} } \right) \, dt^{\la_2}, 
\eeq
\eq{P3ladiff4} 
\int_{b \Om} \frac{\wti{F}^{\la_1, \la_2} (t) - \wti{F}^{\la_1, \la_2} (\zeta) }{t^{\la_1} - \zeta^{\la_1}} \, ds, 
\eeq 
\eq{P3ladiff5}  
\int_{b \Om} \left\{ (l \gm )  (t, \la_2) - (l \gm)  (\zeta, \la_2) \right\}  \left( \frac{1}{t^{\la_1} - \zeta^{\la_1} } - \frac{1}{t^{\la_2} - \zeta^{\la_2} } \right) \, dt^{\la_2}, 
\eeq 
where we set
\begin{gather*}
F_i^{\la_1, \la_2} (\zeta) = e^{\om_i(\zeta, t, \la_1)} - e^{\om_i(\zeta, t, \la_2)}, 
\quad 
\wti{F}^{\la_1, \la_2} (\zeta) = (l \gm \rho') (\zeta, \la_1)  - (l \gm \rho') (\zeta, \la_2) . 
\end{gather*}
As before we let $\| \om_i \|_{\all_p, 0} = \sup_{\la \in [0,1], t\in b \Om} | \om_i (\cdot, t, \la) |_{\all_p}$. Then expression \re{P3ladiff1} is bounded by 
\[ 
C_{1,0}  ( \Si_{i=1,2} \| \om_i \|_{\all_p, 0})  | \del^{\la_1, \la_2} (l \gm \rho') |_0 \int_{b \Om} \frac{1}{|\zeta -t|^{1- \all_p}} \, ds, \quad \all_p = \frac{p-2}{p},
\] 
which converges to $0$ as $|\la_1 - \la_2| \to 0$ by assumption.  
Also expression \re{P3ladiff3} is bounded by 
\[ 
C_{1,0} ( \Si_{i=1,2} \| \om_i \|_{\all_p, 0}) | \Gm^{\la_1} - \Gm^{\la_2} |_1 \int_{b \Om} \frac{1}{|t - \zeta|^{1-\all_p} }   \, dt, 
\] 
which converges to $0$ as $|\la_1 - \la_2| \to 0$, and expression \re{P3ladiff5} is bounded by  
\[
C_{1,0} \| l \gm \|_{\mu, 0}  |\Gm^{\la_1} - \Gm^{\la_2} |_1 \int_{b \Om} \frac{1}{|t- \zeta|^{ 1-\mu}} \, dt 
\]
which converges to $0$ as $|\la_1 - \la_2| \to 0$. Finally in order to show \re{P3ladiff2} and \re{P3ladiff4} converges to $0$ it suffices to see that some H\"older norms of $F_i^{\la_1, \la_2}$ and $\wti{F}^{\la_1, \la_2}$ converge to $0$. This follows from \rl{cinhd} and the following facts: 
\begin{enumerate}
\item $|F_i^{\la_1, \la_2} |_{\all_p}  \leq C \| \om_i \|_{\all_p,0}$; 
\item $| F_i^{\la_1, \la_2} |_{0} \to 0 $ as $| \la_1 - \la_2 | \to 0$, by \rl{cino};   
\item $\Gm^{\la} \in \mc{C}^{1+\mu,0}(\ov{\Om})$, $l^{\la}, \gm^{\la} \in C^{\mu,0}(b \Om^{\la})$. 
\end{enumerate}
Hence we have shown that $\la \mapsto P_{3}(\cdot, \la)$ is a continuous map from $[0,1]$ to $C^0 (b \Om)$. Same can be said for $P_4(\cdot, \la)$, of which the proof is similar and we leave it to the reader. 

Now claim (1) (b) follows from \re{omglp}. By earlier remarks this proves that the map 
$\la \mapsto \left| \mc{M}^{\la} \left( Re [ \ov{l^{\la}} (w_{\gm}^{\la})^{+} ]\right) \right|$ is continuous from $[0,1]$ to $C^{0}(b \Om)$. Combining this with \re{mwgmlaest}, we have shown that for any $0 < \nu' < \nu$,  $
\mc{M} \left( Re [ \ov{l^{\la}} (w_{\gm}^{\la})^{+} ]\right) \in C^{\nu',0}(b \Om^{\la}). 
$
Next, we show that $\mc{M} \left( Re[ \ov{l}  w_{F}] \right) \in C^{\nu,0}(b \Om^{\la})$. We have
\begin{align*}
& \mc{M} ( Re[ \ov{l}  w_{F}] ) (\zeta, \la)
= \frac{1}{2 \pi} P.V. \int_{b \Om} \frac{Re[ \ov{l}  w_{F}] }{t^{\la}-\zeta^{\la}} \, dt^{\la} \\
& \qquad = \frac{1}{2 \pi} \left\{ \int_{b \Om} \frac{ Re[ \ov{l}  w_{F}](t, \la) - Re[ \ov{l}  w_{F}](\zeta, \la)}{t^{\la} - \zeta^{\la}} \, dt^{\la} +  \pi i Re[ \ov{l}  w_{F}] (\zeta, \la) \right\}. 
\end{align*}
By \rl{pvintest}, we get 
\begin{align} \label{mrewfest} 
\| \mc{M} ( Re[ \ov{l}  w_{F}] )  \|_{\nu,0} 
\leq C_{1,0} \| l w_F \|_{\nu,0} 
\leq C_{1,0} \| l \|_{\mu,0} \| w_F \|_{\all_p,0}. 
\end{align}
Using estimate \re{wflahest} in \re{mrewfest} we get 
\eq{mrewfest'}
\| \mc{M} ( Re[ \ov{l}  w_{F}] ) \|_{\nu,0}  
\leq C_{1,0} C_p \| l \|_{\mu,0} (\Si_{i=1,2} \| \om_i  \|_{\all_p,0} ) \sup_{\la \in [0,1]} | F( \cdot, \la) |_{L^p (\Om)}. 
\eeq  
On the other hand, as we have seen, in order to show that $\mc{M} ( Re[ \ov{l}  w_{F}] )$ is continuous in $\la$ in the sup norm, it is suffcient to show $w_{F}^{\la} \in \mc{C}^{\all_p, 0}(b \Om^{\la})$, which is given by \rl{wFle}. Thus we have proved that
$\mc{M} \left( Re[ \ov{l}  w_{F}] \right) \in \mc{C}^{\nu,0}(b \Om^{\la})$, and it satisfies estimate \re{mrewfest'}. Combining with estimates \re{mwgmlaest} 
we obtain that for any $0 < \nu' < \nu$, 
$
\| M^{\la} \gm_0 (\cdot, \la) \|_{\nu', 0 }  \leq M_1,
$
where $M_1$ is some positive constant depending on $\nu'$, $\| \om_i  \|_{\all_p,0}$, $\| l \|_{\mu,0}$, $\| \gm \|_{\mu, 0}$, and $\sup_{\la \in [0,1]} | F(\cdot, \la)| _{L^p(\Om)}$. 
\\ \\ 
\textbf{Step 5: Estimates for $\eta^{\la}$}  

We are now ready to show that $\eta^{\la} \in \mc{C}^{\nu,0}(b \Om^{\la})$. Recall from \re{pArie} that  $\eta(\cdot,\la)$ satisfies the reduced Fredholm integral equation: 
\eq{pArie2}
(I + \mc{N}) \eta (\zeta, \la) :=  \eta(\zeta,\la) + \int_{b \Om} \frac{k(\zeta, t, \la)}{t^{\la } - \zeta^{\la} } \, \eta(t, \la) \, dt^{\la} = \mc{M} \gm_{0} (\zeta, \la),  
\eeq
where $k = \Si_{i=1}^4 k_i$, and the $k_i$-s are given by 
\re{k0}, \re{k1}, \re{k2} and \re{k3}. 

Let $w = w_{\eta}(z, \la) + w_{\gm}(z, \la) + w_F(z, \la) $ where $w_{\eta}^{\la}$, $w_{\gm}^{\la}$ and $w_F^{\la}$ are given by formulas \re{pAcpsrf1} \re{pAcpsrf2} and \re{pAcpsrf3} respectively. For $(\zeta, \la) \in b \Om \times [0,1]$, we have $w (\zeta, \la) = l (\zeta, \la)  \gm (\zeta, \la)+ i l (\zeta, \la) \eta (\zeta, \la)$. Hence we have the relation: for $  ( \zeta, \la) \in b \Om \times [0,1]$
\eq{wetabv} 
w_{\eta} (\zeta, \la) = l (\zeta, \la)  \gm (\zeta, \la)+ i l (\zeta, \la) \eta (\zeta, \la) - w_{\gm} (\zeta, \la) - w_F (\zeta, \la). 
\eeq 

Now, for each fixed $\la$ we know by \rp{pArt} that $\eta^{\la} \in C^{\nu} (b \Om^{\la})$, and we know that $w_{\eta}(\cdot, \la)$ is the solution to the following Problem A:
\begin{gather} \nn
\pa_{\ov{z}} w_{\eta}^{\la}(z) + A^{\la}(z) w_{\eta}(z) + B^{\la}(z) \ov{w_{\eta}^{\la}(z)} = 0 \qquad \text {in $\Om^{\la}$;} 
\\ \label{wetabc} 
Re[\ov{l^{\la}(\zeta)} w_{\eta}^{\la}(\zeta)] = \gm^{\la}(\zeta) - Re[ \ov{l^{\la}} w^{\la}_{\gm}] (\zeta) - Re[ \ov{l^{\la}} w_F^{\la} ] (\zeta) \qquad \text {on $b \Om^{\la}$}. 
\end{gather} 
In addition $w_{\eta}^{\la}$ satisfies conditions \re{c1muucip}. 

We now prove that $\eta^{\la} \in \mc{C}^{\nu,0}(b \Om^{\la})$.  First, we show that $|\eta|_0$ is bounded in $\la$. Seeking contradiction, suppose that for some sequence $\{\la_{j}\}$ converging to $\la_0$, such that $N_{j} := |\eta (\cdot,\la_{j})|_{0} \to \infty$.  
Let $\hht{\eta} (\cdot, \la_j)= \frac{1}{N_j} \eta(\cdot, \la)$. Then $|\hht{\eta}(\cdot,\la)|_{0} = 1$. Dividing both sides of \re{pArie2} by $N_{j}$, we have  
\eq{etahatie}  
(I + \mc{N} ) \hht{\eta}(\zeta, \la_j) 
= N_j^{-1} \mc{M} \gm_0. 
\eeq
The corresponding $w_{\hht{\eta}} (\cdot, \la_j)$ is the solution to the following boundary value problem: 
\begin{gather} \nn 
\pa_{\ov{z}} w_{\hht{\eta} }^{\la_j}(z) + A(z, \la_j) w_{\hht{\eta}}(z, \la_j) + B(z, \la_j) \ov{w_{\hht{\eta} }(z, \la_j)} = 0 \quad \text {in $\Om$;}  
\\  \label{wetahatbc} 
Re[\ov{l(\zeta, \la_j)} w_{\hht{\eta}}(\zeta, \la_j)] = \frac{1}{N_j} \left\{  \gm(\zeta, \la_j) - Re[ \ov{l} w_{\gm}] (\zeta, \la_j) - Re[ \ov{l} w_F ] (\zeta, \la_j) \right\} \quad
\text {on $b \Om$}, 
\end{gather}
and in addition $w_{\hht{\eta}}^{\la_j}$ satisfies the conditions: for $r=1,\dots, N_0$ and $s=1, \dots, N_1$,
\begin{gather} \label{wetahatipc}
w_{\hht{\eta}}(z_r, \la_j) = \frac{1}{N_j} \left\{ a_r(\la_j) + i b_r(\la_j) - w_{\gm} (z_r, \la_j) - w_F (z_r, \la_j) \right\},\\ 
\label{wetahatbpc}  
w_{\eta}^{+} (z_s', \la_j) =  \frac{1}{N_j}  \left\{ l(z_s', \la_j)( \gm(z_s', \la_j) + i c_s(\la_j)) - w_{\gm}^{+} (z_s', \la_j) -   w_{F}(z_s', \la_j)\right\}.
\end{gather}

By \rl{redoptest} and our earlier estimate for $k$, we have for any $0 < \nu' < \nu$, $|\mc{N} \hht{\eta}(\cdot, \la_{j}) |_{\nu'}$ is uniformly bounded in $\la$.  By Arzela-Ascoli theorem, and passing to subsequences if necessary, $\mc{N} \hht{\eta}(\cdot, \la_{j})$ converges to some function in the $C^{\tau}(b \Om)$ norm, for any $0 < \tau < \nu$. It follows from equation \re{etahatie} that $\hht{\eta}(\cdot, \la_{j})$ converges to some function $\wti{\eta}$ in $C^{\tau}(b \Om)$.  

Let 
\begin{align*} 
&2\pi w_{\hht{\eta}}(z, \la_j) =  \int_{b \Om}\left\{ G_{1}(z,t,\la_j) l(t,\la_j) \hht{\eta}(t, \la_j) \, dt^{\la_j} 
-    G_{2}(z,t,\la_j) \ov{l(t,\la_j)} \hht{\eta}(t, \la_j)
\, \ov{dt^{\la_j}}\right\}, 
\\
&w_{\wti{\eta}}^{\la_0}(z) = \frac{1}{2 \pi} \int_{b \Om} G_{1}(z,t,\la_0) l(t,\la_0) \wti{\eta}(t) \, dt^{\la_0}
- \frac{1}{2 \pi} \int_{b \Om} G_{2}(z,t,\la_0) \ov{l(t,\la_0)} \wti{\eta}(t) \, \ov{dt^{\la_0}}. 
\end{align*}   
From \re{etahatie} we have 
$
\hht{\eta}(\zeta,\la_{j}) = - \mc{N}  \hht{\eta} (\zeta, \la_j) + \frac{1}{N_j} \mc{M} \gm_0.  
$
By \rl{redoptest}, we have for any $0 < \nu' < \nu$ and some constants $C$  independent of $\la$, 
\eq{etarhatHe}
| \hht{\eta} (\cdot, \la_j)| _{\nu'} \leq C_{\nu'} C_{1,0} | \hht{\eta} (\cdot, \la_j) |_0 
+ \frac{1}{N_j} \| \mc{M} \gm_0 \|_{\nu'}  
= C.
\eeq

Let $\{ \{ z_r \}_{r=1}^{N_0}, \{ z_s \}_{s=1}^{N_1} \}$ be a normally distributed set in $\Om$. Then it is easy to see that $w_{\hht{\eta}}(z_r, \la_i)$ converges to $w_{\wti{\eta}}^{\la_0} (z_r )$, since for $z_r \in \Om$, the kernels $G_1$ and $G_2$ are uniformly bounded in absolute value by a constant. 

On the other hand, for $z_s' \in b \Om$, we apply the jump formula \re{gCijf1} 
to get
\begin{align*}
w_{\hht{\eta}}^{+} (z_s', \la)
&= \yh \, l ( z_s',  \la) \hht{\eta} (z_s') 
+ P.V. \frac{1}{2 \pi  i} \int_{b \Om} G_1 (z_s', t, \la)  l(t,  \la) \hht{\eta} (t, \la) \, dt^{\la}  \\ 
&\quad- \frac{1}{2 \pi  i}  \int_{b \Om} G_2 (z_s', t, \la)  l(t,  \la) \hht{\eta} (t, \la) \, dt^{\la}. 
\end{align*} 
In the same way as we did earlier for $w^+_{\gm}$, together with \re{etarhatHe}, one can prove that 
\[
| w_{\hht{\eta}}^{+} (z_s', \la_j) - (w_{\wti{\eta}}^{\la_0 } )^{+} (z_s ')| \ra 0,  \quad s = 1, \dots, N_2 
\]
as $|\la_j - \la_0 | \to 0$. 
In fact same proof shows that 
$
| w_{\hht{\eta}} (\cdot, \la_j) - w_{\wti{\eta}}^{\la_0 } (\cdot)|_{C^0(\ov{\Om}) } \ra 0.  
$
Gathering the results and letting $|\la_j - \la_0 | \to 0$ in \re{wetahatbc}, \re{wetahatipc} and \re{wetahatbpc}, we see that $w_{\wti{\eta}}$ is the solution to
\[ 
\begin{cases}
\pa_{\ov{z}} w_{\wti{\eta}}(z) + A^{\la_0}(z) w_{\wti{\eta}}(z) + B^{\la_0}(z) \ov{w_{\wti{\eta}}(z)} = 0 & \text {in $\Om^{\la_0}$;} \\ 
Re[\ov{l^{\la_0}(\zeta)} w_{\wti{\eta}}(\zeta)] = 0 & \text {on $b \Om^{\la_0}$}, 
\end{cases}
\] 
satisfying the conditions
\[ 
w_{\wti{\eta}}^{\la_0} (z_r) = 0, \quad 
(w_{\wti{\eta}}^{\la_0})^{+} ( z_s' ) = 0, \quad r = 1,...,N_1,  \quad s= 1,..., N_2.
\] 
By \rp{c1ut}, we have $ w_{\wti{\eta}}^{\la_0} \equiv 0$ in $\Om^{\la_0}$. By \re{wetabv}, we have
\[
w_{\eta}(t, \la_j) + w_{\gm} (t, \la_j) + w_F (t, \la_j) =  l (t, \la_j) \gm(t, \la_j) + i l (t, \la_j) \eta (t, \la_j), \quad t \in b \Om. 
\]
Dividing on both sides by $N_j$ we  get
\[
w_{\hht{\eta} }(t, \la_j) + \frac{1 }{N_j} w_{\gm} (t, \la_j) +  \frac{1 }{N_j} w_F (t, \la_j)   =  \frac{1}{N_j} l (t, \la_j) \gm(t, \la_j)  + i l (t, \la_j) \hht{\eta} (t, \la_j)  , \quad t \in b \Om. 
\]
Letting $|\la_j - \la_0 | \to 0$, the left-hand side converges uniformly to $w_{\wti{\eta}}^{\la_0} \equiv 0$. Hence the right-hand side converges uniformly to $0$. Since $|l (t, \la_j)|$ is bounded below by a positive constant uniform in $t$ and $\la$, and $\hht{\eta} (t, \la_j)$ converges uniformly to $\wti{\eta}$, we conclude that  $\wti{\eta} \equiv 0$. 

On the other hand we have $|\wti{\eta} |_{0} = \lim_{j \to \infty} |\hht{\eta}(\cdot, \la_{j})|_{0} = 1$, which is a contradiction. This shows that $|\eta(\cdot, \la)|_{0}$ is bounded in $\la$. It then follows by \re{pArie2}, which can be written as $ \eta^{\la} = - \mc{N}^{\la} \eta^{\la} + \mc{M}^{\la} \gm_0^{\la}$, and \rl{redopt} that
for any $0 < \nu' < \nu$,
\begin{align*}
\| \eta \|_{\nu',0} 
&= \| \mc{N}  \eta \|_{\nu', 0} + \| \mc{M} \gm_0 \|_{\nu,0} \leq C_{\nu'} C_{1,0} \| \eta \|_{0,0}. 
\end{align*}

Next we show $\eta(\cdot,\la)$ is continuous in $\la$ in the $C^{0}(b \Om)$-norm. Suppose this is not the case, then there exists some $\la_0$ and a sequence $\la_{j} \to \la_0$, such that $|\eta(\cdot,\la_{j}) - \eta(\cdot, \la_0)|_{0} \geq \ve > 0$. 
Since we have shown that $|\eta(\cdot,\la)|_{\nu'}$ is bounded for any $0 < \nu' < \nu$, passing to a subsequence if necessary, we have that $\eta(\cdot,\la_{j})$ converges to some $\eta^{\ast}$ in the $C^{\nu'}(b \Om)$-norm. 
As before we can show that $w_{\eta}(\zeta, \la_0)$ converges to $w_{\eta_{\ast}} (\zeta)$ uniformly on $b \Om$. Letting $| \la_j - \la_0 | \to 0$ in \re{wetabc} and \re{c1muucip}, we see that $w_{\eta_{\ast}}$ is a solution to the following Problem A: 
\eq{wetastarbvp}
\begin{cases}
\pa_{\ov{z}} w_{\eta_{\ast}}(z) + A^{\la_0}(z) w_{\eta_{\ast} }(z) + B^{\la_0}(z) \ov{w_{\eta_{\ast}}(z)} = 0  , \quad \text{in $\Omega$}; \\ 
Re[\ov{l^{\la_0}(\zeta)} w_{\eta_{\ast}}(\zeta)] =  \gm(\zeta, \la_0 ) - Re[\ov{l^{\la_0}(\zeta)} w^{\la_0}_{\gm}(\zeta)] - Re[\ov{l^{\la_0}(\zeta)} w^{\la_0}_F(\zeta)], \quad \text{on $b\Omega$}, 
\end{cases}
\eeq
and in addition $w_{\eta_{\ast}}$ satisfies the condition:  
\begin{gather*}
w_{\eta_{\ast}}(z_r^{\la_0} ) = a_r(\la_0) + i b_r(\la) - w_{\gm} (z_r, \la_0) - w_F (z_r, \la_0) ,  \quad r = 1,...,N_1; 
\\ 
w_{\eta_{\ast}}^{+} (z_s^{\la_0}) =  l(z_{s}', \la_0)( \gm(z_{s}', \la_0) + i c_{s}(\la_0)) - w_{\gm}^{+} (z_{s}', \la_0) -   w_{F}(z_{s}', \la_0), \quad s = 1,...,N_2. 
\end{gather*}

On the other hand, we have assumed that $\eta(\cdot, \la_0)$ is also a solution to \re{wetastarbvp} satisfying the above conditions. Hence by \rp{c1ut} we have $w_{\eta} (\cdot, \la_0) = w_{\eta_{\ast}} (\cdot)$. Now by \re{wetabv}, we have for $\zeta\in b\Omega$
\begin{gather} \label{wetarla0} 
w_{\eta} (\zeta, \la_0)  = l (\zeta, \la_0) \gm (\zeta, \la_0) + il (\zeta, \la_0) \eta (\zeta, \la_0) - w_{\gm} (\zeta, \la_0) - w_F  (\zeta, \la_0); 
\\
\label{wetarlaj} 
w_{\eta} (\zeta, \la_j)  = l (\zeta, \la_j) \gm (\zeta, \la_j) + il (\zeta, \la_j) \eta (\zeta, \la_j) - w_{\gm} (\zeta, \la_j) - w_F  (\zeta, \la_j).
\end{gather}
By using the fact that $|\eta^{\la} |_{\nu'}$ ($\nu' > 0$) is uniformly bounded in $\la$, we see by the same arguement as before that $w_{\eta}(\cdot, \la_j)$ converges uniformly to $w_{\eta^{\ast}}$ on $b \Om$, as $|\la_j - \la_0 | \to 0$. Passing the limit in equation \re{wetarlaj} we get 
\[ 
w_{\eta} (\zeta, \la_0) = w_{\eta_{\ast}} (\zeta) = l (\zeta, \la_0) \gm (\zeta, \la_0) + il (\zeta, \la_0) \eta_{\ast} (\zeta) - w_{\gm} (\zeta, \la_0) - w_F  (\zeta, \la_0) , \  \zeta \in b \Om. 
\]
Comparing this with equation \re{wetarla0}, we get $\eta_{\ast} (\zeta) = \eta (\zeta, \la_0)$. But
we also have $|\eta^{\ast} (\cdot) - \eta (\cdot, \la_0)|_0 \geq \ve >0$, which is impossible. Hence we conclude that $\eta (\cdot, \la)$ is continuous in $\la$ in the $C^0(b \Om)$ norm, and $\eta^{\la} \in \mc{C}^{\nu', 0} (b\Om^{\la})$ for any $0 < \nu' < \nu $. 

Now by the proof of \rl{wgmle} in Step 2 one can show that $w^{\la}_{\eta} \in \mc{C}^{\nu',0} (b \Om^{\la})$.    

\subsection{The $\mc{C}^{\nu, 0}$ estimate} 
In this section we prove part (ii) of \rt{RHmt}, namely we show that the solution $w^{\la} \in \mc{C}^{\nu-,0}(\ov{\Om^{\la}})$ found in part (i) is in fact in the class $\mc{C}^{\nu, 0} (\ov{\Om^{\la}} )$.  

First we make a reduction to the case $F^{\la} \equiv 0$. Recall that $A^{\la}$, $B^{\la}$ and $F^{\la}$ are taken to be zero outside $\Om^{\la}$. So we can write 
\eq{w_wtiF}
w_{F} (z, \la) = - \frac{1}{\pi} \iint_{\C} G_1(z, \zeta, \la)  F(\zeta, \la) \, dA(\zeta^{\la}) 
- \frac{1}{\pi} \iint_{\C}  G_2(z, \zeta, \la) \ov{F(\zeta, \la)} 
\, dA(\zeta^{\la}). 
\eeq 
By \rl{wFle}, $w_{F}^{\la} \in\mc{C}^{\all_p, 0} (\C)$. Furthermore, $w_F^{\la}$ satisfies the equation 
\[
\pa_{\ov{z^{\la}}} w_F^{\la} (z^{\la}) + A^{\la}(z^{\la}) w_F^{\la}(z^{\la}) + B^{\la} (z^{\la}) \ov{w_F^{\la} (z^{\la})} = F^{\la} (z^{\la}), \quad z \in \Om. 
\]

Now let $w_\ast^{\la}$ be the solution to the following boundary value problem 
\[
\begin{cases}
\pa_{\ov{z^{\la}}} w_\ast^{\la} (z^{\la}) + A^{\la}(z^{\la}) w_\ast^{\la}(z^{\la}) + B^{\la} (z^{\la}) \ov{w_\ast^{\la} (z^{\la})} = 0 & \quad z \in \Om; 
\\ 
Re[\ov{l^{\la}} w_\ast^{\la} ] (\zeta^{\la})  = \gm^{\la}(\zeta^{\la}) - Re[\ov{l^{\la}} w^{\la}_{F} ] (\zeta^{\la}) &\quad \zeta \in b \Om, 
\end{cases} 
\]
and in addition $w_\ast^{\la}$ satisfies the conditions: 
\begin{gather*} 
w_\ast(z_r, \la) = a_r(\la) + i b_r(\la) - w_{F}(z_r, \la)  
\\
w_\ast^{+} (z_s', \la) = l(z_s', \la)( \gm(z_s', \la) + i c_s(\la)) - w_{F}(z_s', \la) 
\end{gather*}
on the normally distributed set $\{ \{ z_r\}_{r=1}^{N_0}, \{ z_s'\}_{s=1}^{N_1} \}$. Then any solution to the original boundary value problem \re{pAmc} satisfying condition \re{mcwucip} equal to $w_\ast^{\la} + w_{F}^{\la}$. Since $w^{\la}_{F} \in \mc{C}^{\all_p}  (\ov{ \Om^{\la}})$, it suffices to consider the following boundary value problem for $w_\ast^{\la}$: 
\eq{wrbvp}
\begin{cases}
\pa_{\ov{z^{\la}}} w_\ast^{\la} (z^{\la}) + A (z, \la) w_\ast^{\la}(z^{\la}) 
+ B (z, \la)  \ov{w_\ast^{\la} (z^{\la})} = 0 & \quad z \in \Om; 
\\ 
Re[\ov{l^{\la}} w_\ast^{\la} ] (\zeta^{\la})  = \gm^{\la}(\zeta^{\la}) &\quad \zeta \in b \Om, 
\end{cases} 
\eeq
where $\gm^{\la} \in \mc{C}^{\nu,0} (b \Om^{\la})$, for $\nu = \min (\all_p, \mu)$. 

We now use an idea of Vekua to reduce problem \re{wrbvp} to one on simply-connected domain. First we introduce some notations. For a bounded domain $\Om$ whose boundary consists of $m+1$ connected component, we denote by $b \Om_0$ the outer boundary of $\Om$, and we let $\Om_0$ be the bounded simply-connected domain interior to the contour $b \Om_0$. We denote the boundary of the holes in the interior of $\Om_0$ by $b \Om_j$, for $1 \leq j \leq m$, and we let $\Om_j$ be the unbounded domain exterior to the contour $b\Om_j$. 

By 
\re{nhgaerf}, we can write the solution of problem \re{wrbvp} in the form
\eq{pAwmcr}
w_\ast(z, \la) = w_{0}(z,\la) + \cdots + w_{m}(z,\la),  \quad \quad z \in \Om, 
\eeq
where for $0  \leq j \leq m$, we set 
\eq{pAwjmcr} 
w_{j}(z,\la) = \frac{1}{2 \pi i} \int_{b \Om_{j}} G_{1}(z, t, \la) w_\ast(t,\la) \, dt^{\la} + G_{2}(z, t, \la) \ov{w_\ast(t, \la)} \, \ov{dt^{\la}}. 
\eeq

Fix $j_0$ with $0 \leq j_0 \leq m$, we claim that for all $j \neq j_0$, $w_{j}^{\la} \in \mc{C}^{\nu,0}(b \Om_{j_0}^{\la})$. 
Assuming this for now, then $w_{j_0}^{\la}$ is a solution to the boundary value problem
\begin{equation} \label{pAwjfde} 
 \begin{gathered} 
\pa_{\ov{z^{\la}} } w^{\la}_{ j_{0}} (z^{\la} ) +  A (z, \la) w^{\la}_{j_0}(z^{\la} ) + B(z, \la) \ov{w_{j_0}^{\la}(z^{\la})} = 0  \quad \text {in $\Om_{j_0}^{\la}$;} 
\\ 
Re[\ov{l(z, \la)} w_{j_0}^{\la}(z^{\la})] = \gm_{j_0}(z, \la)  \quad \text {on $b \Om_{j_0}^{\la}$}, 
\end{gathered} 
\end{equation} 
where 
 $\gm_{j_0}^{\la} \in \mc{C}^{\nu,0}(b \Om_j^{\la})$ with
$ 
\gm_{j_0}(z, \la) := \gm(z, \la) - \sum_{j\neq j_0} Re[\ov{l(z, \la)}w_{j}^{\la}(z^{\la})].
$ 

If $j_0 = 0$, the above problem reduces to the   earlier case on bounded simply-connected domain. For $j_0 \geq 1$, we need to invert the unbounded domain $\Om_{j_0}$. 
We fix some point $a_{j_0}$ in $\C \sm \Om_{j_0}$ such that $\dist(a_{j_0 }, b \Om^{\la}_{j_0}) \geq c > 0$ for all $\la$. 
By a substitution $z^{\la} = \var(\zeta^{\la}) = \frac{1}{\zeta^{ \la} - a_{j_0}}$, we have for each fixed $\la \in [0,1]$, 
\begin{align*} 
\pa_{ \ov{\zeta^{\la}} } w_{j_0}^{\la} \left( \var(\zeta^{\la})) \right) 
&= \pa_{\ov{z^{\la}} } w^{\la}_{j_0} (z^{\la}) \cdot \pa_{\ov{\zeta^{\la}}} \ov{\var} \\ 
&= - \ov{ \pa_{\zeta^{\la}} \var } A^{\la}(z^{\la}) w^{\la}_{j_0} (z^{\la}) - \ov{\pa_{\zeta^{\la}} \var } B^{\la}(z^{\la} ) \ov{w^{\la}_{j_0} (z^{\la})}. 
\end{align*}
Set $\wti w_{j_0} (\zeta, \la) = w^{\la}_{j_0} (\var(\zeta^{\la}))$, 
$ \wti A(\zeta, \la) = (\ov{\pa_{\zeta^{\la}} \var }) A^{\la}(\var^{-1} (\zeta^{\la})) $, 
$\wti B (\zeta, \la)= 
(\ov{\pa_{\zeta^{\la}} \var }) B^{\la}(\var^{-1} (\zeta^{\la})$, $\wti l (\zeta, \la) = l^{\la}(\var(\zeta^{\la}))$ and  $\wti \gm_j (\zeta, \la) = \gm^{\la}_j (\var(\zeta^{\la}))$, where $\Gm^{\la} \in \mc{C}^{1+ \mu ,0}(\C) $ is the extension of $\Gm^{\la}$. 
Problem \re{pAwjfde} then becomes 
\gan
\pa_{\ov{\zeta^{\la}}} \wti w_{j_0} (\zeta, \la) +  \wti A(\zeta, \la) \wti w_{j_0} (\zeta, \la)  + \wti B(\zeta, \la) \ov{\wti w_{j_0} (\zeta, \la)} = 0, \quad \text{in $D^{\la}_{j_0}$};   
\\ 
Re \left[ \ov{\wti l(\zeta, \la)} \wti w_{j_0} (\zeta) \right]  = \wti \gm_{j_0} (\zeta, \la), \quad \text{on $b D^{\la}_{j_0}$, } 
\end{gather*}
where $D^{\la}_{j_0}$ (resp. $bD^{\la}_{j_0}$) is the image of $\Om^{\la}_{j_0}$ (resp. $b \Om^{\la}_{j_0}$), under the map $\var^{-1}: z^{\la} \mapsto \zeta^{\la} = \frac{1}{z^{\la}} + a$. 
Now, $\{D^{\la}_{j_0}\}$ is a family of bounded simply-connected domains in the class $\mc{C}^{1+\mu, 0}$.  
Note that since $A^{\la}, B^{\la}$ vanish outside $\Om^{\la}$ and in particular near infinity, we have $\wti A^{\la} \equiv \wti B^{\la} \equiv 0$ when $ \zeta^{\la}$ is near $a_{j_0}$, and $\wti A(\cdot, \la), \wti B(\cdot,\la)$ satisfy the same assumptions as $A(\cdot, \la) $ and $B(\cdot, \la)$.   
It is clear that $ \wti l, \wti \gm_{j_0} \in \mc{C}^{\nu, 0}(bD_{j_0}^{\la} )$.  

We can now apply \rt{pAscvdt} to the family of bounded simply-connected domains $D_{j_0}^{\la}$, and we get $\wti w_{j_0} \in \mc{C}^{\nu,0}(\ov{D_{j_0}^{\la}})$, for each $0  \leq  j_0 \leq  m$. (Note that the conditions on the normally distributed set are trivially satisfied since we already know from part (i) that $w_\ast^{\la} \in \mc{C}^{\nu^-,0}(\ov{\Om^{\la}})$, and hence by \rp{gCiHne} one has $w_{j_0}^{\la} \in \mc{C}^{\nu^-,0}(\ov{\Om_{j_0}^{\la}})$.)
Since
$
w^{\la}_{j_0}(z^{\la}) = \wti w_{j_0} (\var^{-1} (z^{\la}) ), 
$
we have $w_{j_0}^{\la} \in \mc{C}^{\nu, 0} (\ov{\Om_j^{\la}} )$, for each $0 \leq j_0 \leq m$.  
In view of \re{pAwmcr}, we get $w_\ast^{\la} \in \mc{C}^{\nu,0}(\ov{\Om^{\la}})$.

It remains to prove for each $j_0$ and $j \neq j_0$, we have $w_{j}^{\la} \in \mc{C}^{\nu,0}(b \Om_{j_0}^{\la})$. 
First we show that $w_{j}^{\la}$ is continuous in the $C^{0}(b \Om_{j_0}) $-norm.  Let $ d_\ell = \dist(b \Om_j, b \Om_{j_0})$ and 
$dt^{\la} = a(s,\la) ds$.
By 
\re{pAwjmcr}, we have for $z \in b \Om_{j_0}$, 
\begin{align*}
&|w_j(z,\la_{1}) - w_j(z,\la_{2})| \\ 
& \quad \leq C \sum_{i=1,2} \int_{b \Om_{j}} |\del^{\la_{1}, \la_{2}} G_{i}(z, \cdot)| |w_\ast(t,\la_{1})| \, dt^{\la_1} + |G_{i} (z,t,\la_{2})| | \del^{\la_{1}, \la_{2}}(w_\ast)(t)| \, ds \\
& \quad \leq \| w_\ast \|_{0,0}  \int_{b \Om_{j}} \frac{\Si_{i=1,2} |e^{\om_{i}(z,t,\la_{1})} -   e^{\om_{i}(z,t,\la_{2})}|}{|t^{\la_1} - z^{\la_1}| } \, dt^{\la_1} \\ 
& \qquad  + \| w_\ast \|_{0,0} \int_{b \Om_{j}} \sum_{i=1}^{2} \left| e^{\om_{i}(z,t,\la_{2})} \right|  \left| \frac{1}{|t^{\la_1} - z^{\la_1} |}  - \frac{1}{| t^{\la_2} - z^{\la_2} |} \right| \, dt^{\la_1} \\
& \qquad + \int_{b \Om_{j}} \frac{ \sum_{i=1}^{2}  \left| e^{\om_{i}(z,t,\la_{2})} \right|}{| t^{\la_2} - z^{\la_2} |} 
|\del^{\la_{1}, \la_{2}} (w_\ast)(s) |  \, ds. 
\end{align*}  
For $z \in b \Om_{j_0}$ and $t \in b \Om_j'$, we have $| t^{\la} - z^{\la} |$ is bounded below by some constant uniform in $\la$. Hence the first term in the above sum is bounded by 
\[ 
C_{1,0} C_j \|w_\ast \|_{0,0} \sum_{i=1,2} \sup_{t \in b \Om_{j}} 
| \om_{i}(\cdot ,t, \la_1) - \om_i (\cdot , t, \la_2) |_{C^0 (b \Om_{j_0}) }.   
\]
Similarly, the second term is bounded by 
$
C_{1,0} C_j \|w_\ast \|_{0,0}  \left( \Si_{i=1,2} 
\| \om_i  \|_{0,0} \right) |\Gm^{\la_{1}} - \Gm^{\la_{2}}|_{1} , 
$
and the third term  is bounded by 
$$
C_{1,0} C_j \left( \Si_{i=1,2} \| \om_i \|_{0,0} \right) |\del^{\la_{1},\la_{2}}(w_\ast a)|_{C^{0}(\ov{\Om})}, 
$$
for $\| w_\ast \|_{0,0} = \sup_{t \in b \Om_j, \la \in [0,1]} |w_\ast(t,\la)|$,  $\| \om_i \|_{0,0} = \sup_{z \in b\Om_{j_0}, t \in b \Om_j \la \in [0,1] } |\om_i (z,t, \la)| $. 

By part (i), we have $w_\ast^{\la} \in \mc{C}^{\nu^-,0}(\ov{\Om^{\la}})$ so in particular 
$|w_\ast^{\la_1} - w_\ast^{\la_2} |_0 \to 0$ as $|\la_1-\la_2| \to 0$. Together with \re{cino}, we readily see that the above expressions converge to $0$ as $|\la_{1} - \la_{2}| \to 0$. This shows that for any $j \neq j_0$, $w_{j}(\cdot, \la)$ is continuous in $\la$ in the $C^{0}(b \Om_{j_0})$- norm.   

Next, we show that $|w_j(\cdot, \la)|_{C^{\nu }(b \Om_{j_0})}$ is bounded uniformly in $\la$. In fact we will show the boundedness of $|w_j(\cdot, \la)|_{C^{\all_p}(b \Om_{j_0})}$. 
Recall that $w_j^{\la}$ is given by 
\re{pAwjmcr}: 
\eq{wjcnuest}  
w_{j}(z,\la) = \frac{1}{2 \pi i} \int_{b \Om_{j}} G_{1}(z, t, \la) w_\ast(t,\la) \, dt^{\la} + G_{2}(z, t, \la) \ov{w_\ast(t, \la)} \, \ov{dt^{\la}}, 
\eeq 
for $ z \in \Om_j,$ where $G_1$ and $G_2$ are given by 
\begin{gather*}
G_1 (z, t, \la) = \frac{e^{\om_1 (z, t, \la)} + e^{\om_1 (z, t, \la)} }{t^{\la} - z^{\la}}, \qquad 
G_2 (z, t, \la) = \frac{e^{\om_1(z, t, \la)} - e^{\om_1(z, t, \la)} }{t^{\la} - z^{\la}}, 
\end{gather*}
and $\om_i$ are given by formula \re{omiexp}. By the proof of \rl{ombdd}, we can show that
\[ 
\sup_{t \in b \Om_j} | \om_j(\cdot, t, \la)|_{\all_p} \leq C_{1,0} C(p, \Om) \left( |A(\cdot, \la)|_{L^{p}(\Om)}  + |B (\cdot, \la)|_{L^{p}(\Om)} \right),  \quad 2 < p < \infty. 
\]
Now for $z \in b \Om_{j_0}$ and $t \in b \Om_{j}$ where $j \neq j_0$, $|t^{\la} - z^{\la} |$ is bounded below by a positive constant $c_j$ uniform in $\la$. Hence it follows by the above formula that $ |G_i (\cdot, \la, \la)|_{C^{\all_p} (b \Om_0)} $ is bounded uniformly in $t \in b \Om_j$ and $\la \in [0, 1]$.  Since we know that $w_\ast^{\la} \in \mc{C}^{\nu-,0} (\ov{\Om^{\la}} ) $, in particular $|w_\ast (t, \la) |$ is 
bounded above uniformly in $t \in b \Om_j$ and $\la \in [0, 1]$. \re{wjcnuest} then implies that $|w_j(\cdot, \la)|_{C^{\nu }(b \Om_{j_0})}$ is bounded uniformly in $\la$. This proves the earlier claim and the proof is now complete.

\subsection{Differentiability in space variable: The $\mc{C}^{k+1+\mu, 0}$ estimate} 
\ 

\medskip 

In this subsection we prove part (ii) of \rt{RHmt}, namely the following: 
\begin{prop}
Let $\Om$ be a bounded domain in $\C$ such that $b\Om$ has $m+1$ connected components. Let $k$ be a non-negative integer and $0 
< \mu < 1$. Let $\Gm^{\la}: \Om \to \Om^{\la}$ be a $\mc{C}^{k+1+\mu, 0}$ embedding. For each $\la$, let $w^{\la}$ be the unique solution to Problem \re{pAmc} on $\Om^{\la}$ with index $n > m-1$,    
satisfying condition \re{mcwucip}  
on a normally distributed set $\{z_r^{\la}, z_s' \}$ in $\Om$. 
Suppose $A^{\la}, B^{\la}, F^{\la} \in \mc{C}^{k+\mu, 0}(\ov{\Om ^{\la}})$, $l^{\la}, \gm^{\la} \in \mc{C}^{k+1+\mu, 0}(b \Om^{\la})$, and $a_r, b_r, c_s \in C^0 ([0,1])$. Then $w^{\la} \in \mc{C}^{k+1+\mu, 0}(\ov{\Om^{\la}})$.  Furthermore, there exists some constant $C$ independent of $\la$ such that 
\begin{gather} \label{pAmcsupest}    
\| w^{\la} \|_{0,0} \leq C \left( \| F^{\la} \|_{0,0} + \| \gm^{\la} \|_{\mu,0} + \sum_{r=1}^{N_0} \left( |a_r |_0 + |b_r|_0 \right)  + \sum_{s=1}^{N_1} |c_s|_0 \right), 
\\  \nn 
\| w^{\la} \|_{k+1+\mu,0} \leq C \left( \| F^{\la} \|_{k+\mu,0} + \| \gm^{\la} \|_{k+1+\mu,0} + \sum_{r=1}^{N_0} \left( |a_r |_0 + |b_r|_0 \right) + \sum_{s=1}^{N_1} |c_s|_0 \right). 
\end{gather}  
\end{prop}
\nid \tit{Proof.} 
As before we write $A(z, \la) := A^{\la} (z^{\la})$, $B(z, \la) := B^{\la} (z^{\la})$ and $F(z, \la) := F^{\la} (z^{\la})$. 
Note for fixed $\la$, $w^{\la} \in C^{k+1+\mu}(\ov{\Om^{\la}})$ by \rp{pArt}. By \rp{cinhd}, it suffices to show that $w(\cdot,\la)$ is continuous in $\la$ in the $C^0 (\ov{\Om})$ norm and bounded uniformly in $\la$ in the $C^{k+1+\mu}(\ov{\Om})$-norm. 

Like before we can make a reduction to the case $F^{\la} \equiv 0$.  Define $w_{F^{\la}}$ by formula \re{w_wtiF}. For the reduction to work we need to show that $w_F^{\la} \in \mc{C}^{k+1+\mu, 0} (\ov{\Om^{\la}})$. Now $w_F^{\la}$ satisfies the equation: 
\eq{wFiteqn}  
\pa_{\ov{z^{\la}}} w_{F^{\la}} + A^{\la} w_{F^{\la}} + B^{\la} \ov{w_{F^{\la}}} = F^{\la} \quad \text{ in $ \Om^{\la}$. } 
\eeq
Since by assumption $F^{\la}$ is at least H\"older continuous, \rl{wFle} shows that $w_{F^{\la}} \in\mc{C}^{1^-, 0} (\ov{\Om^{\la}})$. It then follows from \re{wFiteqn} that 
$\pa_{\ov{z^{\la}}} w_{F}^{\la} \in \mc{C}^{\mu, 0} (\ov{\Om^{\la}})$. 
For each fixed $\la$, we apply Cauchy-Green formula to $w^{\la}_{F}$ on $\Om^{\la}$ to get 
\[
w_{F}(z, \la) = T_{\Om^{\la}} (\pa_{\ov{z^{\la}}} w_{F^{\la}}) + \frac{1}{2 \pi i} \int_{b \Om} \frac{w_{F}(t, \la) }{t^{\la} - z^{\la}} d t^{\la}, \quad  z \in \Om 
\] 
where we write $w_{F} (z, \la) = w_{F^{\la}} (z^{\la} )$. Taking $\pa_{z^{\la}}$ on both sides, we have
\eq{wFtizd} 
\pa_{z^{\la}} w_{F^{\la}}(z^{\la}) = \Pi_{\Om^{\la}} (\pa_{\ov{z^{\la}}} w_{F^{\la}}) + \frac{1}{2 \pi i} \int_{b \Om} \frac{w_{F}(t, \la) }{(t^{\la} - z^{\la})^{2}} \, dt^{\la} , 
\eeq 
where $ \Pi_{\Om^{\la}} := \pa_z T_{\Om^{\la}}$. 
Since $w_{F^{\la}} \in\mc{C}^{1^-, 0} (\ov{\Om^{\la}})$, the second term is in $\mc{C}^{1-,0} (\ov{\Om^{\la}})$. 
Now applying \rl{zdTop} to $\pa_{\ov{z^{\la}}} w_{F^{\la}} \in \mc{C}^{\mu, 0} (\ov{\Om^{\la}})$, we have $ \Pi_{\Om^{\la}} (\pa_{\ov{z^{\la}}} w_{F^{\la}}) \in \mc{C}^{\mu, 0} ( \ov{\Om^{\la}})$. 
It then follows from \re{wFtizd} that $ \pa_{z^{\la}} w_{F^{\la}}(z^{\la})  \in \mc{C}^{\mu, 0} (\ov{\Om^{\la}})$. 
This shows that $w_{F^{\la}} \in \mc{C}^{1+\mu,0 } 
(\ov{ \Om^{\la}})$.  We can iterate this argument by using \re{wFiteqn} and \re{wFtizd} so that eventually we get $w_{F^{\la}} \in \mc{C}^{k + 1 + \mu,0 }  (\ov{ \Om^{\la}})$.  
This proves the reduction. 

The problem now is reduced to showing that if $w_\ast^{\la}$ solves \re{wrbvp}, where $\gm^{\la} \in \mc{C}^{k+1+\mu,0 } (b \Om^{\la})$, then $w_\ast^{\la} \in \mc{C}^{k+1+\mu, 0} (\ov{\Om^{\la}})$. 
As in the proof of part (ii), we reduce the problem to one on simply-connected domain, and we use the same notation $w_j$ and $b \Om^j$ as before. 

Note also that by the assumptions and the result of part (ii), $w_\ast^{\la}  \in  \mc{C}^{\frac{p-2}{p }, 0}(\ov{\Om^{\la}} )$ for any $2 < p < \infty$. Hence $w_\ast^{\la}  \in  \mc{C}^{1-, 0}(\ov{\Om^{\la}} )$. 

In view of formula \re{nhgaerf} with $F \equiv 0$, we can write the solution of problem \re{wrbvp} in the form
$
w_\ast(z, \la) = w_{0}(z,\la) + \cdots + w_{m}(z,\la),   z \in \Om, 
$
where we set 
\[ 
w_{j}(z,\la) = \frac{1}{2 \pi i} \int_{b \Om_{j}} G_{1}(z, t, \la) w_\ast(t,\la) \, dt^{\la} + G_{2}(z, t, \la) \ov{w_\ast(t, \la)} \, \ov{dt^{\la}}.
\]

Fix $j_0$ with $0 \leq j_0 \leq m$, we need to show that for all $j \neq j_0$, we have $w_{j}^{\la} \in \mc{C}^{k+1+\mu,0}(b \Om_{j_0}^{\la})$. The rest of the argument follows the same as in part (ii). 

That $w_{j}^{\la}$ is continuous in the $C^{0}(b \Om_{j_0}) $-norm is proved in part (ii), 
so we only need to show that $|w_j(\cdot, \la)|_{C^{k+1+\mu}(b \Om_{j_0})}$ is bounded in $\la$. Recall that
$w_{j}^{\la}$ satisfies the equation: 
\[ 
\pa_{\ov{z^{\la}}} w^{\la}_{j} (z^{\la} ) +  A^{\la}(z^{\la}) w^{\la}_{j} (z^{\la} ) +  B^{\la}(z^{\la}) \ov{w_j^{\la}(z^{\la})} = 0, \quad  z^{\la} \in \Om_{j}^{\la}. 
\] 
By part (i), we have $w_\ast^{\la} \in \mc{C}^{1^-,0}(\ov{\Om^{\la}})$. In view of \rp{gCiHne} and \re{pAwjmcr}, 
$w_{j}^{\la} \in \mc{C}^{1^-, 0}(\ov{\Om_{j}^{ \la}})$.    
Fix an open set $V_{j}$ so that $V_j \subset \subset  (\Om_j^{\la} \cap \Om^{\la})$ and $\dist (b \Om_{j_0}^{\la}, V_j ) \geq \del_0$ for all $\la$. Certainly $w^{\la}_j$ satisfy the same equation on $V_j$: 
\eq{wj'deq}
\pa_{\ov{z} } w^{\la}_{j} (z ) +  A^{\la}(z) w^{\la}_{j}(z) +  B^{\la}(z) \ov{w_{j}^{\la}(z)} = 0, \quad  z \in V_j. 
\eeq
Here $A^{\la}, B^{\la} \in \mc{C}^{k + \mu,0} (\ov{V_j}) \subset \mc{C}^{\mu,0} (\ov{V_j }) $ and $ w^{\la}_{j} \in \mc{C}^{1^-,0} (\ov{V_j})$, thus equation \re{wj'deq} implies that $ \pa_{\ov{z}}w_{j}^{\la} \in \mc{C}^{\mu,0} (\ov{V_j})$.   

Applying Cauchy-Green formula to $V_j$, we have for each fixed $\la$, 
\[ 
 w_{j}^{\la}(z) = T_{V_j} (\pa_{\ov{z}} w_j) (z) + \frac{1}{2 \pi i} \int_{b V_j} \frac{w_j^{\la}(\zeta)}{\zeta- z }  \, d\zeta, \quad z \in V_j. 
\] 
Taking $\pa_{z}$ on both sides we get
\eq{pAwj'zd}
\pa_{z} w_{j}^{\la} (z) = \pa_{z} T_{V_j}  (\pa_{\ov{z}} w_{j}^{\la})(z) + \frac{1}{2 \pi i} \int_{b V_j } \frac{w_j^{\la}(\zeta) }{(\zeta - z)^{2}} \, \, d \zeta,  
\quad z \in V_j. 
\eeq
Now, for $\zeta \in b V_j $ and $z  \in b\Om_{j_0}^{\la}$, $|\zeta - z| $ is bounded below by a positive constant independent of $\la$, so the second term above is bounded by
$C |w_j^{\la} |_{C^0 (b V_j)}$ for some $C$ independent of $\la$.  
On the other hand by \rl{zdTop}, 
$
| \pa_{z} T_{V_{j}} (\pa_{\ov{z}} w_{j}^{\la})  |_{C^{\mu} ( \ov{V_j})}
\leq C | \pa_{\ov{z}} w_{j}^{\la} |_{C^{\mu} ( \ov{V_j})}. 
$ 
Thus
$
| \pa_{z} w_j^{\la} |_{C^{\mu}( \ov{V_j})} 
\leq C \left( | \pa_{\ov{z}} w_{j'}^{\la} |_{C^{\mu}(\ov{V_j})} +  |w_{j}^{\la} |_{C^{0}(\ov{V_j})} \right), 
$
where $C$ is some constant independent of $\la$.  
Hence $|w_{j}^{\la}|_{C^{1+\mu}(V_j)}$ is uniformly bounded in $\la$. 
By repeating the argument using \re{wj'deq} and \re{pAwj'zd}, we can show that 
$|w_{j}^{\la}|_{C^{k+1+\mu}(V_j)}$ is uniformly bounded in $\la$.  
Since $\Gm^{\la} \in \mc{C}^{k+1+\mu, 0}(\ov{\Om})$ and $b \Om_{j_0}  \subset V_j$,  $|w_j^{\la}(z^{\la})|_{C^{k+1+\mu} (b \Om_{j_0}) }$ is bounded uniformly in $\la$.  
This completes the proof of the claim and thus the theorem.  
\hfill \qed 

\subsection{Differentiability in parameter: the $\mc{C}^{k+1+\mu, j}$ estimate} 
\ 

\medskip 

In this subsection we prove part (iii) of \rt{RHmt}, namely the following: 
\th{mt2} 
Let $k, j,\mu$ be as in \rt{mtintro}.
Let $\Om$ be a bounded domain in $\C$ such that $b\Om$ has $m+1$ connected component. Suppose $\Gm^{\la}: \Om \to \Om^{\la}$ is a $\mc{C}^{k+1+\mu,j}$ embedding. For each $\la$, let $w^{\la}$ be the unique solution to Problem \re{pAmc} on $\Om^{\la}$ with index $n > m-1$, satisfying condition \re{mcwucip} 
on a normally distributed set $\{ \zeta_r^{\la}, (\zeta_s')^\la \}$ in $\Om^\la$. 
Suppose $A^{\la}, B^{\la}, F^{\la} \in \mc{C}^{k+\mu,j}(\ov{\Om ^{\la}})$, $l^{\la}, \gm^{\la} \in \mc{C}^{k+1+\mu,j}(b \Om^{\la})$, and $a_r, b_r, c_s \in C^{j} ([0,1])$. Then $w^{\la} \in \mc{C}^{k+1+\mu,j}(\ov{\Om^{\la}})$.  Furthermore, there exists some constant $C$ independent of $\la$ such that 
\begin{equation} \label{lader_est}  
\| w^{\la} \|_{k+1+\mu, j} \leq C \left( \| F^{\la} \|_{k+\mu,j} + \| \gm^{\la} \|_{k+1+\mu,j} + \sum_{r=1}^{N_0} \left( |a_r |_j + |b_r|_j \right) + \sum_{s=1}^{N_1}|c_s|_j \right). 
\end{equation} 
\eth 
\begin{proof}
We proceed with induction on $j$. The $j=0$ case is proved in Theorem \re{RHmt} part (ii). Suppose the statement has been proved for $j-1$. In particular we know that $w^{\la} \in \mc{C}^{k+1+\mu,j-1}(\ov{\Om^{\la}})$ for $k \geq j \geq 1$.

Let $\Psi^{\la} = (\Gm^{\la})^{-1}$ and $v^{\la} (\zeta) = w^{\la} \circ \Gm^{\la} (\zeta)$. Then $w^{\la} = v^{\la} \circ \Psi^{\la}$, $\Psi^{\la}(z^{\la}) = \zeta$. For fixed $\la$, we have
\begin{align*}
\pa_{\ov{z^{\la}}} w^{\la} &= \DD{\Psi^{\la}}{\ov{z^{\la}}} (\Gm^{\la}(\zeta)) \pa_{\zeta} v^{\la} + \DD{\ov{\Psi^{\la}}}{\ov{z^{\la}}} (\Gm^{\la} (\zeta)) \pa_{\ov{\zeta}} v^{\la} = h^{\la}(\zeta) \left( \pa_{\ov{\zeta}} v^{\la} + e^{\la}(\zeta) \pa_{\zeta} v^{\la} \right), 
\end{align*} 
where we denote
\[
e^{\la} (\zeta) := \DD{\Psi^{\la}}{\ov{z^{\la}}} (\Gm^{\la}(\zeta)) \left( \DD{\ov{\Psi^{\la}}}{\ov{z^{\la}}} (\Gm^{\la} (\zeta)) \right)^{-1},  
\quad h^{\la} (\zeta) :=  \DD{\ov{\Psi^{\la}}}{\ov{z^{\la}}} (\Gm^{\la} (\zeta)). 
\]
Note that since $\Psi^0 = \Gm^0 = id$, we have $e^0 \equiv 0$ and $h^0 \equiv 1$ on $ \Om$. Thus for $|\la|$ small, $ |e^{\la}| < 1$. Now $v^{\la}$ is the solution to 
\begin{gather} \label{mcfdbvp} 
\begin{cases}
\pa_{\ov{\zeta}} v^{\la} + e^{\la}(\zeta) \pa_{\zeta} v^{\la} + A_1^{\la}(\zeta) v^{\la}(\zeta) + B_1^{\la}(\zeta) \ov{v^{\la}(\zeta)} = F_1^{\la}(\zeta) & \text {in $\Om$;} \\ 
Re[\ov{ l_1^{\la}(\zeta)} v^{\la}(\zeta)] = \gm_1^{\la}(\zeta) & \text {on $b \Om$}, 
\end{cases} 
\end{gather}
where 
$A_1^{\la} (\zeta) = (h^{\la})^{-1} A^{\la} (\Gm^{\la} (\zeta)), 
 B_1^{\la} (\zeta) = (h^{\la})^{-1} B^{\la} (\Gm^{\la} (\zeta)), $ and 
\begin{gather*}
F_1^{\la} (\zeta) = (h^{\la})^{-1} F^{\la} (\Gm^{\la} (\zeta)), \quad l_1^{\la}(\zeta) = l^{\la} (\Gm^{\la} (\zeta)), \quad 
\gm_1^{\la}(\zeta) = \gm^{\la} (\Gm^{\la} (\zeta)). 
\end{gather*}
By assumption, $e^{\la}, A^{\la}_1, B^{\la}_1, F^{\la}_1 \in \mc{C}^{k+\mu, j} (\Om^{\la})$, and $l_1^{\la}, \gm_1^{\la} \in \mc{C}^{k+1+\mu,j} (b\Om^{\la})$. 
Subtracting equation \re{mcfdbvp} by the same equation with $\la$ replaced by $\la_0$ and dividing the resulting difference by $\la- \la_0$ we get respectively on $\Omega$ and $b\Omega$
\begin{gather} \label{diffqt_eq} 
\begin{cases}
\pa_{\ov{\zeta}} (D^{\la, \la_0} v) + e^{\la_0} \pa_{\zeta} (D^{\la, \la_0} v^{\la}) + A_1^{\la_0} D^{\la, \la_0}v + B_1^{\la_0} \ov{D^{\la, \la_0}v} = F_2^{\la, \la_0}; 
\\ 
Re[\ov{ l_1^{\la_0}} D^{\la, \la_0} v ] =
\gm_2^{\la, \la_0} 
\end{cases} 
\end{gather}
where we used the notation $D^{\la, \la_0} f := ( f^{\la} - f^{\la_0} ) / (\la - \la_0)$, and 
\begin{gather} \label{F2lala0}  
F_2^{\la, \la_0} := D^{\la, \la_0} F_1 - (D^{\la, \la_0} e) \pa_{\zeta} v^{\la} - (D^{\la, \la_0} A_1)v^{\la} - (D^{\la, \la_0} B_1) v^{\la}, 
\\ \label{gm2lala0} 
\gm_2^{\la, \la_0} := D^{\la, \la_0} \gm_1 - Re[\ov{ (D^{\la, \la_0} l_1)} v^{\la}  ].  
\end{gather}
In addition, $D^{\la, \la_0} v$ satisfies the condition
\begin{equation} \label{dq_uc}  
\begin{gathered}	
D^{\la, \la_0} v  (\zeta_r) = D^{\la, \la_0} a_r + i D^{\la, \la_0} b_r,  \quad r = 1, \dots N_0; 
\\   
D^{\la, \la_0} v  (\zeta_s') = D^{\la, \la_0} \left\{ l (\zeta_s', \cdot) [ \gm (\zeta_s', \cdot) + i c_s(\cdot) ]  \right\}, 
\quad  \quad s = 1, \dots N_1. 
\end{gathered}
\end{equation}

Consider the problem 
\begin{gather} \label{ladeqn}  
\begin{cases}
\pa_{\ov \zeta} u^{\la} + e^{\la} \pa_{\zeta} u^{\la} + A_1^{\la} u^{\la} + B_1^{\la} \ov{u^{\la} } = F_2^{\la} & \text {in $\Om$; } \\ 
Re[\ov{ l_1^{\la}} u^{\la}  ] = \gm_2^{\la} & \text {on $b \Om$}, 
\end{cases} 
\end{gather}
where 
\begin{gather} \label{F2la}  
F_2^{\la} := (\pa_{\la} F_1^{\la}) - (\pa_{\la} e^{\la}) \pa_{\zeta} v^{\la}- (\pa_{\la} A_1^{\la}) v^{\la} - (\pa_{\la} B_1^{\la}) v^{\la}, 
\\ \label{gm2la} 
\gm_2^{\la} := (\pa_{\la} \gm_1^{\la}) - Re[\ov{ (\pa_{\la} l_1^{\la} ) } v^{\la}  ].  
\end{gather}
In addition, we require that $u^{\la}$ satisfies the condition: 
\begin{equation} \label{lad_ula_uc}
\begin{gathered} 
u^{\la}(\zeta_r) = a_r'(\la) + i \pa b_r'( \la), 
\quad r = 1, \dots N_0;  
\\  
u^{\la} (\zeta_s') =  \pa_{\la} \left\{  l (\zeta_s', \la) \left[ \gm (\zeta'_s, \la) + i c_s (\la)  \right] \right\}, 
\quad s=1,\dots, N_1.  
\end{gathered}
\end{equation}

By assumption, $\pa_{\la} e^{\la}, \pa_{\la} A_1^{\la}, \pa_{\la} B_1^{\la}, \pa_{\la} F_1^{\la} \in \mc{C}^{k-1+\mu,j-1}(\ov{\Om^{\la}})$, and $\pa_{\la} \gm_1^{\la}, \pa_{\la} l_1^{\la}  \in \mc{C}^{k+\mu, j-1}(b \Om^{\la})$. By the induction hypothesis,  $v^{\la} = w^{\la} \circ \Gm^{\la} \in \mc{C}^{k+1+\mu, j-1}(\ov{\Om^{\la}})$, and there exists some constant $C$ independent of $ \la$ such that 
\begin{equation} \label{lad_vla_est}
\| v^{\la} \|_{\mc{C}^{k+1+\mu, j-1}  (\Om)} 
\leq C \left( \| F ^{\la} \|_{\mc{C}^{k+\mu, j-1} (\Om^{\la})} 
+  \| \gm^{\la} \|_{\mc{C}^{k+1+ \mu,j-1}(b \Om^{\la}) } \right) .  
\end{equation} 
Thus in view of \re{F2la} and \re{gm2la} we have $F_2^{\la} \in \mc{C}^{k-1+\mu , j-1}(\ov{\Om^{\la}})$, and $ \gm_2^{\la} \in \mc{C}^{k+\mu, j-1} (b \Om^{\la}) $. 
To solve problem \re{ladeqn} we would like to make a change of coordinates $\psi^{\la}: \zeta \mapsto \zeta_1^{\la}$ so that 
\eq{isocd}
\pa_{\ov{\zeta}} u^{\la} (\zeta) + e^{\la}(\zeta) \pa_{\zeta} u^{\la} (\zeta) = f_1^{\la}(\zeta_1^{\la}) \pa_{\ov {\zeta_1^{\la}} } u^{\la}_1 (\zeta_1^{\la}), \quad u_1^{\la} (\zeta_1^{\la}) = u^{\la}(\zeta), 
\eeq
for some function $f_1^{\la}$. 
We seek $\psi^{\la}$ in the form $\zeta_1^{\la} = \psi^{\la}(\zeta) = \zeta + g^{\la}(\zeta)$, where $g^{0} \equiv 0$. By \re{isocd} we obtain the following equations for $g^{\la}$:  
\begin{gather} \label{nndfeqn}    
\pa_{\ov{\zeta}} g^{\la} (\zeta) + e^{\la} (\zeta) \pa_{\zeta} g^{\la} (\zeta) + e^{\la} (\zeta) = 0, \quad |e^{\la}|_0 < 1,  \\
\label{f1eqn}
f_1^{\la}(\zeta_1^{\la}) = 1 + \pa_{\ov{\zeta}} \ov{g^{\la}} (\zeta) + e^{\la} (\zeta) \pa_{\zeta} \ov{g^{\la}} (\zeta), 
\quad \zeta = (\psi^{\la})^{-1} (\zeta). 
\end{gather} 
By Theorem 6.4 in \cite{B-G-R14} , there exists a solution $g^{\la}$ to the equation \re{nndfeqn} in the class $\mc{C}^{k+1+\mu,j} (\ov{\Om^{\la}})$. Hence $\psi^{\la} \in \mc{C}^{k+1+\mu,j} (\ov{\Om^{\la}})$. In the following we write $D^{\la} = \psi^{\la} (\ov{\Om^{\la}})$. 
From equation \re{f1eqn} we get $f_1^{\la} \in \mc{C}^{k+ \mu, j}  (\ov{D^{\la}})$. Since $g^0 \equiv 0$, by \re{f1eqn} we see that $|f_1^{\la}|$ is nowhere vanishing on $D^{\la}$ when $\la$ is small.  In view of \re{isocd} and \re{ladeqn} we then get 
\begin{gather*}   
\begin{cases}
\pa_{\ov{\zeta_1^{\la}}} u_1^{\la} + A_2^{\la}(\zeta_1^{\la}) u_1^{\la_1} + B_2^{\la}(\zeta_1^{\la}) \ov{u_1^{\la} } = F_3^{\la} (\zeta_1^{\la})& \text {in $D^{\la}$}; \\ 
Re[\ov{ l_2^{\la} (\zeta_1^{\la}) } u_1^{\la} (\zeta_1^{\la})  ] = \gm_3^{\la} (\zeta_1^{\la}) & \text {on $b D^{\la}$}, 
\end{cases} 
\end{gather*}
where $D^{\la} \in \mc{C}^{k+1+\mu,j} $, $A_2^{\la}, B_2^{\la} \in \mc{C}^{k+\mu,j}(\ov{D^{\la}})$, $F_3^{\la} \in C^{k-1+\mu,j-1} (\ov{D^{\la}})$, $l_2^{\la} \in \mc{C}^{k+\mu,j}(\ov{D^{\la}})$, $\gm_3^{\la} \in \mc{C}^{k+\mu, j} (\ov{D^{\la}})$. The functions are given by  
\begin{gather} \label{A2B2} 
(A_2^{\la}  (\zeta_1^{\la}), B_2^\lambda(\zeta_1^\lambda)) := \left[ f_1^{\la} (\zeta_1^{\la}) \right]^{-1}\left( A_1^{\la} ( (\psi^{\la} )^{-1} (\zeta_1^{\la})), 
B_1^{\la} ( (\psi^{\la} )^{-1} (\zeta_1^{\la}))\right), 
\\ \nn
F_3^{\la}  (\zeta_1^{\la}):= \left[ f_1^{\la} (\zeta_1^{\la}) \right]^{-1} F_2^{\la} ( (\psi^{\la} )^{-1} (\zeta_1^{\la})), 
\\ \label{l2} 
(l_2^{\la}  (\zeta_1^{\la}),\gm_3^{\la}  (\zeta_1^{\la})):= \left[ f_1^{\la} (\zeta_1^{\la}) \right]^{-1}
\left( l_1^{\la} ( (\psi^{\la} )^{-1} (\zeta_1^{\la})), 
\gm_2^{\la} ( (\psi^{\la} )^{-1} (\zeta_1^{\la}))\right) . 
\end{gather}    
By \re{lad_ula_uc}, $u_1^{\la}$ satisfies the estimate   
\begin{gather*}
u_1^{\la}(\psi^{\la} (\zeta_r))  = a_r'(\la) + i b_r'( \la) ,  \quad  r = 1, \cdots, N_0; 
\\  
u_1^{\la} ( \psi^{\la} (\zeta_s')) =  \pa_{\la} \left\{  l (\zeta_s', \la) \left[ \gm (\zeta'_s, \la) + i c_s (\la)  \right] \right\}, \quad s = 1,  \cdots, N_1. 
\end{gather*} 

By the induction hypothesis, $u_1^{\la} \in \mc{C}^{k+\mu, j-1} (\ov{D^{\la}})$. Then $u^{\la}(\zeta) = u_1^{\la} \circ \psi^{\la}(\zeta) \in \mc{C}^{k+\mu, j-1} (\ov{\Om})$, and $u^{\la}$ is the solution to \re{ladeqn}. Furthermore, in view of \re{lad_vla_est} there exists some constant $C$ independent of $\la$ such that 
\begin{align} \label{lad_ula_est} 
&\| u^{\la} \|_{\mc{C}^{k+\mu, j-1} ( \ov{\Om}) }  
\leq C \| u_1^{\la} \|_{C^{k+\mu, j-1} (D^{\la})} 
\\ \nn 
&\qquad\leq C \left( \| F_3^{\la} \|_{\mc{C}^{k-1+\mu, j-1} (\ov{\Om})} 
+ \| \gm_3^{\la} \|_{\mc{C}^{k+\mu, j-1} (b \Om)} + \|(a,b,c) \|_{C^{j}}. 
\right)  
\\ \nn &\qquad\leq C \left( \| F^{\la} \|_{\mc{C}^{k+\mu,j} ( \ov{\Om^{\la}} )} + \| \gm^{\la} \|_{\mc{C}^{k+1+\mu, j} (b\Om^{\la})} + \|(a,b,c) \|_{C^{j}} \right).  
\end{align}
By the induction hypothesis applied for $j=0$, $v^{\la}$ satisfies 
\begin{equation} \label{lad_vla_j0est}
\| v^{\la}  \|_{\mc{C}^{k+1+\mu,0} (\Om)} \leq 
C \left( \| F^{\la} \|_{\mc{C}^{k+\mu, 0} ( \ov{\Om^{\la}} )} + \| \gm^{\la} \|_{\mc{C}^{k+1+\mu, 0} (b\Om^{\la})} + \|(a,b,c) \|_{C^{0}} \right).  
\end{equation}

The induction and thus the theorem will be proved if we show that $u^{\la_0}$ is indeed the $\la$ derivative of $v^{\la}$. Define $G^{\la} := D^{\la,\la_0} v - u^{\la_0}$. Then in view of \re{diffqt_eq} and \re{ladeqn}, $G^{\la}$ is the solution to the problem
\begin{gather*} 
\begin{cases}
\pa_{\ov{\zeta}} G^{\la} + e^{\la_0} \pa_{\zeta} G^{\la} + A_1^{\la_0} G^{\la} + B_1^{\la_0} \ov{G^{\la} } = E^{\la} & \text {in $\Om$;} \\ 
Re[\ov{ l_1^{\la_0}} G^{\la}  ] = \phi^{\la} & \text {on $b \Om$}, 
\end{cases} 
\end{gather*}
where 
$
E^{\la} (\zeta) := F_2^{\la,\la_0} (\zeta) - F_2^{\la_0} (\zeta),  \phi^{\la}(\zeta) := \gm^{\la, \la_0}_2(\zeta) - \gm^{\la_0}_2(\zeta). 
$
In addition, by \re{dq_uc} and \re{lad_ula_uc}, $G^{\la}$ satisfies the conditions:
\begin{equation} \label{Gla_uc}
\begin{gathered}  
G^{\la}  (\zeta_r)  =  (D^{\la, \la_0} a_r - \pa_{\la} a_r )+ i ( D^{\la, \la_0} b_r - \pa_{\la} b_r ); 
\\
G^{\la} (\zeta_s') = D^{\la, \la_0} \left\{ l (\zeta_s', \cdot) [ \gm (\zeta_s', \cdot) + i c_s(\cdot) ]  \right\} 
- \pa_{\la} |_{\la = \la_0} \left\{  l (\zeta_s', \la) \left[ \gm (\zeta'_s, \la) + i c_s (\la)  \right] \right\}. 
\end{gathered}
\end{equation}

Since $e^{\la}, A_1^{\la}, B_1^{\la}, F_1^{\la}, (f_1^{\la})^{-1} \in \mc{C}^{k+\mu,j} (\ov{\Om})$ and $k \geq j \geq 1$, then $\pa_{\la} e^{\la}$, $\pa_{\la} A_1^{\la}$, $ \pa_{\la} B_1^{\la}$, $\pa_{\la} F_1^{\la}$ are in  $ C^{\mu,0} (\ov{\Om})$.  
Since $\gm^{\la}, l^{\la} \in \mc{C}^{k+1+\mu,j} (b \Om)$, then  $ \pa_{\la} \gm^{\la}$, $ \pa_{\la} l^{\la} \in C^{1+\mu,0}(b \Om)$. By \rp{cinhd}, we see that $| \pa_{\la} \gm^{\la} (\la_1) - \pa_{\la} \gm^{\la} (\la_2) |_{\tau}$ converges to $0$ as $|\la_1- \la_2| \to 0$, for any $0 < \tau < 1+\mu$.  It follows from \re{F2lala0}, \re{gm2lala0}, \re{F2la} and \re{gm2la} that 
\eq{Ephicvg0} 
|E^{\la} |_0 \ra 0, \quad | \phi^{\la} |_{\tau} \ra 0, \quad \text{as $| \la - \la_0| \to 0$.}  
\eeq
As before we can do a change of coordinate $\psi^{\la_0}: \zeta \mapsto \zeta_1$ so that in the new coordinates, $G_1^{\la}(\zeta_1) := G^{\la} ( (\psi^{\la_0})^{-1} (\zeta_1) )$ is the solution to the problem 
\begin{gather*} 
\begin{cases} 
\pa_{\ov{\zeta}} G_1^{\la} + A_2^{\la_0} G_1^{\la} + B_2^{\la_0} \ov{G_1^{\la} } = E_1^{\la} & \text {in $\Om_1$;} \\ 
Re[\ov{ l_2^{\la_0}} G_1^{\la}  ] = \phi_1^{\la} & \text {on $b \Om_1$},  
\end{cases} 
\end{gather*}
where $A_2^{\la_0}, B_2^{\la_0}, l_2^{\la_0}$ are defined in the same way as expressions \re{A2B2} and \re{l2} with $\la$ replaced by $\la_0$, and  
\begin{gather*} 
E_1^{\la} (\zeta_1)= \left[ f_1^{\la_0} (\zeta_1) \right]^{-1} E^{\la} ( (\psi^{\la_0} )^{-1} (\zeta_1)), \quad \zeta_1 \in \Om_1; \\
\phi_1^{\la} (\zeta_1) = \left[ f_1^{\la_0} (\zeta_1) \right]^{-1} \phi^{\la} ( (\psi^{\la_0} )^{-1} (\zeta_1)), \quad \zeta_1 \in \Om_1.   
\end{gather*}

In addition, if we set $\wti{\zeta_r} = \psi^{\la_0} (\zeta_r) $ and $\wti{\zeta_s'} =  \psi^{\la_0}(\zeta_s') $, then by \re{Gla_uc} $G_1^{\la}$ satisfies 
\begin{gather*}
G_1^{\la}  ( \wti{\zeta_r })  =  \left\{  (D^{\la, \la_0} a_r - \pa_{\la} a_r )+ i ( D^{\la, \la_0} b_r - \pa_{\la} b_r )  \right\}
\circ ( \psi^{\la_0} )^{-1},  
\\
G^{\la} (\wti{\zeta_s'}) = D^{\la, \la_0} \left\{ l (\zeta_s', \cdot) [ \gm (\zeta_s', \cdot) + i c_s(\cdot) ]  \right\} 
- \pa_{\la} |_{\la = \la_0} \left\{  l (\zeta_s', \la) \left[ \gm (\zeta'_s, \la) + i c_s (\la)  \right] \right\} 
\circ (\psi^{\la_0})^{-1}. 
\end{gather*}

In view of estimate \re{pAmcsupest}, the solution $G_1^{\la}$ satisfies 
\[ 
| G_1^{\la} |_0 \leq C \left( \sum_{r=1}^{N} G_1^{\la} (\wti{\zeta_r})   
+ \sum_{s=1}^{N'}  G_1^{\la} (\wti{\zeta}'_s )  
+ | \phi_1^{\la}|_{\tau} + | E_1^{\la} |_{0} \right), \quad \tau >0.  
\]
As $\la \to \la_0$, the right-hand side converges to $0$.  
Hence $ | G_1^{\la}|_{C^0 (\ov{\Om_1})}\to 0$, and it follows that $|D^{\la, \la_0} v - u^{\la_0} |_{C^0 (\ov{\Om}) }  =| G^{\la} |_{C^0 (\ov{\Om})} 
= |G_1^{\la}|_{C^0(\ov{\Om_1})} 
\to 0$. 
Finally \re{lad_ula_est} and \re{lad_vla_j0est} together imply \re{lader_est}, so the proof is complete. 
\end{proof}

\section{Appendix}

\subsection{Estimates for the Cauchy-Green/ $\dbar$ operator}    
Let $\Om$ be a bounded domain in $\C$. Recall the operator $T_{\Om}$ given by \re{Topt}. 

\begin{prop}[{\cite[Theorem 1.13]{VA62}}] 
Let $\Om$ be a bounded domain. If $ f \in L^1 (\Om)$, then $T_\Om f$ belong to $L^p (\wti{\Om}) $, for any $1 \leq p < 2$ and for any bounded domain $ \wti{\Om}$ in $\C$. 
\end{prop}

\begin{prop}[{\cite[Theorem 1.14]{VA62}}] \label{dbar_opt_dist} 
If $f \in L^1 (\Om)$, then for any $\var \in C^{\infty}_{0} (\Om)$ (i.e. the class of compactly supported smooth functions in $\Om$), one has the following
\[ 
\iint_{\Om} Tf \, \DD{\var}{\ov{z}} d A(z) = - \iint_{\Om} f  \var \, dA(z). 
\] 
In other words, the equation $\pa_{\ov{z}} (Tf) = f$ holds in the sense of distribution on $\Om$. 
\end{prop}

\pr{dbarot} 
Let $\Om$ be a bounded domain. Suppose $f \in L^{p}(\Om)$, for $2 < p < \infty$. Then for $z,z_1,z_2\in \C$ the following estimates hold:
\begin{gather*} 
\left| T_{\Om} f (z) \right| \leq C(p) \left| f  \right|_{L^{p}(\Om)}, \quad  z \in \C; \\
\left| T_{\Om} f (z_1) - T_{\Om} f (z_2) \right| \leq C(p) \left| f  \right|_{L^{p}(\Om)} |z_1 - z_2| ^{\all_p}, \quad \all_p = \frac{p-2}{p}.
\end{gather*} 
\epr

\begin{prop}[{\cite[p.~34]{VA62}}] \label{dbarie}
Let $\Om \subset \C$ be a bounded domain. Suppose $\pa_{\ov z} w \in L^1(\Om)$. Then 
$
w - T_{\Om} (\pa_{\ov z} w) = \Phi(z), 
$
where $\Phi$ is some holomorphic function in $\Om$.
\end{prop} 

\begin{prop}[{\cite[Theorem 1.23]{VA62}}] \label{dbarubd}
Let $2 < p < \infty$. Suppose $f$ satisfies 
\[
M = \int_{\D} \left| f \left( \zeta^{-1} \right) \right|^{p} |\zeta|^{-2p} \, dA(\zeta) < \infty.
\]
Then for $z,z_1,z_2\in\C$, we have
$
|T_{\D^c} f(z) | \leq C(p) M  $ and $$
|T_{\D^c}f(z_1) - T_{\D^c} f(z_2) | \leq C(p) M |z_1 - z_2 |^{\all_p}, \quad \all_p = \frac{p-2}{p}. 
$$
\end{prop} 
Combining \rp{dbarie} and \rp{dbarubd} one gets
\pr{Pnhest1}  
Let $\D$ be the unit disk in the complex plane, and suppose $f \in L^p (\D)$ for $2 < p < \infty$. Let $P_n$ be the operator defined by formula \re{Pnf2}: 
\eq{Pnf2'}
P_n \var (z) = T_{\D} \var (z) + z^{2n+1} T_{\D^c}  \var_1 (z), 
\eeq  
where $\var_1$ is given by the second formula in \re{Pnf2}. Then $P_n \var \in C^{\all_p} (\ov{D})$ for $\all_p = \frac{p-2}{p}$, and 
$
|P_n \var |_{\all_p} \leq C(p) |f|_{L^p (\D)}.
$
\epr

\pr{Pnhest2} 
Let $\Om \in C^{k+1+\mu}$, where $k$ is a non-negative integer, and $0< \mu < 1$. Let $P_{n}$ be the operator given by formula \re{Pnf2'}. Then
$
| P_n \var |_{C^{k+1+\mu}(\ov{\D})} \leq C | \var |_{C^{k+\mu} (\ov{\D})}. 
$
\epr
\begin{proof} Let $\var_1(\zeta) = \ov{\var(1/ \ov{\zeta})} / (\zeta \ov{\zeta^{2}})$ and  
$
T_{\D^c} \var_1 (z) = - \frac{1}{\pi} \iint _{\C \sm \D} \frac{\var_1( \zeta) }{\zeta -z} \, dA(\zeta).
$
We have
\eq{Pnexp}
P_n \var (z) = T_{\D} \var (z) + z^{2n+1} T_{\D^{c}} \var_1 (z). 
\eeq       
By Theorem 1.32 \cite[p.~56]{VA62}, we have
\begin{align} \label{dbarder} 
\pa_{\ov z} T_{\D} \var (z) = \var (z), \quad  \pa_z T_{\D} \var (z) = \Pi_{\D} \var (z), 
\end{align} 
and $\Pi$ is bounded linear operator mapping $C^{k+\mu} (\ov{\D})$ to itself. It follows from \re{dbarder} that $T_{\D}$ is a bounded linear operator mapping $C^{k+\mu} (\ov{\D})$ to $C^{k+1+\mu} (\ov{\D})$. To estimate the second operator $T_{\D^{c}} \var_1$. We first extend $\var_1$ to a function in the class $C^{k+\mu}(\C)$, by extending $\var$ to a function in the class $C^{k+\mu}_0(\C)$. $\var_1$ vanishes in some neighborhood of $0$, and $\var_1 \in L^1(\C)$.  
Write
$
T_{\D^{c}} \var_1 = T_{\C} \var_1 (z) - T_{\D} \var_1 (z).  
$
By \cite[p.~63, Thm 1.34]{VA62}, $T_{\C} \var_1 \in C^{k+1+\mu}(\C)$. Since $T_{\D} \var_1 \in C^{k+1+\mu}(\ov{\D})$, we obtain from the above equation that $T_{\D^c} \var_1 \in C^{k+1+\mu}(\ov{\D})$. The proposition 
follows from 
\re{Pnexp}.   
\end{proof}

\pr{zdTop}      
Let $k$ be a non-negative integer and $0 < \beta < 1$. Let $\Om$ be a domain in $\C$ with $C^{k+1+\beta}$ boundary, and let $\Gm^{\la}: \Om \to \Om^{\la}$ be a $\mc{C}^{k+1+\beta, 0}$ embedding. Let $f^{\la}  \in \mc{C}^{k+\beta, 0}(\ov{\Om^{\la}})$. Then 
$\pa_{z^{\la}} T_{\Om^{\la}} f^{\la}  $ is in $\mc{C}^{k+\beta, 0}(\ov{\Om^{\la}})$, and there exists some constant $C$ independent of $\la$ so that  
$
| \pa_{z^{\la}} T_{\Om^{\la}} f^{\la} |_{k+ \beta} \leq C | f^{\la} |_{k+ \beta}.  
$
\epr 
The proof follows similarly as in \cite[p.~56-61]{VA62}.  

\subsection{Estimates for Cauchy type operators} 
\begin{lemma} \label{pvintest}  
Let $\Om^{\la}$ be a family of bounded domains in $\C$ in the class $\mc{C}^{1,0}$. That is there is a family of embedding $\Gm^{\la}: \Om \to \Om^{\la}$ such that $\Gm^{\la} \in \mc{C}^{1, 0} (\ov{\Om})$. For $0 < \beta < 1$, let $f^{\la}  \in \mc{C}^{\beta, 0} (b \Om)$, and let $g$ be defined by 
\[
g(\zeta, \la) := P.V. \int_{b \Om} \frac{f(t, \la) }{t^{\la} - \zeta^{\la}} \, dt^{\la} 
= \int_{b \Om} \frac{ f (t, \la) - f(\zeta, \la) }{t^{\la} - \zeta^{\la} } \, dt^{\la} + (\pi i ) f(\zeta, \la) . 
\]  
Then $g \in \mc{C}^{\beta, 0} (b \Om^{\la})$, and we have
$
\| g \|_{\beta, 0} \leq C_{1,0} \| f \|_{\beta, 0}.
$
\end{lemma}
\begin{proof} 
Let $\zeta_1, \zeta_2 \in b \Om$. Then we have 
\begin{align} \label{pvintzeta12} 
& |g ( \zeta_1,\la) - g (\zeta_2,\la)| 
\leq \int_{| s | < 2 \del} \frac{| f (t,\la) -  f (\zeta_1,\la)|}{| t^{\la}- \zeta_1^{\la}|} + \frac{|f (t,\la) - f (\zeta_2,\la)|}{| t^{\la}-\zeta_2^{\la}|} \, d t^{\la}
\\ \nn &\qquad + \left| \int_{|s  | \geq 2 \del} \frac{ \ov{f(t,\la)} - \ov{f(\zeta_1,\la)} }{t^{\la}-\zeta_1^{\la}} - \frac{ \ov{f(t,\la)} - \ov{f (\zeta_2,\la)} }{t^{\la}- \zeta_2^{\la}}\, d t^{\la}  \right|
+ \pi |f (\zeta_1, \la) - f(\zeta_2, \la)| 
\\ \nn &\quad \leq C_{1,0} \| l \|_{\beta,0} \, \del^{\beta}
+ \left| \left( \ov{f(\zeta_2,\la)} - \ov{f (\zeta_1,\la)}  \right) \int_{| s | \geq 2 \del}  \frac{d t^{\la}}{t^{\la}- \zeta_1^{\la}} \right|
\\ \nn & \qquad+ \int_{|s| \geq 2 \del}  \left( \ov{f(t,\la)} - \ov{ f( \zeta_2,\la)}  \right) \left( \frac{1}{t^{\la}- \zeta_1^{\la}} - \frac{1}{t^{\la}- \zeta_2^{\la}} \right) \, d t^{\la}.
\end{align}
The first term is bounded by $C_{1,0} \| l \|_{\beta,0} \, \del^{\beta}$. 
The second term is bounded by  
\begin{align*}  
C_{1,0} \| f  \|_{\beta,0} \, \del^{\beta} \left| log \frac{\Gm^{\la} (2 \del)- \Gm^{\la} (0) }{\Gm^{\la} (-2 \del) - \Gm^{\la} (0)} \right|.
\end{align*}
Since $\Gm \in C^{1,0}(\ov{\Om^{\la}})$, we have for positive constants $c_{1,0}$ and $C_{1,0}$, 
\[ 
c_{1,0}
= 
c_{1,0}  \frac{2 \del}{2 \del} 
\leq   
\left| \frac{\Gm^{\la} (2 \del)- \Gm^{\la} (0) }{\Gm^{\la} (-2 \del) - \Gm^{\la} (0)}  \right|
\leq  
C_{1,0} \frac{2 \del}{2 \del}
\leq 
C_{1,0}.  
\] 
Thus 
$
\left| log \frac{\Gm^{\la} (2 \del)- \Gm^{\la} (0) }{\Gm^{\la} (-2 \del) - \Gm^{\la} (0)}  \right|$  is bounded uniformly in $\la$.
Finally the last term in \re{pvintzeta12} is bounded by  
\[ 
C_{1,0} \| f \|_{\beta,0} \, \del^{\beta} \int_{| s  | \geq 2 \del} \frac{| t^{\la} - \zeta_2^{\la} |^{\beta}}{|t ^{\la}- \zeta_1^{\la}| |t^{\la} - \zeta_2^{\la}|} \, d t^{\la}. 
\]
Since
$|t - \zeta_2| \geq | t - \zeta_1| - | \zeta_1 - \zeta_2 | \geq | t - \zeta_1| - c_0 |t - \zeta_1 | \geq (1- c_0) |t - \zeta_1 |$, for some $c_0 > 0$, the above expression is bounded by 
\begin{align*}
C_{1,0} \| f \|_{\beta,0} \, \del^{\beta} \int_{|s| \geq 2 \del} |t ^{\la} - \zeta_1^{\la} |^{\beta-2} \, dt^{\la}    
&\leq C_{1,0}' \| f \|_{\beta,0} \, \del^{\beta}   \int_{|s | \geq 2 \del} |s| ^{\beta-2} \, ds  
\leq C_{1,0}'' \| f \|_{\beta,0}  \, \del^{\beta}.
\end{align*}
Putting together the results we get
$
| v (\zeta_1, \la) - v(\zeta_2, \la) | \leq C_{1,0} \| f^{\la}  \|_{\beta,0}  \| \zeta_1 - \zeta_2|^{\beta}. \qedhere
$
\end{proof}

\begin{lemma} \label{redopt}  
Let $\Om^{\la}$ be a family of bounded domains in $\C$ in the class $\mc{C}^{1,0}$. That is there is a bounded domain $\Om$ of $C^1$ boundary, and a family of embedding $\Gm^{\la}: \Om \to \Om^{\la}$ such that $\Gm^{\la} \in \mc{C}^{1, 0} (\ov{\Om})$. Let $\mc{N}$ be the operator defined by 
\begin{equation*}
\mc{N}f(\zeta, \la) = \int_{b \Om} \frac{1}{t^{\la} - \zeta^{\la} } \left[ v (\zeta,t, \la) - v (t,t,\la) \right] f (t, \la) \, dt^{\la}. 
\end{equation*}
Suppose $v \in C^{\beta} (\ov{\Om})$, for $0 < \beta <1$, and $| v (\cdot, t, \la) |_{\beta}$ is bounded by some constant $A_{\beta}$ uniformly in $\la$ and $t \in b \Om$. Then $\mc{N}$ maps the space $C^{0,0} (b \Om^{\la})$ to the space $C^{\beta', 0}(b \Om^{\la})$, for any $0 < \beta' < \beta$, and 
\eq{redoptest}
\| \mc{N} f \|_{\beta',0} \leq C_{\beta'} C_{1,0} \| f \|_{0,0}. 
\eeq
\end{lemma}
\nid \tit{Proof. } 
First we estimate the sup norm:  
\begin{align*}
|\mc{N}f (\zeta, \la)| &\leq C_{1,0} A_{\beta} |f (\cdot, \la)|_{C^{0} (b \Om) }  \int_{b \Om} \frac{1}{| t- \zeta |^{1-\beta}} \, dt 
\leq C_{1,0}  A_{\beta} |f(\cdot, \la)|_{C^{0}(b \Om)},  
\end{align*}
where the constants are independent of $\la$.  

Next, we estimate the difference of $\mc{N} f $ at two points $\zeta_1, \zeta_2 \in b \Om$. Let
$n (\zeta, t, \la) = v (\zeta,t, \la) - v (t,t,\la)$. Then $|n(\zeta, t, \la)| \leq A_{\beta} |\zeta - t|^{\beta} $. 
Writing $\del = | \zeta - \zeta'|$, we have 
\begin{align*}
&|\mc{N} f(\zeta_1, \la ) - \mc{N} f(\zeta_2, \la )| 
= \left| \int_{b \Om} \left[ \frac{n(\zeta_1, t, \la )}{|\zeta_1^{\la} - t^{\la} | } 
- \frac{n(\zeta_2, t, \la )}{|\zeta_2^{\la} - t^{\la}| } \right] f(t, \la)  \, dt^{\la} \right|   \\
&\quad\leq A_{\beta} |f(\cdot, \la) |_{C^0(b \Om) } \left\{ \int_{|\zeta_1^{\la} - t^{\la} | <2\del} \frac{d t^{ \la}}{|\zeta_1^{\la} - t^{\la}|^{1- \beta} } + \int_{|\zeta_1^{\la} - t^{\la}| < 2 \del} \frac{d t^{\la}}{|\zeta_2^{\la} - t^{\la}|^{1- \beta} } \right\} 
\\ &\qquad +  |f(\cdot, \la) |_{C^0(b \Om) } \int_{|\zeta_1 -t| \geq 2\del} \left| \frac{n(\zeta_1, t, \la)}{|\zeta_1^{\la} -t^{\la}|} 
- \frac{n (\zeta_2, t, \la)}{|\zeta_2^{\la} -t^{\la}| } \right|   \, dt^{\la}.
\end{align*}
The first two integrals are bounded by $C_{1,0 } A_{\beta}' | f(\cdot, \la) |_{C^0 (b \Om) } |\zeta_1- \zeta_2|^{\beta} $. The last integral is bounded by  
\begin{align*} 
& \int_{|\zeta_1 -t| >2\del}\left\{
\frac{|n(\zeta_1,t, \la) - n(\zeta_2,t,\la)|}{|\zeta_1^{\la} - t^{\la}|} 
+  
| n(\zeta_2, t, \la)|  
\left[ \frac{1}{|\zeta_1^{\la}- t^{\la}|  } - \frac{1}{|\zeta_2^{\la} - t^{\la}|} \right]\right\}
\, dt^{\la} \\
& \quad \leq C_{1,0} \int_{| \zeta_1 - t | \geq 2 \del} \frac{| v(\zeta_1, t, \la) - v (\zeta_2, t, \la) | }{ |\zeta_1 - t|} \, dt  + C_{1,0} \int_{|\zeta_1  -t| >2\del}  \frac{|\zeta_1 - \zeta_2| |\zeta_2 - t|^{\beta} }{|\zeta_1 -t| | \zeta_2 - t| } \, dt \\
&\quad \leq C_{1,0} |\zeta_1 - \zeta_2 |^{\beta} (\log |\zeta_1 - \zeta_2| + 1 ) 
+ C_{1,0} |\zeta_1 - \zeta_2| \int_{|\zeta_1 - t| > 2 \del }  \frac{1}{|\zeta_1 -t |^{2- \beta} } \, dt \\ 
&\quad \leq C_{1,0}' |\zeta_1 - \zeta_2|^{\beta'}
\end{align*}
for any $0 < \beta' < \beta$. Combining the estimates we get \re{redoptest}. 
\hfill \qed

\pr{Cint}
Let $0< \beta <1$ and $k, l \geq 0$ be integers. Let $\Om$ be a bounded domain in $\C$ with $b \Om \in C^1$.  For $f \in L^1(b \Om)$, define the Cauchy transform $\mf{C}f$ on $\C \sm b\Om$ by 
$
\mf{C} f (z):= \frac{1}{2 \pi i } \int_{b \Om} \frac{f(\zeta) }{\zeta -z} \, d\zeta. 
$
 Suppose $f \in C^{\beta} (b \Om)$. Then $f$ extends to a function $\mf{C}^{+} f \in C^{\beta} (\ov{\Om})$. Moreover on $b \Om$
\[
\mf{C}^+ f (z) = f(z) + \frac{1}{2 \pi i } \int_{b \Om} \frac{f(z) - f(\zeta)}{z - \zeta} \, d\zeta.   
\]
(ii) Suppose $b \Om \in C^{k+1+\beta}$, and $f \in C^{l+\beta}(b \Om)$ with $k+1 \geq l$. Then $\mf{C}^+ f \in C^{l+ \beta}(\ov{\Om})$. 
\epr
The proof is standard and the reader may refer to \cite{B-G14}.

\pr{gCiHne}
Let $A,B \in L_{p,2}(\C)$.  Let $w(z)$ be the generalized Cauchy integral with density $\var$, defined by 
\begin{equation*}
w(z) = \frac{1}{2 \pi i} \int_{b \Om} G_{1}(z,t) \var(t) \, dt - G_{2}(z,t) \ov{\var(t)} \, \ov{dt}, 
\end{equation*}
where $G_{1}$ and $G_{2}$ are the fundamental kernels of the class $\mc{L}_{p,2}(A,B,\C)$. Suppose $\var(t) \in C^{\mu} (b \Om)$, then $w(z) \in C^{\nu}(\ov{\Om})$, where $\nu = \text{min} (\all_p, \mu)$, $\all_p = \frac{p}{p-2}$.
\epr
\begin{proof}
We write 
$
G_{1}(z,t) = \frac{ \Si_{i=1,2} [e^{\om_{i}(z,t)} - e^{\om_{i}(t,t)}]}{t-z} + \frac{2}{t-z}$ and
$ 
G_{2}(z,t) = \frac{e^{\om_{1}(z,t)} - e^{\om_{2}(z,t)} }{t-z} . 
$	 
By \re{omiHne} and \re{omde} we have
$
|e^{\om_{i}(z,t)} - e^{\om_{i}(t,t)}| \leq C |z-t|^{\all_p} $ for $ i = 1,2$ and $  
|e^{\om_{1}(z,t)} - e^{\om_{2}(z,t)}| \leq C |z-t|^{\all_p}. $
Then
\begin{align*} 
w(z) &= \frac{1}{2 \pi i} \int_{\pa \Om} \frac{\var(t) \, dt}{t-z} 
\\ &\quad+ \int_{\pa \Om}  \frac{\Si_{i=1,2} [e^{\om_{i}(z,t)} - e^{\om_{i}(t,t)}]}{t-z} \var(t) \, dt - \int_{\pa \Om}  \frac{e^{\om_{1}(z,t)} - e^{\om_{2}(z,t)} }{t-z} \ov{\var(t)} \, \ov{dt} \\ 
&=  \frac{1}{2 \pi i} \int_{\pa \Om} \frac{\var(t) \, dt}{t-z} + \int_{\pa \Om} \frac{F(z,t)}{|t-z|^{1-\all_p}} \var(t) \, dt + \int_{\pa \Om} \frac{G(z,t) }{|t-z|^{1-\all_p}} \ov{\var(t)} \, \ov{dt},    
\end{align*}  
where $|F(z,t)| \leq C|z-t|$ and $|G(z,t)| \leq C|z-t|$, for any $z \in \ov{\Om}$ and $t \in \pa \Om$. 
The result then follows from \rp{Cint}.  
\end{proof}   

\pr{gCijft} 
Let $L$ be an arc of the class $C^{1}$ and let $0< \beta < 1$. Define $w$ to be the generalized Cauchy integral with density $\var \in C^{\beta}(L)$.  
\[ 
w(z) := \frac{1}{2 \pi i} \int_{L} G_{1} (z,t) \var(t) \, dt - G_{2} (z,t) \ov{\var(t)} \, \ov{dt}, \quad z \notin L. 
\]
Then $w$ extends to a continuous function from either side of the arc. Furthermore, let us denote by $w^{+}$ and $w^{-}$ the limiting values of $w$ from the region to the left and right of $L$ when going in the positive direction. Then the following formulae hold:
\begin{gather} \label{gCijf1} 
w^{+} (\zeta) = \yh \var(\zeta)  + w(\zeta), \quad
w^{-}(\zeta) = - \yh \var(\zeta) + w(\zeta),  \quad \zeta \in L,
\end{gather} 
where
\[ 
w(\zeta) = \frac{1}{2 \pi i} \,  P.V. \int_{L} G_{1}(\zeta, t)  \var(t) \, dt - G_{2} (\zeta, t) \ov{\var(t)} \, \ov{dt}, \quad \quad  \zeta \in L.  
\] 
The first integral on the right-hand side is understood to be the Cauchy  principal value, and the second integral converging in the ordinary sense. 
\epr
\begin{proof} 
See formula (14.2)-(14.3) from \cite[p.~188]{VA62}. 
The proof uses Plemelj's formula and estimate \re{fke}. 
\end{proof}
\pr{gCith}   
Let $\Om$ be a domain in $\C$ with $C^1$ and compact boundary. Suppose $A,B \in L^p  (\Om) $ and $A \equiv B \equiv 0$ in $\C \sm \Om$. Define  
\[ 
w(z) := \frac{1}{2 \pi i} \int_{b \Om} G_{1}(z,t,\Om) \var(t) \, dt - G_{2}(z,t, \Om) \ov{\var(t)} \, \ov{dt}, 
\] 
where $\var \in C^{\beta} (b \Om)$ for $0<\beta<1$, and $G_i$ are defined by \re{Gidef}.
Then $w$ satisfies the following integral equation (in fact is the unique solution in $C^0 (\ov{\Om})$) 
\eq{gCiie}
w(z) + T_{\Om} (Aw + B\ov{w}) = \Phi(z) , 
\eeq
where
\eq{gCiePhid}
\Phi(z) := \frac{1}{2 \pi i} \int_{b \Om} \frac{\var(t) \, dt}{t-z}. 
\eeq
\epr
\begin{proof}
We will follow some arguments in \cite[p.~183-188]{VA62}. By the well-known jump formula for Cauchy integral, we have 
\[ 
\var(\zeta) = \Phi^{+}(\zeta) - \Phi^{-}(\zeta),  \quad \zeta \in b \Om. 
\]
In addition, $\Phi \in C^{\beta}(\ov{\Om})$.  
By \re{Gieq}, 
$
\pa_{\ov{z}} G_{i} (z, \zeta,  \Om)= 0$   
if $z \in \C \sm \Om$ and $\zeta \in \Om$. 
By \re{fkr}, 
\eq{Giholc}
\pa_{\ov{z}} G_{i} (\zeta, z,  \Om)= 0, \quad \quad \text{if $z \in \C \sm \Om$, \quad $\zeta \in \Om$}; 
\eeq
\[ 
w(z) = \frac{1}{2 \pi i} \int_{b \Om} G_{1}(z, t, \Om) (\Phi^{+} - \Phi^{-}) (t) \, dt 
- G_{2}(z, t, \Om) \ov{(\Phi^{+} - \Phi^{-}) (t)} \,  \ov{d{t}} . 
\]
By \re{Giholc} and Cauchy's theorem, 
$
  \int_{b \Om} G_{1}(z, t, \Om) \Phi^{-} (t) \, dt 
- G_{2}(z, t, \Om) \ov{\Phi^{-}(t)}  \,  \ov{dt}  = 0. 
$ 
Using this and applying Green's identity on the domain $\Om_{\ve} = \{ \zeta \in \Om: |\zeta -z| > \ve \}$,  we obtain 
\begin{align*} 
&w(z) =  \frac{1}{2 \pi i} \int_{b \Om} G_{1}(z, t, \Om) \Phi^{+} (t) \, dt 
-  G_{2}(z, t, \Om) \ov{\Phi^{+}(t)}  \,  \ov{dt}  \\
&= \lim_{\ve \to 0} \left\{ \frac{1}{ \pi} \iint_{\Om_{\ve}} \left( \pa_{\ov{\zeta}} G_{1}(z, \zeta, \Om) \right) \Phi(\zeta) \, d A (\zeta)  + \frac{1}{ \pi} \iint_{\Om_{\ve}} \pa_{ \ov{\zeta} } \left( G_{2}(z, \zeta, \Om) \right ) \ov{\Phi(\zeta)} \, d A(\zeta) \right\} \\
&\quad + 
\lim_{\ve \to 0} \left\{ \frac{1}{2 \pi i} \int_{|t -z| = \ve} G_{1} (z, t, \Om) \Phi(t) -  \Om_{2}(z, t, \Om) \ov{\Phi(t)} \, \ov{d t} \right\}, 
\end{align*} 
where we used that $\Phi$ is holomorphic.

Taking the limit as $\ve \to 0$, and taking into account estimate \re{fke}, we have
\eq{rstf}  
w (z) = \Phi(z) + \iint_{\Om} R_{1}(z, \zeta, \Om) \Phi(\zeta)  d A(\zeta) 
+ \iint_{\Om}  R_{2}(z, \zeta, \Om) \ov{\Phi(\zeta)} \, d A(\zeta) , 
\eeq
where
$ 
R_{1}(z, \zeta, \Om) = \frac{1}{\pi} \pa_{\ov{\zeta}} G_{1}(z, \zeta, \Om)$ and $
R_{2}(z, \zeta, \Om) = \frac{1}{\pi} \pa_{\ov{\zeta}} G_{2}(z, \zeta, \Om). 
$

We now show that expression \re{rstf} represents the unique solution to the integral equation:
\eq{gCiie2}
w(z) - \frac{1}{\pi} \iint_{\Om} \frac{A(\zeta) w(\zeta) + B(\zeta) \ov{w(\zeta)}}{\zeta -z } \, d A(\zeta) = \Phi(z), \quad z \in \Om . 
\eeq
By Vekua \cite[p.~156]{VA62}, there exists a unique solution $w_0$ in the class $L^{q} (\ov{\Om})$, $q = \frac{p}{p-1}$. Our goal is to show that $w$ defined by \re{rstf} is equal to $w_0$. 

Taking $\pa_{\ov{z}}$ on both sides we get $\mc{L}_{A,B} w = 0$ in $\Om$.
Write \re{gCiie2} as
\eq{wiePhi}
w_0 - P_{\Om} w_0 = \Phi, \quad \quad P_{\Om} w_0 := -T_{\Om}(A w_0+ B\ov{w_0})  = T_{\Om} (\pa_{\ov z}  w_0) .  
\eeq 
By \cite[p.~50]{VA62}, there exists an integer $n$ so that $P^{n} w_0$ is in the class $C^{\tau}$ for some $\tau>0$. Since $P(\Phi) \in C^{\frac{p}{p-2}}(\ov{\Om})$ (\rp{Pnhest1}) and
\[ 
w_0 = P^{n} w_0 + P^{n-1} (\Phi) + P^{n-2} (\Phi) + \cdots + \Phi,
\]
we have that $w_0 \in C^{0}(\ov{\Om})$. Since $A,B \in L^{p} (\Om)$, we have $\pa_{\ov{z}} w_0 \in L^p (\Om)$. Hence by the remark in \cite[p.41]{VA62} we have 
\eq{wC-Gf}
w_0 - P_{\Om} w_0 \equiv w_0 - T_{\Om} (\pa_{\ov z}  w_0)  = \frac{1}{2 \pi i} \int_{b \Om} \frac{w_0(t) \, dt}{t -z}. 
\eeq
Now the left-hand sides of \re{wiePhi} and \re{wC-Gf} are the same, so the right-hand side are the same as well and we get
\[ 
\Phi(z) = \frac{1}{2 \pi i} \int_{b \Om} \frac{\var(t) \, dt}{t-z} = \frac{1}{2 \pi i} \int_{b \Om} \frac{w_0(t) \, dt}{t -z}, \quad \quad z \in \Om. 
\]
It follows then by the jump formula for Cauchy integral that $w_0 (\zeta) = \Phi^{+}(\zeta) - \Phi^{-}( \zeta)$, for $ \zeta \in b \Om$. 
By \rp{nhgaerfe}, $w_0$ can be represented as 
\begin{align*}
w_0 (z)  &= \frac{1}{2 \pi i} \int_{\pa \Om} G_{1}(z, t, \Om) w_0 (t) \, dt - G_{2}(z, t, \Om) \ov{w_0 (t)} \, \ov{dt} \\
&= \frac{1}{2 \pi i} \int_{\pa \Om} G_{1}(z, t, \Om) (\Phi^{+} - \Phi^{-}) (t) \, dt  
-  G_{2}(z, t, \Om) \ov{(\Phi^{+} - \Phi^{-}) (t)} \,  \ov{d{t}} = w(z) .  \qedhere
\end{align*}
\end{proof} 

\subsection{Theory of singular integral operators on curves}  
\ 
\\ 

In the following we let $L$ be a piece-wise, non-intersecting $C^1$ curve on the plane.   
\begin{defn}
We call an operator $\mc{N}$ a \emph{Fredholm} operator if it takes the form 
$
\mc{F} \var (t_0) = A(t_0) \var (t_0) + \mc{N} (\var) (t_0),
$
where $A$ is a function nowhere vanishing on  $L$ and $\mc{N}$ is a compact operator on $C^0 (L)$.
\end{defn}

\begin{defn} 
We call an operator defined by 
\[ 
\mc{K} \var (t_0) := A (t_0)  \var(t_0) + \frac{1}{\pi i} P.V. \int_{L} \frac{K(t_0, t) \var(t) \, dt}{t - t_0}, \quad t_0 \in L 
\] 
a \emph{singular integral operator (with Cauchy type kernel)}.  
\end{defn}
The functions $A$ and $K$ are assumed to be H\"older continuous in their respective arguments. We can write $\mc{K} \var$ as
\begin{align*}
\mc{K} \var (t_0) 
&= A (t_0)  \var(t_0) +  \frac{B(t_0) }{\pi i} \,  P.V. \int_L \frac{ \var (t) dt}{t-t_0} 
+ \frac{1}{\pi i} \int_{L} k(t_0, t) \, dt, 
\end{align*} 
where 
$ 
B(t_0) = K(t_0, t_0)$ and $ k(t_0, t) =   \frac{K(t_0, t) \var(t) - K(t_0, t_0) \var(t_0) }{t - t_0}. 
$
The operator $\mc{K}^0$ defined by 
\[
\mc{K}^0  \var (t_0) 
=A (t_0)  \var(t_0) +  \frac{B(t_0) }{\pi i} \,  P.V. \int_L \frac{ \var (t) dt}{t-t_0}  
\]
will be called the \emph{dominant part} of the operator $\mc{K}$. 

We define the \emph{adjoint operator} of $\mc{K}$ by
\[
\mc{K}' (\psi) (t_0) = A(t_0)  \var(t_0) - \frac{1}{\pi i} P.V. \int_{L} \frac{K(t, t_0) \psi(t) \, dt}{t - t_0} . 
\]    
The dominant part of $\mc{K'}$ is  
$
(\mc{K}')^0 (\psi) (t_0) = A(t_0) \psi(t_0) - \frac{ B(t_0) }{\pi i } P.V. \int_L \frac{\psi(t) }{t - t_0 } \, dt. 
$
For any two functions $\var$ and $\psi$ which are H\"older continuous on  $L$, one has
$
\int_{L} ( \mc{K} \var) \psi \, dt  = \int_L \var (\mc{K}' \psi ) \, dt.  
$

\begin{lemma}[Composition of singular integral operators]
Let $\mc{K}_1$ and $\mc{K}_2$ be the two singular integral operators defined by
\begin{gather}
\mc{K}_j (t_0) =  A_i(t_0 ) \var (t_0 ) + \frac{1}{\pi i} \int_{L} \frac{K_j (t_0, t) \var (t)}{t - t_0} \, dt, \quad  t_0 \in L, \quad j=1,2. 
\end{gather}  
Then the composition or the product of $\mc{K}_1$ and $\mc{K}_2$ is  
\begin{equation} \label{soprod}
\begin{aligned}
\mc{K}_1  \mc{K}_2 (t_0 )   
&= \left[ A_1 (t_0) A_2 (t_0) + B_1(t_0) B_2(t_0) \right] \var (t_0) \\ 
&\quad + \frac{1}{\pi i } \int_L \frac{A_1 (t_0) K_2 (t_0, t) + K_1  (t_0, t) A_2 (t) }{t- t_0} \var (t) \, dt
\\  
&\quad + \frac{1}{(\pi i)^2 } \int_L \left[ \int_L \frac{K_1 (t_0, t_1) K_2 (t_1, t) }{(t_1 - t_0) (t - t_1) } \, dt_1 \right] \var (t) \, dt. 
\end{aligned}  
\end{equation} 
\end{lemma}
\begin{proof} See formula (45.9) in \cite[p.~119]{MI92}.
\end{proof} 

Denote the double integral in expression \re{soprod} by $\mc{F}_{12} (\var) $. We claim that $\mc{F}_{12}$ is a compact integral operator. Indeed, we have
\begin{align*}
I (t, t_0) = \int_L \frac{K_1 (t_0, t_1) K_2 (t_1, t) }{(t_1 - t_0) (t - t_1) }  \, d t_1
&= \frac{1}{t-t_0} \left[ v(t_0, t)  - v(t, t) \right] 
\end{align*}
where we set $F(t_1, t, t_0) = K_1 (t_0, t_1) K_2 (t_1, t) $, and 
\begin{gather*}
v(t_0, t) = P.V. \int_L \frac{F(t_1, t, t_0) }{t_1 - t_0}  \, dt_1 
=  \int_L \frac{F(t_1, t, t_0) - F(t_0, t, t_0) }{t_1 - t_0} \, dt_1 + (\pi i) F (t_0, t, t_0),
\\ 
v(t, t) = P.V. \int_L \frac{F(t_1, t, t_0) }{t_1 - t}  \, dt_1 
=  \int_L \frac{F(t_1, t, t_0) - F(t, t, t_0) }{t_1 - t} \, dt_1 + (\pi i) F (t, t, t_0). 
\end{gather*}
Then the double integral becomes 
$
\int_L \frac{1}{t -t_0} \left[ v(t_0, t)  - v(t, t) \right]  \var (t) \, dt. 
$
By the proof of \rl{pvintest}, we see that $v$ is H\"older continuous in the first argument. It follows then by \rl{redopt} that $\mc{F}$ is a compact operator.  

Denote the single integral in the above product $\mc{K}_1 \mc{K}_2$ by $\mc{K}_{12} (\var)$. Then $\mc{K}_{12}$ is a singular integral operator, and we can write it as
\begin{align*} 
&\mc{K}_{12} \var (t_0) 
= \frac{A_1  (t_0)B_2 (t_0) + A_2 (t_0) B_1 (t_0)  }{\pi i} \, P.V. \int_{L} \frac{\var(t)}{t - t_0} \, dt \\ 
& + \int_L \frac{ \left[ A_1 (t_0) K_2 (t_0, t) + K_1  (t_0, t) A_2 (t)  \right] - \left[ A_1 (t_0) K_2 (t_0, t_0) + K_1  (t_0, t_0) A_2 (t_0) \right] }{t- t_0} \var (t)\, dt. 
\end{align*}
The second integral above is clearly a compact integral operator. Then in view of \re{soprod} $\mc{K}_1 \mc{K}_2$ will be a Fredholm operator if one has
\eq{redopteqn} 
0 = A_1  (t_0)B_2 (t_0) + A_2 (t_0) B_1 (t_0) ,  \quad  t_0 \in L. 
\eeq 

\df{redoptdf}
Let $\mc{K}_2$ be a singular integral operator defined on $C^{\beta} (L)$, for some $0 < \beta < 1$, and let $\mc{K}_1$ be another singular integral operator such that the composition $\mc{K}_1 \mc{K}_2$ is a Fredholm operator. Then we call $\mc{K}_1$ a \emph{reducing integral operator} of $\mc{K}_2$.  
\edf

Given an operator $\mc{K}_2$ with the corresponding $A_2$ and $B_2$, one can choose another operator $\mc {K}_1$ whose corresponding $A_1$ and $B_1$ satisfy condition \re{redopteqn}. In particular one can choose $\mc{K}_1$ to be the dominant part of the adjoint of $\mc{K}_2$, since in this case $A_1(t_0) = A_2(t_0)$ and $B_1(t_0) = - B_2 (t_0)$. 

\subsection{Some other results}  
\ 

The following result is stated and used in Vekua's proof \cite{VA62} and we shall provide the details.   
\pr{RHpind}
Let $w$ be a holomorphic function on the unit disk $\D$, continuous up to the boundary circle $S^1$, and satisfying the following boundary condition: 
\[
Re ( z^{-n} w) = \gm (z), \quad n > 0, \quad \text{on $S^1$}, 
\]
where we assume $\gm \in C^{\beta} (S^1)$, for $0 < \beta < 1$. Then 
$ 
w = z^{n} \mc{S} \gm + \sum_{k=0}^{2n} c_{k} z^{k}$ for $ z \in \D, 
$
where $\mc{S} \gm $ is the Schwarz integral 
\[
\mc{S} \gm (z) := \frac{1}{2 \pi i} \int_{S^1} \gm(t) \frac{t+ z}{t-z} \frac{dt}{t}, \quad z \in \D. 
\]
\epr
\begin{proof} 
Since $w$ is holomorphic, one can write  $w (z) = \sum_{k=0}^{\infty} c_{k} z^{k}$. Then
$
z^{-n} w (z) = \sum_{k=0}^{\infty} c_{k} z^{k-n} 
= \sum_{j=-n}^{-1} c_{n+j} z^{j} + \sum_{j=0}^{\infty} c_{n+j} z^{j}. 
$ 
Thus on $S^1$, one has
\begin{align*} 
Re ( \sum_{j=0}^{\infty} c_{n+j} z^{j} ) &= Re (z^{-n} w) - Re ( \sum_{j=-n}^{-1} c_{n+j} z^{j} ) 
= \gm (z) - Re ( \sum_{j=-n}^{-1} c_{n+j} z^{j} ). 
\end{align*}
It is easy to see that $\sum_{j=0}^{\infty} c_{n+j} z^{j}$ is holomorphic in $\D$, and in addition by the assumption the series is continuous up to $S^1$. Hence it can be represented as the Schwarz integral of the real part of its boundary value: 
\[
\sum_{j=0}^{\infty} c_{n+j} z^{j} = \mc{S} (\gm) (z) - \mc{S} (  Re  \sum_{j=-n}^{-1} c_{n+j} \zeta^{j}   ) (z), \quad z \in \D. 
\]
Multiplying both sides by $z^{n}$ we get
\eq{n+jsum}
\sum_{j=0}^{\infty} c_{n+j} z^{n+j} = z^{n} \mc{S} (\gm)(z) - z^{n} \mc{S} (  Re  \sum_{j=-n}^{-1} c_{n+j} \zeta^{j}  ) (z), \quad z \in \D. 
\eeq
For $\zeta \in S^1$, one has 
$
2 Re  \sum_{j=-n}^{-1} c_{n+j} \zeta ^{j}   =   \sum_{j=-n}^{-1} c_{n+j} \zeta^{j} 
+  \sum_{j=-n}^{-1} \ov{c_{n+j}} \zeta ^{-j} 
=   \sum_{j=-n}^{-1} c_{n+j} \zeta^{j} +   \sum_{j=1}^{n} \ov{c_{n-j}} \zeta ^{j}. 
$ 
Let 
\eq{varnsum}
\var (z) = \mc{S} ( Re  \sum_{j=-n}^{-1} c_{n+j} \zeta^{j} )  (z), \quad z \in \D, 
\eeq 
in other words, $\var$ is the holomorphic function in $\D$ whose boundary value has real part equal to 
$  Re \left( \sum_{j=-n}^{-1} c_{n+j} \zeta^{j} \right) $ . 
Then we claim that $ \var$ takes the form $\Si_{k=0}^n  a_{k} z^{k} $.
Indeed, since $\var (z)$ is holomorphic in $ \D$, it has power series expansion $\var = \sum_{k=0}^{\infty} a_{k} z^{k}$, with
\begin{align*} 
a_{k} &= \frac{1}{2 \pi i} \int_{|t|= 1} \frac{\var(t)}{t^{k+1}} 
= \frac{1}{2 \pi i} \int_{|t|= r} \frac{\yh \sum_{j=-n}^{-1} c_{n+j} t^{j} \,dt + \yh \sum_{j=1}^{n} \ov{c_{n-j}} t^{j} }{t^{k+1}} \, dt, \quad k \geq 0. 
\end{align*}
For $k > n$, we have $a_{k} = 0$, so the claim holds. Substituting \re{varnsum} into \re{n+jsum} we obtain 
\begin{align*} 
\sum_{j=0}^{\infty} c_{n+j} z^{n+j} &= z^{n} \mc{S} (\gm) (z) - z^{n} \mc{S} (  Re  \sum_{j=-n}^{-1} c_{n+j} \zeta^{j} ) (z), \quad z \in \D \\
&= z^{n} \mc{S} (\gm) (z) - z^{n} \sum_{k=0}^{n} a_k z^{k} 
= z^{n} \mc{S}(\gm)(z) - \sum_{k=n}^{2n} a_k' z^{k}. 
\end{align*}
Finally, $w(z) = \sum_{j=0}^{n-1} c_{j} z^{j} +  \sum_{j=n}^{\infty} c_{j} z^{j}$ implies
$
w(z) =  z^{n} \mc{S} (\gm) (z) + \sum_{n=0}^{2n} c_{k}' z^{k}, 
$
for some $c_k'$. 
\end{proof}

\pr{L1HnC0}
Let $\Om$ be a bounded domain in $\C$, and $0 < \beta < 1$. Suppose $\{ f_{n} \}$ is a sequence of functions in $C^{\beta} (\ov{\Om})$ satisfying $|f_{n}|_{C^{\beta}(\ov{\Om})} \leq C$ and $|f_n|_{L^1(\Om) } \to 0$. Then $|f_{n}|_{C^{0}(\ov{\Om})} \to 0$.  
\epr 
\nid \tit{Proof.}
Suppose $|f|_{C^{0} (\ov{\Om})} > \ve_1$. Then there exists $x_0 \in \Om$ such that $f (x_0) > \ve_1$. Since $| f(x) - f(x_0) | \leq |f|_{\beta} | x - x_0|^{\beta}$, we have $| f(x)  - f(x_0) | < \frac{\ve_1}{2}$ whenever $| x - x_0 | < \del = \left( \frac{\ve_1}{2 |f|_{\beta} } \right)^{\frac{1}{\beta}}$. Let $B(x_0, \del)$ be the ball of radius $\del$ centered at $x_0$, then for all $x \in B(x_0, \del)$, $| f(x)|  > \frac{\ve_1}{2}$. Consequently
$
| f |_{L^{1} (\Om)} \geq \int_{B(x_0, \del)} | f (x) | \, dx \geq \frac{\ve_1}{2} \pi \del^{2} = C_{\beta} | f |_{\beta}^{- \frac{2}{\beta}}  \ve_1^{\frac{\beta+2}{\beta}}, 
$
where $C_{\beta} = \pi 2^{- \frac{\beta+2}{\beta}} $. Note that if $x_0$ happen to be on the boundary $b \Om$, then the above estimates hold with an extra constant. 

Given $\ve_0$, we set $\ve_1 = C_{\beta}' | f |_{\beta}^{\frac{2}{2 + \beta}} \ve_0^{\frac{\beta}{2+ \beta}}$, where $C'_{\beta} = C_{\beta}^{- \frac{\beta}{2+ \mu}}$. Then by above we see that whenever $| f |_{L^{1} (\Om)}  < \ve_0$, we have $| f |_{C^{0} (\ov{\Om} )}  \leq \ve_1$. By taking $\ve_0 \to 0$ we get the conclusion. 
\hfill \qed

\begin{lemma} \label{argfnt} 
Let $\Gm^{\la}: \Om \to  \Om^{\la}$ be a $\mc{C}^{1+\mu,0}$ embedding, where $0 < \mu < 1$. For $t, t_0 \in b \Om^{\la}$, let $t^{\la} := \Gm^{\la} (t)$ and $t_0^{\la} :=  \Gm^{\la}  (t_0)$. Define 
 $ 
K (t, t_0, \la) := \pa_{\tau_{t_0^{\la}} } \arg (t_0^{\la} - t^{\la} ).  
$ 
\\
(i)We have
$
| K (t, t_0, \la) | \leq C_{1+ \mu, 0} | t_0 - t |^{\mu-1}.  
$
\\
(ii) For any $t_0, t_1, t_2 \in b \Om $ with $| t_1 - t_0 | \geq 2| t_1 - t_2|$, we have
\[
|K(t_1, t_0, \la) - K (t_2, t_0, \la) | \leq C_{1+ \mu, 0} \, \frac{|t_1 - t_2| }{| t_0 - t_1 |^{2- \mu} }. 
\]
\end{lemma}
\nid \tit{Proof.}
(i) We parametrize $b \Om$ by $\rho(s)$ such that $ds$ is the arc-length element agreeing with the standard orientation of $b \Om$. Then $\rho^{\la} (s) = \Gm^{\la} \circ \rho$ gives a parametrization of $b \Om^{\la}$. Let $t= \rho(s)$ and $t_0 = \rho(s_0)$. Then $t^{\la} = \rho^{\la} (s)$ and $t_0^{\la} = \rho^{\la}  (s_0)$, and we can write $K$ as
$ 
K(t, t_0, \la) 
= \pa_{\tau_{t_0^{\la}} } \arg (t_0^{\la} - t^{\la} ) 
= \left| \pa_{s_0} \arg [ \rho^{\la}(s_0) - \rho^{\la}(s)] \right|. 
$ 
Denote by $\nu^{\la}(s_0)$ the outer unit normal vector to $b \Om^{\la}$ at $t_0^{\la} = \rho^{\la}(s_0)$. Then we have
\eq{argder} 
\pa_{s_0}  \arg [\rho^{\la}(s_0) - \rho^{\la}(s)] = \frac{\nu^{\la}(s_0) \cdot (\rho^{\la}(s_0) - \rho^{\la}(s)) }{| \rho^{\la}(s_0) - \rho^{\la}(s) |^{2} }. 
\eeq
Since 
$ 
\nu^{\la}(s_0) \cdot (\rho^{\la}(s_0 ) - \rho^{\la}(s)) = \nu^{\la}(s_0) \cdot \int_{s}^{s_0} (\rho^{\la})'(r) - (\rho^{\la}) '(s_0 ) \, dr 
$ 
and $\rho^{\la} \in \mc{C}^{1+\mu,0}$, the above expression is bounded by $C_{1+\mu,0} |s_0 - s|^{1+ \mu}$. Since $\Gm^{\la} \in \mc{C}^{1+\mu,0}$, we have $c_{1,0} |s_0- s| \leq | \rho^{\la}(s_0) - \rho^{\la}(s) | \leq C_{1,0}  |s_0 - s| $. Hence from \re{argder} we obtain 
$
|K(t,t_0,\la) | \leq C_{1+\mu,0} \frac{|s_0 - s|^{1+ \mu}}{|s_0 - s|^2 } \leq C_{1+\mu, 0} 
| s_0 - s|^{\mu- 1} \leq C_{1+\mu, 0}' | t_0 - t|^{\mu-1}.  
$

(ii) Using formula \re{argder} we have
\begin{align*}
| K(t_1, t_0, \la) &- K(t_2, t_0, \la) | 
= \left| \pa_{s_0} \arg [\rho^{\la} (s_0) - \rho^{\la} (s_1) ] -  \pa_{s_0} \arg [\rho^{\la} (s_0) - \rho^{\la} (s_2) ] \right|   \\
&= \left| \frac{\nu^{\la}(s_0 ) \cdot (\rho^{\la}(s_0) - \rho^{\la}(s_1)) }{| \rho^{\la}(s_0) - \rho^{\la}(s_1) |^{2} } - \frac{\nu^{\la}(s_0) \cdot (\rho^{\la}(s_0) - \rho^{\la}(s_2) }{| \rho^{\la}(s_0) - \rho^{\la}(s_2) |^{2} } \right|. 
\end{align*}
The above expression is bounded by the sum $K_1 + K_2$, where
\begin{gather*}
K_1 = \frac{1}{|\rho^{\la}(s_0) - \rho^{\la}(s_1)|^2 } \left| \nu^{\la} (s_0) \cdot [\rho^{\la} (s_2) - \rho^{\la} (s_1) ]  \right|,  
\\ 
K_2 = \left| \nu^{\la}(s_0) \cdot [\rho^{\la}(s_0) - \rho^{\la}(s_2)] \right|
\left| \frac{1}{| \rho^{\la}(s_0) - \rho^{\la}(s_1) |^{2} } - \frac{1}{| \rho^{\la}(s_0) - \rho^{\la}(s_2) |^{2}}  \right|. 
\end{gather*}
We estimate each $K_i$. For $K_1$ we have 
\begin{align*}
K_1 &= \frac{1}{|\rho^{\la}(s_0) - \rho^{\la}(s_1)|^2 } 
\left| \nu^{ \la} (s_0 ) \cdot \int_{s_1}^{s_2} (\rho^{\la} )'(r) - ( \rho^{\la} )' (s_0) \, dr \right| \\
&\leq C_{1+\mu,0} \frac{1}{|s_0 - s_1|^2 } | s_1 - s_2|^{1+\mu} 
= C_{1+\mu,0}  \frac{|s_1- s_2| }{|s_0 - s_1|^{2- \mu} } 
\leq C_{1+\mu,0}'  \frac{|t_1- t_2| }{|t_0 - t_1|^{2- \mu} }, 
\end{align*}
where we used $|t_1 - t_0 | \geq 2| t_1 - t_2| $.  
For $K_2$, we have 
\begin{align*}
&K_2 = \left| \nu^{\la}(s_0) \cdot [\rho^{\la}(s_0) - \rho^{\la}(s_2)] \right| 
\frac{\left| |\rho^{\la}(s_0) - \rho^{\la} (s_2) |^2 - |\rho^{\la}(s_0) - \rho^{\la} (s_1) |^2  \right| }{| \rho^{\la}(s_0) - \rho^{\la}(s_1) |^{2} | \rho^{\la}(s_0) - \rho^{\la}(s_2) |^{2} } \\
&\leq  \left| \nu^{\la}(s_0) \cdot  [ \rho^{\la}(s_0) - \rho^{\la}(s_2) ]\right| \frac{| \rho^{\la} (s_1) - \rho^{\la} (s_2) | \left( |\rho^{\la} (s_0) - \rho^{\la} (s_2) | + |\rho^{\la} (s_0) - \rho^{\la} (s_1) |  \right)   }{| \rho^{\la}(s_0) - \rho^{\la}(s_1) |^{2} | \rho^{\la}(s_0) - \rho^{\la}(s_2) |^{2} }, 
\end{align*}
where the constant $C_0$ is uniform in $\la$.  Since 
\begin{align*} 
\left| \nu^{\la}(s_0) \cdot [\rho^{\la} (s_0)  - \rho ^{\la} (s_2) ] \right| 
&= \left| \nu^{\la}(s_0) \cdot \int_{s_2}^{s_0} (\rho^{\la})' (r) - (\rho^{\la}) ' (s_0)  \, dr  \right| 
\leq C_{1+\mu,0} |s_0 - s_2|^{1+\mu}, 
\end{align*}
we have 
$ 
K_2 \leq C_{1+ \mu,0}  |s_0 - s_2|^{1+ \mu}  |s_1 - s_2 | \left( \frac{1}{|s_0 - s_1|^2 |s_0 - s_2| } 
+ \frac{1}{|s_0 - s_1| |s_0 - s_2|^2 }  \right). 
$ 
Since $|t_1 - t_0 | \geq 2 |t_1- t_2|$, we have $|t_2 - t_0| \geq |t_1 - t_0 | - | t_1 - t_2| \geq \yh  | t_1 - t_0| $, and  
$ | t_2  - t_0| \leq  |t_1 - t_0 |  + | t _1+ t_2| \leq \frac{3}{2} |  t_1 - t_0|$. Then $c_0 |s_1 - s_0| \leq |s _2 - s_0 | \leq C_0' | s_1 - s_0 | $. 
Consequently, 
$
K_2 \leq C_{1+\mu,0}' |s_1- s_2| \frac{|s_0 - s_1|^{1+ \mu} }{|s_0 - s_1|^3}  
= C_{1+\mu,0}'  \frac{|s_1- s_2|}{|s_0 - s_1|^{2- \mu}}.  $
Combining the bounds for $K_1$ and $K_2$ we get the desired estimate. 
\hfill \qed



\newcommand{\doi}[1]{\href{http://dx.doi.org/#1}{#1}}
\newcommand{\arxiv}[1]{\href{https://arxiv.org/pdf/#1}{arXiv:#1}}

\def\MR#1{\relax\ifhmode\unskip\spacefactor3000 \space\fi%
\href{http://www.ams.org/mathscinet-getitem?mr=#1}{MR#1}}

\bibliographystyle{plain} 

\bibliography{master} 

\end{document}